\documentclass[reqno]{amsart}
\usepackage{a4wide, textcomp}
\usepackage{microtype, fbb}
\usepackage{amsmath, amssymb, amsthm, mathtools, amsfonts}
\usepackage[utf8]{inputenc}
\usepackage[all]{xy}
\usepackage{graphicx, multirow}
\usepackage{enumerate, xcolor, color, enumitem}
\usepackage{appendix, pdfsync}
\usepackage{float}
\usepackage{setspace}
\usepackage{xpatch}
\usepackage{subcaption, caption}
\usepackage{pst-node}
\usepackage{wrapfig}
\usepackage{tikz}
\usepackage{tikz-cd}
\usetikzlibrary{positioning, calc}
\usetikzlibrary{shapes, arrows}
\usepackage[colorlinks,allcolors = blue]{hyperref} 
\usepackage{cleveref}

\DeclareCaptionLabelFormat{onlyalph}{(\Alph{subfigure})}
\makeatletter
\renewcommand{\p@subfigure}{\thefigure(}

\makeatother

\crefformat{subfigure}{figure~#2#1#3}
\Crefformat{subfigure}{Figure~#2#1#3}

\newtheorem{theorem}{Theorem}[section]
\newtheorem{proposition}[theorem]{Proposition}
\newtheorem{conjecture}[theorem]{Conjecture}
\newtheorem{corollary}[theorem]{Corollary}
\newtheorem{lemma}[theorem]{Lemma}
\newtheorem{claim}[theorem]{Claim}

\newtheorem{remark}[theorem]{Remark}

\theoremstyle{definition}
\newtheorem{definition}[theorem]{Definition}

\AddToHook{env/proposition/begin}{\crefalias{theorem}{proposition}}
\AddToHook{env/proposition/end}{\crefalias{proposition}{theorem}}
\AddToHook{env/conjecture/begin}{\crefalias{theorem}{conjecture}}
\AddToHook{env/conjecture/end}{\crefalias{conjecture}{theorem}}
\AddToHook{env/corollary/begin}{\crefalias{theorem}{corollary}}
\AddToHook{env/corollary/end}{\crefalias{corollary}{theorem}}
\AddToHook{env/lemma/begin}{\crefalias{theorem}{lemma}}
\AddToHook{env/lemma/end}{\crefalias{lemma}{theorem}}
\AddToHook{env/claim/begin}{\crefalias{theorem}{claim}}
\AddToHook{env/claim/end}{\crefalias{claim}{theorem}}
\AddToHook{env/question/begin}{\crefalias{theorem}{question}}
\AddToHook{env/question/end}{\crefalias{question}{theorem}}
\AddToHook{env/remark/begin}{\crefalias{theorem}{remark}}
\AddToHook{env/remark/end}{\crefalias{remark}{theorem}}
\AddToHook{env/definition/begin}{\crefalias{theorem}{definition}}
\AddToHook{env/definition/end}{\crefalias{definition}{theorem}}
\AddToHook{env/example/begin}{\crefalias{theorem}{example}}
\AddToHook{env/example/end}{\crefalias{example}{theorem}}

\def\Ind{\text{Ind}}
\newcommand{\IG}[1]{\mathcal{I}_{#1}}
\newcommand{\aG}[1]{\alpha({#1})}
\newcommand{\un}[2]{#1\cup\{#2\}}
\newcommand{\di}[2]{#1\setminus\{#2\}}
\newcommand{\homo}[2]{${#1}:\mathbb{Z}^{#2}$}

\begin{document}
	
	\title{Homology of matching complexes of $3\times n$ grid graphs}
	
	\author{Pratiksha Chauhan}
	\address{School of Mathematical and Statistical Sciences, IIT Mandi, India}
	\email{d22037@students.iitmandi.ac.in\\d24226@students.iitmandi.ac.in\\ samir@iitmandi.ac.in}
	
	\author{Anchal Sharma}
	
	\author{Samir Shukla}
	
	
	\begin{abstract}
		For a finite simple graph $G$, the matching complex $M(G)$ is the simplicial complex whose vertex set is the edge set of $G$ and whose simplices are all the matchings in $G$. The topology of the matching complex of the $m\times n$ grid graph $G_{m\times n}$ is known only for $m = 1,2$, in which cases it is homotopy equivalent to a wedge of spheres. In this article, we study the matching complex $M(G_{3 \times n})$. We prove that for $n\ge2$, its reduced homology vanishes in dimensions $i \leq n-2$ and in top dimension, while $\tilde{H}_{n-1}(M(G_{3\times n}))\neq 0$. We also show that $M(G_{3 \times n})$ is simply connected for $n \geq 3$. Consequently, the topological connectivity of $M(G_{3\times n})$ is $n-2$.
	\end{abstract}
	
	\maketitle
	
	\section{Introduction}
	
	Let $G$ be a finite simple graph, with vertex set $V(G)$ and edge set $E(G)$. A matching in $G$ is a collection of pairwise non-adjacent edges of $G$. The matching complex of $G$, denoted by $M(G)$, is a simplicial complex whose vertex set is $E(G)$ and whose simplices are all the matchings in $G$. Matching complexes have been studied since the work of Garst \cite{garst1979cohen} and Bouc \cite{bouc1992homologie}, in connection with coset complexes, Brown complexes, and Quillen complexes. In particular, Bouc studied the matching complexes of complete graphs $M(K_n)$, and Garst studied the matching complexes of complete bipartite graphs $M(K_{m,n})$.
	
	The matching complexes of complete bipartite graphs admits a natural interpretation as \emph{chessboard complexes}. The chessboard complex $\Delta_{m,n}$ is a simplicial complex whose simplices are the non-taking rook placements on an $m\times n$ chessboard, \textit{i.e.}, placements in which no two rooks share a common row or column. Identifying the squares $(i, j)$ of the $m \times n$ chessboard with the edges of $K_{m,n}$, we obtain that $\Delta_{m,n}$ is the same as $M(K_{m,n})$. Chessboard complexes arise in the study of configuration spaces \cite{bjorner1987some} and halving hyperplanes \cite{Zivaljevic_colored_Tverberg}. 
	The topological connectivity of these complexes was studied by Bj\"{o}rner et al.\ \cite{Bjorner1994}, who proved that $\Delta_{m,n}$ is $(\nu-2)$-connected, where $\nu = \min\{m,n,\lfloor\frac{m+n+1}{3}\rfloor\}$ (here, $\lfloor x\rfloor$ denotes the greatest integer less than or equal to $x$). They also established the corresponding bound for $M(K_n)$ and proved that $M(K_n)$ is $(\nu'-2)$-connected, where $\nu' = \lfloor\frac{n+1}{3}\rfloor$. Recall that a topological space $X$ is said to be $k$-{\it connected} if every map from an $m$-dimensional sphere $S^m \to X$ can be extended to a map from the $(m+1)$-dimensional ball $\mathbb{B}^{m+1} \to X$ for $m = 0, 1, \ldots, k$.
	
	Beyond complete and complete bipartite graphs, matching complexes have been investigated for numerous families of graphs, including paths and cycles \cite{Kozlov_directed_trees}, trees \cite{Milutinovic_trees}, forests \cite{Marietti_forests}, polygonal tilings \cite{bayer2023general,Matsushita_polygonal}, outerplanar graphs \cite{bayer_outerplanar}, complete hypergraphs \cite{Bate_representation}, and categorical product of path graphs \cite{Categorical_product}. 
	Matching complexes also appear in several other areas of mathematics, including representation theory \cite{Bate_representation,bouc1992homologie,Shareshian}, discrete geometry \cite{Athanasiadis,Bjorner1994,Jojic_Tverberg}, and commutative algebra \cite{Jiang_chessboard,Karaguezian_Matching,Reiner_Matching}. 
	For further background, we refer the reader to the survey by Wachs \cite{Wachs_survey} and to \cite[Chapter 11]{JonssonBook}. 
	
	In this article, we focus on the matching complex of the $3 \times n$ grid graph, the Cartesian product of the path graphs $P_3$ and $P_n$.
	\begin{definition}
		The \textit{Cartesian product} of graphs $G$ and $H$, denoted by $G \Box H$, is the graph where $V(G \Box H) = V(G) \times V(H)$ and $E(G \Box H) = \{((g,h),(g',h')) : g = g' \text{ and } (h,h') \in E(H), \text{ or } h = h' \text{ and }(g,g')\in E(G)\}$. 
	\end{definition}
	
	For $n\ge1$, let $P_n$ denote the path graph on $n$ vertices. The $m\times n$ grid graph $G_{m\times n}$ is defined as $G_{m\times n}: = P_m\Box P_n$, the Cartesian product of two path graphs. Matching complexes of grid graphs have been studied by several authors. The case $G_{1\times n}\cong P_n$ was addressed by Kozlov \cite{Kozlov_directed_trees}, who showed that $M(G_{1\times n})$ is contractible or homotopy equivalent to a sphere. 
	In unpublished work, Jonsson \cite{Jonsson_unpublished} studied the homotopical depth and topological connectivity of $M(G_{m\times n})$ for certain cases. Moreover, he remarked that ``it is probably very hard to determine the homotopy type of'' matching complexes of grid graphs. Braun and Hough \cite{braun2016matching} later examined the homology of $M(G_{2\times n})$, and they mentioned that ``the topology of the matching complex for the $2\times n$ grid graph is mysterious''. 
	Matsushita \cite{Matsushita_grid} subsequently proved that $M(G_{2\times n})$ is homotopy equivalent to a wedge of spheres. In addition to matching complexes, several other graph complexes associated with grid graphs have been studied, including independence complexes \cite{adamaszek2012hardsquarescylindersrevisited,Matsushita_grid_independence_1,Matsushita_grid_independence_2}, cut complexes and total cut complexes \cite{Bayer2024Cutcomplex,Bayer2024TotalCutcomplex,bayer2025topology,chandrakar2024topology}, and robust clique complexes \cite{filakovsky_robust_clique}.
	
	Despite significant interest in the matching complexes of grid graphs, the topology of $M(G_{m\times n})$ remains largely unresolved in general. In this article, we study the matching complex of the $3\times n$ grid graph $G_{3\times n}$. A proposed determination of the homotopy type of $M(G_{3\times n})$ was given in \cite{Previous_3Xn}, but a gap was subsequently identified in the proof, leaving the problem open. Here, our main contribution is a sharp topological connectivity bound for $M(G_{3\times n})$. Recall that the topological connectivity of $X$, denoted by $conn(X)$, is the largest integer $k$ such that $X$ is $k$-connected, {\it i.e.}, all the homotopy groups $\pi_i(X) = 0$ for $0 \leq i \leq k$ and $\pi_{k+1}(X) \neq 0$. We set the connectivity of a contractible space to be $\infty$. The main results of this article are as follows. 
	\begin{theorem}\label{theorem:intro_homology}
		Let $n\ge 2$. Then 
		\begin{enumerate}[label = (\roman*)] 
			\item \label{theorem:intro_homology_i} $\tilde{H}_i(M(G_{3\times n})) = 0$ for $i \leq n-2$ and $i = \dim(M(G_{3\times n})) $.
			\item \label{theorem:intro_homology_ii} $\tilde{H}_{n-1}(M(G_{3\times n}))\ne0$.
		\end{enumerate} 
	\end{theorem}
	
	\begin{theorem}\label{theorem:intro_simply_connected}
		For $n\ge3$, $M(G_{3\times n})$ is simply connected.
	\end{theorem}
	
	By the Hurewicz theorem (see \cite[Theorem 4.32]{hatcher2005algebraic}), \Cref{theorem:intro_homology} and \Cref{theorem:intro_simply_connected} together imply the following. 
	\begin{corollary}
		For $n \geq 3$, conn$(M(G_{3\times n})) = n-2$.
	\end{corollary}
	
	Our approach to proving these results is based on a well-known correspondence between matching complexes and independence complexes via line graphs. Recall that an independent set in a graph $G$ is a subset $I \subseteq V(G)$ such that the induced subgraph $G[I]$ contains no edges. The independence complex of $G$, denoted by $\Ind(G)$, is the simplicial complex whose vertex set is $V(G)$ and whose simplices are the independent subsets of $G$.
	The line graph of $G$, denoted by $L(G)$, is the graph with vertex set $V(L(G)) = E(G)$, where two distinct vertices are adjacent if and only if the corresponding edges of $G$ share a common endpoint. 
	One can observe that $M(G) \cong \Ind(L(G))$. 
	
	Let $\Gamma_n: = L(G_{3\times n})$. Then $M(G_{3 \times n}) \cong \Ind(\Gamma_n)$. It is therefore sufficient to study the independence complex $\Ind(\Gamma_n)$. 
	To prove \Cref{theorem:intro_homology}, we establish the corresponding homological results for $\Ind(\Gamma_n)$. Our proof uses two different tools. We first establish that $\tilde{H}_i(\Ind(\Gamma_n)) = 0$ for $i \leq n-2$ and that $\tilde{H}_{n-1}(\Ind(\Gamma_n))\ne0$, by induction on $n$ together with the Mayer-Vietoris sequence. 
	We then prove that $\tilde{H}_{\dim(\Ind(\Gamma_n))}(\Ind(\Gamma_n)) = 0$ (the vanishing of top-dimensional homology) using discrete Morse theory by constructing an acyclic matching with no critical cell in the top dimension. Moreover, to prove \Cref{theorem:intro_simply_connected}, we establish that $\Ind(\Gamma_n)$ is simply connected by induction on $n$, showing that every loop in $\Ind(\Gamma_n)$ is null-homotopic.
	
	This article is organized as follows: In \Cref{section:preliminaries}, we present the necessary preliminaries on graph theory and simplicial complexes. The main results are established in \Cref{section:main}, which is divided into two subsections. In \Cref{section:Lower_dimensions_homology}, we examine the homology in lower dimensions and prove that $\tilde{H}_i(M(G_{3\times n})) = 0$ for $i \leq n-2$ and $\tilde{H}_{n-1}(M(G_{3\times n}))\ne0$. 
	In \Cref{section:top_dimension}, we prove the vanishing of homology in the top dimension, and that $M(G_{3 \times n})$ is simply connected. Finally, in \Cref{section:conclusion}, we discuss a possible approach towards proving the vanishing of homology in dimensions $\geq $ $n$, and make a conjecture about the homotopy type of $M(G_{3 \times n})$. 
	
	\section{Preliminaries}\label{section:preliminaries}
	
	In this section, we introduce the notation, definitions, and results that will be used throughout the article.
	
	\subsection{Graph}
	
	A \textit{graph} $G$ is a pair $(V(G), E(G))$, where $V(G)$ is its vertex set and $E(G) \subseteq \binom{V(G)}{2}$ is the edge set. For any $u,v\in V(G)$, we say that $u$ and $v$ are adjacent if $(u,v) \in E(G)$. 
	We write $u\sim v$ for adjacency and $u\nsim v$ for non-adjacency.
	A {\it subgraph} $H$ of $G$ is a graph with $V(H) \subseteq V(G)$ and $E(H) \subseteq E(G)$. For a subset $U \subseteq V(G)$, the \textit{induced subgraph} $G[U]$ is the subgraph with $V(G[U]) = U$ and $E(G[U]) = \{(a, b) \in E(G) \ | \ a, b \in U\}$. The \textit{open neighbourhood} $N_G(U)$ and the {\it closed neighbourhood} $N_G[U]$ of $U$ in $G$ are defined as $N_G(U): = \{u\in V(G)\ |\ u\sim v\text{ in $G$ for some }v\in U\}$ and $N_G[U]: = N_G(U)\cup U$.
	If $U = \{v\}$, we write $N_G(v)$ and $N_G[v]$ in place of $N_G(\{v\})$ and $N_G[\{v\}]$, respectively. 
	\begin{definition}
		Let $G$ and $H$ be simple graphs (\textit{i.e.}, undirected graphs with no loops or multiple edges).
		An \textit{isomorphism} from $G$ to $H$ is a bijection $f:V(G)\rightarrow V(H)$ such that, for every $u,v\in V(G),$ $(u,v)\in E(G)$ if and only if $(f(u),f(v))\in E(H).$ If such an isomorphism exists, then $G$ and $H$ are said to be \textit{isomorphic}, and it is denoted by $G\cong H.$
	\end{definition}
	
	For $n\geq 1,$ the {\it path graph} $P_n,$ is the graph with vertex set $V(P_n) = \{1,2,3,\cdots,n\},$ and edge set $E(P_n) = \{(i,i+1) \ | \ 1\leq i\leq n-1\}$. For $n\geq 3,$ the {\it cycle graph} $C_n,$ is a graph with vertex set $V(C_n) = \{1,2,3,\cdots,n\},$ and edge set $E(C_n) = \{(i,i+1) \ | \ 1\le i\le n-1\}\cup \{(1,n)\}.$
	
	For more details on the graph terminology used in this article, we refer the reader to \cite{bondy1976graph} and \cite{west}. 
	
	\subsection{Simplicial complex}
	
	A {\it finite abstract simplicial complex} $\Delta$ is a collection of finite sets such that if $\tau \in\Delta$ and $\sigma \subseteq\tau$, then $\sigma\in\Delta$. The elements of $\Delta$ are called {\it simplices} of $\Delta$. If $\sigma \subseteq\tau$, we say that $\sigma$ is a {\it face} of $\tau$. The \textit{dimension of a simplex} $\sigma$ is equal to $|\sigma| - 1$. The \textit{dimension of an abstract simplicial complex} is the maximum of the dimensions of its simplices. 
	If a simplex has dimension $d$, it is said to be $d$-{\it dimensional}. The $0$-dimensional simplices are called \textit{vertices} of $\Delta$, and the set of vertices is denoted by $V(\Delta)$. A \textit{subcomplex} $\Delta'$ of $\Delta$ is a simplicial complex such that $\sigma\in\Delta'$ implies $\sigma\in\Delta$.
	
	Throughout the article, we consider a simplicial complex as a topological space, namely, its {\it geometric realization} (see \cite{Kozlov2008} for details).
	\begin{definition}
		Let $\Delta_1$ and $\Delta_2$ be two abstract simplicial complexes with disjoint vertex sets. The \textit{join} of $\Delta_1$ and $\Delta_2$ is the abstract simplicial complex $\Delta_1 * \Delta_2$, whose vertex set is $V (\Delta_1) \cup V (\Delta_2)$ and whose set of simplices is given by
		$$\Delta_1 * \Delta_2 = \{\sigma\subseteq V (\Delta_1) \cup V (\Delta_2) \ | \ \sigma\cap V (\Delta_1) \in \Delta_1\text{ and }\sigma\cap V (\Delta_2) \in \Delta_2\}.$$
	\end{definition} 
	
	For a simplex $\tau$ of $\Delta$, we consider the following constructions.
	\begin{itemize}
		\item The \textit{link} of $\tau$ in $\Delta$ is the simplicial complex defined by 
		$$lk(\tau,\Delta) : = \{\sigma\in \Delta \ | \ \sigma\cap \tau = \emptyset, \text{ and } \sigma \cup \tau \in \Delta\}.$$
		
		\item The \textit{deletion} of $\tau$ in $\Delta$ is the simplicial complex defined by 
		$$del(\tau,\Delta) : = \{\sigma \in \Delta \ | \ \tau \not\subseteq\sigma\}.$$
		
		\item The \textit{star} of $\tau$ in $\Delta$ is the simplicial complex defined by 
		$$st(\tau,\Delta) : = \{\sigma\in \Delta \ | \ \sigma \cup \tau \in \Delta\}.$$
	\end{itemize}
	
	The following lemma expresses the homotopy type of a simplicial complex in terms of the link and deletion of a vertex.
	\begin{lemma}[{\cite[Example 0.14]{hatcher2005algebraic}}]\label{theorem:lk_del} 
		Let $\Delta$ be a simplicial complex and $v$ be a vertex of $\Delta$. If $lk(v,\Delta)$ is contractible in $del(v,\Delta)$, then 
		$$\Delta\simeq del(v,\Delta)\vee \Sigma(lk(v,\Delta)).$$
	\end{lemma}
	
	We refer the reader to \cite{hatcher2005algebraic,munkres_book} for further details on the terminology used in algebraic topology.
	
	Recall that the matching complex of a graph $G$ is the same as the independence complex of its line graph $ L(G)$, \textit{i.e.}, $M(G)=\Ind(L(G))$. We therefore present some results related to the independence complex $\Ind(G)$.
	
	Before stating these results, we fix some notation. For $S\subseteq V(G)$, let $G-S : = G[V(G)\setminus S]$. For $u,v\in V(G)$, if $(u,v)\notin E(G)$, then let $G\cup\{(u,v)\}$ denote the graph with vertex set $V(G)$ and edge set $E(G)\cup\{(u,v)\}$; if $(u,v)\in E(G)$, then let $G-\{(u,v)\}$ denote the graph with vertex set $V(G)$ and edge set $E(G)\setminus\{(u,v)\}$.
	
	We now describe the link and deletion of a vertex in $\Ind(G)$. Let $v\in V(G)$. Then
	\begin{equation}\label{equation:lk_del_ind_complex}
		lk(v,\Ind(G))\simeq \Ind(G-N_G[v])\text{ and }del(v,\Ind(G))\simeq \Ind(G-\{v\}).
	\end{equation}
	
	\Cref{theorem:lk_del} and \eqref{equation:lk_del_ind_complex} yield the following corollary for independence complexes.
	\begin{corollary}[{\cite[Theorem 3.3]{Adamaszek_power_of_cyles}}]\label{corollary:N[u]subsetN[v]_lk(v)_del(v)} 
		Let $G$ be a graph and $u,v\in V(G), u\neq v$ such that $N_{G}[u]\subseteq N_{G}[v]$. Then
		$$\Ind(G)\simeq \Ind(G-\{v\})\vee \Sigma(\Ind(G-N_G[v])).$$
	\end{corollary}
	
	We will also use the following results to establish the homotopy equivalences of independence complexes.
	\begin{lemma}[{\cite[Lemma 3.2]{engstrom2009complexes}}]\label{theorem:N(u)subsetN(v)}
		Let $G$ be a graph and $u,v\in V(G), u\neq v$ such that $N_{G}(u)\subseteq N_{G}(v)$. Then
		$$\Ind(G)\simeq \Ind(G-\{v\}).$$
	\end{lemma}
	
	\begin{lemma}[{\cite[Theorem 3.7]{engstrom2009complexes}}]\label{theorem:simplicial_vertex}
		Let $v$ be a simplicial vertex of a graph $G$ and $N_{G}(v) = \{w_1,\cdots, w_k\}$. Then
		$$\Ind(G)\simeq \bigvee^{k}_{i = 1}\Sigma(\Ind(G-N_{G}[w_i])).$$ 
	\end{lemma}
	
	\begin{lemma}[{\cite[Proposition 3.4]{Adamaszek_power_of_cyles}}]\label{lemma:edge_contractible}
		Let $G$ be a graph and $\{u,v\}$ be a $1$-simplex in $\Ind(G)$. If $\Ind(G-N_G[\{u,v\}])$ is contractible, then
		$$\Ind(G)\simeq \Ind(G'),$$
		where $V(G') = V(G)$ and $E(G') = E(G)\cup \{(u,v)\}.$
	\end{lemma}
	
	The following result relates the independence complexes of a disjoint union of graphs to the independence complexes of its components.
	\begin{lemma}\label{lemma:ind_disjoint_union}
		Let $G_1\sqcup G_2$ denote the disjoint union of two graphs $G_1$ and $G_2$. Then
		$$\Ind(G_1\sqcup G_2)\simeq \Ind(G_1) *\Ind(G_2).$$
	\end{lemma}
	
	As a consequence of \Cref{lemma:ind_disjoint_union}, we obtain the following corollary.
	\begin{corollary}\label{corollary:isolated_vertex}
		Let $G\sqcup \{v\}$ denote the disjoint union of a graph $G$ with an isolated vertex $v$. Then $\Ind(G\sqcup \{v\})$ is contractible.
	\end{corollary}
	
	Before stating the next result, we recall the notion of a path. Let $G$ be a graph. For $u,v\in V(G)$, a {\it path} from $u$ to $v$ is a sequence of distinct vertices $u = v_0,v_1, \ldots, v_n = v$ such that $v_i \sim v_{i+1}$ for all $0 \le i\le n-1$.
	\begin{lemma}[{\cite[Theorem 11]{csorba2009subdivision}}]\label{theorem:replace_edge}
		Let $G$ be a graph and $e = (u,v)$ an edge. If $G'$ is obtained from $G$ by replacing $e$ with a path $uabcv$ with three new vertices, then $\Ind(G')\simeq \Sigma(\Ind(G)).$
	\end{lemma}
	
	We conclude this section with the homotopy types of the independence complexes of path and cycle graphs.
	\begin{lemma}[{\cite[Proposition 4.6]{Kozlov_directed_trees}}]\label{lemma:path_graph}
		For $n\geq 1,$
		$$\Ind(P_n)\simeq 
		\begin{cases}
			\mathbb{S}^{k-1}& \text{ if } n = 3k,\\
			\text{pt} & \text{ if } n = 3k+1,\\
			\mathbb{S}^{k} & \text{ if } n = 3k+2.
		\end{cases}$$
		where pt denotes a space with one point.
	\end{lemma}
	
	\begin{lemma}[{\cite[Proposition 5.2]{Kozlov_directed_trees}}]\label{theorem:cycle_graph}
		For $n\geq 3,$
		$$\Ind(C_n)\simeq 
		\begin{cases}
			\mathbb{S}^{k-1}\vee \mathbb{S}^{k-1} & \text{ if } n = 3k,\\
			\mathbb{S}^{k-1} & \text{ if } n = 3k\pm 1.
		\end{cases}$$
	\end{lemma}
	
	\section{Main Results}\label{section:main}
	
	The aim of this section is to prove \Cref{theorem:intro_homology,theorem:intro_simply_connected}. For a positive integer $m$, let $[m]: = \{1,2,\ldots,m\}$. Recall that $\Gamma_n$ denotes the line graph of $3\times n$ grid graph. For $n\geq 1$, the vertex and edge sets of $\Gamma_n$ are given by:
	\begin{align*}
		V(\Gamma_n)& = \{a_i,b_j,c_i,d_j,e_i: i\in [n-1]\text{ and } j\in [n]\},\\
		E(\Gamma_n)& = \{(a_i,a_{i+1}),(c_i,c_{i+1}),(e_i,e_{i+1}):i\in [n-2]\}\sqcup \{(a_i,b_i),\\
		&\quad\ ~ (a_i,b_{i+1}),(b_i,c_i),(c_i,d_i),(c_i,b_{i+1}),(c_i,d_{i+1}),(d_i,e_i),\\
		&\quad\ (e_i, d_{i+1}):i\in [n-1]\}
		\sqcup \{(b_i,d_i):i\in [n]\}.
	\end{align*}
	\begin{figure}[h!]
		\centering
		\begin{subfigure}{0.44\textwidth}
			\centering
			\begin{tikzpicture}[xshift = 1cm, scale = 0.25, vertices/.style = {draw, fill = black, circle, inner sep = 0.5pt}] 
				
				\node[vertices,label = left:{{\tiny $11$}}] (11) at (12, 3) {};
				\node[vertices,label = left:{{\tiny $12$}}] (12) at (12, 0) {}; 
				\node[vertices,label = left:{{\tiny $13$}}] (13) at (12, -3) {};
				
				\node[vertices,label = above:{{\tiny $21$}}] (21) at (15, 3) {};
				\node[vertices,label = {[xshift = -.2cm]below:{{\tiny $22$}}}] (22) at (15, 0) {}; 
				\node[vertices,label = below:{{\tiny $23$}}] (23) at (15, -3) {}; 
				
				\node[vertices,label = above:{{\tiny $31$}}] (31) at (18, 3) {};
				\node[vertices,label = {[xshift = -.2cm]below:{{\tiny $32$}}}] (32) at (18, 0) {}; 
				\node[vertices,label = below:{{\tiny $33$}}] (33) at (18, -3) {}; 
				
				\node[vertices,label = above:{{\tiny $41$}}] (a) at (21, 3) {};
				\node[vertices,label = {[xshift = -.2cm]below:{{\tiny $42$}}}] (b) at (21, 0) {}; 
				\node[vertices,label = below:{{\tiny $43$}}] (c) at (21, -3) {}; 
				
				\node[vertices,label = above:{{\tiny $(n-1)1$}}] (1) at ($(a)+(5,0)$) {};
				\node[vertices,label = below:{{\tiny $(n-1)2$}}] (2) at ($(b)+(5,0)$) {}; 
				\node[vertices,label = below:{{\tiny $(n-1)3$}}] (3) at ($(c)+(5,0)$) {}; 
				
				\node[vertices,label = right:{{\tiny $n1$}}] (n1) at ($(1)+(3,0)$) {};
				\node[vertices,label = right:{{\tiny $n2$}}] (n2) at ($(2)+(3,0)$) {}; 
				\node[vertices,label = right:{{\tiny $n3$}}] (n3) at ($(3)+(3,0)$) {}; 
				
				\draw (a) -- ++(1,0);
				\draw (b) -- ++(1,0);
				\draw (c) -- ++(1,0);
				\draw (1) -- ++(-1,0);
				\draw (2) -- ++(-1,0);
				\draw (3) -- ++(-1,0);
				
				\node at ($(a)!0.5!(1)$) {$\ldots$};
				\node at ($(b)!0.5!(2)$) {$\ldots$};
				\node at ($(c)!0.5!(3)$) {$\ldots$};
				
				\foreach \vertex/\neighbors in {
					12/{11,22,13},
					22/{21,32,23},
					32/{31,33},
					b/{32,a,c},
					a/{31},
					c/{33},
					2/{1,n2,3},
					n2/{n1,n3},
					21/{11,31},
					23/{13,33},
					1/{n1},
					3/{n3}}
				{
					\foreach \neighbor in \neighbors
					\draw [-] (\vertex)--(\neighbor);
				}
			\end{tikzpicture}
			\subcaption{$G_{3\times n}, \ n \geq 1$}
			\label{fig:G_3_n}
		\end{subfigure}
		\hfill
		\begin{subfigure}{0.13\textwidth}
			\centering
			\begin{tikzpicture}[xshift = 1cm, scale = 0.30, vertices/.style = {draw, fill = black, circle, inner sep = 0.5pt}] 
				\node[vertices,label = above:{{\tiny $b_1$}}] (b1) at (12, 3) {};
				\node[vertices,label = below:{{\tiny $d_1$}}] (d1) at (12, 0) {};
				
				\draw [-] (b1)--(d1);
				
				\node at (12,-2.58) {};
			\end{tikzpicture} 
			\subcaption{$\Gamma_1$} 
			\label{fig:Gamma_1} 
		\end{subfigure}
		\begin{subfigure}{0.4\textwidth}
			\centering
			\begin{tikzpicture}[scale = 0.28, vertices/.style = {draw, fill = black, circle, inner sep = 0.5pt}]
				
				\node[vertices,label = above:{{\tiny $b_1$}}] (b1) at (12, 3) {};
				\node[vertices,label = below:{{\tiny $d_1$}}] (d1) at (12, 0) {};
				
				\node[vertices,label = above:{{\tiny $a_1$}}] (a1) at (13.5, 4.5) {};
				\node[vertices,label = below:{{\tiny $c_1$}}] (c1) at (13.5, 1.5) {};
				\node[vertices,label = below:{{\tiny $e_1$}}] (e1) at (13.5, -1.5) {};
				
				\node[vertices,label = above:{{\tiny $b_2$}}] (b2) at (15, 3) {};
				\node[vertices,label = below:{{\tiny $d_2$}}] (d2) at (15, 0) {};
				
				\node[vertices,label = above:{{\tiny $a_2$}}] (a) at (16.5, 4.5) {};
				\node[vertices,label = below:{{\tiny $c_2$}}] (c) at (16.5, 1.5) {};
				\node[vertices,label = below:{{\tiny $e_2$}}] (e) at (16.5, -1.5) {};
				
				\node[vertices,label = above:{{\tiny $b_3$}}] (b) at (18, 3) {};
				\node[vertices,label = below:{{\tiny $d_3$}}] (d) at (18, 0) {};
				
				\node[vertices, label = {[xshift = .3cm,yshift = -.12cm]above:{{{\tiny $b_{(n-1)}$}}}}] (1) at ($(b)+(4,0)$) {};
				\node[vertices, label = {[xshift = .3cm]below:{{{\tiny $d_{(n-1)}$}}}}] (2) at ($(d)+(4,0)$) {};
				
				\node[vertices,label = {[xshift = .3cm,yshift = -.13cm]above:{{{\tiny $a_{(n-1)}$}}}}] (3) at ($(a)+(7,0)$) {}; 
				\node[vertices,label = {[xshift = .27cm]below:{{{\tiny $c_{(n-1)}$}}}}] (4) at ($(c)+(7,0)$) {};
				\node[vertices,label = {[xshift = .3cm]below:{{{\tiny $e_{(n-1)}$}}}}] (5) at ($(e)+(7,0)$) {};
				
				\node[vertices,label = {[xshift = .06cm]above:{{{\tiny $b_n$}}}}] (6) at ($(b)+(7,0)$) {};
				\node[vertices,label = {[xshift = .06cm]below:{{{\tiny $d_n$}}}}] (7) at ($(d)+(7,0)$) {};
				
				\draw (a) -- ++(2.3,0);
				\draw (c) -- ++(2.3,0);
				\draw (e) -- ++(2.3,0);
				\draw (3) -- ++(-2.3,0);
				\draw (4) -- ++(-2.3,0);
				\draw (5) -- ++(-2.3,0);
				\draw (b) -- ++(1,1);
				\draw (b) -- ++(1,-1);
				\draw (d) -- ++(1,1);
				\draw (d) -- ++(1,-1);
				\draw (1) -- ++(-1,1);
				\draw (1) -- ++(-1,-1);
				\draw (2) -- ++(-1,1);
				\draw (2) -- ++(-1,-1);
				
				\node at ($(b)!0.5!(1)$) {$\ldots$};
				\node at ($(d)!0.5!(2)$) {$\ldots$};
				\node at ($(a)!0.5!(3)$) {$\ldots$};
				\node at ($(c)!0.5!(4)$) {$\ldots$};
				\node at ($(e)!0.5!(5)$) {$\ldots$};
				
				\foreach \vertex/\neighbors in {
					b1/{a1,c1,d1},
					d1/{c1,e1},
					a1/{b2},
					c1/{b2,d2},
					e1/{d2},
					b2/{d2},
					a/{a1,b2,b},
					b/{c,d},
					c/{c1,b2,d2,d},
					d/{e},
					e/{d2,e1},
					1/{2,3,4},
					2/{4,5},
					3/{6},
					4/{6,7},
					5/{7},
					6/{7}}
				{
					\foreach \neighbor in \neighbors
					\draw [-] (\vertex)--(\neighbor);
				}
			\end{tikzpicture}
			\subcaption{$\Gamma_n, \ n \geq 2$}
			\label{fig:Gamma_n_2}
		\end{subfigure}\vspace{-.2cm}
		\caption{}
		\label{fig:Gamma_n}
	\end{figure}
	Since the matching complex of a graph coincides with the independence complex of its line graph, we have $M(G_{3 \times n}) = \Ind(\Gamma_n)$. 
	Using this, we formulate \Cref{theorem:intro_homology} in terms of $\Ind(\Gamma_n)$ as follows. 
	\begin{theorem}\label{theorem:ind_homology}
		Let $n\ge 2$. Then 
		\begin{enumerate}[label = (\roman*)] 
			\item \label{theorem:ind_homology_i} $\tilde{H}_i(\Ind(\Gamma_n)) = 0$ for $i \leq n-2$ and $i = \dim(\Ind(\Gamma_n)) $.
			\item \label{theorem:ind_homology_ii} $\tilde{H}_{n-1}(\Ind(\Gamma_n))\ne0$.
		\end{enumerate} 
	\end{theorem}
	
	We complete the proof of \Cref{theorem:ind_homology} in two parts. In \Cref{section:Lower_dimensions_homology}, we show that $\tilde{H}_i(\Ind(\Gamma_n)) = 0$ for $i\le n-2$ (\Cref{proposition:homology_up_to_n-2}) and $\tilde{H}_{n-1}(\Ind(\Gamma_n))\ne0$ (\Cref{proposition:homology_n-1}). In \Cref{section:top_dimension}, we show that $\tilde{H}_{\dim(\Ind(\Gamma_n))}(\Ind(\Gamma_n)) = 0$ (\Cref{proposition:top_dimension}). 
	
	Moreover, to prove \Cref{theorem:intro_simply_connected}, it suffices to show that $\Ind(\Gamma_n)$ is simply connected (\Cref{theorem:simply_connected}).
	
	\subsection{}\label{section:Lower_dimensions_homology}
	
	In this section, we show that for $n\geq 1$, $\tilde{H}_i(\Ind(\Gamma_n)) = 0$ when $i\le n-2$, and $\tilde{H}_{n-1}(\Ind(\Gamma_n))\neq 0$. To establish these results, we introduce six families of auxiliary graphs $\{\mathcal{U}_n\}_{n\ge1}$, $\{\mathcal{V}_n\}_{n\ge1}$, $\{\mathcal{W}_n\}_{n\ge1}$, $\{\mathcal{X}_n\}_{n\ge1}$, $\{\mathcal{Y}_n\}_{n\ge1}$, and $\{\mathcal{Z}_n\}_{n\ge1}$ (defined in \Cref{subsection:graph_definitions}), where the $n$-th member of each family is constructed from $\Gamma_n$ by adjoining certain vertices and edges. 
	
	We proceed by dividing this section into four subsections. In \Cref{section:Gamma_123}, we determine the homotopy type of the independence complexes of $\Gamma_n$ for $n = 1,2,3$. In \Cref{subsection:graph_definitions}, we define six families of auxiliary graphs and, for each family, we compute the homotopy type of the independence complexes for $n = 1,2$.
	In \Cref{subsection:recursive_homotopy_equivalences}, we analyze the independence complexes of these auxiliary graphs for general $n\ge3$ and derive the recursive relation between the independence complexes of these graphs. These results are then used in \Cref{subsection:lower_dimensions_homology} to obtain the desired homological results for $\Ind(\Gamma_n)$. 
	\subsubsection{Homotopy type of $\Ind(\Gamma_n)$ for $n = 1,2,3$}\label{section:Gamma_123}
	${}$
	
	\vspace{0.1 cm}
	
	\textbf{Case $n = 1:$} Since $\Gamma_1\cong P_2$, $\Ind(\Gamma_1)\simeq \mathbb{S}^0$. 
	
	\vspace{0.1 cm}
	\begin{wrapfigure}{r}{.15\textwidth}
		\vspace{-28pt}
		\centering
		\begin{tikzpicture}[scale = 0.30, vertices/.style = {draw, fill = black, circle, inner sep = 0.5pt}]
			
			\node[vertices,label = above:{{\tiny $b_1$}}] (b1) at (12, 3) {};
			\node[vertices,label = below:{{\tiny $d_1$}}] (d1) at (12, 0) {};
			
			\node[vertices,label = above:{{\tiny $a_1$}}] (a1) at (13.5, 4.5) {};
			\node[vertices,label = below:{{\tiny $c_1$}}] (c1) at (13.5, 1.5) {};
			\node[vertices,label = below:{{\tiny $e_1$}}] (e1) at (13.5, -1.5) {};
			
			\node[vertices,label = above:{{\tiny $b_2$}}] (b2) at (15, 3) {};
			\node[vertices,label = below:{{\tiny $d_2$}}] (d2) at (15, 0) {};
			
			\foreach \vertex/\neighbors in {
				b1/{a1,c1,d1},
				d1/{c1,e1},
				a1/{b2},
				c1/{b2,d2},
				e1/{d2},
				b2/{d2}}
			{
				\foreach \neighbor in \neighbors
				\draw [-] (\vertex)--(\neighbor);
			}
		\end{tikzpicture}\vspace{-.42cm}
		\caption{$\Gamma_2$}
		\label{fig:Gamma_2}
	\end{wrapfigure}
	\textbf{Case $n = 2:$} We have $N_{\Gamma_2}(a_1)\subseteq N_{\Gamma_2}(c_1)$. By \Cref{theorem:N(u)subsetN(v)}, $\Ind(\Gamma_2)\simeq \Ind(\Gamma_2 -\{c_1\})$. Since $\Gamma_2 -\{c_1\}\cong C_6$, $\Ind(\Gamma_2)\simeq \mathbb{S}^1 \vee \mathbb{S}^1$ by \Cref{theorem:cycle_graph}.
	
	\vspace{0.1 cm}
	\textbf{Case $n = 3:$} Note that $\Gamma_3-N_{\Gamma_3}[\{b_1,b_2\}]\cong P_4$ (see \Cref{fig:G3}). By \Cref{lemma:path_graph}, $\Ind(P_4)\simeq \text{pt}$. Hence $\Ind(\Gamma_3)\simeq \Ind(\Gamma_3\cup\{(b_1,b_2)\})$ by \Cref{lemma:edge_contractible}. Let $\Gamma_3': = \Gamma_3\cup \{(b_1,b_2)\}$. Then $\Ind(\Gamma_3)\simeq \Ind(\Gamma_3')$.
	
	We have $N_{\Gamma_3'}[a_1]\subseteq N_{\Gamma_3'}[b_2]$ (see \Cref{fig:G3_b1b2}). By \Cref{corollary:N[u]subsetN[v]_lk(v)_del(v)}, $\Ind(\Gamma_3')\simeq \Ind(\Gamma_3'-\{b_2\})\vee \Sigma(\Ind(\Gamma_3'-N_{\Gamma_3'}[b_2]))$. Since $\Gamma_3'-N_{\Gamma_3'}[b_2]\cong P_5$ (\Cref{fig:lk_G3_b1b2}), \Cref{lemma:path_graph} implies $\Ind(\Gamma_3'-N_{\Gamma_3'}[b_2])\simeq \mathbb{S}^1$. Hence $\Ind(\Gamma_3)\simeq \Ind(\Gamma_3') \simeq \Ind(\Gamma_3'-\{b_2\})\vee \mathbb{S}^2$.
	\begin{figure}[H]
		\begin{subfigure}{0.23\textwidth}
			\centering
			\begin{tikzpicture}[scale = 0.30, vertices/.style = {draw, fill = black, circle, inner sep = 0.5pt}]
				
				\node[vertices,label = above:{{\tiny $b_1$}}] (b1) at (12, 3) {};
				\node[vertices,label = below:{{\tiny $d_1$}}] (d1) at (12, 0) {};
				
				\node[vertices,label = above:{{\tiny $a_1$}}] (a1) at (13.5, 4.5) {};
				\node[vertices,label = left:{{\tiny $c_1$}}] (c1) at (13.5, 1.5) {};
				\node[vertices,label = below:{{\tiny $e_1$}}] (e1) at (13.5, -1.5) {};
				
				\node[vertices,label = above:{{\tiny $b_2$}}] (b2) at (15, 3) {};
				\node[vertices,label = below:{{\tiny $d_2$}}] (d2) at (15, 0) {};
				
				\node[vertices,label = above:{{\tiny $a_2$}}] (a2) at (16.5, 4.5) {};
				\node[vertices,label = right:{{\tiny $c_2$}}] (c2) at (16.5, 1.5) {};
				\node[vertices,label = below:{{\tiny $e_2$}}] (e2) at (16.5, -1.5) {};
				
				\node[vertices,label = above:{{\tiny $b_3$}}] (b3) at (18, 3) {};
				\node[vertices,label = below:{{\tiny $d_3$}}] (d3) at (18, 0) {};
				
				\foreach \to/\from in
				{a1/a2, a1/b1,a1/b2,a2/b2, b1/d1, b1/c1,d1/e1,c1/d2, d1/c1,b2/c1,c1/c2,d2/e1, d2/e2, e1/e2, b2/d2, b2/c2,d2/c2, a2/b3, c2/b3, c2/d3, d3/e2, b3/d3} \draw [-] (\to)--(\from);
				
			\end{tikzpicture}
			\caption{$\Gamma_3$}
			\label{fig:G3}
		\end{subfigure} 
		\hfill
		\begin{subfigure}{0.23\textwidth}
			\centering
			\begin{tikzpicture}[scale = 0.30, vertices/.style = {draw, fill = black, circle, inner sep = 0.5pt}]
				
				\node[vertices,label = above:{{\tiny $b_1$}}] (b1) at (12, 3) {};
				\node[vertices,label = below:{{\tiny $d_1$}}] (d1) at (12, 0) {};
				
				\node[vertices,label = above:{{\tiny $a_1$}}] (a1) at (13.5, 4.5) {};
				\node[vertices,label = left:{{\tiny $c_1$}}] (c1) at (13.5, 1.5) {};
				\node[vertices,label = below:{{\tiny $e_1$}}] (e1) at (13.5, -1.5) {};
				
				\node[vertices,label = above:{{\tiny $b_2$}}] (b2) at (15, 3) {};
				\node[vertices,label = below:{{\tiny $d_2$}}] (d2) at (15, 0) {};
				
				\node[vertices,label = above:{{\tiny $a_2$}}] (a2) at (16.5, 4.5) {};
				\node[vertices,label = right:{{\tiny $c_2$}}] (c2) at (16.5, 1.5) {};
				\node[vertices,label = below:{{\tiny $e_2$}}] (e2) at (16.5, -1.5) {};
				
				\node[vertices,label = above:{{\tiny $b_3$}}] (b3) at (18, 3) {};
				\node[vertices,label = below:{{\tiny $d_3$}}] (d3) at (18, 0) {};
				
				\foreach \to/\from in
				{a1/a2, a1/b1,a1/b2,a2/b2, b1/d1, b1/c1,d1/e1,c1/d2, d1/c1,b2/c1,c1/c2,d2/e1, d2/e2, e1/e2, b2/d2, b2/c2,d2/c2, a2/b3, c2/b3, c2/d3, d3/e2, b3/d3,b1/b2} \draw [-] (\to)--(\from);
				
			\end{tikzpicture}
			\caption{$\Gamma_3'$}
			\label{fig:G3_b1b2}
		\end{subfigure} 
		\hfill
		\begin{subfigure}{0.23\textwidth}
			\centering
			\begin{tikzpicture}[scale = 0.30, vertices/.style = {draw, fill = black, circle, inner sep = 0.5pt}]
				
				\node[vertices,label = below:{{\tiny $d_1$}}] (d1) at (12, 0) {};
				\node[vertices,label = below:{{\tiny $e_1$}}] (e1) at (13.5, -1.5) {};
				
				\node[vertices,label = below:{{\tiny $e_2$}}] (e2) at (16.5, -1.5) {};
				
				\node[vertices,label = above:{{\tiny $b_3$}}] (b3) at (18, 3) {};
				\node[vertices,label = below:{{\tiny $d_3$}}] (d3) at (18, 0) {};
				
				\foreach \to/\from in
				{d1/e1, e1/e2, d3/e2, b3/d3} \draw [-] (\to)--(\from);
				
			\end{tikzpicture}
			\caption{$\Gamma_3'-N_{\Gamma_3'}[b_2]$}
			\label{fig:lk_G3_b1b2}
		\end{subfigure}
		\hfill
		\begin{subfigure}{0.23\textwidth}
			\centering
			\begin{tikzpicture}[scale = 0.30, vertices/.style = {draw, fill = black, circle, inner sep = 0.5pt}]
				
				\node[vertices,label = above:{{\tiny $b_1$}}] (b1) at (12, 3) {};
				\node[vertices,label = below:{{\tiny $d_1$}}] (d1) at (12, 0) {};
				
				\node[vertices,label = above:{{\tiny $a_1$}}] (a1) at (13.5, 4.5) {};
				\node[vertices,label = left:{{\tiny $c_1$}}] (c1) at (13.5, 1.5) {};
				\node[vertices,label = below:{{\tiny $e_1$}}] (e1) at (13.5, -1.5) {};
				
				\node[vertices,label = below:{{\tiny $d_2$}}] (d2) at (15, 0) {};
				
				\node[vertices,label = above:{{\tiny $a_2$}}] (a2) at (16.5, 4.5) {};
				\node[vertices,label = right:{{\tiny $c_2$}}] (c2) at (16.5, 1.5) {};
				\node[vertices,label = below:{{\tiny $e_2$}}] (e2) at (16.5, -1.5) {};
				
				\node[vertices,label = above:{{\tiny $b_3$}}] (b3) at (18, 3) {};
				\node[vertices,label = below:{{\tiny $d_3$}}] (d3) at (18, 0) {};
				
				\foreach \to/\from in
				{a1/a2, a1/b1,b1/d1, b1/c1,d1/e1,c1/d2, d1/c1,c1/c2,d2/e1, d2/e2, e1/e2, d2/c2, a2/b3, c2/b3, c2/d3, d3/e2, b3/d3} \draw [-] (\to)--(\from);
				
			\end{tikzpicture}
			\caption{$\Gamma_3'-\{b_2\}$}
			\label{fig:del_G3'}
		\end{subfigure} 
		\vspace{-.2cm}
		\caption{}
	\end{figure}
	Since $\Gamma_3'-\{b_2\}-N_{\Gamma_3'-\{b_2\}}[\{d_1,d_2\}]\cong P_4$ (\Cref{fig:del_G3'}) and $\Ind(P_4)\simeq \text{pt}$,  \Cref{lemma:edge_contractible} implies $\Ind(\Gamma_3'-\{b_2\})\simeq \Ind((\Gamma_3'-\{b_2\})\cup \{(d_1,d_2)\})$. Let $\Gamma_3'': = (\Gamma_3'-\{b_2\})\cup \{(d_1,d_2)\}$ (see \Cref{fig:G3''}). Then $\Ind(\Gamma_3'-\{b_2\})\simeq \Ind(\Gamma_3'')$, and therefore $\Ind(\Gamma_3)\simeq \Ind(\Gamma_3'')\vee \mathbb{S}^2$.
	
	Now, $N_{\Gamma_3''}[e_1]\subseteq N_{\Gamma_3''}[d_2]$ implies $\Ind(\Gamma_3'')\simeq \Ind(\Gamma_3''-\{d_2\})\vee \Sigma(\Ind(\Gamma_3''-N_{\Gamma_3''}[d_2]))$. Since $\Gamma_3''-N_{\Gamma_3''}[d_2]\cong P_5$, we have $\Ind(\Gamma_3''-N_{\Gamma_3''}[d_2])\simeq \mathbb{S}^1$. Hence $\Ind(\Gamma_3'') \simeq \Ind(\Gamma_3''-\{d_2\})\vee \mathbb{S}^2$. This implies $\Ind(\Gamma_3)\simeq \Ind(\Gamma_3''-\{d_2\})\vee \mathbb{S}^2 \vee \mathbb{S}^2$.
	
	Observe that $\Gamma_3''-\{d_2\}-N_{\Gamma_3''-\{d_2\}}[c_1,e_1]\cong P_4$ (see \Cref{fig:del_G3''}). Using \Cref{lemma:edge_contractible}, we have $\Ind(\Gamma_3''-\{d_2\})\simeq \Ind((\Gamma_3''-\{d_2\})\cup \{(c_1,e_1)\})$. Let $\Gamma_3''': = (\Gamma_3''-\{d_2\})\cup \{(c_1,e_1)\}$ (see \Cref{fig:G3'''}). Then $\Ind(\Gamma_3''-\{d_2\})\simeq \Ind(\Gamma_3''')$, which gives $\Ind(\Gamma_3)\simeq \Ind(\Gamma_3''')\vee\mathbb{S}^2\vee\mathbb{S}^2$.
	\begin{figure}[H]
		\begin{subfigure}{0.23\textwidth}
			\centering
			\begin{tikzpicture}[scale = 0.30, vertices/.style = {draw, fill = black, circle, inner sep = 0.5pt}]
				
				\node[vertices,label = above:{{\tiny $b_1$}}] (b1) at (12, 3) {};
				\node[vertices,label = below:{{\tiny $d_1$}}] (d1) at (12, 0) {};
				
				\node[vertices,label = above:{{\tiny $a_1$}}] (a1) at (13.5, 4.5) {};
				\node[vertices,label = left:{{\tiny $c_1$}}] (c1) at (13.5, 1.5) {};
				\node[vertices,label = below:{{\tiny $e_1$}}] (e1) at (13.5, -1.5) {};
				
				\node[vertices,label = below:{{\tiny $d_2$}}] (d2) at (15, 0) {};
				
				\node[vertices,label = above:{{\tiny $a_2$}}] (a2) at (16.5, 4.5) {};
				\node[vertices,label = right:{{\tiny $c_2$}}] (c2) at (16.5, 1.5) {};
				\node[vertices,label = below:{{\tiny $e_2$}}] (e2) at (16.5, -1.5) {};
				
				\node[vertices,label = above:{{\tiny $b_3$}}] (b3) at (18, 3) {};
				\node[vertices,label = below:{{\tiny $d_3$}}] (d3) at (18, 0) {};
				
				\foreach \to/\from in
				{a1/a2, a1/b1,b1/d1, b1/c1,d1/e1,c1/d2, d1/c1,c1/c2,d2/e1, d2/e2, e1/e2, d2/c2, a2/b3, c2/b3, c2/d3, d3/e2, b3/d3,d2/d1} \draw [-] (\to)--(\from);
			\end{tikzpicture}
			\caption{$\Gamma_3''$}
			\label{fig:G3''}
		\end{subfigure} 
		\begin{subfigure}{0.23\textwidth}
			\centering
			\begin{tikzpicture}[scale = 0.30, vertices/.style = {draw, fill = black, circle, inner	sep = 0.5pt}]
				
				\node[vertices,label = above:{{\tiny $b_1$}}] (b1) at (12, 3) {};
				\node[vertices,label = below:{{\tiny $d_1$}}] (d1) at (12, 0) {};
				
				\node[vertices,label = above:{{\tiny $a_1$}}] (a1) at (13.5, 4.5) {};
				\node[vertices,label = left:{{\tiny $c_1$}}] (c1) at (13.5, 1.5) {};
				\node[vertices,label = below:{{\tiny $e_1$}}] (e1) at (13.5, -1.5) {};
				
				\node[vertices,label = above:{{\tiny $a_2$}}] (a2) at (16.5, 4.5) {};
				\node[vertices,label = right:{{\tiny $c_2$}}] (c2) at (16.5, 1.5) {};
				\node[vertices,label = below:{{\tiny $e_2$}}] (e2) at (16.5, -1.5) {};
				
				\node[vertices,label = above:{{\tiny $b_3$}}] (b3) at (18, 3) {};
				\node[vertices,label = below:{{\tiny $d_3$}}] (d3) at (18, 0) {};
				
				\foreach \to/\from in
				{a1/a2, a1/b1,b1/d1, b1/c1,d1/e1, d1/c1,c1/c2, e1/e2, a2/b3, c2/b3, c2/d3, d3/e2, b3/d3} \draw [-] (\to)--(\from);
			\end{tikzpicture}
			\caption{$\Gamma_3''-\{d_2\}$}
			\label{fig:del_G3''}
		\end{subfigure} 
		\hfill
		\begin{subfigure}{0.23\textwidth}
			\centering
			\begin{tikzpicture}[scale = 0.30, vertices/.style = {draw, fill = black, circle, inner sep = 0.5pt}]
				
				\node[vertices,label = above:{{\tiny $b_1$}}] (b1) at (12, 3) {};
				\node[vertices,label = below:{{\tiny $d_1$}}] (d1) at (12, 0) {};
				
				\node[vertices,label = above:{{\tiny $a_1$}}] (a1) at (13.5, 4.5) {};
				\node[vertices,label = left:{{\tiny $c_1$}}] (c1) at (13.5, 1.5) {};
				\node[vertices,label = below:{{\tiny $e_1$}}] (e1) at (13.5, -1.5) {};
				
				\node[vertices,label = above:{{\tiny $a_2$}}] (a2) at (16.5, 4.5) {};
				\node[vertices,label = right:{{\tiny $c_2$}}] (c2) at (16.5, 1.5) {};
				\node[vertices,label = below:{{\tiny $e_2$}}] (e2) at (16.5, -1.5) {};
				
				\node[vertices,label = above:{{\tiny $b_3$}}] (b3) at (18, 3) {};
				\node[vertices,label = below:{{\tiny $d_3$}}] (d3) at (18, 0) {};
				
				\foreach \to/\from in
				{a1/a2, a1/b1,b1/d1, b1/c1,d1/e1, d1/c1,c1/c2, e1/e2, a2/b3, c2/b3, c2/d3, d3/e2, b3/d3,c1/e1} \draw [-] (\to)--(\from);
				
			\end{tikzpicture}
			\caption{$\Gamma_3'''$}
			\label{fig:G3'''}
		\end{subfigure} 
		\begin{subfigure}{0.25\textwidth}
			\centering
			\begin{tikzpicture}[scale = 0.30, vertices/.style = {draw, fill = black, circle, inner sep = 0.5pt}]
				
				\node[vertices,label = above:{{\tiny $b_1$}}] (b1) at (12, 3) {};
				\node[vertices,label = below:{{\tiny $d_1$}}] (d1) at (12, 0) {};
				
				\node[vertices,label = above:{{\tiny $a_1$}}] (a1) at (13.5, 4.5) {};
				\node[vertices,label = below:{{\tiny $e_1$}}] (e1) at (13.5, -1.5) {};
				
				\node[vertices,label = above:{{\tiny $a_2$}}] (a2) at (16.5, 4.5) {};
				\node[vertices,label = right:{{\tiny $c_2$}}] (c2) at (16.5, 1.5) {};
				\node[vertices,label = below:{{\tiny $e_2$}}] (e2) at (16.5, -1.5) {};
				
				\node[vertices,label = above:{{\tiny $b_3$}}] (b3) at (18, 3) {};
				\node[vertices,label = below:{{\tiny $d_3$}}] (d3) at (18, 0) {};
				
				\foreach \to/\from in
				{a1/a2, a1/b1,b1/d1, d1/e1, e1/e2, a2/b3, c2/b3, c2/d3, d3/e2, b3/d3} \draw [-] (\to)--(\from);
				
			\end{tikzpicture}
			\caption{$\Gamma_3'''-\{c_1\}$}
			\label{fig:G3'''_c1}
		\end{subfigure} 
		\vspace{-.25cm}
		\caption{}
		\label{Fig:Gamma_3 second}
	\end{figure}
	
	Further, $N_{\Gamma_3'''}[d_1]\subseteq N_{\Gamma_3'''}[c_1]$ implies $\Ind(\Gamma_3''')\simeq \Ind(\Gamma_3'''-\{c_1\})\vee \Sigma(\Ind(\Gamma_3'''-N_{\Gamma_3'''}[c_1]))$ by \Cref{corollary:N[u]subsetN[v]_lk(v)_del(v)}. Since $\Gamma_3'''-N_{\Gamma_3'''}[c_1]\cong P_5$, we have $\Ind(\Gamma_3'''-N_{\Gamma_3'''}[c_1])\simeq \mathbb{S}^1$. Hence $\Ind(\Gamma_3''')\simeq \Ind(\Gamma_3'''-\{c_1\})\vee \Sigma(\mathbb{S}^1)\simeq \Ind(\Gamma_3'''-\{c_1\})\vee \mathbb{S}^2$. Therefore, $\Ind(\Gamma_3)\simeq \Ind(\Gamma_3'''-\{c_1\})\vee(\vee_3 \mathbb{S}^2)$. 
	
	Since $c_2$ is a simplicial vertex in $\Gamma_3'''-\{c_1\}$ with $N_{\Gamma_3'''-\{c_1\}}(c_2) = \{b_3,d_3\}$ (\Cref{fig:G3'''_c1}), \Cref{theorem:simplicial_vertex} implies that $\Ind(\Gamma_3'''-\{c_1\})\simeq \Sigma(\Ind(\Gamma_3'''-\{c_1\}-N_{\Gamma_3'''-\{c_1\}}[b_3])) \vee \Sigma(\Ind(\Gamma_3'''-\{c_1\}-N_{\Gamma_3'''-\{c_1\}}[d_3]))$. Observe that $\Gamma_3'''-\{c_1\}-N_{\Gamma_3'''-\{c_1\}}[b_3]\cong P_5$ and $\Gamma_3'''-\{c_1\}-N_{\Gamma_3'''-\{c_1\}}[d_3]\cong P_5$. Therefore $\Ind(\Gamma_3'''-\{c_1\})\simeq \vee_2\Sigma(\Ind(P_5))$. Since $\Ind(P_5)\simeq \mathbb{S}^1$, it follows that $\Ind(\Gamma_3'''-\{c_1\})\simeq \vee_2\Sigma(\mathbb{S}^1) \simeq \vee_2\mathbb{S}^2$. Hence, $\Ind(\Gamma_3)\simeq\vee_5\mathbb{S}^2.$ 
	
	We conclude that
	\begin{equation}\label{equation:Gamma_123}
		\Ind(\Gamma_1)\simeq\mathbb{S}^0,\ \Ind(\Gamma_2)\simeq\mathbb{S}^1\vee\mathbb{S}^1 \text{ and } \Ind(\Gamma_3)\simeq\vee_5\mathbb{S}^2.
	\end{equation}
	
	\subsubsection{Definitions of auxiliary graphs and their independence complexes for $n = 1,2$}\label{subsection:graph_definitions} ${}$
	
	\vspace{0.1 cm}
	$(1)$ $\mathcal{U}_n$. For $n\geq 1,$ the graph $\mathcal{U}_n$ is defined by:
	\begin{align*} 
		V(\mathcal{U}_n)& = V(\Gamma_n)\sqcup \{u\},\\
		E(\mathcal{U}_1)& = E(\Gamma_1)\sqcup \{(u,d_1)\},\text{ and}\\
		E(\mathcal{U}_n)& = E(\Gamma_n)\sqcup \{(u,d_1),(u,e_1)\} \text{ for } n\ge 2.
	\end{align*}
	\begin{figure}[H]
		\centering
		\begin{subfigure}{0.2\textwidth}
			\centering
			\begin{tikzpicture}[xshift = -7cm, scale = 0.26, vertices/.style = {draw, fill = black, circle, inner sep = 0.5pt}] 
				
				\node[vertices,label = below:{{\tiny $u$}}] (u) at (10.5, -1.5) {};
				\node[vertices,label = above:{{\tiny $b_1$}}] (b1) at (12, 3) {};
				\node[vertices,label = below:{{\tiny $d_1$}}] (d1) at (12, 0) {};
				\node at (12,-2.58) {};
				\draw [-] (b1)--(d1);
				\draw [-] (u)--(d1);
			\end{tikzpicture}
			\subcaption{$\mathcal{U}_1$}
			\label{fig:U1}
		\end{subfigure}
		\begin{subfigure}{0.30\textwidth}
			\centering
			\begin{tikzpicture}[scale = 0.26, vertices/.style = {draw, fill = black, circle, inner sep = 0.5pt}]
				
				\node[vertices,label = below:{{\tiny $u$}}] (u) at (10.5, -1.5) {}; 
				\node[vertices,label = above:{{\tiny $b_1$}}] (b1) at (12, 3) {}; 
				\node[vertices,label = below:{{\tiny $d_1$}}] (d1) at (12, 0) {};
				\node[vertices,label = above:{{\tiny $a_1$}}] (a1) at (13.5, 4.5) {}; 
				\node[vertices,label = below:{{\tiny $c_1$}}] (c1) at (13.5, 1.5) {}; 
				\node[vertices,label = below:{{\tiny $e_1$}}] (e1) at (13.5, -1.5) {};
				\node[vertices,label = above:{{\tiny $b_2$}}] (b2) at (15, 3) {}; \node[vertices,label = below:{{\tiny $d_2$}}] (d2) at (15, 0) {}; 
				
				\foreach \vertex/\neighbors in { 
					u/{d1,e1},
					b1/{a1,c1,d1},
					d1/{c1,e1},
					a1/{b2},
					c1/{b2,d2}, 
					e1/{d2},
					b2/{d2}} 
				{
					\foreach \neighbor in \neighbors \draw [-] (\vertex)--(\neighbor); }
			\end{tikzpicture}
			\subcaption{$\mathcal{U}_2$}\label{fig:U_2}
		\end{subfigure}
		\begin{subfigure}{0.4\textwidth}
			\centering
			\begin{tikzpicture}[scale = 0.26, vertices/.style = {draw, fill = black, circle, inner sep = 0.5pt}]
				
				\node[vertices,label = below:{{\tiny $u$}}] (u) at (10.5, -1.5) {}; 
				
				\node[vertices,label = above:{{\tiny $b_1$}}] (b1) at (12, 3) {}; 
				\node[vertices,label = below:{{\tiny $d_1$}}] (d1) at (12, 0) {}; 
				
				\node[vertices,label = above:{{\tiny $a_1$}}] (a1) at (13.5, 4.5) {}; 
				\node[vertices,label = below:{{\tiny $c_1$}}] (c1) at (13.5, 1.5) {}; 
				\node[vertices,label = below:{{\tiny $e_1$}}] (e1) at (13.5, -1.5) {}; 
				\node[vertices,label = above:{{\tiny $b_2$}}] (b2) at (15, 3) {}; 
				\node[vertices,label = below:{{\tiny $d_2$}}] (d2) at (15, 0) {}; 
				
				\node[vertices,label = above:{{\tiny $a_2$}}] (a) at (16.5, 4.5) {}; 
				\node[vertices,label = below:{{\tiny $c_2$}}] (c) at (16.5, 1.5) {}; 
				\node[vertices,label = below:{{\tiny $e_2$}}] (e) at (16.5, -1.5) {}; 
				\node[vertices,label = {[yshift = -.06cm]above:{{{\tiny $b_{3}$}}}}] (b) at (18,3) {};
				\node[vertices,label = below:{{\tiny $d_3$}}] (d) at (18, 0) {};
				
				\node[vertices, label = {[xshift = .33cm,yshift = -.12cm]above:{{{\tiny $b_{(n-1)}$}}}}] (1) at ($(b)+(4,0)$) {}; 
				\node[vertices, label = {[xshift = .3cm]below:{{{\tiny $d_{(n-1)}$}}}}] (2) at ($(d)+(4,0)$) {};
				\node[vertices,label = {[xshift = .3cm,yshift = -.13cm]above:{{{\tiny $a_{(n-1)}$}}}}] (3) at ($(a)+(7,0)$) {};
				\node[vertices,label = {[xshift = 0.02cm]right:{{{\tiny $c_{(n-1)}$}}}}] (4) at ($(c)+(7,0)$) {}; 
				\node[vertices,label = {[xshift = .3cm]below:{{{\tiny $e_{(n-1)}$}}}}] (5) at ($(e)+(7,0)$) {};
				\node[vertices,label = right:{{{\tiny $b_n$}}}] (6) at ($(b)+(7,0)$) {};
				
				\node[vertices,label = below:{{{\tiny $d_n$}}}] (7) at ($(d)+(7,0)$) {};
				\draw (a) -- ++(2.3,0); 
				\draw (c) -- ++(2.3,0); 
				\draw (e) -- ++(2.3,0); 
				\draw (3) -- ++(-2.3,0); 
				\draw (4) -- ++(-2.3,0); 
				\draw (5) -- ++(-2.3,0); 
				\draw (b) -- ++(1,1); 
				\draw (b) -- ++(1,-1); 
				\draw (d) -- ++(1,1); 
				\draw (d) -- ++(1,-1); 
				\draw (1) -- ++(-1,1); 
				\draw (1) -- ++(-1,-1); 
				\draw (2) -- ++(-1,1); 
				\draw (2) -- ++(-1,-1);
				
				\node at ($(b)!0.5!(1)$) {$\ldots$}; 
				\node at ($(d)!0.5!(2)$) {$\ldots$}; 
				\node at ($(a)!0.5!(3)$) {$\ldots$}; 
				\node at ($(c)!0.5!(4)$) {$\ldots$}; 
				\node at ($(e)!0.5!(5)$) {$\ldots$};
				
				\foreach \vertex/\neighbors in { 
					u/{d1,e1},  
					b1/{a1,c1,d1},  
					d1/{c1,e1},  
					a1/{b2},  
					c1/{b2,d2}, 
					e1/{d2},  
					b2/{d2},  
					a/{a1,b2,b},  
					b/{c,d},  
					c/{c1,b2,d2,d},  
					d/{e},  
					e/{d2,e1},  
					1/{2,3,4},  
					2/{4,5},  
					3/{6},  
					4/{6,7},  
					5/{7},  
					6/{7}} 
				{
					\foreach \neighbor in \neighbors \draw [-] (\vertex)--(\neighbor); 
				}
				
			\end{tikzpicture}
			\subcaption{$\mathcal{U}_n,\ n\geq 2$}\label{fig:Un_2}
		\end{subfigure}
		\vspace{-0.25cm}
		\caption{}\label{fig:Un}
	\end{figure}
	
	\textbf{Case $n = 1:$} Since $\mathcal{U}_1\cong P_3$, we have $\Ind(\mathcal{U}_1)\simeq \mathbb{S}^0.$
	
	\vspace{0.1cm}
	\textbf{Case $n = 2:$} Note that $u$ is a simplicial vertex in $\mathcal{U}_2$ with $N_{\mathcal{U}_2}(u) = \{d_1,e_1\}$. By \Cref{theorem:simplicial_vertex}, $\Ind(\mathcal{U}_2)\simeq \Sigma(\Ind(\mathcal{U}_2-N_{\mathcal{U}_2}[d_1]))\vee \Sigma(\Ind(\mathcal{U}_2-N_{\mathcal{U}_2}[e_1]))$. Since $\mathcal{U}_2-N_{\mathcal{U}_2}[d_1]\cong P_3$ and $\mathcal{U}_2-N_{\mathcal{U}_2}[e_1]\cong C_4$, we get $\Ind(\mathcal{U}_2)\simeq \Sigma(\mathbb{S}^0)\vee \Sigma(\mathbb{S}^0)\simeq \mathbb{S}^1 \vee \mathbb{S}^1$. 
	
	Thus, 
	\begin{equation}\label{equation:U_12}
		\Ind(\mathcal{U}_1)\simeq\mathbb{S}^0 \text{ and } \Ind(\mathcal{U}_2)\simeq\vee_2\mathbb{S}^1.
	\end{equation}
	
	$(2)$ $\mathcal{V}_n$. For $n\geq 1,$ the graph $\mathcal{V}_n$ is defined by:
	\begin{align*}
		V(\mathcal{V}_n)& = V(\Gamma_n)\sqcup \{v_1,v_2,v_3,v_4\},\\
		E(\mathcal{V}_1)& = E(\Gamma_1)\sqcup \{(v_1,v_3),(v_1,v_4),(v_2,b_1),(v_3,b_1),(v_3,d_1),(v_4,d_1)\},\text{ and}\\
		E(\mathcal{V}_n)& = E(\Gamma_n)\sqcup \{(v_1,v_3),(v_1,v_4),(v_2,a_1),(v_2,b_1),(v_3,b_1),(v_3,c_1),(v_3,d_1), \\
		&\qquad\qquad\qquad(v_4,d_1),(v_4,e_1)\}\text{ for }n\ge 2.
	\end{align*}
	\begin{figure}[H]
		\centering
		\begin{subfigure}{0.2\textwidth}
			\centering
			\begin{tikzpicture}[xshift = 1cm, scale = 0.3, vertices/.style = {draw, fill = black, circle, inner sep = 0.5pt}]
				
				\node[vertices,label = below:{{\tiny $v_1$}}] (v1) at (9, 0) {}; 
				\node[vertices,label = above:{{\tiny $v_2$}}] (v2) at (10.5, 4.5) {}; 
				\node[vertices,label = below:{{\tiny $v_3$}}] (v3) at (10.5, 1.5) {}; 
				\node[vertices,label = below:{{\tiny $v_4$}}] (v4) at (10.5, -1.5) {}; 
				
				\node[vertices,label = above:{{\tiny $b_1$}}] (b1) at (12, 3) {};
				\node[vertices,label = below:{{\tiny $d_1$}}] (d1) at (12, 0) {};
				
				\node at (10.5,-2.58) {};
				
				\draw [-] (b1)--(d1);
				\draw [-] (v2)--(b1);
				\draw [-] (v3)--(b1);
				\draw [-] (v3)--(d1);
				\draw [-] (v3)--(v1);
				\draw [-] (v4)--(v1);
				\draw [-] (v4)--(d1); 
			\end{tikzpicture} 
			\subcaption{$\mathcal{V}_1$} 
			\label{fig:v2}
		\end{subfigure}
		\begin{subfigure}{0.3\textwidth}
			\centering
			\begin{tikzpicture}[scale = 0.3, vertices/.style = {draw, fill = black, circle, inner sep = 0.5pt}]
				
				\node[vertices,label = below:{{\tiny $v_1$}}] (v1) at (9, 0) {};
				\node[vertices,label = above:{{\tiny $v_2$}}] (v2) at (10.5, 4.5) {}; 
				\node[vertices,label = below:{{\tiny $v_3$}}] (v3) at (10.5, 1.5) {}; 
				\node[vertices,label = below:{{\tiny $v_4$}}] (v4) at (10.5, -1.5) {}; 
				
				\node[vertices,label = above:{{\tiny $b_1$}}] (b1) at (12, 3) {};
				\node[vertices,label = below:{{\tiny $d_1$}}] (d1) at (12, 0) {};
				
				\node[vertices,label = above:{{\tiny $a_1$}}] (a1) at (13.5, 4.5) {};
				\node[vertices,label = below:{{\tiny $c_1$}}] (c1) at (13.5, 1.5) {};
				\node[vertices,label = below:{{\tiny $e_1$}}] (e1) at (13.5, -1.5) {};
				
				\node[vertices,label = above:{{\tiny $b_2$}}] (b2) at (15, 3) {};
				\node[vertices,label = below:{{\tiny $d_2$}}] (d2) at (15, 0) {};
				
				\foreach \vertex/\neighbors in {
					v2/{a1,b1},
					v3/{v1,b1,c1,d1},
					v4/{v1,d1,e1},
					b1/{a1,c1,d1},
					d1/{c1,e1},
					a1/{b2},
					c1/{b2,d2},
					e1/{d2},
					b2/{d2}}
				{
					\foreach \neighbor in \neighbors
					\draw [-] (\vertex)--(\neighbor);
				}
			\end{tikzpicture}
			\subcaption{$\mathcal{V}_2$}
			\label{fig:V_2}
		\end{subfigure}
		\begin{subfigure}{0.4\textwidth}
			\centering
			\begin{tikzpicture}[scale = 0.3, vertices/.style = {draw, fill = black, circle, inner sep = 0.5pt}]
				
				\node[vertices,label = below:{{\tiny $v_1$}}] (v1) at (9, 0) {};
				\node[vertices,label = above:{{\tiny $v_2$}}] (v2) at (10.5, 4.5) {}; 
				\node[vertices,label = below:{{\tiny $v_3$}}] (v3) at (10.5, 1.5) {}; 
				\node[vertices,label = below:{{\tiny $v_4$}}] (v4) at (10.5, -1.5) {}; 
				
				\node[vertices,label = above:{{\tiny $b_1$}}] (b1) at (12, 3) {};
				\node[vertices,label = below:{{\tiny $d_1$}}] (d1) at (12, 0) {};
				
				\node[vertices,label = above:{{\tiny $a_1$}}] (a1) at (13.5, 4.5) {};
				\node[vertices,label = below:{{\tiny $c_1$}}] (c1) at (13.5, 1.5) {};
				\node[vertices,label = below:{{\tiny $e_1$}}] (e1) at (13.5, -1.5) {};
				
				\node[vertices,label = above:{{\tiny $b_2$}}] (b2) at (15, 3) {};
				\node[vertices,label = below:{{\tiny $d_2$}}] (d2) at (15, 0) {};
				
				\node[vertices,label = above:{{\tiny $a_2$}}] (a) at (16.5, 4.5) {};
				\node[vertices,label = below:{{\tiny $c_2$}}] (c) at (16.5, 1.5) {};
				\node[vertices,label = below:{{\tiny $e_2$}}] (e) at (16.5, -1.5) {};
				
				\node[vertices,label = above:{{\tiny $b_3$}}] (b) at (18, 3) {};
				\node[vertices,label = below:{{\tiny $d_3$}}] (d) at (18, 0) {};
				
				\node[vertices, label = {[xshift = .32cm,yshift = -.12cm]above:{{{\tiny $b_{(n-1)}$}}}}] (1) at ($(b)+(4,0)$) {};
				\node[vertices, label = {[xshift = .3cm]below:{{{\tiny $d_{(n-1)}$}}}}] (2) at ($(d)+(4,0)$) {};
				
				\node[vertices,label = {[xshift = .3cm,yshift = -.13cm]above:{{{\tiny $a_{(n-1)}$}}}}] (3) at ($(a)+(7,0)$) {}; 
				\node[vertices,label = {[xshift = .27cm]below:{{{\tiny $c_{(n-1)}$}}}}] (4) at ($(c)+(7,0)$) {};
				\node[vertices,label = {[xshift = .3cm]below:{{{\tiny $e_{(n-1)}$}}}}] (5) at ($(e)+(7,0)$) {};
				
				\node[vertices,label = above:{{{\tiny $b_n$}}}] (6) at ($(b)+(7,0)$) {};
				\node[vertices,label = below:{{{\tiny $d_n$}}}] (7) at ($(d)+(7,0)$) {};
				
				\draw (a) -- ++(2.3,0);
				\draw (c) -- ++(2.3,0);
				\draw (e) -- ++(2.3,0);
				\draw (3) -- ++(-2.3,0);
				\draw (4) -- ++(-2.3,0);
				\draw (5) -- ++(-2.3,0);
				\draw (b) -- ++(1,1);
				\draw (b) -- ++(1,-1);
				\draw (d) -- ++(1,1);
				\draw (d) -- ++(1,-1);
				\draw (1) -- ++(-1,1);
				\draw (1) -- ++(-1,-1);
				\draw (2) -- ++(-1,1);
				\draw (2) -- ++(-1,-1);
				
				\node at ($(b)!0.5!(1)$) {$\ldots$};
				\node at ($(d)!0.5!(2)$) {$\ldots$};
				\node at ($(a)!0.5!(3)$) {$\ldots$};
				\node at ($(c)!0.5!(4)$) {$\ldots$};
				\node at ($(e)!0.5!(5)$) {$\ldots$};
				
				\foreach \vertex/\neighbors in {
					v2/{a1,b1},
					v3/{v1,b1,c1,d1},
					v4/{v1,d1,e1},
					b1/{a1,c1,d1},
					d1/{c1,e1},
					a1/{b2},
					c1/{b2,d2},
					e1/{d2},
					b2/{d2},
					a/{a1,b2,b},
					b/{c,d},
					c/{c1,b2,d2,d},
					d/{e},
					e/{d2,e1},
					1/{2,3,4},
					2/{4,5},
					3/{6},
					4/{6,7},
					5/{7},
					6/{7}}
				{
					\foreach \neighbor in \neighbors
					\draw [-] (\vertex)--(\neighbor);
				}
				
			\end{tikzpicture}
			\subcaption{$\mathcal{V}_n, \ n \geq 2$}
			\label{fig:Vn_2}
		\end{subfigure}
		\vspace{-.25cm}
		\caption{}
		\label{fig:Vn}
	\end{figure}
	
	\textbf{Case $n = 1:$} Since $v_2$ is a simplicial vertex in $\mathcal{V}_1$ with $N_{\mathcal{V}_1}(v_2) = \{b_1\}$, $\Ind(\mathcal{V}_1)\simeq \Sigma(\Ind(\mathcal{V}_1-N_{\mathcal{V}_1}[b_1]))$ by \Cref{theorem:simplicial_vertex}. Now, $\mathcal{V}_1-N_{\mathcal{V}_1}[b_1]\cong P_2$ implies $\Ind(\mathcal{V}_1)\simeq\Sigma(\Ind(P_2))\simeq \mathbb{S}^1$. \vspace{0.1cm}
	
	\textbf{Case $n = 2:$} In $\mathcal{V}_2$, $v_2$ is a simplicial vertex with $N_{\mathcal{V}_2}(v_2) = \{a_1,b_1\}$. Hence $\Ind(\mathcal{V}_2)\simeq \Sigma(\Ind(\mathcal{V}_2-N_{\mathcal{V}_2}[a_1]))\vee \Sigma(\Ind(\mathcal{V}_2-N_{\mathcal{V}_2}[b_1]))$. Since $\mathcal{V}_2-N_{\mathcal{V}_2}[a_1]\cong \Gamma_2$ and $\mathcal{V}_2-N_{\mathcal{V}_2}[b_1]\cong P_5$, it follows that $\Ind(\mathcal{V}_2)\simeq\Sigma(\Ind (\Gamma_2))\vee \Sigma(\Ind(P_5))$. We have $\Ind(\Gamma_2)\simeq\vee_2\mathbb{S}^1$ by \eqref{equation:Gamma_123} and $\Ind(P_5)\simeq\mathbb{S}^1$ by \Cref{lemma:path_graph}. Thus $\Ind(\mathcal{V}_2) \simeq \Sigma(\vee_2\mathbb{S}^1) \vee\Sigma(\mathbb{S}^1)\simeq \vee_3 \mathbb{S}^2.$ 
	
	Therefore, 
	\begin{equation}\label{equation:V_12}
		\Ind(\mathcal{V}_1)\simeq\mathbb{S}^1 \text{ and } \Ind(\mathcal{V}_2)\simeq\vee_3\mathbb{S}^2.
	\end{equation}
	
	$(3)$ $\mathcal{W}_n$. For $n\geq 1,$ the graph $\mathcal{W}_n$ is defined by:
	\begin{align*}
		V(\mathcal{W}_n)& = V(\Gamma_n)\sqcup \{w_1,w_2,w_3\},\\
		E(\mathcal{W}_1)& = E(\Gamma_1)\sqcup \{(w_1,w_2),(w_1,w_3),(w_2,b_1),(w_2,d_1),(w_3,d_1)\},\text{ and}\\
		E(\mathcal{W}_n)& = E(\Gamma_n)\sqcup \{(w_1,w_2),(w_1,w_3),(w_2,b_1),(w_2,c_1),(w_2,d_1),(w_3,d_1),(w_3,e_1)\}\text{ for $n\geq 2$}.
	\end{align*} 
	\begin{figure}[H]
		\centering
		\begin{subfigure}{0.25\textwidth}
			\centering
			\begin{tikzpicture}[xshift = 1cm, scale = 0.3, vertices/.style = {draw, fill = black, circle, inner sep = 0.5pt}]
				
				\node[vertices,label = below:{{\tiny $w_1$}}] (w1) at (9, 0) {};
				\node[vertices,label = below:{{\tiny $w_2$}}] (w2) at (10.5, 1.5) {}; 
				\node[vertices,label = below:{{\tiny $w_3$}}] (w3) at (10.5, -1.5) {};
				
				\node[vertices,label = above:{{\tiny $b_1$}}] (b1) at (12, 3) {};
				\node[vertices,label = below:{{\tiny $d_1$}}] (d1) at (12, 0) {};
				
				\node at (10.5,-2.58) {};
				
				\draw [-] (b1)--(d1);
				\draw [-] (w1)--(b1);
				\draw [-] (w2)--(d1);
				\draw [-] (w3)--(w1);
				\draw [-] (w3)--(d1);   
			\end{tikzpicture} 
			\subcaption{$\mathcal{W}_1$} 
			\label{fig:W1}
		\end{subfigure}
		\begin{subfigure}{0.25\textwidth}
			\centering
			\begin{tikzpicture}[scale = 0.3, vertices/.style = {draw, fill = black, circle, inner sep = 0.5pt}]
				
				\node[vertices,label = below:{{\tiny $w_1$}}] (w1) at (9, 0) {}; 
				\node[vertices,label = below:{{\tiny $w_2$}}] (w2) at (10.5, 1.5) {}; 
				\node[vertices,label = below:{{\tiny $w_3$}}] (w3) at (10.5, -1.5) {}; 
				
				\node[vertices,label = above:{{\tiny $b_1$}}] (b1) at (12, 3) {};
				\node[vertices,label = below:{{\tiny $d_1$}}] (d1) at (12, 0) {};
				
				\node[vertices,label = above:{{\tiny $a_1$}}] (a1) at (13.5, 4.5) {};
				\node[vertices,label = below:{{\tiny $c_1$}}] (c1) at (13.5, 1.5) {};
				\node[vertices,label = below:{{\tiny $e_1$}}] (e1) at (13.5, -1.5) {};
				
				\node[vertices,label = above:{{\tiny $b_2$}}] (b2) at (15, 3) {};
				\node[vertices,label = below:{{\tiny $d_2$}}] (d2) at (15, 0) {};
				
				\foreach \vertex/\neighbors in {
					w2/{w1,b1,c1,d1},
					w3/{w1,d1,e1},
					b1/{a1,c1,d1},
					d1/{c1,e1},
					a1/{b2},
					c1/{b2,d2},
					e1/{d2},
					b2/{d2}}
				{
					\foreach \neighbor in \neighbors
					\draw [-] (\vertex)--(\neighbor);
				}  
			\end{tikzpicture}
			\subcaption{$\mathcal{W}_2$}
			\label{fig:W_2}
		\end{subfigure}
		\begin{subfigure}{0.44\textwidth}
			\centering
			\begin{tikzpicture}[scale = 0.3, vertices/.style = {draw, fill = black, circle, inner sep = 0.5pt}]
				
				\node[vertices,label = below:{{\tiny $w_1$}}] (w1) at (9, 0) {}; 
				\node[vertices,label = below:{{\tiny $w_2$}}] (w2) at (10.5, 1.5) {}; 
				\node[vertices,label = below:{{\tiny $w_3$}}] (w3) at (10.5, -1.5) {}; 
				
				\node[vertices,label = above:{{\tiny $b_1$}}] (b1) at (12, 3) {};
				\node[vertices,label = below:{{\tiny $d_1$}}] (d1) at (12, 0) {};
				
				\node[vertices,label = above:{{\tiny $a_1$}}] (a1) at (13.5, 4.5) {};
				\node[vertices,label = below:{{\tiny $c_1$}}] (c1) at (13.5, 1.5) {};
				\node[vertices,label = below:{{\tiny $e_1$}}] (e1) at (13.5, -1.5) {};
				
				\node[vertices,label = above:{{\tiny $b_2$}}] (b2) at (15, 3) {};
				\node[vertices,label = below:{{\tiny $d_2$}}] (d2) at (15, 0) {};
				
				\node[vertices,label = above:{{\tiny $a_2$}}] (a) at (16.5, 4.5) {};
				\node[vertices,label = below:{{\tiny $c_2$}}] (c) at (16.5, 1.5) {};
				\node[vertices,label = below:{{\tiny $e_2$}}] (e) at (16.5, -1.5) {};
				
				\node[vertices,label = above:{{\tiny $b_3$}}] (b) at (18, 3) {};
				\node[vertices,label = below:{{\tiny $d_3$}}] (d) at (18, 0) {};
				
				\node[vertices, label = {[xshift = .32cm,yshift = -.12cm]above:{{{\tiny $b_{(n-1)}$}}}}] (1) at ($(b)+(4,0)$) {};
				\node[vertices, label = {[xshift = .3cm]below:{{{\tiny $d_{(n-1)}$}}}}] (2) at ($(d)+(4,0)$) {};
				
				\node[vertices,label = {[xshift = .3cm,yshift = -.13cm]above:{{{\tiny $a_{(n-1)}$}}}}] (3) at ($(a)+(7,0)$) {}; 
				\node[vertices,label = {[xshift = .27cm]below:{{{\tiny $c_{(n-1)}$}}}}] (4) at ($(c)+(7,0)$) {};
				\node[vertices,label = {[xshift = .3cm]below:{{{\tiny $e_{(n-1)}$}}}}] (5) at ($(e)+(7,0)$) {};
				
				\node[vertices,label = above:{{{\tiny $b_n$}}}] (6) at ($(b)+(7,0)$) {};
				\node[vertices,label = below:{{{\tiny $d_n$}}}] (7) at ($(d)+(7,0)$) {};
				
				\draw (a) -- ++(2.3,0);
				\draw (c) -- ++(2.3,0);
				\draw (e) -- ++(2.3,0);
				\draw (3) -- ++(-2.3,0);
				\draw (4) -- ++(-2.3,0);
				\draw (5) -- ++(-2.3,0);
				\draw (b) -- ++(1,1);
				\draw (b) -- ++(1,-1);
				\draw (d) -- ++(1,1);
				\draw (d) -- ++(1,-1);
				\draw (1) -- ++(-1,1);
				\draw (1) -- ++(-1,-1);
				\draw (2) -- ++(-1,1);
				\draw (2) -- ++(-1,-1);
				
				\node at ($(b)!0.5!(1)$) {$\ldots$};
				\node at ($(d)!0.5!(2)$) {$\ldots$};
				\node at ($(a)!0.5!(3)$) {$\ldots$};
				\node at ($(c)!0.5!(4)$) {$\ldots$};
				\node at ($(e)!0.5!(5)$) {$\ldots$};
				
				\foreach \vertex/\neighbors in {
					w2/{w1,b1,c1,d1},
					w3/{w1,d1,e1},
					b1/{a1,c1,d1},
					d1/{c1,e1},
					a1/{b2},
					c1/{b2,d2},
					e1/{d2},
					b2/{d2},
					a/{a1,b2,b},
					b/{c,d},
					c/{c1,b2,d2,d},
					d/{e},
					e/{d2,e1},
					1/{2,3,4},
					2/{4,5},
					3/{6},
					4/{6,7},
					5/{7},
					6/{7}}
				{
					\foreach \neighbor in \neighbors
					\draw [-] (\vertex)--(\neighbor);
				}
			\end{tikzpicture}
			\subcaption{$\mathcal{W}_n, \ n \geq 2$}
			\label{fig:Wn_2}
		\end{subfigure}
		\vspace{-.25cm}
		\caption{}
		\label{fig:Wn}
	\end{figure}
	
	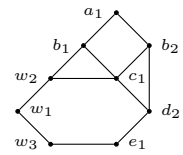
\begin{wrapfigure}{r}{.2\textwidth}
		\vspace{-12pt}
		\centering
		\begin{tikzpicture}[scale = 0.29, vertices/.style = {draw, fill = black, circle, inner sep = 0.5pt}]
			
			\node[vertices,label = right:{{\tiny $w_1$}}] (w1) at (9, 0) {}; 
			\node[vertices,label = left:{{\tiny $w_2$}}] (w2) at (10.5, 1.5) {}; 
			\node[vertices,label = left:{{\tiny $w_3$}}] (w3) at (10.5, -1.5) {}; 
			
			\node[vertices,label = left:{{\tiny $b_1$}}] (b1) at (12, 3) {};
			
			\node[vertices,label = left:{{\tiny $a_1$}}] (a1) at (13.5, 4.5) {};
			\node[vertices,label = right:{{\tiny $c_1$}}] (c1) at (13.5, 1.5) {};
			\node[vertices,label = right:{{\tiny $e_1$}}] (e1) at (13.5, -1.5) {};
			
			\node[vertices,label = right:{{\tiny $b_2$}}] (b2) at (15, 3) {};
			\node[vertices,label = right:{{\tiny $d_2$}}] (d2) at (15, 0) {};
			
			\foreach \vertex/\neighbors in {
				w2/{w1,b1},
				w3/{w1,e1},
				c1/{b1,b2,d2,w2},
				b1/{a1},
				a1/{b2},
				e1/{d2},
				b2/{d2}}
			{
				\foreach \neighbor in \neighbors
				\draw [-] (\vertex)--(\neighbor);
			}
		\end{tikzpicture}  
		\vspace{-.25cm}
		\caption{$\mathcal{W}_2-\{d_1\}$}
		\label{fig:W2_d1}
	\end{wrapfigure}
	
	\textbf{Case $n = 1:$} Observe that $N_{\mathcal{W}_1}(w_1) \subseteq N_{\mathcal{W}_1}(d_1)$. By \Cref{theorem:N(u)subsetN(v)}, $\Ind(\mathcal{W}_1) \simeq \Ind(\mathcal{W}_1-\{d_1\})$. Since $\mathcal{W}_1-\{d_1\}\cong P_4$, it follows that $\Ind(\mathcal{W}_1)\simeq \text{pt}$.
	
	\vspace{0.1cm}
	
	\textbf{Case $n = 2:$} We have $N_{\mathcal{W}_2}(w_1)\subseteq N_{\mathcal{W}_2}(d_1)$, which implies $\Ind(\mathcal{W}_2)\simeq \Ind(\mathcal{W}_2-\{d_1\})$. Further, since $N_{\mathcal{W}_2-\{d_1\}}(a_1)\subseteq N_{\mathcal{W}_2-\{d_1\}}(c_1)$, $\Ind(\mathcal{W}_2-\{d_1\}) \simeq \Ind(\mathcal{W}_2-\{c_1,d_1\})$. Note that $\mathcal{W}_2-\{c_1,d_1\}\cong C_8$. Hence $\Ind(\mathcal{W}_2)\simeq \Ind(C_8)\simeq \mathbb{S}^2$ by \Cref{theorem:cycle_graph}. 
	
	Therefore,
	\begin{equation}\label{equation:W_12}
		\Ind(\mathcal{W}_1)\simeq \text{pt}\text{ and } \Ind(\mathcal{W}_2)\simeq \mathbb{S}^2. 
	\end{equation}
	
	$(4)$ $\mathcal{X}_n$. For $n\geq 1,$ the graph $\mathcal{X}_n$ is defined by:
	\begin{align*} 
		V(\mathcal{X}_n)& = V(\Gamma_n)\sqcup \{x_1,x_2,x_3,x_4,x_5, x_6,x_7,x_8\},\\
		E(\mathcal{X}_1)& = E(\Gamma_1)\sqcup \{(x_1,x_2),(x_1,x_3),(x_2,x_4),(x_3,x_5),(x_4,x_6),(x_4,x_7),(x_5,x_6),\\
		&\qquad\qquad\qquad(x_5,x_8),(x_6,x_7),(x_6,x_8),(x_7,b_1),(x_8,b_1),(x_8,d_1)\},\text{ and}\\
		E(\mathcal{X}_n)& = E(\Gamma_n)\sqcup \{(x_1,x_2),(x_1,x_3),(x_2,x_4),(x_3,x_5),(x_4,x_6),(x_4,x_7),(x_5,x_6),(x_5,x_8),\\
		&\qquad\qquad\qquad(x_6,x_7),(x_6,x_8),(x_7,a_1),(x_7,b_1),(x_8,b_1),(x_8,c_1),(x_8,d_1)\}\text{ for $n\geq 2$}.
	\end{align*} 
	\begin{figure}[H]
		\centering
		\begin{subfigure}{0.23\textwidth}
			\centering
			\begin{tikzpicture}[scale = 0.3, vertices/.style = {draw, fill = black, circle, inner sep = 0.5pt}]
				
				\node[vertices,label = above:{{\tiny $x_1$}}] (x1) at (3, 3) {};
				\node[vertices,label = above:{{\tiny $x_2$}}] (x2) at (4.5, 4.5) {}; 
				\node[vertices,label = below:{{\tiny $x_3$}}] (x3) at (4.5, 1.5) {}; 
				\node[vertices,label = above:{{\tiny $x_4$}}] (x4) at (7.5, 4.5) {}; 
				\node[vertices,label = below:{{\tiny $x_5$}}] (x5) at (7.5, 1.5) {}; 
				\node[vertices,label = above:{{\tiny $x_6$}}] (x6) at (9, 3) {};
				\node[vertices,label = above:{{\tiny $x_7$}}] (x7) at (10.5, 4.5) {}; 
				\node[vertices,label = below:{{\tiny $x_8$}}] (x8) at (10.5, 1.5) {};
				
				\node[vertices,label = above:{{\tiny $b_1$}}] (b1) at (11.8, 3) {};
				\node[vertices,label = below:{{\tiny $d_1$}}] (d1) at (11.8, 0) {};
				
				\node at (10.5,-2.59) {};
				
				\foreach \vertex/\neighbors in {
					b1/{d1},
					x7/{x6,x4,b1},
					x8/{x6,x5,b1,d1},
					x6/{x4,x5},
					x2/{x4,x1},
					x3/{x5,x1}}
				{
					\foreach \neighbor in \neighbors
					\draw [-] (\vertex)--(\neighbor);
				}
			\end{tikzpicture} 
			\subcaption{$\mathcal{X}_1$} 
			\label{fig:x7}
		\end{subfigure}
		\begin{subfigure}{0.25\textwidth}
			\centering
			\begin{tikzpicture}[scale = 0.3, vertices/.style = {draw, fill = black, circle, inner sep = 0.5pt}]
				
				\node[vertices,label = above:{{\tiny $x_1$}}] (x1) at (3, 3) {};
				\node[vertices,label = above:{{\tiny $x_2$}}] (x2) at (4.5, 4.5) {}; 
				\node[vertices,label = below:{{\tiny $x_3$}}] (x3) at (4.5, 1.5) {}; 
				\node[vertices,label = above:{{\tiny $x_4$}}] (x4) at (7.5, 4.5) {}; 
				\node[vertices,label = below:{{\tiny $x_5$}}] (x5) at (7.5, 1.5) {}; 
				\node[vertices,label = above:{{\tiny $x_6$}}] (x6) at (9, 3) {};
				\node[vertices,label = above:{{\tiny $x_7$}}] (x7) at (10.5, 4.5) {}; 
				\node[vertices,label = below:{{\tiny $x_8$}}] (x8) at (10.5, 1.5) {};
				
				\node[vertices,label = above:{{\tiny $b_1$}}] (b1) at (12, 3) {};
				\node[vertices,label = below:{{\tiny $d_1$}}] (d1) at (12, 0) {};
				
				\node[vertices,label = above:{{\tiny $a_1$}}] (a1) at (13.5, 4.5) {};
				\node[vertices,label = below:{{\tiny $c_1$}}] (c1) at (13.5, 1.5) {};
				\node[vertices,label = below:{{\tiny $e_1$}}] (e1) at (13.5, -1.5) {};
				
				\node[vertices,label = above:{{\tiny $b_2$}}] (b2) at (15, 3) {};
				\node[vertices,label = below:{{\tiny $d_2$}}] (d2) at (15, 0) {};
				
				\foreach \vertex/\neighbors in {
					x7/{x6,x4,b1},
					x8/{x6,x5,b1,c1,d1},
					x6/{x4,x5},
					x2/{x4,x1},
					x3/{x5,x1},
					a1/{x7},
					b1/{a1,c1,d1},
					d1/{c1,e1},
					a1/{b2},
					c1/{b2,d2},
					e1/{d2},
					b2/{d2}}
				{
					\foreach \neighbor in \neighbors
					\draw [-] (\vertex)--(\neighbor);
				}
			\end{tikzpicture}
			\subcaption{$\mathcal{X}_2$}
			\label{fig:X_2}
		\end{subfigure}
		\begin{subfigure}{0.42\textwidth}
			\centering
			\begin{tikzpicture}[scale = 0.3, vertices/.style = {draw, fill = black, circle, inner sep = 0.5pt}]
				
				\node[vertices,label = above:{{\tiny $x_1$}}] (x1) at (3, 3) {};
				\node[vertices,label = above:{{\tiny $x_2$}}] (x2) at (4.5, 4.5) {}; 
				\node[vertices,label = below:{{\tiny $x_3$}}] (x3) at (4.5, 1.5) {}; 
				\node[vertices,label = above:{{\tiny $x_4$}}] (x4) at (7.5, 4.5) {}; 
				\node[vertices,label = below:{{\tiny $x_5$}}] (x5) at (7.5, 1.5) {}; 
				\node[vertices,label = above:{{\tiny $x_6$}}] (x6) at (9, 3) {};
				\node[vertices,label = above:{{\tiny $x_7$}}] (x7) at (10.5, 4.5) {}; 
				\node[vertices,label = below:{{\tiny $x_8$}}] (x8) at (10.5, 1.5) {};
				
				\node[vertices,label = above:{{\tiny $b_1$}}] (b1) at (12, 3) {};
				\node[vertices,label = below:{{\tiny $d_1$}}] (d1) at (12, 0) {};
				
				\node[vertices,label = above:{{\tiny $a_1$}}] (a) at (13.5, 4.5) {};
				\node[vertices,label = below:{{\tiny $c_1$}}] (c) at (13.5, 1.5) {};
				\node[vertices,label = below:{{\tiny $e_1$}}] (e) at (13.5, -1.5) {};
				
				\node[vertices,label = above:{{\tiny $b_2$}}] (b) at (15, 3) {};
				\node[vertices,label = below:{{\tiny $d_2$}}] (d) at (15, 0) {};
				
				\node[vertices, label = {[xshift = .32cm,yshift = -.12cm]above:{{{\tiny $b_{(n-1)}$}}}}] (1) at ($(b)+(4,0)$) {};
				\node[vertices, label = {[xshift = .3cm]below:{{{\tiny $d_{(n-1)}$}}}}] (2) at ($(d)+(4,0)$) {};
				
				\node[vertices,label = {[xshift = .3cm,yshift = -.13cm]above:{{{\tiny $a_{(n-1)}$}}}}] (3) at ($(a)+(7,0)$) {}; 
				\node[vertices,label = {[xshift = .27cm]below:{{{\tiny $c_{(n-1)}$}}}}] (4) at ($(c)+(7,0)$) {};
				\node[vertices,label = {[xshift = .3cm]below:{{{\tiny $e_{(n-1)}$}}}}] (5) at ($(e)+(7,0)$) {};
				
				\node[vertices,label = above:{{{\tiny $b_n$}}}] (6) at ($(b)+(7,0)$) {};
				\node[vertices,label = below:{{{\tiny $d_n$}}}] (7) at ($(d)+(7,0)$) {};
				
				\draw (a) -- ++(2.3,0);
				\draw (c) -- ++(2.3,0);
				\draw (e) -- ++(2.3,0);
				\draw (3) -- ++(-2.3,0);
				\draw (4) -- ++(-2.3,0);
				\draw (5) -- ++(-2.3,0);
				\draw (b) -- ++(1,1);
				\draw (b) -- ++(1,-1);
				\draw (d) -- ++(1,1);
				\draw (d) -- ++(1,-1);
				\draw (1) -- ++(-1,1);
				\draw (1) -- ++(-1,-1);
				\draw (2) -- ++(-1,1);
				\draw (2) -- ++(-1,-1);
				
				\node at ($(b)!0.5!(1)$) {$\ldots$};
				\node at ($(d)!0.5!(2)$) {$\ldots$};
				\node at ($(a)!0.5!(3)$) {$\ldots$};
				\node at ($(c)!0.5!(4)$) {$\ldots$};
				\node at ($(e)!0.5!(5)$) {$\ldots$};
				
				\foreach \vertex/\neighbors in {
					x7/{x6,x4,b1},
					x8/{x6,x5,b1,d1},
					x6/{x4,x5},
					x2/{x4,x1},
					x3/{x5,x1},
					b1/{d1},
					a/{x7,b1,b},
					b/{c,d},
					c/{x8,b1,d1,d},
					d/{e},
					e/{d1},
					1/{2,3,4},
					2/{4,5},
					3/{6},
					4/{6,7},
					5/{7},
					6/{7}}
				{
					\foreach \neighbor in \neighbors
					\draw [-] (\vertex)--(\neighbor);
				}
			\end{tikzpicture}
			\subcaption{$\mathcal{X}_n, \ n \geq 2$}
			\label{fig:Xn_2}
		\end{subfigure}
		\vspace{-.25cm}
		\caption{}
		\label{fig:Xn}
	\end{figure}
	\textbf{Case $n = 1:$} Since $d_1$ is a simplicial vertex in $\mathcal{X}_1$ with $N_{\mathcal{X}_1}(d_1) = \{x_8,b_1\},$ we have $\Ind(\mathcal{X}_1)\simeq \Sigma(\Ind(\mathcal{X}_1-N_{\mathcal{X}_1}[x_8])) \vee \Sigma(\Ind(\mathcal{X}_1-N_{\mathcal{X}_1}[b_1]))$ by \Cref{theorem:simplicial_vertex}. 
	Now, $\mathcal{X}_1-N_{\mathcal{X}_1}[x_8]\cong P_5$ and $\mathcal{X}_1-N_{\mathcal{X}_1}[b_1]\cong C_6$ imply $\Ind(\mathcal{X}_1)\simeq \Sigma(\mathbb{S}^1)\vee \Sigma(\vee_2\mathbb{S}^1)\simeq \vee_3 \mathbb{S}^2$. 
	
	\vspace{0.1cm}
	\textbf{Case $n = 2:$} Note that $N_{\mathcal{X}_2}(e_1)\subseteq N_{\mathcal{X}_2}(c_1)$. By \Cref{theorem:N(u)subsetN(v)}, $\Ind(\mathcal{X}_2)\simeq \Ind(\mathcal{X}_2-\{c_1\})$. Let $\mathcal{X}_2'$ be the graph obtained from $\mathcal{X}_2-\{c_1\}$ by replacing the paths $x_4x_2x_1x_3x_5$ and $a_1b_2d_2e_1d_1$ with the edges $(x_4,x_5)$ and $(a_1,d_1)$, respectively (see \Cref{fig:X2'}).
	By \Cref{theorem:replace_edge}, $\Ind(\mathcal{X}_2-\{c_1\})\simeq \Sigma^2(\Ind(\mathcal{X}_2'))$. This implies $\Ind(\mathcal{X}_2)\simeq \Sigma^2(\Ind(\mathcal{X}_2')).$
	\begin{figure}[H]
		\centering
		\begin{subfigure}{0.25\textwidth}
			\centering
			\begin{tikzpicture}[scale = 0.3, vertices/.style = {draw, fill = black, circle, inner sep = 0.5pt}]
				
				\node[vertices,label = above:{{\tiny $x_1$}}] (x1) at (4.5, 3) {};
				\node[vertices,label = above:{{\tiny $x_2$}}] (x2) at (5.7, 4.5) {}; 
				\node[vertices,label = below:{{\tiny $x_3$}}] (x3) at (5.7, 1.5) {}; 
				\node[vertices,label = above:{{\tiny $x_4$}}] (x4) at (7.5, 4.5) {}; 
				\node[vertices,label = below:{{\tiny $x_5$}}] (x5) at (7.5, 1.5) {}; 
				\node[vertices,label = above:{{\tiny $x_6$}}] (x6) at (9, 3) {};
				\node[vertices,label = above:{{\tiny $x_7$}}] (x7) at (10.5, 4.5) {}; 
				\node[vertices,label = below:{{\tiny $x_8$}}] (x8) at (10.5, 1.5) {};
				
				\node[vertices,label = above:{{\tiny $b_1$}}] (b1) at (12, 3) {};
				\node[vertices,label = below:{{\tiny $d_1$}}] (d1) at (12, 0) {};
				
				\node[vertices,label = above:{{\tiny $a_1$}}] (a1) at (13.5, 4.5) {};
				\node[vertices,label = below:{{\tiny $e_1$}}] (e1) at (13.5, -1.5) {};
				
				\node[vertices,label = above:{{\tiny $b_2$}}] (b2) at (15, 3) {};
				\node[vertices,label = below:{{\tiny $d_2$}}] (d2) at (15, 0) {};
				
				\foreach \vertex/\neighbors in {
					x7/{x6,x4,a1,b1},
					x8/{x6,x5,b1,d1},
					x6/{x4,x5},
					x2/{x4,x1},
					x3/{x5,x1},
					b1/{a1,d1},
					d1/{e1},
					a1/{b2},
					e1/{d2},
					b2/{d2}}
				{
					\foreach \neighbor in \neighbors
					\draw [-] (\vertex)--(\neighbor);
				}
			\end{tikzpicture}
			\subcaption{$\mathcal{X}_2-\{c_1\}$}
			\label{fig:X_2_c_1}
		\end{subfigure}
		\begin{subfigure}{0.25\textwidth}
			\centering
			\begin{tikzpicture}[scale = 0.3, vertices/.style = {draw, fill = black, circle, inner sep = 0.5pt}]
				
				\node[vertices,label = above:{{\tiny $x_4$}}] (x4) at (7.5, 4.5) {}; 
				\node[vertices,label = below:{{\tiny $x_5$}}] (x5) at (7.5, 1.5) {}; 
				\node[vertices,label = above:{{\tiny $x_6$}}] (x6) at (9, 3) {};
				\node[vertices,label = above:{{\tiny $x_7$}}] (x7) at (10.5, 4.5) {}; 
				\node[vertices,label = below:{{\tiny $x_8$}}] (x8) at (10.5, 1.5) {};
				
				\node[vertices,label = above:{{\tiny $b_1$}}] (b1) at (12, 3) {};
				\node[vertices,label = below:{{\tiny $d_1$}}] (d1) at (12, 0) {};
				
				\node[vertices,label = above:{{\tiny $a_1$}}] (a1) at (13.5, 4.5) {};
				
				\node at (12,-2.4) {};
				
				\foreach \vertex/\neighbors in {
					x7/{x6,x4,a1,b1},
					x8/{x6,x5,b1,d1},
					x6/{x4,x5},
					x4/{x5},
					b1/{a1,d1},
					d1/{a1}}
				{
					\foreach \neighbor in \neighbors
					\draw [-] (\vertex)--(\neighbor);
				}
			\end{tikzpicture}
			\subcaption{$\mathcal{X}_2'$}
			\label{fig:X2'}
		\end{subfigure}
		\begin{subfigure}{0.25\textwidth}
			\centering
			\begin{tikzpicture}[scale = 0.3, vertices/.style = {draw, fill = black, circle, inner sep = 0.5pt}]
				
				\node[vertices,label = above:{{\tiny $x_4$}}] (x4) at (7.5, 4.5) {}; 
				\node[vertices,label = below:{{\tiny $x_5$}}] (x5) at (7.5, 1.5) {}; 
				\node[vertices,label = above:{{\tiny $x_6$}}] (x6) at (9, 3) {};
				\node[vertices,label = above:{{\tiny $x_7$}}] (x7) at (10.5, 4.5) {}; 
				\node[vertices,label = below:{{\tiny $x_8$}}] (x8) at (10.5, 1.5) {};
				
				\node[vertices,label = below:{{\tiny $d_1$}}] (d1) at (12, 0) {};
				
				\node[vertices,label = above:{{\tiny $a_1$}}] (a1) at (13.5, 4.5) {};
				
				\node at (12,-2.4) {};
				
				\foreach \vertex/\neighbors in {
					x7/{x6,x4,a1},
					x8/{x6,x5,d1},
					x6/{x4,x5},
					x4/{x5},
					d1/{a1}}
				{
					\foreach \neighbor in \neighbors
					\draw [-] (\vertex)--(\neighbor);
				}
				
			\end{tikzpicture}
			\subcaption{$\mathcal{X}_2''$}
		\end{subfigure}
		\vspace{-.25cm}
		\caption{}
	\end{figure} 
	
	Further, $N_{\mathcal{X}_2'}[a_1]\subseteq N_{\mathcal{X}_2'}[b_1]$. By \Cref{corollary:N[u]subsetN[v]_lk(v)_del(v)}, $\Ind(\mathcal{X}_2')\simeq \Ind(\mathcal{X}_2'-\{b_1\})\vee\Sigma(\Ind(\mathcal{X}_2'-N_{\mathcal{X}_2'}[b_1]))$. We have $\mathcal{X}_2'-N_{\mathcal{X}_2'}[b_1])\cong C_3$. Let $\mathcal{X}_2'': = \mathcal{X}_2'-\{b_1\}$. Then $\Ind(\mathcal{X}_2')\simeq \Ind(\mathcal{X}_2'')\vee\Sigma(\Ind(C_3))$.
	Using the fact that $N_{\mathcal{X}_2''}[x_4]\subseteq N_{\mathcal{X}_2''}[x_6]$, we have $\Ind(\mathcal{X}_2'')\simeq \Ind(\mathcal{X}_2''-\{x_6\})\vee\Sigma(\Ind(\mathcal{X}_2''-N_{\mathcal{X}_2''}[x_6]))$. Since $\mathcal{X}_2''-\{x_6\}\cong C_6$ and $\mathcal{X}_2''-N_{\mathcal{X}_2''}[x_6]\simeq P_2$, it follows that $\Ind(\mathcal{X}_2'')\simeq \Ind(C_6)\vee\Sigma(\Ind(P_2))$. Therefore $\Ind(\mathcal{X}_2)\simeq \Sigma^2(\Ind(\mathcal{X}_2'))\simeq \Sigma^2(\Ind(C_6))\vee\Sigma^3(\Ind(P_2))\vee\Sigma^3(\Ind(C_3))$. From \Cref{lemma:path_graph,theorem:cycle_graph}, $\Ind(\mathcal{X}_2)\simeq \Sigma^2(\vee_2\mathbb{S}^1)\vee\Sigma^3(\mathbb{S}^0)\vee\Sigma^3(\vee_2\mathbb{S}^0)\simeq \vee_5\mathbb{S}^3$. 
	
	Thus,
	\begin{equation}\label{equation:X_12}
		\Ind(\mathcal{X}_1)\simeq \vee_3 \mathbb{S}^2\text{ and }
		\Ind(\mathcal{X}_2)\simeq \vee_5 \mathbb{S}^3. 
	\end{equation}
	
	$(5)$ $\mathcal{Y}_n$. For $n\geq 1,$ the graph $\mathcal{Y}_n$ is defined by:
	\begin{align*} 
		V(\mathcal{Y}_n)& = V(\Gamma_n)\sqcup \{y_1,y_2,y_3,y_4,y_5, y_6\},\\
		E(\mathcal{Y}_1)& = E(\Gamma_1)\sqcup \{(y_1,y_3),(y_1,y_5),(y_2,y_6),(y_3,y_4),(y_3,y_5),(y_4,b_1),(y_5,b_1),(y_5,d_1),(y_6,d_1)\},\text{ and}\\
		E(\mathcal{Y}_n)& = E(\Gamma_n)\sqcup \{(y_1,y_3),(y_1,y_5),(y_2,y_6),(y_3,y_4),(y_3,y_5),(y_4,a_1),(y_4,b_1),\\
		&\qquad\qquad\qquad(y_5,b_1),(y_5,c_1),(y_5,d_1),(y_6,d_1),(y_6,e_1)\}
		\text{ for $n\geq 2.$}
	\end{align*} 
	\begin{figure}[H]
		\centering
		\begin{subfigure}{0.20\textwidth}
			\centering
			\begin{tikzpicture}[scale = 0.3, vertices/.style = {draw, fill = black, circle, inner sep = 0.5pt}]
				
				\node[vertices,label = below:{{\tiny $y_1$}}] (y1) at (7.5, 1.5) {};
				\node[vertices,label = below:{{\tiny $y_2$}}] (y2) at (7.5, -1.5) {};
				\node[vertices,label = above:{{\tiny $y_3$}}] (y3) at (9, 3) {};
				\node[vertices,label = above:{{\tiny $y_4$}}] (y4) at (10.5, 4.5) {}; 
				\node[vertices,label = below:{{\tiny $y_5$}}] (y5) at (10.5, 1.5) {}; 
				\node[vertices,label = below:{{\tiny $y_6$}}] (y6) at (10.5, -1.5) {}; 
				
				\node[vertices,label = above:{{\tiny $b_1$}}] (b1) at (12, 3) {};
				\node[vertices,label = below:{{\tiny $d_1$}}] (d1) at (12, 0) {};
				
				\node at (10.5,-2.58) {};
				
				\foreach \vertex/\neighbors in {
					b1/{d1},
					y1/{y3,y5},
					y2/{y6},
					y3/{y4,y5},
					y4/{b1},
					y5/{b1,d1},
					y6/{d1}}
				{
					\foreach \neighbor in \neighbors
					\draw [-] (\vertex)--(\neighbor);
				}
			\end{tikzpicture} 
			\subcaption{$\mathcal{Y}_1$} 
			\label{fig:Y1} 
		\end{subfigure}
		\begin{subfigure}{0.22\textwidth}
			\centering
			\begin{tikzpicture}[scale = 0.3, vertices/.style = {draw, fill = black, circle, inner sep = 0.5pt}]
				
				\node[vertices,label = below:{{\tiny $y_1$}}] (y1) at (7.5, 1.5) {};
				\node[vertices,label = below:{{\tiny $y_2$}}] (y2) at (7.5, -1.5) {};
				\node[vertices,label = above:{{\tiny $y_3$}}] (y3) at (9, 3) {};
				\node[vertices,label = above:{{\tiny $y_4$}}] (y4) at (10.5, 4.5) {}; 
				\node[vertices,label = below:{{\tiny $y_5$}}] (y5) at (10.5, 1.5) {}; 
				\node[vertices,label = below:{{\tiny $y_6$}}] (y6) at (10.5, -1.5) {};
				
				\node[vertices,label = above:{{\tiny $b_1$}}] (b1) at (12, 3) {};
				\node[vertices,label = below:{{\tiny $d_1$}}] (d1) at (12, 0) {};
				
				\node[vertices,label = above:{{\tiny $a_1$}}] (a1) at (13.5, 4.5) {};
				\node[vertices,label = below:{{\tiny $c_1$}}] (c1) at (13.5, 1.5) {};
				\node[vertices,label = below:{{\tiny $e_1$}}] (e1) at (13.5, -1.5) {};
				
				\node[vertices,label = above:{{\tiny $b_2$}}] (b2) at (15, 3) {};
				\node[vertices,label = below:{{\tiny $d_2$}}] (d2) at (15, 0) {};
				
				\foreach \vertex/\neighbors in {
					y1/{y3,y5},
					y2/{y6},
					y3/{y4,y5},
					y4/{a1,b1},
					y5/{b1,c1,d1},
					y6/{d1,e1},
					b1/{a1,c1,d1},
					d1/{c1,e1},
					a1/{b2},
					c1/{b2,d2},
					e1/{d2},
					b2/{d2}}
				{
					\foreach \neighbor in \neighbors
					\draw [-] (\vertex)--(\neighbor);
				}
			\end{tikzpicture}
			\caption{$\mathcal{Y}_2$}
		\end{subfigure}
		\begin{subfigure}{0.45\textwidth}
			\centering
			\begin{tikzpicture}[scale = 0.3, vertices/.style = {draw, fill = black, circle, inner
					sep = 0.5pt}]
				
				\node[vertices,label = below:{{\tiny $y_1$}}] (y1) at (7.5, 1.5) {};
				\node[vertices,label = below:{{\tiny $y_2$}}] (y2) at (7.5, -1.5) {};
				\node[vertices,label = above:{{\tiny $y_3$}}] (y3) at (9, 3) {};
				\node[vertices,label = above:{{\tiny $y_4$}}] (y4) at (10.5, 4.5) {}; 
				\node[vertices,label = below:{{\tiny $y_5$}}] (y5) at (10.5, 1.5) {}; 
				\node[vertices,label = below:{{\tiny $y_6$}}] (y6) at (10.5, -1.5) {};
				
				\node[vertices,label = above:{{\tiny $b_1$}}] (b1) at (12, 3) {};
				\node[vertices,label = below:{{\tiny $d_1$}}] (d1) at (12, 0) {};
				
				\node[vertices,label = above:{{\tiny $a_1$}}] (a1) at (13.5, 4.5) {};
				\node[vertices,label = below:{{\tiny $c_1$}}] (c1) at (13.5, 1.5) {};
				\node[vertices,label = below:{{\tiny $e_1$}}] (e1) at (13.5, -1.5) {};
				
				\node[vertices,label = above:{{\tiny $b_2$}}] (b2) at (15, 3) {};
				\node[vertices,label = below:{{\tiny $d_2$}}] (d2) at (15, 0) {};
				
				\node[vertices,label = above:{{\tiny $a_2$}}] (a) at (16.5, 4.5) {};
				\node[vertices,label = below:{{\tiny $c_2$}}] (c) at (16.5, 1.5) {};
				\node[vertices,label = below:{{\tiny $e_2$}}] (e) at (16.5, -1.5) {};
				
				\node[vertices,label = above:{{\tiny $b_3$}}] (b) at (18, 3) {};
				\node[vertices,label = below:{{\tiny $d_3$}}] (d) at (18, 0) {};
				
				\node[vertices, label = {[xshift = .32cm,yshift = -.12cm]above:{{{\tiny $b_{(n-1)}$}}}}] (1) at ($(b)+(4,0)$) {};
				\node[vertices, label = {[xshift = 0cm]right:{{{\tiny $d_{(n-1)}$}}}}] (2) at ($(d)+(4,0)$) {};
				
				\node[vertices,label = {[xshift = .3cm,yshift = -.13cm]above:{{{\tiny $a_{(n-1)}$}}}}] (3) at ($(a)+(7,0)$) {}; 
				\node[vertices,label = {[xshift = 0cm]right:{{{\tiny $c_{(n-1)}$}}}}] (4) at ($(c)+(7,0)$) {};
				\node[vertices,label = {[xshift = .3cm]below:{{{\tiny $e_{(n-1)}$}}}}] (5) at ($(e)+(7,0)$) {};
				
				\node[vertices,label = right:{{{\tiny $b_n$}}}] (6) at ($(b)+(7,0)$) {};
				\node[vertices,label = right:{{{\tiny $d_n$}}}] (7) at ($(d)+(7,0)$) {};
				
				\draw (a) -- ++(2.3,0);
				\draw (c) -- ++(2.3,0);
				\draw (e) -- ++(2.3,0);
				\draw (3) -- ++(-2.3,0);
				\draw (4) -- ++(-2.3,0);
				\draw (5) -- ++(-2.3,0);
				\draw (b) -- ++(1,1);
				\draw (b) -- ++(1,-1);
				\draw (d) -- ++(1,1);
				\draw (d) -- ++(1,-1);
				\draw (1) -- ++(-1,1);
				\draw (1) -- ++(-1,-1);
				\draw (2) -- ++(-1,1);
				\draw (2) -- ++(-1,-1);
				
				\node at ($(b)!0.5!(1)$) {$\ldots$};
				\node at ($(d)!0.5!(2)$) {$\ldots$};
				\node at ($(a)!0.5!(3)$) {$\ldots$};
				\node at ($(c)!0.5!(4)$) {$\ldots$};
				\node at ($(e)!0.5!(5)$) {$\ldots$};
				
				\foreach \vertex/\neighbors in {
					y1/{y3,y5},
					y2/{y6},
					y3/{y4,y5},
					y4/{a1,b1},
					y5/{b1,c1,d1},
					y6/{d1,e1},
					b1/{a1,c1,d1},
					d1/{c1,e1},
					a1/{b2},
					c1/{b2,d2},
					e1/{d2},
					b2/{d2},
					a/{a1,b2,b},
					b/{c,d},
					c/{c1,b2,d2,d},
					d/{e},
					e/{d2,e1},
					1/{2,3,4},
					2/{4,5},
					3/{6},
					4/{6,7},
					5/{7},
					6/{7}}
				{
					\foreach \neighbor in \neighbors
					\draw [-] (\vertex)--(\neighbor);
				}
			\end{tikzpicture}
			\subcaption{$\mathcal{Y}_n, \ n \geq 2$}
			\label{fig:Yn_2}
		\end{subfigure}
		\vspace{-.25cm}
		\caption{}
		\label{fig:Yn}
	\end{figure}
	\textbf{Case $n = 1:$} Since $N_{\mathcal{Y}_1}(y_4)\subseteq N_{\mathcal{Y}_1}(y_5)$, $\Ind(\mathcal{Y}_1)\simeq \Ind(\mathcal{Y}_1-\{y_5\})$. Observe that $\mathcal{Y}_1-\{y_5\}\cong P_7$. Hence $\Ind(\mathcal{Y}_1)\simeq \text{pt}$ by \Cref{lemma:path_graph}.
	
	\vspace{0.1cm}
	\begin{wrapfigure}{r}{.52\textwidth}
		\vspace{-22pt}
		\centering
		\begin{subfigure}{0.20\textwidth}
			\centering
			\begin{tikzpicture}[scale = 0.255, vertices/.style = {draw, fill = black, circle, inner sep = 0.5pt}] 
				
				\node[vertices,label = below:{{\tiny $y_1$}}] (y1) at (7.5, 1.5) {};
				\node[vertices,label = above:{{\tiny $y_3$}}] (y3) at (9, 3) {};
				\node[vertices,label = above:{{\tiny $y_4$}}] (y4) at (10.5, 4.5) {}; 
				\node[vertices,label = below:{{\tiny $y_5$}}] (y5) at (10.5, 1.5) {}; 
				
				\node[vertices,label = above:{{\tiny $b_1$}}] (b1) at (12, 3) {};
				
				\node[vertices,label = above:{{\tiny $a_1$}}] (a1) at (13.5, 4.5) {};
				\node[vertices,label = below:{{\tiny $c_1$}}] (c1) at (13.5, 1.5) {};
				
				\node[vertices,label = above:{{\tiny $b_2$}}] (b2) at (15, 3) {};
				\node[vertices,label = below:{{\tiny $d_2$}}] (d2) at (15, 0) {};
				
				\node at (12,-2.44) {};
				
				\foreach \vertex/\neighbors in {
					y1/{y3,y5},
					y3/{y4,y5},
					y4/{a1,b1},
					y5/{b1,c1},
					b1/{a1,c1},
					a1/{b2},
					c1/{b2,d2},
					b2/{d2}}
				{
					\foreach \neighbor in \neighbors
					\draw [-] (\vertex)--(\neighbor);
				}
			\end{tikzpicture}
			\vspace{-0.32cm}
			\caption{$\mathcal{Y}_2'$}
		\end{subfigure}
		\begin{subfigure}{0.20\textwidth}
			\centering
			\begin{tikzpicture}[scale = 0.255, vertices/.style = {draw, fill = black, circle, inner sep = 0.5pt}] 
				
				\node[vertices,label = above:{{\tiny $y_4$}}] (y4) at (10.5, 4.5) {}; 
				\node[vertices,label = above:{{\tiny $a_1$}}] (a1) at (13.5, 4.5) {};
				\node[vertices,label = above:{{\tiny $b_2$}}] (b2) at (15, 3) {};
				\node[vertices,label = below:{{\tiny $d_2$}}] (d2) at (15, 0) {};
				
				\node at (12,-2.44) {};
				
				\foreach \vertex/\neighbors in {
					y4/{a1},
					a1/{b2},
					b2/{d2}}
				{
					\foreach \neighbor in \neighbors
					\draw [-] (\vertex)--(\neighbor);
				}
				
			\end{tikzpicture}
			\vspace{-0.32cm}
			\caption{$\mathcal{Y}_2'-N_{\mathcal{Y}_2'}[y_5]$}
		\end{subfigure}
		\begin{subfigure}{0.10\textwidth}
			\centering
			\begin{tikzpicture}[scale = 0.255, vertices/.style = {draw, fill = black, circle, inner sep = 0.5pt}]
				
				\node[vertices,label = above:{{\tiny $b_1$}}] (b1) at (12, 3) {};
				
				\node[vertices,label = above:{{\tiny $a_1$}}] (a1) at (13.5, 4.5) {};
				\node[vertices,label = below:{{\tiny $c_1$}}] (c1) at (13.5, 1.5) {};
				
				\node[vertices,label = above:{{\tiny $b_2$}}] (b2) at (15, 3) {};
				\node[vertices,label = below:{{\tiny $d_2$}}] (d2) at (15, 0) {};
				
				\node at (12,-2.37) {};
				
				\foreach \vertex/\neighbors in {
					b1/{a1,c1},
					a1/{b2},
					c1/{b2,d2},
					b2/{d2}}
				{
					\foreach \neighbor in \neighbors
					\draw [-] (\vertex)--(\neighbor);
				}
			\end{tikzpicture}
			\vspace{-0.32cm}
			\caption{$\mathcal{Y}_2''$}
		\end{subfigure}
		\vspace{-.32cm}
		\caption{}
		\vspace{-.5cm} 
	\end{wrapfigure}
	
	\textbf{Case $n = 2:$} Observe that $y_1$ is a simplicial vertex in $\mathcal{Y}_2$ with $N_{\mathcal{Y}_2}(y_2) = \{y_6\}$. 
	Hence $\Ind(\mathcal{Y}_2)\simeq \Sigma(\Ind(\mathcal{Y}_2-N_{\mathcal{Y}_2}[y_6]))$ by \Cref{theorem:simplicial_vertex}.
	
	Let $\mathcal{Y}_2': = \mathcal{Y}_2-N_{\mathcal{Y}_2}[y_6]$. Then $\Ind(\mathcal{Y}_2)\simeq \Sigma(\Ind(\mathcal{Y}_2'))$. Since $y_1$ is a simplicial vertex in $\mathcal{Y}_2'$ with $N_{\mathcal{Y}_2'}(y_1) = \{y_3,y_5\}$, $\Ind( \mathcal{Y}_2' )\simeq \Sigma(\Ind(\mathcal{Y}_2'-N_{\mathcal{Y}_2'}[y_3]))\vee \Sigma(\Ind(\mathcal{Y}_2'-N_{\mathcal{Y}_2'}[y_5]))$. We have $\mathcal{Y}_2'-N_{\mathcal{Y}_2'}[y_5]\cong P_4$. Let $\mathcal{Y}_2'': = \mathcal{Y}_2'-N_{\mathcal{Y}_2'}[y_3]$. Then $\Ind( \mathcal{Y}_2' )\simeq \Sigma(\Ind(\mathcal{Y}_2''))\vee \Sigma(\Ind(P_4))$. 
	Further, since $N_{\mathcal{Y}_2''}(b_1)\subseteq N_{\mathcal{Y}_2''}(b_2)$, $\Ind(\mathcal{Y}_2'')\simeq \Ind(\mathcal{Y}_2''-\{b_2\}) \simeq \Ind(P_4) \simeq pt$. Hence $\Ind(\mathcal{Y}_2') \simeq pt$, and thus $\Ind(\mathcal{Y}_2)\simeq \Sigma(\Ind(\mathcal{Y}_2')) \simeq pt$. 
	
	Therefore,
	\begin{equation}\label{equation:Y_12}
		\Ind(\mathcal{Y}_1)\simeq \text{pt} \text{ and } \Ind(\mathcal{Y}_2)\simeq \text{pt}.
	\end{equation} 
	
	$(6)$ $\mathcal{Z}_n$. For $n\geq 1,$ the graph $\mathcal{Z}_n$ is defined by:
	\begin{align*} 
		V(\mathcal{Z}_n)& = V(\Gamma_n)\sqcup \{z_1,z_2,z_3,z_4,z_5, z_6,z_7,z_8,z_9\},\\
		E(\mathcal{Z}_1)& = E(\Gamma_1)\sqcup \{(z_1,z_2),(z_1,z_3),(z_2,z_4),(z_3,z_5),(z_4,z_6),(z_4,z_8),(z_5,z_9),\\
		&\qquad\qquad\qquad(z_6,z_7),(z_6,z_8),(z_7,b_1),(z_8,b_1),(z_8,d_1),(z_9,d_1)\},\text{ and}\\
		E(\mathcal{Z}_n)& = E(\Gamma_n)\sqcup \{(z_1,z_2),(z_1,z_3),(z_2,z_4),(z_3,z_5),(z_4,z_6),(z_4,z_8),(z_5,z_9),(z_6,z_7),(z_6,z_8),\\
		&\qquad\qquad\qquad(z_7,a_1),(z_7,b_1),(z_8,b_1),(z_8,c_1),(z_8,d_1),(z_9,d_1),(z_9,e_1)\text{ for $n\geq 2.$}
	\end{align*} 
	\begin{figure}[H]
		\centering
		\begin{subfigure}{0.2\textwidth}
			\centering
			\begin{tikzpicture}[scale = 0.3, vertices/.style = {draw, fill = black, circle, inner sep = 0.5pt}]
				
				\node[vertices,label = below:{{\tiny $z_1$}}] (z1) at (3, 0) {};
				\node[vertices,label = below:{{\tiny $z_2$}}] (z2) at (4.5, 1.5) {}; 
				\node[vertices,label = below:{{\tiny $z_3$}}] (z3) at (4.5, -1.5) {};
				\node[vertices,label = below:{{\tiny $z_4$}}] (z4) at (7.5, 1.5) {}; 
				\node[vertices,label = below:{{\tiny $z_5$}}] (z5) at (7.5, -1.5) {};
				\node[vertices,label = above:{{\tiny $z_6$}}] (z6) at (9, 3) {};
				\node[vertices,label = above:{{\tiny $z_7$}}] (z7) at (10.5, 4.5) {}; 
				\node[vertices,label = below:{{\tiny $z_8$}}] (z8) at (10.5, 1.5) {}; 
				\node[vertices,label = below:{{\tiny $z_9$}}] (z9) at (10.5, -1.5) {}; 
				
				\node[vertices,label = above:{{\tiny $b_1$}}] (b1) at (12, 3) {};
				\node[vertices,label = below:{{\tiny $d_1$}}] (d1) at (12, 0) {};
				
				\node at (10.5,-2.58) {};
				
				\foreach \vertex/\neighbors in {
					z1/{z2,z3},
					z2/{z4},
					z3/{z5},
					z4/{z6,z8},
					z5/{z9},
					z6/{z7,z8},
					z7/{b1},
					z8/{b1,d1},
					z9/{d1},
					b1/{d1}}
				{
					\foreach \neighbor in \neighbors
					\draw [-] (\vertex)--(\neighbor);
				}
			\end{tikzpicture} 
			\subcaption{$\mathcal{Z}_1$} 
			\label{fig:Z1}
		\end{subfigure}
		\begin{subfigure}{0.3\textwidth}
			\centering
			\begin{tikzpicture}[scale = 0.3, vertices/.style = {draw, fill = black, circle, inner sep = 0.5pt}]
				
				\node[vertices,label = below:{{\tiny $z_1$}}] (z1) at (3, 0) {};
				\node[vertices,label = below:{{\tiny $z_2$}}] (z2) at (4.5, 1.5) {}; 
				\node[vertices,label = below:{{\tiny $z_3$}}] (z3) at (4.5, -1.5) {};
				\node[vertices,label = below:{{\tiny $z_4$}}] (z4) at (7.5, 1.5) {}; 
				\node[vertices,label = below:{{\tiny $z_5$}}] (z5) at (7.5, -1.5) {};
				\node[vertices,label = above:{{\tiny $z_6$}}] (z6) at (9, 3) {};
				\node[vertices,label = above:{{\tiny $z_7$}}] (z7) at (10.5, 4.5) {}; 
				\node[vertices,label = below:{{\tiny $z_8$}}] (z8) at (10.5, 1.5) {}; 
				\node[vertices,label = below:{{\tiny $z_9$}}] (z9) at (10.5, -1.5) {}; 
				
				\node[vertices,label = above:{{\tiny $b_1$}}] (b1) at (12, 3) {};
				\node[vertices,label = below:{{\tiny $d_1$}}] (d1) at (12, 0) {};
				
				\node[vertices,label = above:{{\tiny $a_1$}}] (a) at (13.5, 4.5) {};
				\node[vertices,label = below:{{\tiny $c_1$}}] (c) at (13.5, 1.5) {};
				\node[vertices,label = below:{{\tiny $e_1$}}] (e) at (13.5, -1.5) {};
				
				\node[vertices,label = above:{{\tiny $b_2$}}] (b) at (15, 3) {};
				\node[vertices,label = below:{{\tiny $d_2$}}] (d) at (15, 0) {};
				
				\foreach \vertex/\neighbors in {
					z1/{z2,z3},
					z2/{z4},
					z3/{z5},
					z4/{z6,z8},
					z5/{z9},
					z6/{z7,z8},
					z7/{a,b1},
					z8/{b1,c,d1},
					z9/{d1,e},
					b1/{d1},
					a/{b1,b},
					b/{c,d},
					c/{b1,d1,d},
					d/{e},
					e/{d1}}
				{
					\foreach \neighbor in \neighbors
					\draw [-] (\vertex)--(\neighbor);
				}
			\end{tikzpicture}
			\subcaption{$\mathcal{Z}_2$}
			\label{fig:Z_2}
		\end{subfigure}
		\begin{subfigure}{0.4\textwidth}
			\centering
			\begin{tikzpicture}[scale = 0.3, vertices/.style = {draw, fill = black, circle, inner sep = 0.5pt}]
				
				\node[vertices,label = below:{{\tiny $z_1$}}] (z1) at (3, 0) {};
				\node[vertices,label = below:{{\tiny $z_2$}}] (z2) at (4.5, 1.5) {}; 
				\node[vertices,label = below:{{\tiny $z_3$}}] (z3) at (4.5, -1.5) {};
				\node[vertices,label = below:{{\tiny $z_4$}}] (z4) at (7.5, 1.5) {}; 
				\node[vertices,label = below:{{\tiny $z_5$}}] (z5) at (7.5, -1.5) {};
				\node[vertices,label = above:{{\tiny $z_6$}}] (z6) at (9, 3) {};
				\node[vertices,label = above:{{\tiny $z_7$}}] (z7) at (10.5, 4.5) {}; 
				\node[vertices,label = below:{{\tiny $z_8$}}] (z8) at (10.5, 1.5) {}; 
				\node[vertices,label = below:{{\tiny $z_9$}}] (z9) at (10.5, -1.5) {}; 
				
				\node[vertices,label = above:{{\tiny $b_1$}}] (b1) at (12, 3) {};
				\node[vertices,label = below:{{\tiny $d_1$}}] (d1) at (12, 0) {};
				
				\node[vertices,label = above:{{\tiny $a_1$}}] (a) at (13.5, 4.5) {};
				\node[vertices,label = below:{{\tiny $c_1$}}] (c) at (13.5, 1.5) {};
				\node[vertices,label = below:{{\tiny $e_1$}}] (e) at (13.5, -1.5) {};
				
				\node[vertices,label = above:{{\tiny $b_2$}}] (b) at (15, 3) {};
				\node[vertices,label = below:{{\tiny $d_2$}}] (d) at (15, 0) {};
				
				\node[vertices, label = {[xshift = .32cm,yshift = -.12cm]above:{{{\tiny $b_{(n-1)}$}}}}] (1) at ($(b)+(4,0)$) {};
				\node[vertices, label = {[xshift = .3cm]below:{{{\tiny $d_{(n-1)}$}}}}] (2) at ($(d)+(4,0)$) {};
				
				\node[vertices,label = {[xshift = .3cm,yshift = -.13cm]above:{{{\tiny $a_{(n-1)}$}}}}] (3) at ($(a)+(7,0)$) {}; 
				\node[vertices,label = {[xshift = .27cm]below:{{{\tiny $c_{(n-1)}$}}}}] (4) at ($(c)+(7,0)$) {};
				\node[vertices,label = {[xshift = .3cm]below:{{{\tiny $e_{(n-1)}$}}}}] (5) at ($(e)+(7,0)$) {};
				
				\node[vertices,label = above:{{{\tiny $b_n$}}}] (6) at ($(b)+(7,0)$) {};
				\node[vertices,label = below:{{{\tiny $d_n$}}}] (7) at ($(d)+(7,0)$) {};
				
				\draw (a) -- ++(2.3,0);
				\draw (c) -- ++(2.3,0);
				\draw (e) -- ++(2.3,0);
				\draw (3) -- ++(-2.3,0);
				\draw (4) -- ++(-2.3,0);
				\draw (5) -- ++(-2.3,0);
				\draw (b) -- ++(1,1);
				\draw (b) -- ++(1,-1);
				\draw (d) -- ++(1,1);
				\draw (d) -- ++(1,-1);
				\draw (1) -- ++(-1,1);
				\draw (1) -- ++(-1,-1);
				\draw (2) -- ++(-1,1);
				\draw (2) -- ++(-1,-1);
				
				\node at ($(b)!0.5!(1)$) {$\ldots$};
				\node at ($(d)!0.5!(2)$) {$\ldots$};
				\node at ($(a)!0.5!(3)$) {$\ldots$};
				\node at ($(c)!0.5!(4)$) {$\ldots$};
				\node at ($(e)!0.5!(5)$) {$\ldots$};
				
				\foreach \vertex/\neighbors in {
					z1/{z2,z3},
					z2/{z4},
					z3/{z5},
					z4/{z6,z8},
					z5/{z9},
					z6/{z7,z8},
					z7/{a,b1},
					z8/{b1,c,d1},
					z9/{d1,e},
					b1/{d1},
					a/{b1,b},
					b/{c,d},
					c/{b1,d1,d},
					d/{e},
					e/{d1},
					1/{2,3,4},
					2/{4,5},
					3/{6},
					4/{6,7},
					5/{7},
					6/{7}}
				{
					\foreach \neighbor in \neighbors
					\draw [-] (\vertex)--(\neighbor);
				}
			\end{tikzpicture}
			\subcaption{$\mathcal{Z}_n, \ n \geq 2$}
			\label{fig:Zn_2}
		\end{subfigure}
		\vspace{-.2cm}
		\caption{} 
		\label{fig:Zn}
	\end{figure}
	
	\textbf{Case $n = 1:$} Since $N_{\mathcal{Z}_1}(z_7)\subseteq N_{\mathcal{Z}_1}(z_8)$, we have $\Ind(\mathcal{Z}_1)\simeq \Ind(\mathcal{Z}_1-\{z_8\})$. Observe that $\mathcal{Z}_1-\{z_8\}\cong C_{10}$. Hence $\Ind(\mathcal{Z}_1)\simeq \mathbb{S}^2$ by \Cref{theorem:cycle_graph}. 
	
	\textbf{Case $n = 2:$} In $\mathcal{Z}_2$, replacing the path $z_4z_2z_1z_3z_5$ with the edge $(z_4,z_5)$, we get $\mathcal{Z}_2^{(1)}$ (see \Cref{fig:Z2_1}). By \Cref{theorem:replace_edge}, $\Ind(\mathcal{Z}_2)\simeq \Sigma(\Ind(\mathcal{Z}_2^{(1)}))$. 
	Note that $\mathcal{Z}_2^{(1)}-N_{\mathcal{Z}_2^{(1)}}[\{b_1,d_2\}]\cong P_4$. 
	Since $\Ind(P_4)\simeq \text{pt}$, $\Ind(\mathcal{Z}_2^{(1)})\simeq \Ind(\mathcal{Z}_2^{(1)}\cup\{(b_1,d_2)\})$ by \Cref{lemma:edge_contractible}.
	We have $N_{\mathcal{Z}_2^{(1)}\cup\{(b_1,d_2)\}}(b_2)\subseteq N_{\mathcal{Z}_2^{(1)}\cup\{(b_1,d_2)\}}(b_1)$ (see \Cref{fig:Z2'}). Hence $\Ind(\mathcal{Z}_2^{(1)}\cup\{(b_1,d_2)\})\simeq \Ind(\mathcal{Z}_2^{(1)}\cup\{(b_1,d_2)\}-\{b_1\})$. Let $\mathcal{Z}_2^{(2)}: = \mathcal{Z}_2^{(1)}\cup\{(b_1,d_2)\}-\{b_1\}$. Then $\Ind(\mathcal{Z}_2^{(1)})\simeq \Ind(\mathcal{Z}_2^{(2)})$, which implies $\Ind(\mathcal{Z}_2)\simeq \Sigma(\Ind(\mathcal{Z}_2^{(2)})).$
	\begin{figure}[H]
		\centering
		\begin{subfigure}{0.27\textwidth}
			\centering
			\begin{tikzpicture}[scale = 0.250, vertices/.style = {draw, fill = black, circle, inner sep = 0.5pt}]
				
				\node[vertices,label = below:{{\tiny $z_5$}}] (z5) at (10,-2) {};
				\node[vertices,label = above:{{\tiny $z_4$}}] (z4) at (10, 2) {};
				\node[vertices,label = above:{{\tiny $z_6$}}] (z6) at (12, 4) {};
				\node[vertices,label = below:{{\tiny $z_9$}}] (z9) at (14, -2) {};
				\node[vertices,label = below:{{\tiny $z_8$}}] (z8) at (14,2) {};
				\node[vertices,label = above:{{\tiny $z_7$}}] (z7) at (14, 6) {};
				
				\node[vertices,label = below:{{\tiny $d_1$}}] (d1) at (16, 0) {};
				\node[vertices,label = above:{{\tiny $b_1$}}] (b1) at (16, 4) {};
				
				\node[vertices,label = above:{{\tiny $a_1$}}] (a1) at ($(z7)+(4,0)$) {};
				\node[vertices,label = below:{{\tiny $c_1$}}] (c1) at ($(z8)+(4,0)$) {};
				\node[vertices,label = below:{{\tiny $e_1$}}] (e1) at ($(z9)+(4,0)$) {};
				
				\node[vertices,label = below:{{\tiny $d_2$}}] (d2) at ($(d1)+(4,0)$) {};
				\node[vertices,label = above:{{\tiny $b_2$}}] (b2) at ($(b1)+(4,0)$) {};
				
				\foreach \vertex/\neighbors in {
					z7/{b1,a1},
					z4/{z5},
					z9/{d1,e1,z5},
					z6/{z4,z7},
					z8/{c1,z4,z6},
					a1/{b2},
					c1/{d1,d2,b2},
					b1/{z8,d1,c1,a1},
					d1/{z8,e1},
					d2/{e1,b2}}
				{
					\foreach \neighbor in \neighbors
					\draw [-] (\vertex)--(\neighbor);
				}
				
			\end{tikzpicture}
			\caption{$\mathcal{Z}_2^{(1)}$}
			\label{fig:Z2_1}
		\end{subfigure}
		\begin{subfigure}{0.30\textwidth}
			\centering
			\begin{tikzpicture}[scale = 0.250, vertices/.style = {draw, fill = black, circle, inner sep = 0.5pt}]
				
				\node[vertices,label = below:{{\tiny $z_5$}}] (z5) at (10,-2) {};
				\node[vertices,label = above:{{\tiny $z_4$}}] (z4) at (10, 2) {};
				\node[vertices,label = above:{{\tiny $z_6$}}] (z6) at (12, 4) {};
				\node[vertices,label = below:{{\tiny $z_9$}}] (z9) at (14, -2) {};
				\node[vertices,label = below:{{\tiny $z_8$}}] (z8) at (14,2) {};
				\node[vertices,label = above:{{\tiny $z_7$}}] (z7) at (14, 6) {};
				
				\node[vertices,label = below:{{\tiny $d_1$}}] (d1) at (16, 0) {};
				\node[vertices,label = above:{{\tiny $b_1$}}] (b1) at (16, 4) {};
				
				\node[vertices,label = above:{{\tiny $a_1$}}] (a1) at ($(z7)+(4,0)$) {};
				\node[vertices,label = below:{{\tiny $c_1$}}] (c1) at ($(z8)+(4,0)$) {};
				\node[vertices,label = below:{{\tiny $e_1$}}] (e1) at ($(z9)+(4,0)$) {};
				
				\node[vertices,label = below:{{\tiny $d_2$}}] (d2) at ($(d1)+(4,0)$) {};
				\node[vertices,label = above:{{\tiny $b_2$}}] (b2) at ($(b1)+(4,0)$) {};
				
				\foreach \vertex/\neighbors in {
					z7/{b1,a1},
					z4/{z5},
					z9/{d1,e1,z5},
					z6/{z4,z7},
					z8/{c1,z4,z6},
					a1/{b2},
					c1/{d1,d2,b2},
					b1/{z8,d1,c1,a1},
					d1/{z8,e1},
					d2/{e1,b2}}
				{
					\foreach \neighbor in \neighbors
					\draw [-] (\vertex)--(\neighbor);
				}
				\draw [-, bend left = 25] (b1) to (d2);
				
			\end{tikzpicture}
			\caption{$\mathcal{Z}_2^{(1)}\cup\{(b_1,d_2)\}$}
			\label{fig:Z2'}
		\end{subfigure}
		\begin{subfigure}{0.30\textwidth}
			\centering
			\begin{tikzpicture}[scale = 0.250, vertices/.style = {draw, fill = black, circle, inner sep = 0.5pt}]
				
				\node[vertices,label = below:{{\tiny $z_5$}}] (z5) at (10,-2) {};
				\node[vertices,label = above:{{\tiny $z_4$}}] (z4) at (10, 2) {};
				\node[vertices,label = above:{{\tiny $z_6$}}] (z6) at (12, 4) {};
				\node[vertices,label = below:{{\tiny $z_9$}}] (z9) at (14, -2) {};
				\node[vertices,label = below:{{\tiny $z_8$}}] (z8) at (14,2) {};
				\node[vertices,label = above:{{\tiny $z_7$}}] (z7) at (14, 6) {};
				
				\node[vertices,label = below:{{\tiny $d_1$}}] (d1) at (16, 0) {};
				\node[vertices,label = above:{{\tiny $a_1$}}] (a1) at ($(z7)+(4,0)$) {};
				\node[vertices,label = below:{{\tiny $c_1$}}] (c1) at ($(z8)+(4,0)$) {};
				\node[vertices,label = below:{{\tiny $e_1$}}] (e1) at ($(z9)+(4,0)$) {};
				
				\node[vertices,label = below:{{\tiny $d_2$}}] (d2) at ($(d1)+(4,0)$) {};
				\node[vertices,label = above:{{\tiny $b_2$}}] (b2) at ($(b1)+(4,0)$) {};
				
				\foreach \vertex/\neighbors in {
					z7/{a1},
					z4/{z5},
					z9/{d1,e1,z5},
					z6/{z4,z7},
					z8/{c1,z4,z6},
					a1/{b2},
					c1/{d1,d2,b2},
					d1/{z8,e1},
					d2/{e1,b2}}
				{
					\foreach \neighbor in \neighbors
					\draw [-] (\vertex)--(\neighbor);
				}
			\end{tikzpicture}
			\caption{$\mathcal{Z}_2^{(2)}$}
			\label{fig:Z2_2}
		\end{subfigure} 
		\caption{}
		\label{fig:Z2s}
	\end{figure}
	
	Observe that $\mathcal{Z}_2^{(2)}-N_{\mathcal{Z}_2^{(2)}}[\{z_8,z_9\}]\cong P_4$. Hence $\Ind(\mathcal{Z}_2^{(2)})\simeq \Ind(\mathcal{Z}_2^{(2)}\cup\{(z_8,z_9)\})$ (see \Cref{fig:Z2_2_(z8z9)}). 
	We have $N_{\mathcal{Z}_2^{(2)}\cup\{(z_8,z_9)\}}(z_5)\subseteq N_{\mathcal{Z}_2^{(2)}\cup\{(z_8,z_9)\}}(z_8)$. By \Cref{theorem:N(u)subsetN(v)}, $\Ind(\mathcal{Z}_2^{(2)}\cup\{(z_8,z_9)\})\simeq \Ind(\mathcal{Z}_2^{(2)}\cup\{(z_8,z_9)\}-\{z_8\})$. 
	In $\mathcal{Z}_2^{(2)}\cup\{(z_8,z_9)\}-\{z_8\}$, replacing the path $z_7z_6z_4z_5z_9$ with the edge $(z_7,z_9)$, we get $\mathcal{Z}_2^{(3)}$ (see \Cref{fig:Z2_3}). By \Cref{theorem:replace_edge}, $\Ind(\mathcal{Z}_2^{(2)})\simeq \Sigma(\Ind(\mathcal{Z}_2^{(3)}))$. It follows that $\Ind(\mathcal{Z}_2)\simeq \Sigma^{2}(\Ind(\mathcal{Z}_2^{(3)})).$ 
	
	\begin{wrapfigure}{r}{0.52\textwidth} 
		\centering
		\begin{subfigure}{0.2\textwidth}
			\centering
			\begin{tikzpicture}[scale = 0.250, vertices/.style = {draw, fill = black, circle, inner sep = 0.5pt}]
				
				\node[vertices,label = below:{{\tiny $z_5$}}] (z5) at (10,-2) {};
				\node[vertices,label = above:{{\tiny $z_4$}}] (z4) at (10, 2) {};
				\node[vertices,label = above:{{\tiny $z_6$}}] (z6) at (12, 4) {};
				\node[vertices,label = below:{{\tiny $z_9$}}] (z9) at (14, -2) {};
				\node[vertices,label = above:{{\tiny $z_8$}}] (z8) at (14,2) {};
				\node[vertices,label = above:{{\tiny $z_7$}}] (z7) at (14, 6) {};
				
				\node[vertices,label = below:{{\tiny $d_1$}}] (d1) at (16, 0) {};
				\node[vertices,label = above:{{\tiny $a_1$}}] (a1) at ($(z7)+(4,0)$) {};
				\node[vertices,label = below:{{\tiny $c_1$}}] (c1) at ($(z8)+(4,0)$) {};
				\node[vertices,label = below:{{\tiny $e_1$}}] (e1) at ($(z9)+(4,0)$) {};
				
				\node[vertices,label = below:{{\tiny $d_2$}}] (d2) at ($(d1)+(4,0)$) {};
				\node[vertices,label = above:{{\tiny $b_2$}}] (b2) at ($(16, 4)+(4,0)$) {};
				
				\foreach \vertex/\neighbors in {
					z7/{a1},
					z4/{z5},
					z9/{d1,e1,z5},
					z6/{z4,z7},
					z8/{c1,z4,z9,z6},
					a1/{b2},
					c1/{d1,d2,b2},
					d1/{z8,e1},
					d2/{e1,b2}}
				{
					\foreach \neighbor in \neighbors
					\draw [-] (\vertex)--(\neighbor);
				}
			\end{tikzpicture} 
			\caption{$\mathcal{Z}_2^{(2)}\cup \{(z_8,z_9)\}$}
			\label{fig:Z2_2_(z8z9)}
		\end{subfigure} 
		\begin{subfigure}{0.15\textwidth}
			\centering
			\begin{tikzpicture}[scale = 0.250, vertices/.style = {draw, fill = black, circle, inner sep = 0.5pt}]
				
				\node[vertices,label = below:{{\tiny $z_9$}}] (z9) at (14, -2) {};
				\node[vertices,label = above:{{\tiny $z_7$}}] (z7) at (14, 6) {};
				
				\node[vertices,label = below:{{\tiny $d_1$}}] (d1) at (16, 0) {};
				\node[vertices,label = above:{{\tiny $a_1$}}] (a1) at ($(z7)+(4,0)$) {};
				\node[vertices,label = below:{{\tiny $c_1$}}] (c1) at ($(z8)+(4,0)$) {};
				\node[vertices,label = below:{{\tiny $e_1$}}] (e1) at ($(z9)+(4,0)$) {};
				
				\node[vertices,label = below:{{\tiny $d_2$}}] (d2) at ($(d1)+(4,0)$) {};
				\node[vertices,label = above:{{\tiny $b_2$}}] (b2) at ($(16, 4)+(4,0)$) {};
				
				\foreach \vertex/\neighbors in {
					z7/{a1,z9},
					z9/{d1,e1},
					a1/{b2},
					c1/{d1,d2,b2},
					d1/{e1},
					d2/{e1,b2}}
				{
					\foreach \neighbor in \neighbors
					\draw [-] (\vertex)--(\neighbor);
				}
			\end{tikzpicture}
			\caption{$\mathcal{Z}_2^{(3)}$}
			\label{fig:Z2_3}
		\end{subfigure}
		\begin{subfigure}{0.15\textwidth}
			\centering
			\begin{tikzpicture}[scale = 0.250, vertices/.style = {draw, fill = black, circle, inner sep = 0.5pt}]
				
				\node[vertices,label = below:{{\tiny $z_9$}}] (z9) at (14, -2) {};
				\node[vertices,label = above:{{\tiny $z_7$}}] (z7) at (14, 6) {};
				\node[vertices,label = above:{{\tiny $a_1$}}] (a1) at ($(z7)+(4,0)$) {};
				\node[vertices,label = below:{{\tiny $c_1$}}] (c1) at ($(z8)+(4,0)$) {};
				\node[vertices,label = below:{{\tiny $e_1$}}] (e1) at ($(z9)+(4,0)$) {};
				
				\node[vertices,label = below:{{\tiny $d_2$}}] (d2) at ($(d1)+(4,0)$) {};
				\node[vertices,label = above:{{\tiny $b_2$}}] (b2) at ($(16, 4)+(4,0)$) {};
				
				\foreach \vertex/\neighbors in {
					z7/{a1,z9},
					z9/{e1},
					a1/{b2},
					c1/{d2,b2},
					d2/{e1,b2}}
				{
					\foreach \neighbor in \neighbors
					\draw [-] (\vertex)--(\neighbor);
				}
			\end{tikzpicture} 
			\caption{$\mathcal{Z}_2^{(4)}$}
			\label{fig:Z2_4}
		\end{subfigure} 
		\vspace{-.2cm} 
		\caption{}
		\label{fig:Z2seq}
		\vspace{-1cm} 
	\end{wrapfigure}
	
	Since $\mathcal{Z}_2^{(3)}-N_{\mathcal{Z}_2^{(3)}}[\{d_1,a_1\}]\cong P_1$ and $\Ind(P_1)\simeq \text{pt},$ \Cref{lemma:edge_contractible} implies that $\Ind(\mathcal{Z}_2^{(3)})\simeq \Ind(\mathcal{Z}_2^{(3)}\cup\{(d_1,a_1)\})$. 
	Now, note that $N_{\mathcal{Z}_2^{(3)}\cup\{(d_1,a_1)\}}(z_7)\subseteq N_{\mathcal{Z}_2^{(3)}\cup\{(d_1,a_1)\}}(d_1)$. Hence $\Ind(\mathcal{Z}_2^{(3)}\cup\{(d_1,a_1)\})\simeq \Ind(\mathcal{Z}_2^{(3)}\cup\{(d_1,a_1)\}-\{d_1\})$. 
	Let $\mathcal{Z}_2^{(4)}: = \mathcal{Z}_2^{(3)}\cup\{(d_1,a_1)\}-\{d_1\}$. Then $\Ind(\mathcal{Z}_2^{(3)})\simeq \Ind(\mathcal{Z}_2^{(4)})$, and therefore $\Ind(\mathcal{Z}_2)\simeq \Sigma^2(\Ind(\mathcal{Z}_2^{(4)})).$
	
	Observe that $c_1$ is a simplicial vertex in $\mathcal{Z}_2^{(4)}$ with $N_{\mathcal{Z}_2^{(4)}}(c_1) = \{b_2,d_2\}$ (see \Cref{fig:Z2_4}). By \Cref{theorem:simplicial_vertex}, $\Ind(\mathcal{Z}_2^{(4)})\simeq\Sigma(\Ind(\mathcal{Z}_2^{(4)}-N_{\mathcal{Z}_2^{(4)}}[b_2])) \vee \Sigma(\Ind(\mathcal{Z}_2^{(4)}-N_{\mathcal{Z}_2^{(4)}}[d_2]))$. 
	We have $\mathcal{Z}_2^{(4)}-N_{\mathcal{Z}_2^{(4)}}[b_2]\cong P_3$ and $\mathcal{Z}_2^{(4)}-N_{\mathcal{Z}_2^{(4)}}[d_2]\cong P_3$. It follows that $\Ind(\mathcal{Z}_2)\simeq \Sigma^2(\Ind(\mathcal{Z}_2^{(4)}))\simeq \vee_2\Sigma^3(\Ind(P_3))$. Since $\Ind(P_3)\simeq \mathbb{S}^0$ by \Cref{lemma:path_graph}, we get $\Ind(\mathcal{Z}_2)\simeq\Sigma^3(\vee_2(\mathbb{S}^0))\simeq \vee_2\mathbb{S}^3$.
	
	Thus,
	\begin{equation}\label{equation:Z_12}
		\Ind(\mathcal{Z}_1)\simeq \mathbb{S}^2 \text{ and } \Ind(\mathcal{Z}_2)\simeq \vee_2 \mathbb{S}^3.
	\end{equation}
	
	\subsubsection{Recursive homotopy equivalences}\label{subsection:recursive_homotopy_equivalences}
	
	In this section, we study the independence complexes of the six auxiliary graph families defined in \Cref{subsection:graph_definitions}. We derive recursive relations among the independence complexes of these auxiliary graphs and $\Ind(\Gamma_n)$. For $\mathcal{U}_n$, $\mathcal{V}_n$, $\mathcal{X}_n$ and $\mathcal{Y}_n$, we directly establish the recursive homotopy equivalences for their independence complexes. For $\mathcal{W}_n$ and $\mathcal{Z}_n$, we instead determine homotopy equivalences for the link and deletion of suitable vertices. 
	\begin{claim}\label{claim:Un_relation}
		For $n\ge3$,
		$$\Ind(\mathcal{U}_n)\simeq
		\begin{cases}
			\vee_4\mathbb{S}^2 &\text{ if $n = 3$},\\
			\Sigma(\Ind(\mathcal{U}_{n-1}))\vee \Sigma(\Ind(\mathcal{Z}_{n-3}))\vee \Sigma^3(\Ind(\mathcal{U}_{n-3})) &\text{ if $n\ge4$}.
		\end{cases}$$
	\end{claim}
	
	\begin{proof} 
		Note that for $n\ge3$, $u$ is a simplicial vertex in $\mathcal{U}_n$ with $N_{\mathcal{U}_n}(u) = \{d_1, e_1\}$ (see \Cref{fig:Un_2}). By \Cref{theorem:simplicial_vertex}, $\Ind(\mathcal{U}_n)\simeq \Sigma(\Ind(\mathcal{U}_n-N_{\mathcal{U}_n}[d_1]))\vee \Sigma(\Ind(\mathcal{U}_n-N_{\mathcal{U}_n}[e_1]))$.
		We have $\mathcal{U}_n-N_{\mathcal{U}_n}[d_1]\cong \mathcal{U}_{n-1}$ (see \Cref{fig:Un-N_d1}). Let $\mathcal{U}_n': = \mathcal{U}_n-N_{\mathcal{U}_n}[e_1]$. Then $\Ind(\mathcal{U}_n)\simeq \Sigma(\Ind(\mathcal{U}_{n-1}))\vee\Sigma(\Ind(\mathcal{U}_n'))$.
		\begin{figure}[h!]
			\centering
			\begin{subfigure}{0.3\textwidth}
				\centering
				\begin{tikzpicture}[scale = 0.27, vertices/.style = {draw, fill = black, circle, inner sep = 0.5pt}]
					
					\node[vertices,label = above:{{\tiny $a_1$}}] (a1) at (13.5, 4.5) {}; 
					
					\node[vertices,label = above:{{\tiny $b_2$}}] (b2) at (15, 3) {};
					\node[vertices,label = below:{{\tiny $d_2$}}] (d2) at (15, 0) {};
					
					\node[vertices,label = above:{{\tiny $a_2$}}] (a2) at (16.5, 4.5) {};
					\node[vertices,label = below:{{\tiny $c_2$}}] (c2) at (16.5, 1.5) {};
					\node[vertices,label = below:{{\tiny $e_2$}}] (e2) at (16.5, -1.5) {};
					
					\node[vertices,label = above:{{\tiny $b_3$}}] (b3) at (18, 3) {};
					\node[vertices,label = below:{{\tiny $d_3$}}] (d3) at (18, 0) {};
					
					\node[vertices,label = above:{{\tiny $a_3$}}] (a) at (19.5, 4.5) {};
					\node[vertices,label = below:{{\tiny $c_3$}}] (c) at (19.5, 1.5) {};
					\node[vertices,label = below:{{\tiny $e_3$}}] (e) at (19.5, -1.5) {};
					
					\node[vertices,label = above:{{\tiny $b_4$}}] (b) at (21, 3) {};
					\node[vertices,label = below:{{\tiny $d_4$}}] (d) at (21, 0) {};
					
					\node[vertices, label = {[xshift = .32cm,yshift = -.12cm]above:{{{\tiny $b_{(n-1)}$}}}}] (1) at ($(b)+(4,0)$) {};
					\node[vertices, label = {[xshift = .3cm]below:{{{\tiny $d_{(n-1)}$}}}}] (2) at ($(d)+(4,0)$) {};
					
					\node[vertices,label = {[xshift = .3cm,yshift = -.13cm]above:{{{\tiny $a_{(n-1)}$}}}}] (3) at ($(a)+(7,0)$) {}; 
					\node[vertices,label = {[xshift = 0cm]right:{{{\tiny $c_{(n-1)}$}}}}] (4) at ($(c)+(7,0)$) {};
					\node[vertices,label = {[xshift = .3cm]below:{{{\tiny $e_{(n-1)}$}}}}] (5) at ($(e)+(7,0)$) {};
					
					\node[vertices,label = right:{{{\tiny $b_n$}}}] (6) at ($(b)+(7,0)$) {};
					\node[vertices,label = {[xshift = 0cm]right:{{{\tiny $d_n$}}}}] (7) at ($(d)+(7,0)$) {};
					
					\draw (a) -- ++(2.3,0);
					\draw (c) -- ++(2.3,0);
					\draw (e) -- ++(2.3,0);
					\draw (3) -- ++(-2.3,0);
					\draw (4) -- ++(-2.3,0);
					\draw (5) -- ++(-2.3,0);
					\draw (b) -- ++(1,1);
					\draw (b) -- ++(1,-1);
					\draw (d) -- ++(1,1);
					\draw (d) -- ++(1,-1);
					\draw (1) -- ++(-1,1);
					\draw (1) -- ++(-1,-1);
					\draw (2) -- ++(-1,1);
					\draw (2) -- ++(-1,-1);
					
					\node at ($(b)!0.5!(1)$) {$\ldots$};
					\node at ($(d)!0.5!(2)$) {$\ldots$};
					\node at ($(a)!0.5!(3)$) {$\ldots$};
					\node at ($(c)!0.5!(4)$) {$\ldots$};
					\node at ($(e)!0.5!(5)$) {$\ldots$};
					
					\foreach \vertex/\neighbors in {
						a1/{a2,b2},
						b2/{a2,c2,d2},
						d2/{c2,e2},
						a2/{b3},
						c2/{b3,d3},
						e2/{d3},
						b3/{d3},
						a/{a2,b3,b},
						b/{c,d},
						c/{c2,b3,d3,d},
						d/{e},
						e/{d3,e2},
						1/{2,3,4},
						2/{4,5},
						3/{6},
						4/{6,7},
						5/{7},
						6/{7}}
					{
						\foreach \neighbor in \neighbors
						\draw [-] (\vertex)--(\neighbor);
					}
				\end{tikzpicture}
				\subcaption{$\mathcal{U}_n-N_{\mathcal{U}_n}[d_1]$}
				\label{fig:Un-N_d1}
			\end{subfigure}
			\hfill
			\begin{subfigure}{0.33\textwidth}
				\centering
				\begin{tikzpicture}[scale = 0.27, vertices/.style = {draw, fill = black, circle, inner sep = 0.5pt}]
					
					\node[vertices,label = above:{{\tiny $b_1$}}] (b1) at (12, 3) {};
					
					\node[vertices,label = above:{{\tiny $a_1$}}] (a1) at (13.5, 4.5) {}; 
					\node[vertices,label = below:{{\tiny $c_1$}}] (c1) at (13.5, 1.5) {};
					
					\node[vertices,label = above:{{\tiny $b_2$}}] (b2) at (15, 3) {};
					
					\node[vertices,label = above:{{\tiny $a_2$}}] (a2) at (16.5, 4.5) {};
					\node[vertices,label = below:{{\tiny $c_2$}}] (c2) at (16.5, 1.5) {};
					
					\node[vertices,label = above:{{\tiny $b_3$}}] (b3) at (18, 3) {};
					\node[vertices,label = below:{{\tiny $d_3$}}] (d3) at (18, 0) {};
					
					\node[vertices,label = above:{{\tiny $a_3$}}] (a) at (19.5, 4.5) {};
					\node[vertices,label = below:{{\tiny $c_3$}}] (c) at (19.5, 1.5) {};
					\node[vertices,label = below:{{\tiny $e_3$}}] (e) at (19.5, -1.5) {};
					
					\node[vertices,label = above:{{\tiny $b_4$}}] (b) at (21, 3) {};
					\node[vertices,label = below:{{\tiny $d_4$}}] (d) at (21, 0) {};
					
					\node[vertices, label = {[xshift = .32cm,yshift = -.12cm]above:{{{\tiny $b_{(n-1)}$}}}}] (1) at ($(b)+(4,0)$) {};
					\node[vertices, label = {[xshift = .3cm]below:{{{\tiny $d_{(n-1)}$}}}}] (2) at ($(d)+(4,0)$) {};
					
					\node[vertices,label = {[xshift = .3cm,yshift = -.13cm]above:{{{\tiny $a_{(n-1)}$}}}}] (3) at ($(a)+(7,0)$) {}; 
					\node[vertices,label = {[xshift = 0cm]right:{{{\tiny $c_{(n-1)}$}}}}] (4) at ($(c)+(7,0)$) {};
					\node[vertices,label = {[xshift = .3cm]below:{{{\tiny $e_{(n-1)}$}}}}] (5) at ($(e)+(7,0)$) {};
					
					\node[vertices,label = right:{{{\tiny $b_n$}}}] (6) at ($(b)+(7,0)$) {};
					\node[vertices,label = right:{{{\tiny $d_n$}}}] (7) at ($(d)+(7,0)$) {};
					
					\draw (a) -- ++(2.3,0);
					\draw (c) -- ++(2.3,0);
					\draw (e) -- ++(2.3,0);
					\draw (3) -- ++(-2.3,0);
					\draw (4) -- ++(-2.3,0);
					\draw (5) -- ++(-2.3,0);
					\draw (b) -- ++(1,1);
					\draw (b) -- ++(1,-1);
					\draw (d) -- ++(1,1);
					\draw (d) -- ++(1,-1);
					\draw (1) -- ++(-1,1);
					\draw (1) -- ++(-1,-1);
					\draw (2) -- ++(-1,1);
					\draw (2) -- ++(-1,-1);
					
					\node at ($(b)!0.5!(1)$) {$\ldots$};
					\node at ($(d)!0.5!(2)$) {$\ldots$};
					\node at ($(a)!0.5!(3)$) {$\ldots$};
					\node at ($(c)!0.5!(4)$) {$\ldots$};
					\node at ($(e)!0.5!(5)$) {$\ldots$};
					
					\foreach \vertex/\neighbors in {
						a1/{b1,a2,b2},
						c1/{b1,b2,c2},
						b2/{a2,c2},
						a2/{b3},
						c2/{b3,d3},
						b3/{d3},
						a/{a2,b3,b},
						b/{c,d},
						c/{c2,b3,d3,d},
						d/{e},
						e/{d3},
						1/{2,3,4},
						2/{4,5},
						3/{6},
						4/{6,7},
						5/{7},
						6/{7}}
					{
						\foreach \neighbor in \neighbors
						\draw [-] (\vertex)--(\neighbor);
					}
				\end{tikzpicture}
				\subcaption{$\mathcal{U}_n'$}
			\end{subfigure}
			\begin{subfigure}{0.34\textwidth}
				\centering
				\begin{tikzpicture}[scale = 0.27, vertices/.style = {draw, fill = black, circle, inner sep = 0.5pt}]
					
					\node[vertices,label = above:{{\tiny $b_1$}}] (b1) at (12, 3) {};
					
					\node[vertices,label = above:{{\tiny $a_1$}}] (a1) at (13.5, 4.5) {}; 
					\node[vertices,label = below:{{\tiny $c_1$}}] (c1) at (13.5, 1.5) {};
					
					\node[vertices,label = above:{{\tiny $a_2$}}] (a2) at (16.5, 4.5) {};
					\node[vertices,label = below:{{\tiny $c_2$}}] (c2) at (16.5, 1.5) {};
					
					\node[vertices,label = above:{{\tiny $b_3$}}] (b3) at (18, 3) {};
					\node[vertices,label = below:{{\tiny $d_3$}}] (d3) at (18, 0) {};
					
					\node[vertices,label = above:{{\tiny $a_3$}}] (a) at (19.5, 4.5) {};
					\node[vertices,label = below:{{\tiny $c_3$}}] (c) at (19.5, 1.5) {};
					\node[vertices,label = below:{{\tiny $e_3$}}] (e) at (19.5, -1.5) {};
					
					\node[vertices,label = above:{{\tiny $b_4$}}] (b) at (21, 3) {};
					\node[vertices,label = below:{{\tiny $d_4$}}] (d) at (21, 0) {};
					
					\node[vertices, label = {[xshift = .32cm,yshift = -.12cm]above:{{{\tiny $b_{(n-1)}$}}}}] (1) at ($(b)+(4,0)$) {};
					\node[vertices, label = {[xshift = .3cm]below:{{{\tiny $d_{(n-1)}$}}}}] (2) at ($(d)+(4,0)$) {};
					
					\node[vertices,label = {[xshift = .3cm,yshift = -.13cm]above:{{{\tiny $a_{(n-1)}$}}}}] (3) at ($(a)+(7,0)$) {}; 
					\node[vertices,label = {[xshift = 0cm]right:{{{\tiny $c_{(n-1)}$}}}}] (4) at ($(c)+(7,0)$) {};
					\node[vertices,label = {[xshift = .3cm]below:{{{\tiny $e_{(n-1)}$}}}}] (5) at ($(e)+(7,0)$) {};
					
					\node[vertices,label = right:{{{\tiny $b_n$}}}] (6) at ($(b)+(7,0)$) {};
					\node[vertices,label = right:{{{\tiny $d_n$}}}] (7) at ($(d)+(7,0)$) {};
					
					\draw (a) -- ++(2.3,0);
					\draw (c) -- ++(2.3,0);
					\draw (e) -- ++(2.3,0);
					\draw (3) -- ++(-2.3,0);
					\draw (4) -- ++(-2.3,0);
					\draw (5) -- ++(-2.3,0);
					\draw (b) -- ++(1,1);
					\draw (b) -- ++(1,-1);
					\draw (d) -- ++(1,1);
					\draw (d) -- ++(1,-1);
					\draw (1) -- ++(-1,1);
					\draw (1) -- ++(-1,-1);
					\draw (2) -- ++(-1,1);
					\draw (2) -- ++(-1,-1);
					
					\node at ($(b)!0.5!(1)$) {$\ldots$};
					\node at ($(d)!0.5!(2)$) {$\ldots$};
					\node at ($(a)!0.5!(3)$) {$\ldots$};
					\node at ($(c)!0.5!(4)$) {$\ldots$};
					\node at ($(e)!0.5!(5)$) {$\ldots$};
					
					\foreach \vertex/\neighbors in {
						a1/{b1,a2},
						c1/{b1,c2},
						a2/{b3},
						c2/{b3,d3},
						b3/{d3},
						a/{a2,b3,b},
						b/{c,d},
						c/{c2,b3,d3,d},
						d/{e},
						e/{d3},
						1/{2,3,4},
						2/{4,5},
						3/{6},
						4/{6,7},
						5/{7},
						6/{7}}
					{
						\foreach \neighbor in \neighbors
						\draw [-] (\vertex)--(\neighbor);
					}
				\end{tikzpicture}
				\subcaption{$\mathcal{U}_n'-\{b_2\}$}
				\label{fig:Un'_b2}
			\end{subfigure}
			\vspace{-.18cm}
			\caption{}
			\label{fig:Un_seq}
		\end{figure}
		
		For $n = 3$, $\mathcal{U}_3'\cong \mathcal{U}_2$. Hence $\Ind(\mathcal{U}_3)\simeq \vee_2\Sigma(\Ind(\mathcal{U}_2))\simeq \vee_4 \mathbb{S}^2$ by \eqref{equation:U_12}. 
		
		Now, let $n\ge4$. Since $N_{\mathcal{U}_n'}(b_1)\subseteq N_{\mathcal{U}_n'}(b_2)$, \Cref{theorem:N(u)subsetN(v)} implies $\Ind(\mathcal{U}_n')\simeq \Ind(\mathcal{U}_n'-\{b_2\})$. Note that $(\mathcal{U}_n'-\{b_2\})-N_{\mathcal{U}_n'-\{b_2\}}[\{a_1,b_3\}]$ contains an isolated vertex $c_1$ (\Cref{fig:Un'_b2}), and hence $\Ind((\mathcal{U}_n'-\{b_2\})-N_{\mathcal{U}_n'-\{b_2\}}[\{a_1,b_3\}])$ is contractible. By \Cref{lemma:edge_contractible}, $ \Ind(\mathcal{U}_n'-\{b_2\})\simeq \Ind((\mathcal{U}_n'-\{b_2\})\cup \{(a_1,b_3)\})$. Let $\mathcal{U}_n'': = (\mathcal{U}_n'-\{b_2\})\cup \{(a_1,b_3)\}$ (see \Cref{fig:U_n''}). Then $\Ind(\mathcal{U}_n')\simeq \Ind(\mathcal{U}_n'')$, which implies $\Ind(\mathcal{U}_n)\simeq\Sigma(\Ind(\mathcal{U}_{n-1}))\vee \Sigma(\Ind(\mathcal{U}_n''))$. 
		
		Since $N_{\mathcal{U}_n''}[a_2]\subseteq N_{\mathcal{U}_n''}[b_3]$, \Cref{corollary:N[u]subsetN[v]_lk(v)_del(v)} implies $\Ind(\mathcal{U}_n'')\simeq \Ind(\mathcal{U}_n''-\{b_3\})\vee\Sigma(\Ind(\mathcal{U}_n''-N_{\mathcal{U}_n''}[b_3]))$. Observe that $\mathcal{U}_n''-N_{\mathcal{U}_n''}[b_3]\cong\mathcal{U}_{n-3}\sqcup P_2$ (\Cref{fig:U_n''-N_b3}) and $\mathcal{U}_n''-\{b_3\}\cong\mathcal{Z}_{n-3}$ (\Cref{fig_U_n''-b3}).
		Therefore $\Ind(\mathcal{U}_n'')\simeq \Ind(\mathcal{Z}_{n-3})\vee \Sigma(\Ind(\mathcal{U}_{n-3}\sqcup P_2))$. 
		\begin{figure}[H]
			\centering
			\begin{subfigure}{0.33\textwidth}
				\centering
				\begin{tikzpicture}[scale = 0.27, vertices/.style = {draw, fill = black, circle, inner sep = 0.5pt}]
					
					\node[vertices,label = above:{{\tiny $b_1$}}] (b1) at (12.3, 3) {};
					
					\node[vertices,label = above:{{\tiny $a_1$}}] (a1) at (13.5, 4.5) {}; 
					\node[vertices,label = below:{{\tiny $c_1$}}] (c1) at (13.5, 1.5) {};
					
					\node[vertices,label = above:{{\tiny $a_2$}}] (a2) at (16.5, 4.5) {};
					\node[vertices,label = below:{{\tiny $c_2$}}] (c2) at (16.5, 1.5) {};
					
					\node[vertices,label = above:{{\tiny $b_3$}}] (b3) at (18, 3) {};
					\node[vertices,label = below:{{\tiny $d_3$}}] (d3) at (18, 0) {};
					
					\node[vertices,label = above:{{\tiny $a_3$}}] (a) at (19.5, 4.5) {};
					\node[vertices,label = below:{{\tiny $c_3$}}] (c) at (19.5, 1.5) {};
					\node[vertices,label = below:{{\tiny $e_3$}}] (e) at (19.5, -1.5) {};
					
					\node[vertices,label = above:{{\tiny $b_4$}}] (b) at (21, 3) {};
					\node[vertices,label = below:{{\tiny $d_4$}}] (d) at (21, 0) {};
					
					\node[vertices, label = {[xshift = .32cm,yshift = -.12cm]above:{{{\tiny $b_{(n-1)}$}}}}] (1) at ($(b)+(4,0)$) {};
					\node[vertices, label = {[xshift = .3cm]below:{{{\tiny $d_{(n-1)}$}}}}] (2) at ($(d)+(4,0)$) {};
					
					\node[vertices,label = {[xshift = .3cm,yshift = -.13cm]above:{{{\tiny $a_{(n-1)}$}}}}] (3) at ($(a)+(7,0)$) {}; 
					\node[vertices,label = {[xshift = 0cm]right:{{{\tiny $c_{(n-1)}$}}}}] (4) at ($(c)+(7,0)$) {};
					\node[vertices,label = {[xshift = .3cm]below:{{{\tiny $e_{(n-1)}$}}}}] (5) at ($(e)+(7,0)$) {};
					
					\node[vertices,label = right:{{{\tiny $b_n$}}}] (6) at ($(b)+(7,0)$) {};
					\node[vertices,label = right:{{{\tiny $d_n$}}}] (7) at ($(d)+(7,0)$) {};
					
					\draw (a) -- ++(2.3,0);
					\draw (c) -- ++(2.3,0);
					\draw (e) -- ++(2.3,0);
					\draw (3) -- ++(-2.3,0);
					\draw (4) -- ++(-2.3,0);
					\draw (5) -- ++(-2.3,0);
					\draw (b) -- ++(1,1);
					\draw (b) -- ++(1,-1);
					\draw (d) -- ++(1,1);
					\draw (d) -- ++(1,-1);
					\draw (1) -- ++(-1,1);
					\draw (1) -- ++(-1,-1);
					\draw (2) -- ++(-1,1);
					\draw (2) -- ++(-1,-1);
					
					\node at ($(b)!0.5!(1)$) {$\ldots$};
					\node at ($(d)!0.5!(2)$) {$\ldots$};
					\node at ($(a)!0.5!(3)$) {$\ldots$};
					\node at ($(c)!0.5!(4)$) {$\ldots$};
					\node at ($(e)!0.5!(5)$) {$\ldots$};
					
					\foreach \vertex/\neighbors in {
						a1/{b1,a2,b3},
						c1/{b1,c2},
						a2/{b3},
						c2/{b3,d3},
						b3/{d3},
						a/{a2,b3,b},
						b/{c,d},
						c/{c2,b3,d3,d},
						d/{e},
						e/{d3},
						1/{2,3,4},
						2/{4,5},
						3/{6},
						4/{6,7},
						5/{7},
						6/{7}}
					{
						\foreach \neighbor in \neighbors
						\draw [-] (\vertex)--(\neighbor);
					}
				\end{tikzpicture}
				\subcaption{$\mathcal{U}_n''$}
				\label{fig:U_n''}
			\end{subfigure}
			\hfill
			\begin{subfigure}{0.3\textwidth}
				\centering
				\begin{tikzpicture}[scale = 0.27, vertices/.style = {draw, fill = black, circle, inner sep = 0.5pt}]
					
					\node[vertices,label = above:{{\tiny $b_1$}}] (b1) at (17, 3) {};
					
					\node[vertices,label = below:{{\tiny $c_1$}}] (c1) at (18.5, 1.5) {};
					
					\node[vertices,label = below:{{\tiny $e_3$}}] (e3) at (19.5, -1.5) {};
					
					\node[vertices,label = above:{{\tiny $b_4$}}] (b4) at (21, 3) {};
					\node[vertices,label = below:{{\tiny $d_4$}}] (d4) at (21, 0) {};
					
					\node[vertices,label = above:{{\tiny $a_4$}}] (a) at (22.5,4.5) {};
					\node[vertices,label = below:{{\tiny $c_4$}}] (c) at (22.5,1.5) {};
					\node[vertices,label = below:{{\tiny $e_4$}}] (e) at (22.5,-1.5) {};
					
					\node[vertices,label = above:{{\tiny $b_5$}}] (b) at (24, 3) {};
					\node[vertices,label = below:{{\tiny $d_5$}}] (d) at (24, 0) {};
					
					\node[vertices, label = {[xshift = .32cm,yshift = -.12cm]above:{{{\tiny $b_{(n-1)}$}}}}] (1) at ($(b)+(4,0)$) {};
					\node[vertices, label = {[xshift = .3cm]below:{{{\tiny $d_{(n-1)}$}}}}] (2) at ($(d)+(4,0)$) {};
					
					\node[vertices,label = {[xshift = .3cm,yshift = -.13cm]above:{{{\tiny $a_{(n-1)}$}}}}] (3) at ($(a)+(7,0)$) {}; 
					\node[vertices,label = {[xshift = 0cm]right:{{{\tiny $c_{(n-1)}$}}}}] (4) at ($(c)+(7,0)$) {};
					\node[vertices,label = {[xshift = .3cm]below:{{{\tiny $e_{(n-1)}$}}}}] (5) at ($(e)+(7,0)$) {};
					
					\node[vertices,label = right:{{{\tiny $b_n$}}}] (6) at ($(b)+(7,0)$) {};
					\node[vertices,label = right:{{{\tiny $d_n$}}}] (7) at ($(d)+(7,0)$) {};
					
					\draw (a) -- ++(2.3,0);
					\draw (c) -- ++(2.3,0);
					\draw (e) -- ++(2.3,0);
					\draw (3) -- ++(-2.3,0);
					\draw (4) -- ++(-2.3,0);
					\draw (5) -- ++(-2.3,0);
					\draw (b) -- ++(1,1);
					\draw (b) -- ++(1,-1);
					\draw (d) -- ++(1,1);
					\draw (d) -- ++(1,-1);
					\draw (1) -- ++(-1,1);
					\draw (1) -- ++(-1,-1);
					\draw (2) -- ++(-1,1);
					\draw (2) -- ++(-1,-1);
					
					\node at ($(b)!0.5!(1)$) {$\ldots$};
					\node at ($(d)!0.5!(2)$) {$\ldots$};
					\node at ($(a)!0.5!(3)$) {$\ldots$};
					\node at ($(c)!0.5!(4)$) {$\ldots$};
					\node at ($(e)!0.5!(5)$) {$\ldots$};
					
					\foreach \vertex/\neighbors in {
						c1/{b1},
						d4/{e3,b4},
						a/{b4,b},
						b/{c,d},
						c/{b4,d4,d},
						d/{e},
						e/{e3,d4},
						1/{2,3,4},
						2/{4,5},
						3/{6},
						4/{6,7},
						5/{7},
						6/{7}}
					{
						\foreach \neighbor in \neighbors
						\draw [-] (\vertex)--(\neighbor);
					}
				\end{tikzpicture}
				\subcaption{$\mathcal{U}_n''-N_{\mathcal{U}_n''}[b_3]$}
				\label{fig:U_n''-N_b3}
			\end{subfigure}
			\begin{subfigure}{0.35\textwidth}
				\centering
				\begin{tikzpicture}[scale = 0.27, vertices/.style = {draw, fill = black, circle, inner sep = 0.5pt}]
					
					\node[vertices,label = above:{{\tiny $b_1$}}] (b1) at (12, 3) {};
					
					\node[vertices,label = above:{{\tiny $a_1$}}] (a1) at (13.5, 4.5) {}; 
					\node[vertices,label = below:{{\tiny $c_1$}}] (c1) at (13.5, 1.5) {};

					\node[vertices,label = above:{{\tiny $a_2$}}] (a2) at (16.5, 4.5) {};
					\node[vertices,label = below:{{\tiny $c_2$}}] (c2) at (16.5, 1.5) {};
					
					\node[vertices,label = below:{{\tiny $d_3$}}] (d3) at (18, 0) {};
					
					\node[vertices,label = above:{{\tiny $a_3$}}] (a) at (19.5, 4.5) {};
					\node[vertices,label = below:{{\tiny $c_3$}}] (c) at (19.5, 1.5) {};
					\node[vertices,label = below:{{\tiny $e_3$}}] (e) at (19.5, -1.5) {};
					
					\node[vertices,label = above:{{\tiny $b_4$}}] (b) at (21, 3) {};
					\node[vertices,label = below:{{\tiny $d_4$}}] (d) at (21, 0) {};
					
					\node[vertices, label = {[xshift = .32cm,yshift = -.12cm]above:{{{\tiny $b_{(n-1)}$}}}}] (1) at ($(b)+(4,0)$) {};
					\node[vertices, label = {[xshift = .3cm]below:{{{\tiny $d_{(n-1)}$}}}}] (2) at ($(d)+(4,0)$) {};
					
					\node[vertices,label = {[xshift = .3cm,yshift = -.13cm]above:{{{\tiny $a_{(n-1)}$}}}}] (3) at ($(a)+(7,0)$) {}; 
					\node[vertices,label = {[xshift = 0cm]right:{{{\tiny $c_{(n-1)}$}}}}] (4) at ($(c)+(7,0)$) {};
					\node[vertices,label = {[xshift = .3cm]below:{{{\tiny $e_{(n-1)}$}}}}] (5) at ($(e)+(7,0)$) {};
					
					\node[vertices,label = right:{{{\tiny $b_n$}}}] (6) at ($(b)+(7,0)$) {};
					\node[vertices,label = right:{{{\tiny $d_n$}}}] (7) at ($(d)+(7,0)$) {};
					
					\draw (a) -- ++(2.3,0);
					\draw (c) -- ++(2.3,0);
					\draw (e) -- ++(2.3,0);
					\draw (3) -- ++(-2.3,0);
					\draw (4) -- ++(-2.3,0);
					\draw (5) -- ++(-2.3,0);
					\draw (b) -- ++(1,1);
					\draw (b) -- ++(1,-1);
					\draw (d) -- ++(1,1);
					\draw (d) -- ++(1,-1);
					\draw (1) -- ++(-1,1);
					\draw (1) -- ++(-1,-1);
					\draw (2) -- ++(-1,1);
					\draw (2) -- ++(-1,-1);
					
					\node at ($(b)!0.5!(1)$) {$\ldots$};
					\node at ($(d)!0.5!(2)$) {$\ldots$};
					\node at ($(a)!0.5!(3)$) {$\ldots$};
					\node at ($(c)!0.5!(4)$) {$\ldots$};
					\node at ($(e)!0.5!(5)$) {$\ldots$};
					
					\foreach \vertex/\neighbors in {
						a1/{b1,a2},
						c1/{b1,c2},
						c2/{d3},
						a/{a2,b},
						b/{c,d},
						c/{c2,d3,d},
						d/{e},
						e/{d3},
						1/{2,3,4},
						2/{4,5},
						3/{6},
						4/{6,7},
						5/{7},
						6/{7}}
					{
						\foreach \neighbor in \neighbors
						\draw [-] (\vertex)--(\neighbor);
					}
				\end{tikzpicture}
				\subcaption{$\mathcal{U}_n''-\{b_3\}$}
				\label{fig_U_n''-b3}
			\end{subfigure}
			\vspace{-.18cm}
			\caption{}
			\label{fig:Un_seq2}
		\end{figure}
		By \Cref{lemma:ind_disjoint_union}, $\Ind(\mathcal{U}_{n-3}\sqcup P_2)\simeq \Ind(\mathcal{U}_{n-3})* \Ind(P_2) \simeq \Sigma(\Ind(\mathcal{U}_{n-3}))$. It follows that $\Ind(\mathcal{U}_n'')\simeq \Ind(\mathcal{Z}_{n-3}) \vee \Sigma^2 (\Ind(\mathcal{U}_{n-3}))$. Thus $\Ind(\mathcal{U}_n)\simeq\Sigma(\Ind(\mathcal{U}_{n-1}))\vee \Sigma(\Ind(\mathcal{Z}_{n-3}))\vee \Sigma^3(\Ind(\mathcal{U}_{n-3})).$
	\end{proof}
	
	\begin{claim}\label{claim:Vn_relation}
		For $n\ge3$,
		$\Ind(\mathcal{V}_n)\simeq \Sigma(\Ind(\mathcal{X}_{n-2}))\vee \Sigma^2 (\Ind(\Gamma_{n-1})).$
	\end{claim}
	
	\begin{proof}
		For $n\ge3$, $v_2$ is a simplicial vertex in $\mathcal{V}_n$ with $N_{\mathcal{V}_n}(v_2) = \{a_1, b_1\}$ (see \Cref{fig:Vn_2}). Using \Cref{theorem:simplicial_vertex}, we get $\Ind(\mathcal{V}_n)\simeq \Sigma(\Ind(\mathcal{V}_n-N_{\mathcal{V}_n}[a_1]))\vee \Sigma(\Ind(\mathcal{V}_n-N_{\mathcal{V}_n}[b_1]))$.
		Let $\mathcal{V}_n': = \mathcal{V}_n-N_{\mathcal{V}_n}[a_1]$ and $\mathcal{V}_n'': = \mathcal{V}_n-N_{\mathcal{V}_n}[b_1]$. Then $\Ind(\mathcal{V}_n)\simeq \Sigma(\Ind(\mathcal{V}_n'))\vee \Sigma(\Ind(\mathcal{V}_n''))$. 
		\begin{figure}[H]
			\centering
			\begin{subfigure}{0.40\textwidth}
				\centering
				\begin{tikzpicture}[scale = 0.28, vertices/.style = {draw, fill = black, circle, inner sep = 0.5pt}]
					
					\node[vertices,label = below:{{\tiny $v_1$}}] (v1) at (9, 0) {};
					\node[vertices,label = below:{{\tiny $v_3$}}] (v3) at (10.5, 1.5) {}; 
					\node[vertices,label = below:{{\tiny $v_4$}}] (v4) at (10.5, -1.5) {}; 
					
					\node[vertices,label = below:{{\tiny $d_1$}}] (d1) at (12, 0) {};
					
					\node[vertices,label = below:{{\tiny $c_1$}}] (c1) at (13.5, 1.5) {};
					\node[vertices,label = below:{{\tiny $e_1$}}] (e1) at (13.5, -1.5) {};
					
					\node[vertices,label = below:{{\tiny $d_2$}}] (d2) at (15, 0) {};
					
					\node[vertices,label = below:{{\tiny $c_2$}}] (c2) at (16.5, 1.5) {};
					\node[vertices,label = below:{{\tiny $e_2$}}] (e2) at (16.5, -1.5) {};
					
					\node[vertices,label = above:{{\tiny $b_3$}}] (b3) at (18, 3) {};
					\node[vertices,label = below:{{\tiny $d_3$}}] (d3) at (18, 0) {};
					
					\node[vertices,label = above:{{\tiny $a_3$}}] (a) at (19.5, 4.5) {};
					\node[vertices,label = below:{{\tiny $c_3$}}] (c) at (19.5, 1.5) {};
					\node[vertices,label = below:{{\tiny $e_3$}}] (e) at (19.5, -1.5) {};
					
					\node[vertices,label = above:{{\tiny $b_4$}}] (b) at (21, 3) {};
					\node[vertices,label = below:{{\tiny $d_4$}}] (d) at (21, 0) {};
					
					\node[vertices, label = {[xshift = .32cm,yshift = -.12cm]above:{{{\tiny $b_{(n-1)}$}}}}] (1) at ($(b)+(4,0)$) {};
					\node[vertices, label = {[xshift = .3cm]below:{{{\tiny $d_{(n-1)}$}}}}] (2) at ($(d)+(4,0)$) {};
					
					\node[vertices,label = {[xshift = .3cm,yshift = -.13cm]above:{{{\tiny $a_{(n-1)}$}}}}] (3) at ($(a)+(7,0)$) {}; 
					\node[vertices,label = {[xshift = 0cm]right:{{{\tiny $c_{(n-1)}$}}}}] (4) at ($(c)+(7,0)$) {};
					\node[vertices,label = {[xshift = .3cm]below:{{{\tiny $e_{(n-1)}$}}}}] (5) at ($(e)+(7,0)$) {};
					
					\node[vertices,label = right:{{{\tiny $b_n$}}}] (6) at ($(b)+(7,0)$) {};
					\node[vertices,label = right:{{{\tiny $d_n$}}}] (7) at ($(d)+(7,0)$) {};
					
					\draw (a) -- ++(2.3,0);
					\draw (c) -- ++(2.3,0);
					\draw (e) -- ++(2.3,0);
					\draw (3) -- ++(-2.3,0);
					\draw (4) -- ++(-2.3,0);
					\draw (5) -- ++(-2.3,0);
					\draw (b) -- ++(1,1);
					\draw (b) -- ++(1,-1);
					\draw (d) -- ++(1,1);
					\draw (d) -- ++(1,-1);
					\draw (1) -- ++(-1,1);
					\draw (1) -- ++(-1,-1);
					\draw (2) -- ++(-1,1);
					\draw (2) -- ++(-1,-1);
					
					\node at ($(b)!0.5!(1)$) {$\ldots$};
					\node at ($(d)!0.5!(2)$) {$\ldots$};
					\node at ($(a)!0.5!(3)$) {$\ldots$};
					\node at ($(c)!0.5!(4)$) {$\ldots$};
					\node at ($(e)!0.5!(5)$) {$\ldots$};
					
					\foreach \vertex/\neighbors in {
						v3/{v1,c1,d1},
						v4/{v1,d1,e1},
						d1/{c1,e1},
						c1/{d2},
						e1/{d2,e2},
						c2/{c1,b3,d3},
						d2/{c2,e2},
						e2/{d3},
						b3/{d3},
						a/{b3,b},
						b/{c,d},
						c/{c2,b3,d3,d},
						d/{e},
						e/{d3,e2},
						1/{2,3,4},
						2/{4,5},
						3/{6},
						4/{6,7},
						5/{7},
						6/{7}}
					{
						\foreach \neighbor in \neighbors
						\draw [-] (\vertex)--(\neighbor);
					}
				\end{tikzpicture}
				\subcaption{$\mathcal{V}_n'$}
			\end{subfigure}
			\begin{subfigure}{0.47\textwidth}
				\centering
				\begin{tikzpicture}[scale = 0.28, vertices/.style = {draw, fill = black, circle, inner sep = 0.5pt}]
					
					\node[vertices,label = below:{{\tiny $v_1$}}] (v1) at (9, 0) {};
					\node[vertices,label = below:{{\tiny $v_4$}}] (v4) at (10.5, -1.5) {}; 
					
					\node[vertices,label = below:{{\tiny $e_1$}}] (e1) at (13.5, -1.5) {};
					
					\node[vertices,label = above:{{\tiny $b_2$}}] (b2) at (15, 3) {};
					\node[vertices,label = below:{{\tiny $d_2$}}] (d2) at (15, 0) {};
					
					\node[vertices,label = above:{{\tiny $a_2$}}] (a) at (16.5, 4.5) {};
					\node[vertices,label = below:{{\tiny $c_2$}}] (c) at (16.5, 1.5) {};
					\node[vertices,label = below:{{\tiny $e_2$}}] (e) at (16.5, -1.5) {};
					
					\node[vertices,label = above:{{\tiny $b_3$}}] (b) at (18, 3) {};
					\node[vertices,label = below:{{\tiny $d_3$}}] (d) at (18, 0) {};
					
					\node[vertices, label = {[xshift = .32cm,yshift = -.12cm]above:{{{\tiny $b_{(n-1)}$}}}}] (1) at ($(b)+(4,0)$) {};
					\node[vertices, label = {[xshift = .3cm]below:{{{\tiny $d_{(n-1)}$}}}}] (2) at ($(d)+(4,0)$) {};
					
					\node[vertices,label = {[xshift = .3cm,yshift = -.13cm]above:{{{\tiny $a_{(n-1)}$}}}}] (3) at ($(a)+(7,0)$) {}; 
					\node[vertices,label = {[xshift = 0cm]right:{{{\tiny $c_{(n-1)}$}}}}] (4) at ($(c)+(7,0)$) {};
					\node[vertices,label = {[xshift = .3cm]below:{{{\tiny $e_{(n-1)}$}}}}] (5) at ($(e)+(7,0)$) {};
					
					\node[vertices,label = right:{{{\tiny $b_n$}}}] (6) at ($(b)+(7,0)$) {};
					\node[vertices,label = right:{{{\tiny $d_n$}}}] (7) at ($(d)+(7,0)$) {};
					
					\draw (a) -- ++(2.3,0);
					\draw (c) -- ++(2.3,0);
					\draw (e) -- ++(2.3,0);
					\draw (3) -- ++(-2.3,0);
					\draw (4) -- ++(-2.3,0);
					\draw (5) -- ++(-2.3,0);
					\draw (b) -- ++(1,1);
					\draw (b) -- ++(1,-1);
					\draw (d) -- ++(1,1);
					\draw (d) -- ++(1,-1);
					\draw (1) -- ++(-1,1);
					\draw (1) -- ++(-1,-1);
					\draw (2) -- ++(-1,1);
					\draw (2) -- ++(-1,-1);
					
					\node at ($(b)!0.5!(1)$) {$\ldots$};
					\node at ($(d)!0.5!(2)$) {$\ldots$};
					\node at ($(a)!0.5!(3)$) {$\ldots$};
					\node at ($(c)!0.5!(4)$) {$\ldots$};
					\node at ($(e)!0.5!(5)$) {$\ldots$};
					
					\foreach \vertex/\neighbors in {
						v4/{v1,e1},
						e1/{d2},
						b2/{d2},
						a/{b2,b},
						b/{c,d},
						c/{b2,d2,d},
						d/{e},
						e/{d2,e1},
						1/{2,3,4},
						2/{4,5},
						3/{6},
						4/{6,7},
						5/{7},
						6/{7}}
					{
						\foreach \neighbor in \neighbors
						\draw [-] (\vertex)--(\neighbor);
					}
				\end{tikzpicture}
				\subcaption{$\mathcal{V}_n''$}
			\end{subfigure}
			\vspace{-.2cm}
			\caption{}
		\end{figure}
		
		Since $N_{\mathcal{V}_n'}(v_1)\subseteq N_{\mathcal{V}_n'}(d_1)$, $\Ind(\mathcal{V}_ n')\simeq \Ind(\mathcal{V}_n'-\{d_1\})$ by \Cref{theorem:N(u)subsetN(v)}. Note that $\mathcal{V}_n'-\{d_1\}\cong\mathcal{X}_{n-2}$. Therefore $\Ind(\mathcal{V}_n')\simeq \Ind(\mathcal{X}_{n-2})$. 
		Further, $v_1$ is a simplicial vertex in $\mathcal{V}_n''$ with $N_{\mathcal{V}_n''}(v_1) = \{v_4\}$. Hence $\Ind(\mathcal{V}_n'')\simeq \Sigma(\Ind(\mathcal{V}_n''-N_{\mathcal{V}_n''}[v_4]))$. Since $\mathcal{V}_n''-N_{\mathcal{V}_n''}[v_4]\cong\Gamma_{n-1}$, we have $\Ind(\mathcal{V}_n'')\simeq \Sigma (\Ind(\Gamma_{n-1}))$. Thus $\Ind(\mathcal{V}_n)\simeq \Sigma(\Ind(\mathcal{X}_{n-2}))\vee \Sigma^2 (\Ind(\Gamma_{n-1})).$
	\end{proof}
	
	\begin{claim}\label{claim:Wn_relation}
		For $n\ge3$, let $W_n': = \mathcal{W}_n-\{d_1\}$. Then $\Ind(\mathcal{W}_n)\simeq \Ind(\mathcal{W}_n')$. Moreover, 
		$$lk(w_1,\Ind(\mathcal{W}_n'))\simeq 
		\begin{cases} 
			\vee_3\mathbb{S}^2& \text{ if $n = 3$,}\\
			\Sigma^2 (\Ind(\Gamma_{n-2}))\vee \Sigma(\Ind(\mathcal{X}_{n-3})) & \text{ if $n\ge4$,}
		\end{cases}$$ 
		and $del(w_1,\Ind(\mathcal{W}_n'))\simeq \Ind(\mathcal{Y}_{n-1}).$
	\end{claim}
	
	\begin{proof}
		Since $N_{\mathcal{W}_n}(w_1)\subseteq N_{\mathcal{W}_n}(d_1)$ (\Cref{fig:Wn_2}), $\Ind(\mathcal{W}_n)\simeq \Ind(\mathcal{W}_n-\{d_1\})$. Hence $\Ind(\mathcal{W}_n)\simeq \Ind(\mathcal{W}_n')$.
		
		We first compute the deletion of the vertex $w_1$ in $\Ind(\mathcal{W}_n')$. Using \eqref{equation:lk_del_ind_complex}, $del(w_1,\Ind(\mathcal{W}_n'))\simeq \Ind(\mathcal{W}_n'-\{w_1\})$. Since $\mathcal{W}_n'-\{w_1\}\cong \mathcal{Y}_{n-1}$ (\Cref{fig:Wn'-w1}), it follows that $del(w_1,\Ind(\mathcal{W}_n'))\simeq \Ind(\mathcal{Y}_{n-1}).$  
		\begin{figure}[h!]
			\centering
			\begin{subfigure}{0.32\textwidth}
				\centering
				\begin{tikzpicture}[scale = 0.27, vertices/.style = {draw, fill = black, circle, inner sep = 0.5pt}]
					
					\node[vertices,label = below:{{\tiny $w_1$}}] (w1) at (9, 0) {}; 
					\node[vertices,label = below:{{\tiny $w_2$}}] (w2) at (10.5, 1.5) {}; 
					\node[vertices,label = below:{{\tiny $w_3$}}] (w3) at (10.5, -1.5) {}; 
					
					\node[vertices,label = above:{{\tiny $b_1$}}] (b1) at (12, 3) {};
					
					\node[vertices,label = above:{{\tiny $a_1$}}] (a1) at (13.5, 4.5) {};
					\node[vertices,label = below:{{\tiny $c_1$}}] (c1) at (13.5, 1.5) {};
					\node[vertices,label = below:{{\tiny $e_1$}}] (e1) at (13.5, -1.5) {};
					
					\node[vertices,label = above:{{\tiny $b_2$}}] (b2) at (15, 3) {};
					\node[vertices,label = below:{{\tiny $d_2$}}] (d2) at (15, 0) {};
					
					\node[vertices,label = above:{{\tiny $a_2$}}] (a) at (16.5, 4.5) {};
					\node[vertices,label = below:{{\tiny $c_2$}}] (c) at (16.5, 1.5) {};
					\node[vertices,label = below:{{\tiny $e_2$}}] (e) at (16.5, -1.5) {};
					
					\node[vertices,label = above:{{\tiny $b_3$}}] (b) at (18, 3) {};
					\node[vertices,label = below:{{\tiny $d_3$}}] (d) at (18, 0) {};
					
					\node[vertices, label = {[xshift = .25cm,yshift = -.12cm]above:{{{\tiny $b_{(n-1)}$}}}}] (1) at ($(b)+(4,0)$) {};
					\node[vertices, label = {[xshift = .25cm]below:{{{\tiny $d_{(n-1)}$}}}}] (2) at ($(d)+(4,0)$) {};
					
					\node[vertices,label = {[xshift = .3cm,yshift = -.13cm]above:{{{\tiny $a_{(n-1)}$}}}}] (3) at ($(a)+(7,0)$) {}; 
					\node[vertices,label = {[xshift = .27cm]below:{{{\tiny $c_{(n-1)}$}}}}] (4) at ($(c)+(7,0)$) {};
					\node[vertices,label = {[xshift = .3cm]below:{{{\tiny $e_{(n-1)}$}}}}] (5) at ($(e)+(7,0)$) {};
					
					\node[vertices,label = above:{{{\tiny $b_n$}}}] (6) at ($(b)+(7,0)$) {};
					\node[vertices,label = below:{{{\tiny $d_n$}}}] (7) at ($(d)+(7,0)$) {};
					
					\draw (a) -- ++(2.3,0);
					\draw (c) -- ++(2.3,0);
					\draw (e) -- ++(2.3,0);
					\draw (3) -- ++(-2.3,0);
					\draw (4) -- ++(-2.3,0);
					\draw (5) -- ++(-2.3,0);
					\draw (b) -- ++(1,1);
					\draw (b) -- ++(1,-1);
					\draw (d) -- ++(1,1);
					\draw (d) -- ++(1,-1);
					\draw (1) -- ++(-1,1);
					\draw (1) -- ++(-1,-1);
					\draw (2) -- ++(-1,1);
					\draw (2) -- ++(-1,-1);
					
					\node at ($(b)!0.5!(1)$) {$\ldots$};
					\node at ($(d)!0.5!(2)$) {$\ldots$};
					\node at ($(a)!0.5!(3)$) {$\ldots$};
					\node at ($(c)!0.5!(4)$) {$\ldots$};
					\node at ($(e)!0.5!(5)$) {$\ldots$};
					
					\foreach \vertex/\neighbors in {
						w2/{w1,b1,c1},
						w3/{w1,e1},
						b1/{a1,c1},
						a1/{b2},
						c1/{b2,d2},
						e1/{d2},
						b2/{d2},
						a/{a1,b2,b},
						b/{c,d},
						c/{c1,b2,d2,d},
						d/{e},
						e/{d2,e1},
						1/{2,3,4},
						2/{4,5},
						3/{6},
						4/{6,7},
						5/{7},
						6/{7}}
					{
						\foreach \neighbor in \neighbors
						\draw [-] (\vertex)--(\neighbor);
					}
				\end{tikzpicture}
				\subcaption{$\mathcal{W}_n'$}
			\end{subfigure}
			\hfill
			\begin{subfigure}{0.3\textwidth}
				\centering
				\begin{tikzpicture}[scale = 0.27, vertices/.style = {draw, fill = black, circle, inner sep = 0.5pt}]
					
					\node[vertices,label = below:{{\tiny $w_2$}}] (w2) at (10.5, 1.5) {}; 
					\node[vertices,label = below:{{\tiny $w_3$}}] (w3) at (10.5, -1.5) {}; 
					
					\node[vertices,label = above:{{\tiny $b_1$}}] (b1) at (12, 3) {};
					
					\node[vertices,label = above:{{\tiny $a_1$}}] (a1) at (13.5, 4.5) {};
					\node[vertices,label = below:{{\tiny $c_1$}}] (c1) at (13.5, 1.5) {};
					\node[vertices,label = below:{{\tiny $e_1$}}] (e1) at (13.5, -1.5) {};
					
					\node[vertices,label = above:{{\tiny $b_2$}}] (b2) at (15, 3) {};
					\node[vertices,label = below:{{\tiny $d_2$}}] (d2) at (15, 0) {};
					
					\node[vertices,label = above:{{\tiny $a_2$}}] (a) at (16.5, 4.5) {};
					\node[vertices,label = below:{{\tiny $c_2$}}] (c) at (16.5, 1.5) {};
					\node[vertices,label = below:{{\tiny $e_2$}}] (e) at (16.5, -1.5) {};
					
					\node[vertices,label = above:{{\tiny $b_3$}}] (b) at (18, 3) {};
					\node[vertices,label = below:{{\tiny $d_3$}}] (d) at (18, 0) {};
					
					\node[vertices, label = {[xshift = .25cm,yshift = -.12cm]above:{{{\tiny $b_{(n-1)}$}}}}] (1) at ($(b)+(4,0)$) {};
					\node[vertices, label = {[xshift = .25cm]below:{{{\tiny $d_{(n-1)}$}}}}] (2) at ($(d)+(4,0)$) {};
					
					\node[vertices,label = {[xshift = .3cm,yshift = -.13cm]above:{{{\tiny $a_{(n-1)}$}}}}] (3) at ($(a)+(7,0)$) {}; 
					\node[vertices,label = {[xshift = .27cm]below:{{{\tiny $c_{(n-1)}$}}}}] (4) at ($(c)+(7,0)$) {};
					\node[vertices,label = {[xshift = .3cm]below:{{{\tiny $e_{(n-1)}$}}}}] (5) at ($(e)+(7,0)$) {};
					
					\node[vertices,label = above:{{{\tiny $b_n$}}}] (6) at ($(b)+(7,0)$) {};
					\node[vertices,label = below:{{{\tiny $d_n$}}}] (7) at ($(d)+(7,0)$) {};
					
					\draw (a) -- ++(2.3,0);
					\draw (c) -- ++(2.3,0);
					\draw (e) -- ++(2.3,0);
					\draw (3) -- ++(-2.3,0);
					\draw (4) -- ++(-2.3,0);
					\draw (5) -- ++(-2.3,0);
					\draw (b) -- ++(1,1);
					\draw (b) -- ++(1,-1);
					\draw (d) -- ++(1,1);
					\draw (d) -- ++(1,-1);
					\draw (1) -- ++(-1,1);
					\draw (1) -- ++(-1,-1);
					\draw (2) -- ++(-1,1);
					\draw (2) -- ++(-1,-1);
					
					\node at ($(b)!0.5!(1)$) {$\ldots$};
					\node at ($(d)!0.5!(2)$) {$\ldots$};
					\node at ($(a)!0.5!(3)$) {$\ldots$};
					\node at ($(c)!0.5!(4)$) {$\ldots$};
					\node at ($(e)!0.5!(5)$) {$\ldots$};
					
					\foreach \vertex/\neighbors in {
						w2/{b1,c1},
						w3/{e1},
						b1/{a1,c1},
						a1/{b2},
						c1/{b2,d2},
						e1/{d2},
						b2/{d2},
						a/{a1,b2,b},
						b/{c,d},
						c/{c1,b2,d2,d},
						d/{e},
						e/{d2,e1},
						1/{2,3,4},
						2/{4,5},
						3/{6},
						4/{6,7},
						5/{7},
						6/{7}}
					{
						\foreach \neighbor in \neighbors
						\draw [-] (\vertex)--(\neighbor);
					}
				\end{tikzpicture}
				\subcaption{$\mathcal{W}_n'-\{w_1\}$}
				\label{fig:Wn'-w1}
			\end{subfigure}
			\hfill
			\begin{subfigure}{0.27\textwidth}
				\centering
				\begin{tikzpicture}[scale = 0.27, vertices/.style = {draw, fill = black, circle, inner sep = 0.5pt}]
					
					\node[vertices,label = above:{{\tiny $b_1$}}] (b1) at (12, 3) {};
					
					\node[vertices,label = above:{{\tiny $a_1$}}] (a1) at (13.5, 4.5) {};
					\node[vertices,label = below:{{\tiny $c_1$}}] (c1) at (13.5, 1.5) {};
					\node[vertices,label = below:{{\tiny $e_1$}}] (e1) at (13.5, -1.5) {};
					
					\node[vertices,label = above:{{\tiny $b_2$}}] (b2) at (15, 3) {};
					\node[vertices,label = below:{{\tiny $d_2$}}] (d2) at (15, 0) {};
					
					\node[vertices,label = above:{{\tiny $a_2$}}] (a) at (16.5, 4.5) {};
					\node[vertices,label = below:{{\tiny $c_2$}}] (c) at (16.5, 1.5) {};
					\node[vertices,label = below:{{\tiny $e_2$}}] (e) at (16.5, -1.5) {};
					
					\node[vertices,label = above:{{\tiny $b_3$}}] (b) at (18, 3) {};
					\node[vertices,label = below:{{\tiny $d_3$}}] (d) at (18, 0) {};
					
					\node[vertices, label = {[xshift = .25cm,yshift = -.12cm]above:{{{\tiny $b_{(n-1)}$}}}}] (1) at ($(b)+(4,0)$) {};
					\node[vertices, label = {[xshift = .25cm]below:{{{\tiny $d_{(n-1)}$}}}}] (2) at ($(d)+(4,0)$) {};
					
					\node[vertices,label = {[xshift = .3cm,yshift = -.13cm]above:{{{\tiny $a_{(n-1)}$}}}}] (3) at ($(a)+(7,0)$) {}; 
					\node[vertices,label = {[xshift = .27cm]below:{{{\tiny $c_{(n-1)}$}}}}] (4) at ($(c)+(7,0)$) {};
					\node[vertices,label = {[xshift = .3cm]below:{{{\tiny $e_{(n-1)}$}}}}] (5) at ($(e)+(7,0)$) {};
					
					\node[vertices,label = above:{{{\tiny $b_n$}}}] (6) at ($(b)+(7,0)$) {};
					\node[vertices,label = below:{{{\tiny $d_n$}}}] (7) at ($(d)+(7,0)$) {};
					
					\draw (a) -- ++(2.3,0);
					\draw (c) -- ++(2.3,0);
					\draw (e) -- ++(2.3,0);
					\draw (3) -- ++(-2.3,0);
					\draw (4) -- ++(-2.3,0);
					\draw (5) -- ++(-2.3,0);
					\draw (b) -- ++(1,1);
					\draw (b) -- ++(1,-1);
					\draw (d) -- ++(1,1);
					\draw (d) -- ++(1,-1);
					\draw (1) -- ++(-1,1);
					\draw (1) -- ++(-1,-1);
					\draw (2) -- ++(-1,1);
					\draw (2) -- ++(-1,-1);
					
					\node at ($(b)!0.5!(1)$) {$\ldots$};
					\node at ($(d)!0.5!(2)$) {$\ldots$};
					\node at ($(a)!0.5!(3)$) {$\ldots$};
					\node at ($(c)!0.5!(4)$) {$\ldots$};
					\node at ($(e)!0.5!(5)$) {$\ldots$};
					
					\foreach \vertex/\neighbors in {
						b1/{a1,c1},
						a1/{b2},
						c1/{b2,d2},
						e1/{d2},
						b2/{d2},
						a/{a1,b2,b},
						b/{c,d},
						c/{c1,b2,d2,d},
						d/{e},
						e/{d2,e1},
						1/{2,3,4},
						2/{4,5},
						3/{6},
						4/{6,7},
						5/{7},
						6/{7}}
					{
						\foreach \neighbor in \neighbors
						\draw [-] (\vertex)--(\neighbor);
					}
				\end{tikzpicture}
				\subcaption{$\mathcal{W}_n''$}
				\label{fig:Wn''}
			\end{subfigure}
			\vspace{-.18cm}
			\caption{}
			\label{fig:Wnseq1}
		\end{figure}
		
		For the link, we have $lk(w_1,\Ind(\mathcal{W}_n'))\simeq \Ind(\mathcal{W}_n'-N_{\mathcal{W}_n'}[w_1])$ by \eqref{equation:lk_del_ind_complex}. Let $\mathcal{W}_n'': = \mathcal{W}_n'-N_{\mathcal{W}_n'}[w_1]$ (see \Cref{fig:Wn''}). Then $lk(w_1,\Ind(\mathcal{W}_n'))\simeq \Ind(\mathcal{W}_n'')$. Since $e_1$ is a simplicial vertex in $\mathcal{W}_n''$ with $N_{\mathcal{W}_n''}(e_1) = \{d_2,e_2\}$, \Cref{theorem:simplicial_vertex} implies $\Ind( \mathcal{W}_n'')\simeq \Sigma(\Ind(\mathcal{W}_n''-N_{\mathcal{W}_n''}[d_2]))\vee \Sigma(\Ind(\mathcal{W}_n''-N_{\mathcal{W}_n''}[e_2]))$. 
		
		We first assume $n = 3$. Then $\mathcal{W}_3''-N_{\mathcal{W}_3''}[d_2]\cong P_5$ and $\mathcal{W}_3''-N_{\mathcal{W}_3''}[e_2]\cong \Gamma_2$. It follows that $\Ind(\mathcal{W}_3'')\simeq\Sigma(\Ind(P_5))\vee\Sigma(\Ind(\Gamma_2))$, and therefore $lk(w_1,\Ind(\mathcal{W}_3'))\simeq \Ind(\mathcal{W}_3'')\simeq\vee_3\mathbb{S}^2$ by \Cref{lemma:path_graph} and \eqref{equation:Gamma_123}.
		
		Now, assume $n\ge4$. Let $\mathcal{W}_n''': = \mathcal{W}_n''-N_{\mathcal{W}_n''}[d_2]$. Since $b_1$ is a simplicial vertex in $\mathcal{W}_n'''$ with $N_{\mathcal{W}_n'''}(b_1) = \{a_1\}$, it follows that $\Ind(\mathcal{W}_n''')\simeq \Sigma(\Ind(\mathcal{W}_n'''-N_{\mathcal{W}_n'''}[a_1]))$. Further, $\mathcal{W}_n'''-N_{\mathcal{W}_n'''}[a_1]\cong \Gamma_{n-2}$ implies $\Ind(\mathcal{W}_n''-N_{\mathcal{W}_n''}[d_2])\simeq \Ind(\mathcal{W}_n''')\simeq \Sigma(\Ind(\Gamma_{n-2}))$.
		\begin{figure}[H]
			\centering
			\begin{subfigure}{0.43\textwidth}
				\centering
				\begin{tikzpicture}[scale = 0.27, vertices/.style = {draw, fill = black, circle, inner sep = 0.5pt}]
					
					\node[vertices,label = above:{{\tiny $b_1$}}] (b1) at (12, 3) {};
					
					\node[vertices,label = above:{{\tiny $a_1$}}] (a1) at (13.5, 4.5) {};
					
					\node[vertices,label = above:{{\tiny $a_2$}}] (a2) at (16.5, 4.5) {};
					
					\node[vertices,label = above:{{\tiny $b_3$}}] (b3) at (18, 3) {};
					\node[vertices,label = below:{{\tiny $d_3$}}] (d3) at (18, 0) {};
					
					\node[vertices,label = above:{{\tiny $a_3$}}] (a) at (19.5, 4.5) {};
					\node[vertices,label = below:{{\tiny $c_3$}}] (c) at (19.5, 1.5) {};
					\node[vertices,label = below:{{\tiny $e_3$}}] (e) at (19.5, -1.5) {};
					
					\node[vertices,label = above:{{\tiny $b_4$}}] (b) at (21, 3) {};
					\node[vertices,label = below:{{\tiny $d_4$}}] (d) at (21, 0) {};
					
					\node[vertices, label = {[xshift = .25cm,yshift = -.12cm]above:{{{\tiny $b_{(n-1)}$}}}}] (1) at ($(b)+(4,0)$) {};
					\node[vertices, label = {[xshift = .25cm]below:{{{\tiny $d_{(n-1)}$}}}}] (2) at ($(d)+(4,0)$) {};
					
					\node[vertices,label = {[xshift = .3cm,yshift = -.13cm]above:{{{\tiny $a_{(n-1)}$}}}}] (3) at ($(a)+(7,0)$) {}; 
					\node[vertices,label = {[xshift = 0cm]right:{{{\tiny $c_{(n-1)}$}}}}] (4) at ($(c)+(7,0)$) {};
					\node[vertices,label = {[xshift = .3cm]below:{{{\tiny $e_{(n-1)}$}}}}] (5) at ($(e)+(7,0)$) {};
					
					\node[vertices,label = right:{{{\tiny $b_n$}}}] (6) at ($(b)+(7,0)$) {};
					\node[vertices,label = right:{{{\tiny $d_n$}}}] (7) at ($(d)+(7,0)$) {};
					
					\draw (a) -- ++(2.3,0);
					\draw (c) -- ++(2.3,0);
					\draw (e) -- ++(2.3,0);
					\draw (3) -- ++(-2.3,0);
					\draw (4) -- ++(-2.3,0);
					\draw (5) -- ++(-2.3,0);
					\draw (b) -- ++(1,1);
					\draw (b) -- ++(1,-1);
					\draw (d) -- ++(1,1);
					\draw (d) -- ++(1,-1);
					\draw (1) -- ++(-1,1);
					\draw (1) -- ++(-1,-1);
					\draw (2) -- ++(-1,1);
					\draw (2) -- ++(-1,-1);
					
					\node at ($(b)!0.5!(1)$) {$\ldots$};
					\node at ($(d)!0.5!(2)$) {$\ldots$};
					\node at ($(a)!0.5!(3)$) {$\ldots$};
					\node at ($(c)!0.5!(4)$) {$\ldots$};
					\node at ($(e)!0.5!(5)$) {$\ldots$};
					
					\foreach \vertex/\neighbors in {
						a1/{a2},
						b1/{a1},
						b3/{a2,d3},
						a/{a2,b3,b},
						b/{c,d},
						c/{b3,d3,d},
						d/{e},
						e/{d3},
						1/{2,3,4},
						2/{4,5},
						3/{6},
						4/{6,7},
						5/{7},
						6/{7}}
					{
						\foreach \neighbor in \neighbors
						\draw [-] (\vertex)--(\neighbor);
					}
				\end{tikzpicture}
				\subcaption{$\mathcal{W}_n'''$}
			\end{subfigure}
			\begin{subfigure}{0.45\textwidth}
				\centering
				\begin{tikzpicture}[scale = 0.27, vertices/.style = {draw, fill = black, circle, inner sep = 0.5pt}]
					
					\node[vertices,label = above:{{\tiny $b_1$}}] (b1) at (12, 3) {};
					
					\node[vertices,label = above:{{\tiny $a_1$}}] (a1) at (13.5, 4.5) {};
					\node[vertices,label = below:{{\tiny $c_1$}}] (c1) at (13.5, 1.5) {};
					
					\node[vertices,label = above:{{\tiny $b_2$}}] (b2) at (15, 3) {};
					
					\node[vertices,label = above:{{\tiny $a_2$}}] (a2) at (16.5, 4.5) {};
					\node[vertices,label = below:{{\tiny $c_2$}}] (c2) at (16.5, 1.5) {};
					
					\node[vertices,label = above:{{\tiny $b_3$}}] (b3) at (18, 3) {};
					
					\node[vertices,label = above:{{\tiny $a_3$}}] (a3) at (19.5, 4.5) {};
					\node[vertices,label = below:{{\tiny $c_3$}}] (c3) at (19.5, 1.5) {};
					
					\node[vertices,label = above:{{\tiny $b_4$}}] (b4) at (21, 3) {};
					\node[vertices,label = below:{{\tiny $d_4$}}] (d4) at (21, 0) {};
					
					\node[vertices,label = above:{{\tiny $a_4$}}] (a) at (22.5, 4.5) {};
					\node[vertices,label = below:{{\tiny $c_4$}}] (c) at (22.5, 1.5) {};
					\node[vertices,label = below:{{\tiny $e_4$}}] (e) at (22.5, -1.5) {};
					
					\node[vertices,label = above:{{\tiny $b_5$}}] (b) at (24, 3) {};
					\node[vertices,label = below:{{\tiny $d_5$}}] (d) at (24, 0) {};
					
					\node[vertices, label = {[xshift = .25cm,yshift = -.12cm]above:{{{\tiny $b_{(n-1)}$}}}}] (1) at ($(b)+(4,0)$) {};
					\node[vertices, label = {[xshift = .25cm]below:{{{\tiny $d_{(n-1)}$}}}}] (2) at ($(d)+(4,0)$) {};
					
					\node[vertices,label = {[xshift = .3cm,yshift = -.13cm]above:{{{\tiny $a_{(n-1)}$}}}}] (3) at ($(a)+(7,0)$) {}; 
					\node[vertices,label = {[xshift = 0cm]right:{{{\tiny $c_{(n-1)}$}}}}] (4) at ($(c)+(7,0)$) {};
					\node[vertices,label = {[xshift = .3cm]below:{{{\tiny $e_{(n-1)}$}}}}] (5) at ($(e)+(7,0)$) {};
					
					\node[vertices,label = right:{{{\tiny $b_n$}}}] (6) at ($(b)+(7,0)$) {};
					\node[vertices,label = right:{{{\tiny $d_n$}}}] (7) at ($(d)+(7,0)$) {};
					
					\draw (a) -- ++(2.3,0);
					\draw (c) -- ++(2.3,0);
					\draw (e) -- ++(2.3,0);
					\draw (3) -- ++(-2.3,0);
					\draw (4) -- ++(-2.3,0);
					\draw (5) -- ++(-2.3,0);
					\draw (b) -- ++(1,1);
					\draw (b) -- ++(1,-1);
					\draw (d) -- ++(1,1);
					\draw (d) -- ++(1,-1);
					\draw (1) -- ++(-1,1);
					\draw (1) -- ++(-1,-1);
					\draw (2) -- ++(-1,1);
					\draw (2) -- ++(-1,-1);
					
					\node at ($(b)!0.5!(1)$) {$\ldots$};
					\node at ($(d)!0.5!(2)$) {$\ldots$};
					\node at ($(a)!0.5!(3)$) {$\ldots$};
					\node at ($(c)!0.5!(4)$) {$\ldots$};
					\node at ($(e)!0.5!(5)$) {$\ldots$};
					
					\foreach \vertex/\neighbors in {
						a1/{a2},
						b1/{a1,c1},
						c1/{c2},
						a2/{a3},
						b2/{a1,a2,c1,c2},
						b3/{a2,c2,a3,c3},
						c3/{d4},
						b4/{a3,c3,d4},
						a/{a3,b4,b},
						b/{c,d},
						c/{c2,b4,d4,d},
						d/{e},
						e/{d4},
						1/{2,3,4},
						2/{4,5},
						3/{6},
						4/{6,7},
						5/{7},
						6/{7}}
					{
						\foreach \neighbor in \neighbors
						\draw [-] (\vertex)--(\neighbor);
					}
				\end{tikzpicture}
				\subcaption{$\mathcal{W}_n''-N_{\mathcal{W}_n''}[e_2]$}
			\end{subfigure}
			\vspace{-.18cm}
			\caption{}
			\label{fig:Wn_seq}
		\end{figure}
		
		Observe that $N_{\mathcal{W}_n''-N_{\mathcal{W}_n''}[e_2]}(b_1)\subseteq N_{\mathcal{W}_n''-N_{\mathcal{W}_n''}[e_2]}(b_2)$. By \Cref{theorem:N(u)subsetN(v)}, $\Ind(\mathcal{W}_n''-N_{\mathcal{W}_n''}[e_2])\simeq \Ind((\mathcal{W}_n''-N_{\mathcal{W}_n''}[e_2])-\{b_2\})$. Since $(\mathcal{W}_n''-N_{\mathcal{W}_n''}[e_2])-\{b_2\}\cong\mathcal{X}_{n-3}$, we have $\Ind(\mathcal{W}_n''-N_{\mathcal{W}_n''}[e_2])\simeq \Ind(\mathcal{X}_{n-3})$. 
		
		Therefore $lk(w_1,\Ind(\mathcal{W}_n'))\simeq \Ind( \mathcal{W}_n'')\simeq \Sigma^2 (\Ind(\Gamma_{n-2}))\vee \Sigma(\Ind(\mathcal{X}_{n-3}))$ for $n\ge4$.
	\end{proof}
	
	\begin{claim}\label{claim:Xn_relation}
		For $n\ge3$,
		$\Ind(\mathcal{X}_n)\simeq\vee_2 \Sigma^2(\Ind(\Gamma_n))\vee \Sigma^3 (\Ind(\mathcal{U}_{n-1})).$
	\end{claim}
	
	\begin{proof}
		Let $\mathcal{X}_n'$ be the graph obtained from $\mathcal{X}_n$ by replacing the path $x_4x_2x_1x_3x_5$ with the edge $(x_4,x_5)$ (see \Cref{fig:Xn_2,fig:Xn'}). 
		By \Cref{theorem:replace_edge}, $\Ind(\mathcal{X}_n)\simeq \Sigma(\Ind(\mathcal{X}_n'))$. 
		Since $N_{\mathcal{X}_n'}[x_4]\subseteq N_{\mathcal{X}_n'}[x_6]$, \Cref{corollary:N[u]subsetN[v]_lk(v)_del(v)} implies $\Ind(\mathcal{X}_n')\simeq \Ind(\mathcal{X}_n'-\{x_6\})\vee\Sigma(\Ind(\mathcal{X}_n'-N_{\mathcal{X}_n'}[x_6]))$. 
		We have $\mathcal{X}_n'-N_{\mathcal{X}_n'}[x_6]\cong \Gamma_n$. Hence $\Ind(\mathcal{X}_n')\simeq \Ind(\mathcal{X}_n'-\{x_6\})\vee\Sigma(\Ind(\Gamma_n))$, and therefore $\Ind(\mathcal{X}_n)\simeq \Sigma (\Ind(\mathcal{X}_n'-\{x_6\}))\vee\Sigma^2(\Ind(\Gamma_n))$.
		\begin{figure}[H]
			\begin{subfigure}{0.30\textwidth}
				\centering
				\begin{tikzpicture}[scale = 0.27, vertices/.style = {draw, fill = black, circle, inner sep = 0.5pt}]
					
					\node[vertices,label = above:{{\tiny $x_4$}}] (x4) at (8, 4.5) {}; 
					\node[vertices,label = below:{{\tiny $x_5$}}] (x5) at (8, 1.5) {}; 
					\node[vertices,label = above:{{\tiny $x_6$}}] (x6) at (9.2, 3) {};
					\node[vertices,label = above:{{\tiny $x_7$}}] (x7) at (10.5, 4.5) {}; 
					\node[vertices,label = below:{{\tiny $x_8$}}] (x8) at (10.5, 1.5) {};
					
					\node[vertices,label = above:{{\tiny $b_1$}}] (b1) at (12, 3) {};
					\node[vertices,label = below:{{\tiny $d_1$}}] (d1) at (12, 0) {};
					
					\node[vertices,label = above:{{\tiny $a_1$}}] (a) at (13.5, 4.5) {};
					\node[vertices,label = below:{{\tiny $c_1$}}] (c) at (13.5, 1.5) {};
					\node[vertices,label = below:{{\tiny $e_1$}}] (e) at (13.5, -1.5) {};
					
					\node[vertices,label = above:{{\tiny $b_2$}}] (b) at (15, 3) {};
					\node[vertices,label = below:{{\tiny $d_2$}}] (d) at (15, 0) {};
					
					\node[vertices, label = {[xshift = .32cm,yshift = -.12cm]above:{{{\tiny $b_{(n-1)}$}}}}] (1) at ($(b)+(3.3,0)$) {};
					\node[vertices, label = {[xshift = .3cm]below:{{{\tiny $d_{(n-1)}$}}}}] (2) at ($(d)+(3.3,0)$) {};
					
					\node[vertices,label = {[xshift = .3cm,yshift = -.13cm]above:{{{\tiny $a_{(n-1)}$}}}}] (3) at ($(a)+(6.3,0)$) {}; 
					\node[vertices,label = {[xshift = .27cm]below:{{{\tiny $c_{(n-1)}$}}}}] (4) at ($(c)+(6.3,0)$) {};
					\node[vertices,label = {[xshift = .3cm]below:{{{\tiny $e_{(n-1)}$}}}}] (5) at ($(e)+(6.3,0)$) {};
					
					\node[vertices,label = above:{{{\tiny $b_n$}}}] (6) at ($(b)+(6.3,0)$) {};
					\node[vertices,label = below:{{{\tiny $d_n$}}}] (7) at ($(d)+(6.3,0)$) {};
					
					\draw (a) -- ++(2.3,0);
					\draw (c) -- ++(2.3,0);
					\draw (e) -- ++(2.3,0);
					\draw (3) -- ++(-2.3,0);
					\draw (4) -- ++(-2.3,0);
					\draw (5) -- ++(-2.3,0);
					\draw (b) -- ++(1,1);
					\draw (b) -- ++(1,-1);
					\draw (d) -- ++(1,1);
					\draw (d) -- ++(1,-1);
					\draw (1) -- ++(-1,1);
					\draw (1) -- ++(-1,-1);
					\draw (2) -- ++(-1,1);
					\draw (2) -- ++(-1,-1);
					
					\node at ($(b)!0.5!(1)$) {$\ldots$};
					\node at ($(d)!0.5!(2)$) {$\ldots$};
					\node at ($(a)!0.5!(3)$) {$\ldots$};
					\node at ($(c)!0.5!(4)$) {$\ldots$};
					\node at ($(e)!0.5!(5)$) {$\ldots$};
					
					\foreach \vertex/\neighbors in {
						x7/{x6,x4,b1},
						x8/{x6,x5,b1,d1},
						x6/{x4,x5},
						x4/{x5},
						b1/{d1},
						a/{x7,b1,b},
						b/{c,d},
						c/{x8,b1,d1,d},
						d/{e},
						e/{d1},
						1/{2,3,4},
						2/{4,5},
						3/{6},
						4/{6,7},
						5/{7},
						6/{7}}
					{
						\foreach \neighbor in \neighbors
						\draw [-] (\vertex)--(\neighbor);
					}
				\end{tikzpicture}
				\subcaption{$\mathcal{X}_n'$}
				\label{fig:Xn'}
			\end{subfigure}
			\hfill
			\begin{subfigure}{0.30\textwidth}
				\centering
				\begin{tikzpicture}[scale = 0.3, vertices/.style = {draw, fill = black, circle, inner sep = 0.5pt}]
					
					\node[vertices,label = above:{{\tiny $x_4$}}] (x4) at (8.7, 4.5) {}; 
					\node[vertices,label = below:{{\tiny $x_5$}}] (x5) at (8.7, 1.5) {}; 
					\node[vertices,label = above:{{\tiny $x_7$}}] (x7) at (10.5, 4.5) {}; 
					\node[vertices,label = below:{{\tiny $x_8$}}] (x8) at (10.5, 1.5) {};
					
					\node[vertices,label = above:{{\tiny $b_1$}}] (b1) at (12, 3) {};
					\node[vertices,label = below:{{\tiny $d_1$}}] (d1) at (12, 0) {};
					
					\node[vertices,label = above:{{\tiny $a_1$}}] (a) at (13.5, 4.5) {};
					\node[vertices,label = below:{{\tiny $c_1$}}] (c) at (13.5, 1.5) {};
					\node[vertices,label = below:{{\tiny $e_1$}}] (e) at (13.5, -1.5) {};
					
					\node[vertices,label = above:{{\tiny $b_2$}}] (b) at (15, 3) {};
					\node[vertices,label = below:{{\tiny $d_2$}}] (d) at (15, 0) {};
					
					\node[vertices, label = {[xshift = .32cm,yshift = -.12cm]above:{{{\tiny $b_{(n-1)}$}}}}] (1) at ($(b)+(3.3,0)$) {};
					\node[vertices, label = {[xshift = .3cm]below:{{{\tiny $d_{(n-1)}$}}}}] (2) at ($(d)+(3.3,0)$) {};
					
					\node[vertices,label = {[xshift = .3cm,yshift = -.13cm]above:{{{\tiny $a_{(n-1)}$}}}}] (3) at ($(a)+(6.3,0)$) {}; 
					\node[vertices,label = {[xshift = .27cm]below:{{{\tiny $c_{(n-1)}$}}}}] (4) at ($(c)+(6.3,0)$) {};
					\node[vertices,label = {[xshift = .3cm]below:{{{\tiny $e_{(n-1)}$}}}}] (5) at ($(e)+(6.3,0)$) {};
					
					\node[vertices,label = above:{{{\tiny $b_n$}}}] (6) at ($(b)+(6.3,0)$) {};
					\node[vertices,label = below:{{{\tiny $d_n$}}}] (7) at ($(d)+(6.3,0)$) {};
					
					\draw (a) -- ++(2.3,0);
					\draw (c) -- ++(2.3,0);
					\draw (e) -- ++(2.3,0);
					\draw (3) -- ++(-2.3,0);
					\draw (4) -- ++(-2.3,0);
					\draw (5) -- ++(-2.3,0);
					\draw (b) -- ++(1,1);
					\draw (b) -- ++(1,-1);
					\draw (d) -- ++(1,1);
					\draw (d) -- ++(1,-1);
					\draw (1) -- ++(-1,1);
					\draw (1) -- ++(-1,-1);
					\draw (2) -- ++(-1,1);
					\draw (2) -- ++(-1,-1);
					
					\node at ($(b)!0.5!(1)$) {$\ldots$};
					\node at ($(d)!0.5!(2)$) {$\ldots$};
					\node at ($(a)!0.5!(3)$) {$\ldots$};
					\node at ($(c)!0.5!(4)$) {$\ldots$};
					\node at ($(e)!0.5!(5)$) {$\ldots$};
					
					\foreach \vertex/\neighbors in {
						x7/{x4,b1},
						x8/{x5,b1,d1},
						x4/{x5},
						b1/{d1},
						a/{x7,b1,b},
						b/{c,d},
						c/{x8,b1,d1,d},
						d/{e},
						e/{d1},
						1/{2,3,4},
						2/{4,5},
						3/{6},
						4/{6,7},
						5/{7},
						6/{7}}
					{
						\foreach \neighbor in \neighbors
						\draw [-] (\vertex)--(\neighbor);
					}
				\end{tikzpicture}
				\subcaption{$\mathcal{X}_n'-\{x_6\}$}
				\label{fig:Xn'_x6}
			\end{subfigure}
			\begin{subfigure}{0.333\textwidth}
				\centering 
				\begin{tikzpicture}[scale = 0.3, vertices/.style = {draw, fill = black, circle, inner
						sep = 0.5pt}]
					
					\node[vertices,label = above:{{\tiny $x_4$}}] (x4) at (9.2, 4.5) {}; 
					\node[vertices,label = below:{{\tiny $x_5$}}] (x5) at (9.2, 1.5) {}; 
					\node[vertices,label = above:{{\tiny $x_7$}}] (x7) at (10.7, 4.5) {}; 
					\node[vertices,label = below:{{\tiny $x_8$}}] (x8) at (10.7, 1.5) {};
					
					\node[vertices,label = above:{{\tiny $b_1$}}] (b1) at (12, 3) {};
					\node[vertices,label = below:{{\tiny $d_1$}}] (d1) at (12, 0) {};
					
					\node[vertices,label = above:{{\tiny $a_1$}}] (a) at (13.5, 4.5) {};
					\node[vertices,label = below:{{\tiny $c_1$}}] (c) at (13.5, 1.5) {};
					\node[vertices,label = below:{{\tiny $e_1$}}] (e) at (13.5, -1.5) {};
					
					\node[vertices,label = above:{{\tiny $b_2$}}] (b) at (15, 3) {};
					\node[vertices,label = below:{{\tiny $d_2$}}] (d) at (15, 0) {};
					
					\node[vertices, label = {[xshift = .32cm,yshift = -.12cm]above:{{{\tiny $b_{(n-1)}$}}}}] (1) at ($(b)+(3.35,0)$) {};
					\node[vertices, label = {[xshift = .3cm]below:{{{\tiny $d_{(n-1)}$}}}}] (2) at ($(d)+(3.35,0)$) {};
					
					\node[vertices,label = {[xshift = .3cm,yshift = -.13cm]above:{{{\tiny $a_{(n-1)}$}}}}] (3) at ($(a)+(6.35,0)$) {}; 
					\node[vertices,label = {[xshift = .27cm]below:{{{\tiny $c_{(n-1)}$}}}}] (4) at ($(c)+(6.35,0)$) {};
					\node[vertices,label = {[xshift = .3cm]below:{{{\tiny $e_{(n-1)}$}}}}] (5) at ($(e)+(6.35,0)$) {};
					
					\node[vertices,label = above:{{{\tiny $b_n$}}}] (6) at ($(b)+(6.35,0)$) {};
					\node[vertices,label = below:{{{\tiny $d_n$}}}] (7) at ($(d)+(6.35,0)$) {};
					
					\draw (a) -- ++(2.3,0);
					\draw (c) -- ++(2.3,0);
					\draw (e) -- ++(2.3,0);
					\draw (3) -- ++(-2.3,0);
					\draw (4) -- ++(-2.3,0);
					\draw (5) -- ++(-2.3,0);
					\draw (b) -- ++(1,1);
					\draw (b) -- ++(1,-1);
					\draw (d) -- ++(1,1);
					\draw (d) -- ++(1,-1);
					\draw (1) -- ++(-1,1);
					\draw (1) -- ++(-1,-1);
					\draw (2) -- ++(-1,1);
					\draw (2) -- ++(-1,-1);
					
					\node at ($(b)!0.5!(1)$) {$\ldots$};
					\node at ($(d)!0.5!(2)$) {$\ldots$};
					\node at ($(a)!0.5!(3)$) {$\ldots$};
					\node at ($(c)!0.5!(4)$) {$\ldots$};
					\node at ($(e)!0.5!(5)$) {$\ldots$};
					
					\foreach \vertex/\neighbors in {
						x7/{x4,b1},
						x8/{x5,d1},
						x4/{x5},
						b1/{d1},
						a/{x7,b1,b},
						b/{c,d},
						c/{x8,b1,d1,d},
						d/{e},
						e/{d1},
						1/{2,3,4},
						2/{4,5},
						3/{6},
						4/{6,7},
						5/{7},
						6/{7}}
					{
						\foreach \neighbor in \neighbors
						\draw [-] (\vertex)--(\neighbor);
					}
				\end{tikzpicture}
				\subcaption{$\mathcal{X}_n''$}
				\label{fig:Xn''}
			\end{subfigure}
			\vspace{-.18cm}
			\caption{}
			\label{}
		\end{figure}  
		
		Observe that $x_4$ is an isolated vertex in $(\mathcal{X}_n'-\{x_6\})-N_{\mathcal{X}_n'-\{x_6\}}[\{x_8,b_1\}]$ (see \Cref{fig:Xn'_x6}). Hence $\Ind((\mathcal{X}_n'-\{x_6\})-N_{\mathcal{X}_n'-\{x_6\}}[\{x_8,b_1\}])$ is contractible. From \Cref{lemma:edge_contractible}, $ \Ind(\mathcal{X}_n'-\{x_6\})\simeq \Ind((\mathcal{X}_n'-\{x_6\})- \{(x_8,b_1)\}).$ Let $\mathcal{X}_n'': = (\mathcal{X}_n'-\{x_6\})- \{(x_8,b_1)\}$ (see \Cref{fig:Xn''}). Then $\Ind(\mathcal{X}_n'-\{x_6\})\simeq \Ind(\mathcal{X}_n'')$. It follows that $\Ind(\mathcal{X}_n)\simeq \Sigma(\Ind(\mathcal{X}_n''))\vee \Sigma^2(\Ind(\Gamma_n)).$
		
		Now, since $\mathcal{X}_n''-N_{\mathcal{X}_n''}[\{x_7,d_1\}]$ contains an isolated vertex $x_5$, $\Ind(\mathcal{X}_n''-N_{\mathcal{X}_n''}[\{x_7,d_1\}])$ is contractible. Hence $ \Ind(\mathcal{X}_n'')\simeq \Ind(\mathcal{X}_n''\cup \{(x_7,d_1)\})$. Moreover, in $(\mathcal{X}_n''\cup \{(x_7,d_1)\})-N_{\mathcal{X}_n''\cup \{(x_7,d_1)\}}[\{x_7,c_1\}]$, $x_5$ is an isolated vertex. Therefore $ \Ind(\mathcal{X}_n''\cup \{(x_7,d_1)\})\simeq \Ind((\mathcal{X}_n''\cup \{(x_7,d_1)\})\cup \{(x_7,c_1)\}).$ Let $\mathcal{X}_n''': = (\mathcal{X}_n''\cup \{(x_7,d_1)\})\cup \{(x_7,c_1)\}$ (see \Cref{fig:Xn'''}). Then $\Ind(\mathcal{X}_n'')\simeq \Ind(\mathcal{X}_n''')$. This yields $\Ind(\mathcal{X}_n) \simeq \Sigma(\Ind(\mathcal{X}_n''')) \vee\Sigma^2(\Ind(\Gamma_n)).$
		
		We have $N_{\mathcal{X}_n'''}[b_1]\subseteq N_{\mathcal{X}_n'''}[x_7]$. By \Cref{corollary:N[u]subsetN[v]_lk(v)_del(v)}, $\Ind(\mathcal{X}_n''')\simeq \Ind(\mathcal{X}_n'''-\{x_7\})\vee\Sigma(\Ind(\mathcal{X}_n'''-N_{\mathcal{X}_n'''}[x_7]))$. Note that $\mathcal{X}_n'''-N_{\mathcal{X}_n'''}[x_7]\cong \mathcal{U}_{n-1}\sqcup P_2$ (see \Cref{fig:Xn'''_lk_x7}). Hence $\Ind(\mathcal{X}_n''')\simeq \Ind(\mathcal{X}_n'''-\{x_7\})\vee\Sigma^2(\Ind(\mathcal{U}_{n-1}))$.
		\begin{figure}[h!]
			\centering
			\begin{subfigure}{0.31\textwidth}
				\centering
				\begin{tikzpicture}[scale = 0.27, vertices/.style = {draw, fill = black, circle, inner sep = 0.5pt}]
					
					\node[vertices,label = above:{{\tiny $x_4$}}] (x4) at (9.2, 4.5) {}; 
					\node[vertices,label = below:{{\tiny $x_5$}}] (x5) at (9.2, 1.5) {}; 
					\node[vertices,label = above:{{\tiny $x_7$}}] (x7) at (10.7, 4.5) {}; 
					\node[vertices,label = below:{{\tiny $x_8$}}] (x8) at (10.7, 1.5) {};
					
					\node[vertices,label = above:{{\tiny $b_1$}}] (b1) at (12, 3) {};
					\node[vertices,label = below:{{\tiny $d_1$}}] (d1) at (12, 0) {};
					
					\node[vertices,label = above:{{\tiny $a_1$}}] (a) at (13.5, 4.5) {};
					\node[vertices,label = below:{{\tiny $c_1$}}] (c) at (13.5, 1.5) {};
					\node[vertices,label = below:{{\tiny $e_1$}}] (e) at (13.5, -1.5) {};
					
					\node[vertices,label = above:{{\tiny $b_2$}}] (b) at (15, 3) {};
					\node[vertices,label = below:{{\tiny $d_2$}}] (d) at (15, 0) {};
					
					\node[vertices, label = {[xshift = .32cm,yshift = -.12cm]above:{{{\tiny $b_{(n-1)}$}}}}] (1) at ($(b)+(3.35,0)$) {};
					\node[vertices, label = {[xshift = .3cm]below:{{{\tiny $d_{(n-1)}$}}}}] (2) at ($(d)+(3.35,0)$) {};
					
					\node[vertices,label = {[xshift = .3cm,yshift = -.13cm]above:{{{\tiny $a_{(n-1)}$}}}}] (3) at ($(a)+(6.35,0)$) {}; 
					\node[vertices,label = {[xshift = .27cm]below:{{{\tiny $c_{(n-1)}$}}}}] (4) at ($(c)+(6.35,0)$) {};
					\node[vertices,label = {[xshift = .3cm]below:{{{\tiny $e_{(n-1)}$}}}}] (5) at ($(e)+(6.35,0)$) {};
					
					\node[vertices,label = above:{{{\tiny $b_n$}}}] (6) at ($(b)+(6.35,0)$) {};
					\node[vertices,label = below:{{{\tiny $d_n$}}}] (7) at ($(d)+(6.35,0)$) {};
					
					\draw (a) -- ++(2.3,0);
					\draw (c) -- ++(2.3,0);
					\draw (e) -- ++(2.3,0);
					\draw (3) -- ++(-2.3,0);
					\draw (4) -- ++(-2.3,0);
					\draw (5) -- ++(-2.3,0);
					\draw (b) -- ++(1,1);
					\draw (b) -- ++(1,-1);
					\draw (d) -- ++(1,1);
					\draw (d) -- ++(1,-1);
					\draw (1) -- ++(-1,1);
					\draw (1) -- ++(-1,-1);
					\draw (2) -- ++(-1,1);
					\draw (2) -- ++(-1,-1);
					
					\node at ($(b)!0.5!(1)$) {$\ldots$};
					\node at ($(d)!0.5!(2)$) {$\ldots$};
					\node at ($(a)!0.5!(3)$) {$\ldots$};
					\node at ($(c)!0.5!(4)$) {$\ldots$};
					\node at ($(e)!0.5!(5)$) {$\ldots$};
					
					\foreach \vertex/\neighbors in {
						x7/{x4,b1,d1},
						x8/{x5,d1},
						x4/{x5},
						b1/{d1},
						a/{x7,b1,b},
						b/{c,d},
						c/{x8,b1,d1,d},
						d/{e},
						e/{d1},
						1/{2,3,4},
						2/{4,5},
						3/{6},
						4/{6,7},
						5/{7},
						6/{7}}
					{
						\foreach \neighbor in \neighbors
						\draw [-] (\vertex)--(\neighbor);
					}
					\draw [-, bend left = 25] (x7) to (c);
				\end{tikzpicture} 
				\subcaption{$\mathcal{X}_n'''$}
				\label{fig:Xn'''}
			\end{subfigure}
			\begin{subfigure}{0.30\textwidth}
				\centering
				\begin{tikzpicture}[scale = 0.3, vertices/.style = {draw, fill = black, circle, inner sep = 0.5pt}]
					
					\node[vertices,label = below:{{\tiny $x_5$}}] (x5) at (11.3, 1.5) {}; 
					\node[vertices,label = below:{{\tiny $x_8$}}] (x8) at (13.3, 1.5) {};
					
					\node[vertices,label = below:{{\tiny $e_1$}}] (e1) at (13.5, -1.5) {};
					
					\node[vertices,label = above:{{\tiny $b_2$}}] (b2) at (15, 3) {};
					\node[vertices,label = below:{{\tiny $d_2$}}] (d2) at (15, 0) {};
					
					\node[vertices,label = above:{{\tiny $a_2$}}] (a) at (16.5, 4.5) {};
					\node[vertices,label = below:{{\tiny $c_2$}}] (c) at (16.5, 1.5) {};
					\node[vertices,label = below:{{\tiny $e_2$}}] (e) at (16.5, -1.5) {};
					
					\node[vertices,label = above:{{\tiny $b_3$}}] (b) at (18, 3) {};
					\node[vertices,label = below:{{\tiny $d_3$}}] (d) at (18, 0) {};
					
					\node[vertices, label = {[xshift = .32cm,yshift = -.12cm]above:{{{\tiny $b_{(n-1)}$}}}}] (1) at ($(b)+(3.3,0)$) {};
					\node[vertices, label = {[xshift = .3cm]below:{{{\tiny $d_{(n-1)}$}}}}] (2) at ($(d)+(3.3,0)$) {};
					
					\node[vertices,label = {[xshift = .3cm,yshift = -.13cm]above:{{{\tiny $a_{(n-1)}$}}}}] (3) at ($(a)+(6.3,0)$) {}; 
					\node[vertices,label = {[xshift = .27cm]below:{{{\tiny $c_{(n-1)}$}}}}] (4) at ($(c)+(6.3,0)$) {};
					\node[vertices,label = {[xshift = .3cm]below:{{{\tiny $e_{(n-1)}$}}}}] (5) at ($(e)+(6.3,0)$) {};
					
					\node[vertices,label = above:{{{\tiny $b_n$}}}] (6) at ($(b)+(6.3,0)$) {};
					\node[vertices,label = below:{{{\tiny $d_n$}}}] (7) at ($(d)+(6.3,0)$) {};
					
					\draw (a) -- ++(2.3,0);
					\draw (c) -- ++(2.3,0);
					\draw (e) -- ++(2.3,0);
					\draw (3) -- ++(-2.3,0);
					\draw (4) -- ++(-2.3,0);
					\draw (5) -- ++(-2.3,0);
					\draw (b) -- ++(1,1);
					\draw (b) -- ++(1,-1);
					\draw (d) -- ++(1,1);
					\draw (d) -- ++(1,-1);
					\draw (1) -- ++(-1,1);
					\draw (1) -- ++(-1,-1);
					\draw (2) -- ++(-1,1);
					\draw (2) -- ++(-1,-1);
					
					\node at ($(b)!0.5!(1)$) {$\ldots$};
					\node at ($(d)!0.5!(2)$) {$\ldots$};
					\node at ($(a)!0.5!(3)$) {$\ldots$};
					\node at ($(c)!0.5!(4)$) {$\ldots$};
					\node at ($(e)!0.5!(5)$) {$\ldots$};
					
					\foreach \vertex/\neighbors in {
						x8/{x5},
						d2/{e1,b2},
						a/{b2,b},
						b/{c,d},
						c/{b2,d2,d},
						d/{e},
						e/{e1,d2},
						1/{2,3,4},
						2/{4,5},
						3/{6},
						4/{6,7},
						5/{7},
						6/{7}}
					{
						\foreach \neighbor in \neighbors
						\draw [-] (\vertex)--(\neighbor);
					}
				\end{tikzpicture}
				\subcaption{$\mathcal{X}_n'''-N_{\mathcal{X}_n'''}[x_7]$}
				\label{fig:Xn'''_lk_x7}
			\end{subfigure}
			\begin{subfigure}{0.30\textwidth}
				\centering
				\begin{tikzpicture}[scale = 0.3, vertices/.style = {draw, fill = black, circle, inner sep = 0.5pt}]
					
					\node[vertices,label = above:{{\tiny $x_4$}}] (x4) at (9.1, 4.5) {}; 
					\node[vertices,label = below:{{\tiny $x_5$}}] (x5) at (9.1, 1.5) {}; 
					\node[vertices,label = below:{{\tiny $x_8$}}] (x8) at (10.7, 1.5) {};
					
					\node[vertices,label = above:{{\tiny $b_1$}}] (b1) at (12, 3) {};
					\node[vertices,label = below:{{\tiny $d_1$}}] (d1) at (12, 0) {};
					
					\node[vertices,label = above:{{\tiny $a_1$}}] (a) at (13.5, 4.5) {};
					\node[vertices,label = below:{{\tiny $c_1$}}] (c) at (13.5, 1.5) {};
					\node[vertices,label = below:{{\tiny $e_1$}}] (e) at (13.5, -1.5) {};
					
					\node[vertices,label = above:{{\tiny $b_2$}}] (b) at (15, 3) {};
					\node[vertices,label = below:{{\tiny $d_2$}}] (d) at (15, 0) {};
					
					\node[vertices, label = {[xshift = .32cm,yshift = -.12cm]above:{{{\tiny $b_{(n-1)}$}}}}] (1) at ($(b)+(3.3,0)$) {};
					\node[vertices, label = {[xshift = .3cm]below:{{{\tiny $d_{(n-1)}$}}}}] (2) at ($(d)+(3.3,0)$) {};
					
					\node[vertices,label = {[xshift = .3cm,yshift = -.13cm]above:{{{\tiny $a_{(n-1)}$}}}}] (3) at ($(a)+(6.3,0)$) {}; 
					\node[vertices,label = {[xshift = .27cm]below:{{{\tiny $c_{(n-1)}$}}}}] (4) at ($(c)+(6.3,0)$) {};
					\node[vertices,label = {[xshift = .3cm]below:{{{\tiny $e_{(n-1)}$}}}}] (5) at ($(e)+(6.3,0)$) {};
					
					\node[vertices,label = above:{{{\tiny $b_n$}}}] (6) at ($(b)+(6.3,0)$) {};
					\node[vertices,label = below:{{{\tiny $d_n$}}}] (7) at ($(d)+(6.3,0)$) {};
					
					\draw (a) -- ++(2.3,0);
					\draw (c) -- ++(2.3,0);
					\draw (e) -- ++(2.3,0);
					\draw (3) -- ++(-2.3,0);
					\draw (4) -- ++(-2.3,0);
					\draw (5) -- ++(-2.3,0);
					\draw (b) -- ++(1,1);
					\draw (b) -- ++(1,-1);
					\draw (d) -- ++(1,1);
					\draw (d) -- ++(1,-1);
					\draw (1) -- ++(-1,1);
					\draw (1) -- ++(-1,-1);
					\draw (2) -- ++(-1,1);
					\draw (2) -- ++(-1,-1);
					
					\node at ($(b)!0.5!(1)$) {$\ldots$};
					\node at ($(d)!0.5!(2)$) {$\ldots$};
					\node at ($(a)!0.5!(3)$) {$\ldots$};
					\node at ($(c)!0.5!(4)$) {$\ldots$};
					\node at ($(e)!0.5!(5)$) {$\ldots$};
					
					\foreach \vertex/\neighbors in {
						x8/{x5,d1},
						x4/{x5},
						b1/{d1},
						a/{b1,b},
						b/{c,d},
						c/{x8,b1,d1,d},
						d/{e},
						e/{d1},
						1/{2,3,4},
						2/{4,5},
						3/{6},
						4/{6,7},
						5/{7},
						6/{7}}
					{
						\foreach \neighbor in \neighbors
						\draw [-] (\vertex)--(\neighbor);
					}
				\end{tikzpicture}
				\subcaption{$\mathcal{X}_n'''-\{x_7\}$}
				\label{fig:Xn'''_x7}
			\end{subfigure}
			\vspace{-.2cm}
			\caption{}
		\end{figure}
		Further, $x_4$ is a simplicial vertex in $\mathcal{X}_n'''-\{x_7\}$ with $N_{\mathcal{X}_n'''-\{x_7\}}[x_4] = \{x_5\}$ (see \Cref{fig:Xn'''_x7}). By \Cref{theorem:simplicial_vertex}, $\Ind(\mathcal{X}_n'''-\{x_7\})\simeq \Sigma(\Ind((\mathcal{X}_n'''-\{x_7\})-N_{\mathcal{X}_n'''-\{x_7\}}[x_5]))$. Note that $(\mathcal{X}_n'''-\{x_7\})-N_{\mathcal{X}_n'''-\{x_7\}}[x_5]\cong\Gamma_n$. Hence $\Ind(\mathcal{X}_n'''-\{x_7\})\simeq \Sigma( \Ind(\Gamma_n))$. 
		It follows that $\Ind(\mathcal{X}_n''')\simeq \Sigma( \Ind(\Gamma_n))\vee\Sigma^2(\Ind(\mathcal{U}_{n-1}))$. Thus $\Ind(\mathcal{X}_n)\simeq \vee_2 \Sigma^2(\Ind(\Gamma_n))\vee \Sigma^3 (\Ind(\mathcal{U}_{n-1})).$
	\end{proof}
	
	\begin{claim}\label{claim:Yn_relation}
		For $n\ge3$, $\Ind(\mathcal{Y}_n)\simeq \Sigma^2(\Ind(\mathcal{W}_{n-1}))\vee \Sigma^3(\Ind(\mathcal{W}_{n-2})).$
	\end{claim}
	
	\begin{proof}
		Since $y_2$ is a simplicial vertex in $\mathcal{Y}_n$ with $N_{\mathcal{Y}_n}(y_2) = \{y_6\}$ (\Cref{fig:Yn_2}), \Cref{theorem:simplicial_vertex} gives $\Ind(\mathcal{Y}_n)\simeq \Sigma(\Ind(\mathcal{Y}_n-N_{\mathcal{Y}_n}[y_6]))$. Let $\mathcal{Y}_n': = \mathcal{Y}_n-N_{\mathcal{Y}_n}[y_6]$ (see \Cref{fig:Yn'}). Then $\Ind(\mathcal{Y}_n)\simeq \Sigma(\Ind(\mathcal{Y}_n')).$
		\begin{figure}[H]
			\centering
			\begin{subfigure}{0.40\textwidth}
				\centering
				\begin{tikzpicture}[scale = 0.27, vertices/.style = {draw, fill = black, circle, inner
						sep = 0.5pt}]
					
					\node[vertices,label = below:{{\tiny $y_1$}}] (y1) at (7.5, 1.5) {};
					\node[vertices,label = above:{{\tiny $y_3$}}] (y3) at (9, 3) {};
					\node[vertices,label = above:{{\tiny $y_4$}}] (y4) at (10.5, 4.5) {}; 
					\node[vertices,label = below:{{\tiny $y_5$}}] (y5) at (10.5, 1.5) {}; 
					
					\node[vertices,label = above:{{\tiny $b_1$}}] (b1) at (12, 3) {};
					
					\node[vertices,label = above:{{\tiny $a_1$}}] (a1) at (13.5, 4.5) {};
					\node[vertices,label = below:{{\tiny $c_1$}}] (c1) at (13.5, 1.5) {};
					
					\node[vertices,label = above:{{\tiny $b_2$}}] (b2) at (15, 3) {};
					\node[vertices,label = below:{{\tiny $d_2$}}] (d2) at (15, 0) {};
					
					\node[vertices,label = above:{{\tiny $a_2$}}] (a) at (16.5, 4.5) {};
					\node[vertices,label = below:{{\tiny $c_2$}}] (c) at (16.5, 1.5) {};
					\node[vertices,label = below:{{\tiny $e_2$}}] (e) at (16.5, -1.5) {};
					
					\node[vertices,label = above:{{\tiny $b_3$}}] (b) at (18, 3) {};
					\node[vertices,label = below:{{\tiny $d_3$}}] (d) at (18, 0) {};
					
					\node[vertices, label = {[xshift = .32cm,yshift = -.12cm]above:{{{\tiny $b_{(n-1)}$}}}}] (1) at ($(b)+(4,0)$) {};
					\node[vertices, label = {[xshift = .3cm]below:{{{\tiny $d_{(n-1)}$}}}}] (2) at ($(d)+(4,0)$) {};
					
					\node[vertices,label = {[xshift = .3cm,yshift = -.13cm]above:{{{\tiny $a_{(n-1)}$}}}}] (3) at ($(a)+(7,0)$) {}; 
					\node[vertices,label = {[xshift = 0cm]right:{{{\tiny $c_{(n-1)}$}}}}] (4) at ($(c)+(7,0)$) {};
					\node[vertices,label = {[xshift = .3cm]below:{{{\tiny $e_{(n-1)}$}}}}] (5) at ($(e)+(7,0)$) {};
					
					\node[vertices,label = right:{{{\tiny $b_n$}}}] (6) at ($(b)+(7,0)$) {};
					\node[vertices,label = right:{{{\tiny $d_n$}}}] (7) at ($(d)+(7,0)$) {};
					
					\draw (a) -- ++(2.3,0);
					\draw (c) -- ++(2.3,0);
					\draw (e) -- ++(2.3,0);
					\draw (3) -- ++(-2.3,0);
					\draw (4) -- ++(-2.3,0);
					\draw (5) -- ++(-2.3,0);
					\draw (b) -- ++(1,1);
					\draw (b) -- ++(1,-1);
					\draw (d) -- ++(1,1);
					\draw (d) -- ++(1,-1);
					\draw (1) -- ++(-1,1);
					\draw (1) -- ++(-1,-1);
					\draw (2) -- ++(-1,1);
					\draw (2) -- ++(-1,-1);
					
					\node at ($(b)!0.5!(1)$) {$\ldots$};
					\node at ($(d)!0.5!(2)$) {$\ldots$};
					\node at ($(a)!0.5!(3)$) {$\ldots$};
					\node at ($(c)!0.5!(4)$) {$\ldots$};
					\node at ($(e)!0.5!(5)$) {$\ldots$};
					
					\foreach \vertex/\neighbors in {
						y1/{y3,y5},
						y3/{y4,y5},
						y4/{a1,b1},
						y5/{b1,c1},
						b1/{a1,c1},
						a1/{b2},
						c1/{b2,d2},
						b2/{d2},
						a/{a1,b2,b},
						b/{c,d},
						c/{c1,b2,d2,d},
						d/{e},
						e/{d2},
						1/{2,3,4},
						2/{4,5},
						3/{6},
						4/{6,7},
						5/{7},
						6/{7}}
					{
						\foreach \neighbor in \neighbors
						\draw [-] (\vertex)--(\neighbor);
					}   
				\end{tikzpicture}
				\subcaption{$\mathcal{Y}_n'$}
				\label{fig:Yn'}
			\end{subfigure}
			\begin{subfigure}{0.29\textwidth}
				\centering
				\begin{tikzpicture}[scale = 0.3, vertices/.style = {draw, fill = black, circle, inner sep = 0.5pt}]
					
					\node[vertices,label = above:{{\tiny $b_1$}}] (b1) at (12, 3) {};
					
					\node[vertices,label = above:{{\tiny $a_1$}}] (a1) at (13.5, 4.5) {};
					\node[vertices,label = below:{{\tiny $c_1$}}] (c1) at (13.5, 1.5) {};
					
					\node[vertices,label = above:{{\tiny $b_2$}}] (b2) at (15, 3) {};
					\node[vertices,label = below:{{\tiny $d_2$}}] (d2) at (15, 0) {};
					
					\node[vertices,label = above:{{\tiny $a_2$}}] (a) at (16.5, 4.5) {};
					\node[vertices,label = below:{{\tiny $c_2$}}] (c) at (16.5, 1.5) {};
					\node[vertices,label = below:{{\tiny $e_2$}}] (e) at (16.5, -1.5) {};
					
					\node[vertices,label = above:{{\tiny $b_3$}}] (b) at (18, 3) {};
					\node[vertices,label = below:{{\tiny $d_3$}}] (d) at (18, 0) {};
					
					\node[vertices, label = {[xshift = .32cm,yshift = -.12cm]above:{{{\tiny $b_{(n-1)}$}}}}] (1) at ($(b)+(4,0)$) {};
					\node[vertices, label = {[xshift = .3cm]below:{{{\tiny $d_{(n-1)}$}}}}] (2) at ($(d)+(4,0)$) {};
					
					\node[vertices,label = {[xshift = .3cm,yshift = -.13cm]above:{{{\tiny $a_{(n-1)}$}}}}] (3) at ($(a)+(7,0)$) {}; 
					\node[vertices,label = {[xshift = 0cm]right:{{{\tiny $c_{(n-1)}$}}}}] (4) at ($(c)+(7,0)$) {};
					\node[vertices,label = {[xshift = .3cm]below:{{{\tiny $e_{(n-1)}$}}}}] (5) at ($(e)+(7,0)$) {};
					
					\node[vertices,label = right:{{{\tiny $b_n$}}}] (6) at ($(b)+(7,0)$) {};
					\node[vertices,label = right:{{{\tiny $d_n$}}}] (7) at ($(d)+(7,0)$) {};
					
					\draw (a) -- ++(2.3,0);
					\draw (c) -- ++(2.3,0);
					\draw (e) -- ++(2.3,0);
					\draw (3) -- ++(-2.3,0);
					\draw (4) -- ++(-2.3,0);
					\draw (5) -- ++(-2.3,0);
					\draw (b) -- ++(1,1);
					\draw (b) -- ++(1,-1);
					\draw (d) -- ++(1,1);
					\draw (d) -- ++(1,-1);
					\draw (1) -- ++(-1,1);
					\draw (1) -- ++(-1,-1);
					\draw (2) -- ++(-1,1);
					\draw (2) -- ++(-1,-1);
					
					\node at ($(b)!0.5!(1)$) {$\ldots$};
					\node at ($(d)!0.5!(2)$) {$\ldots$};
					\node at ($(a)!0.5!(3)$) {$\ldots$};
					\node at ($(c)!0.5!(4)$) {$\ldots$};
					\node at ($(e)!0.5!(5)$) {$\ldots$};
					
					\foreach \vertex/\neighbors in {
						b1/{a1,c1},
						a1/{b2},
						c1/{b2,d2},
						b2/{d2},
						a/{a1,b2,b},
						b/{c,d},
						c/{c1,b2,d2,d},
						d/{e},
						e/{d2},
						1/{2,3,4},
						2/{4,5},
						3/{6},
						4/{6,7},
						5/{7},
						6/{7}}
					{
						\foreach \neighbor in \neighbors
						\draw [-] (\vertex)--(\neighbor);
					}
				\end{tikzpicture}
				\subcaption{$\mathcal{Y}_n'-N_{\mathcal{Y}_n'}[y_3]$}
				\label{fig:Yn'_lk_y3}
			\end{subfigure}
			\begin{subfigure}{0.29\textwidth}
				\centering
				\begin{tikzpicture}[scale = 0.3, vertices/.style = {draw, fill = black, circle, inner sep = 0.5pt}]
					
					\node[vertices,label = above:{{\tiny $y_4$}}] (y4) at (11.5, 4.5) {};
					
					\node[vertices,label = above:{{\tiny $a_1$}}] (a1) at (13.5, 4.5) {};
					
					\node[vertices,label = above:{{\tiny $b_2$}}] (b2) at (15, 3) {};
					\node[vertices,label = below:{{\tiny $d_2$}}] (d2) at (15, 0) {};
					
					\node[vertices,label = above:{{\tiny $a_2$}}] (a) at (16.5, 4.5) {};
					\node[vertices,label = below:{{\tiny $c_2$}}] (c) at (16.5, 1.5) {};
					\node[vertices,label = below:{{\tiny $e_2$}}] (e) at (16.5, -1.5) {};
					
					\node[vertices,label = above:{{\tiny $b_3$}}] (b) at (18, 3) {};
					\node[vertices,label = below:{{\tiny $d_3$}}] (d) at (18, 0) {};
					
					\node[vertices, label = {[xshift = .32cm,yshift = -.12cm]above:{{{\tiny $b_{(n-1)}$}}}}] (1) at ($(b)+(4,0)$) {};
					\node[vertices, label = {[xshift = .3cm]below:{{{\tiny $d_{(n-1)}$}}}}] (2) at ($(d)+(4,0)$) {};
					
					\node[vertices,label = {[xshift = .3cm,yshift = -.13cm]above:{{{\tiny $a_{(n-1)}$}}}}] (3) at ($(a)+(7,0)$) {}; 
					\node[vertices,label = {[xshift = .27cm]below:{{{\tiny $c_{(n-1)}$}}}}] (4) at ($(c)+(7,0)$) {};
					\node[vertices,label = {[xshift = .3cm]below:{{{\tiny $e_{(n-1)}$}}}}] (5) at ($(e)+(7,0)$) {};
					
					\node[vertices,label = above:{{{\tiny $b_n$}}}] (6) at ($(b)+(7,0)$) {};
					\node[vertices,label = below:{{{\tiny $d_n$}}}] (7) at ($(d)+(7,0)$) {};
					
					\draw (a) -- ++(2.3,0);
					\draw (c) -- ++(2.3,0);
					\draw (e) -- ++(2.3,0);
					\draw (3) -- ++(-2.3,0);
					\draw (4) -- ++(-2.3,0);
					\draw (5) -- ++(-2.3,0);
					\draw (b) -- ++(1,1);
					\draw (b) -- ++(1,-1);
					\draw (d) -- ++(1,1);
					\draw (d) -- ++(1,-1);
					\draw (1) -- ++(-1,1);
					\draw (1) -- ++(-1,-1);
					\draw (2) -- ++(-1,1);
					\draw (2) -- ++(-1,-1);
					
					\node at ($(b)!0.5!(1)$) {$\ldots$};
					\node at ($(d)!0.5!(2)$) {$\ldots$};
					\node at ($(a)!0.5!(3)$) {$\ldots$};
					\node at ($(c)!0.5!(4)$) {$\ldots$};
					\node at ($(e)!0.5!(5)$) {$\ldots$};
					
					\foreach \vertex/\neighbors in {
						y4/{a1},
						a1/{b2},
						b2/{d2},
						a/{a1,b2,b},
						b/{c,d},
						c/{b2,d2,d},
						d/{e},
						e/{d2},
						1/{2,3,4},
						2/{4,5},
						3/{6},
						4/{6,7},
						5/{7},
						6/{7}}
					{
						\foreach \neighbor in \neighbors
						\draw [-] (\vertex)--(\neighbor);
					}
				\end{tikzpicture}
				\subcaption{$\mathcal{Y}_n''$}
				\label{fig:Yn''}
			\end{subfigure}
			\vspace{-.2cm}
			\caption{}
			\label{fig:Yn_seq}
		\end{figure}
		Now, $y_1$ is a simplicial vertex in $\mathcal{Y}_n'$ with $N_{\mathcal{Y}_n'}(y_1) = \{y_3,y_5\}.$ Hence $\Ind(\mathcal{Y}_n')\simeq \Sigma(\Ind(\mathcal{Y}_n'-N_{\mathcal{Y}_n'}[y_3]))\vee \Sigma(\Ind(\mathcal{Y}_n'-N_{\mathcal{Y}_n'}[y_5])).$
		We have $\mathcal{Y}_n'-N_{\mathcal{Y}_n'}[y_3]\cong\mathcal{W}_{n-1}$ (see \Cref{fig:Yn'_lk_y3}). Let $\mathcal{Y}_n'': = \mathcal{Y}_n'-N_{\mathcal{Y}_n'}[y_5]$. Then $\Ind( \mathcal{Y}_n')\simeq \Sigma(\Ind(\mathcal{W}_{n-1}))\vee \Sigma(\Ind(\mathcal{Y}_n''))$. Therefore $\Ind(\mathcal{Y}_n)\simeq \Sigma^2 (\Ind(\mathcal{W}_{n-1}))\vee \Sigma^2 (\Ind(\mathcal{Y}_n''))$.
		
		In $\mathcal{Y}_n''$, $y_4$ is a simplicial vertex with $N_{\mathcal{Y}_n''}(y_4) = \{a_1\}$ (see \Cref{fig:Yn''}). Hence $\Ind(\mathcal{Y}_n'')\simeq \Sigma(\Ind(\mathcal{Y}_n''-N_{\mathcal{Y}_n''}[a_1]))$. Since $\mathcal{Y}_n''-N_{\mathcal{Y}_n''}[a_1]\cong\mathcal{W}_{n-2}$, it follows that $\Ind(\mathcal{Y}_n)\simeq \Sigma^2(\Ind(\mathcal{W}_{n-1}))\vee \Sigma^3(\Ind(\mathcal{W}_{n-2}))$.
	\end{proof}
	
	\begin{claim}\label{claim:Zn_relation}
		For $n\ge3$, $lk(z_1,\Ind(\mathcal{Z}_n))\simeq \Ind(\mathcal{Y}_n)$ and $del(z_1,\Ind(\mathcal{Z}_n))\simeq \Sigma^2 (\Ind(\mathcal{U}_n)).$
	\end{claim}
	
	\begin{proof}
		By \eqref{equation:lk_del_ind_complex}, we have $lk(z_1,\Ind(\mathcal{Z}_n))\simeq \Ind(\mathcal{Z}_n-N_{\mathcal{Z}_n}[z_1])$. Since $\mathcal{Z}_n-N_{\mathcal{Z}_n}[z_1]\cong\mathcal{Y}_n$ (\Cref{fig:lk_Zn_z1}), $lk(z_1,\Ind(\mathcal{Z}_n))\simeq \Ind(\mathcal{Y}_n).$
		\begin{figure}[h!]
			\centering
			\begin{subfigure}{0.31\textwidth}
				\centering
				\begin{tikzpicture}[scale = 0.27, vertices/.style = {draw, fill = black, circle, inner sep = 0.5pt}]
					
					\node[vertices,label = below:{{\tiny $z_4$}}] (z4) at (7.5, 1.5) {}; 
					\node[vertices,label = below:{{\tiny $z_5$}}] (z5) at (7.5, -1.5) {};
					\node[vertices,label = above:{{\tiny $z_6$}}] (z6) at (9, 3) {};
					\node[vertices,label = above:{{\tiny $z_7$}}] (z7) at (10.5, 4.5) {}; 
					\node[vertices,label = below:{{\tiny $z_8$}}] (z8) at (10.5, 1.5) {}; 
					\node[vertices,label = below:{{\tiny $z_9$}}] (z9) at (10.5, -1.5) {}; 
					
					\node[vertices,label = above:{{\tiny $b_1$}}] (b1) at (12, 3) {};
					\node[vertices,label = below:{{\tiny $d_1$}}] (d1) at (12, 0) {};
					
					\node[vertices,label = above:{{\tiny $a_1$}}] (a) at (13.5, 4.5) {};
					\node[vertices,label = below:{{\tiny $c_1$}}] (c) at (13.5, 1.5) {};
					\node[vertices,label = below:{{\tiny $e_1$}}] (e) at (13.5, -1.5) {};
					
					\node[vertices,label = above:{{\tiny $b_2$}}] (b) at (15, 3) {};
					\node[vertices,label = below:{{\tiny $d_2$}}] (d) at (15, 0) {};
					
					\node[vertices, label = {[xshift = .32cm,yshift = -.12cm]above:{{{\tiny $b_{(n-1)}$}}}}] (1) at ($(b)+(4,0)$) {};
					\node[vertices, label = {[xshift = .3cm]below:{{{\tiny $d_{(n-1)}$}}}}] (2) at ($(d)+(4,0)$) {};
					
					\node[vertices,label = {[xshift = .3cm,yshift = -.13cm]above:{{{\tiny $a_{(n-1)}$}}}}] (3) at ($(a)+(7,0)$) {}; 
					\node[vertices,label = {[xshift = .27cm]below:{{{\tiny $c_{(n-1)}$}}}}] (4) at ($(c)+(7,0)$) {};
					\node[vertices,label = {[xshift = .3cm]below:{{{\tiny $e_{(n-1)}$}}}}] (5) at ($(e)+(7,0)$) {};
					
					\node[vertices,label = above:{{{\tiny $b_n$}}}] (6) at ($(b)+(7,0)$) {};
					\node[vertices,label = below:{{{\tiny $d_n$}}}] (7) at ($(d)+(7,0)$) {};
					
					\draw (a) -- ++(2.3,0);
					\draw (c) -- ++(2.3,0);
					\draw (e) -- ++(2.3,0);
					\draw (3) -- ++(-2.3,0);
					\draw (4) -- ++(-2.3,0);
					\draw (5) -- ++(-2.3,0);
					\draw (b) -- ++(1,1);
					\draw (b) -- ++(1,-1);
					\draw (d) -- ++(1,1);
					\draw (d) -- ++(1,-1);
					\draw (1) -- ++(-1,1);
					\draw (1) -- ++(-1,-1);
					\draw (2) -- ++(-1,1);
					\draw (2) -- ++(-1,-1);
					
					\node at ($(b)!0.5!(1)$) {$\ldots$};
					\node at ($(d)!0.5!(2)$) {$\ldots$};
					\node at ($(a)!0.5!(3)$) {$\ldots$};
					\node at ($(c)!0.5!(4)$) {$\ldots$};
					\node at ($(e)!0.5!(5)$) {$\ldots$};
					
					\foreach \vertex/\neighbors in {
						z4/{z6,z8},
						z5/{z9},
						z6/{z7,z8},
						z7/{a,b1},
						z8/{b1,c,d1},
						z9/{d1,e},
						b1/{d1},
						a/{b1,b},
						b/{c,d},
						c/{b1,d1,d},
						d/{e},
						e/{d1},
						1/{2,3,4},
						2/{4,5},
						3/{6},
						4/{6,7},
						5/{7},
						6/{7}}
					{
						\foreach \neighbor in \neighbors
						\draw [-] (\vertex)--(\neighbor);
					}
				\end{tikzpicture}
				\subcaption{$\mathcal{Z}_n-N_{\mathcal{Z}_n}[z_1]$}
				\label{fig:lk_Zn_z1}
			\end{subfigure}
			\begin{subfigure}{0.34\textwidth}
				\centering
				\begin{tikzpicture}[scale = 0.3, vertices/.style = {draw, fill = black, circle, inner sep = 0.5pt}]
					
					\node[vertices,label = below:{{\tiny $z_2$}}] (z2) at (5.5, 1.5) {}; 
					\node[vertices,label = below:{{\tiny $z_3$}}] (z3) at (5.5, -1.5) {};
					\node[vertices,label = below:{{\tiny $z_4$}}] (z4) at (7.5, 1.5) {}; 
					\node[vertices,label = below:{{\tiny $z_5$}}] (z5) at (7.5, -1.5) {};
					\node[vertices,label = above:{{\tiny $z_6$}}] (z6) at (9, 3) {};
					\node[vertices,label = above:{{\tiny $z_7$}}] (z7) at (10.5, 4.5) {}; 
					\node[vertices,label = below:{{\tiny $z_8$}}] (z8) at (10.5, 1.5) {}; 
					\node[vertices,label = below:{{\tiny $z_9$}}] (z9) at (10.5, -1.5) {}; 
					
					\node[vertices,label = above:{{\tiny $b_1$}}] (b1) at (12, 3) {};
					\node[vertices,label = below:{{\tiny $d_1$}}] (d1) at (12, 0) {};
					
					\node[vertices,label = above:{{\tiny $a_1$}}] (a) at (13.5, 4.5) {};
					\node[vertices,label = below:{{\tiny $c_1$}}] (c) at (13.5, 1.5) {};
					\node[vertices,label = below:{{\tiny $e_1$}}] (e) at (13.5, -1.5) {};
					
					\node[vertices,label = above:{{\tiny $b_2$}}] (b) at (15, 3) {};
					\node[vertices,label = below:{{\tiny $d_2$}}] (d) at (15, 0) {};
					
					\node[vertices, label = {[xshift = .32cm,yshift = -.12cm]above:{{{\tiny $b_{(n-1)}$}}}}] (1) at ($(b)+(4,0)$) {};
					\node[vertices, label = {[xshift = .3cm]below:{{{\tiny $d_{(n-1)}$}}}}] (2) at ($(d)+(4,0)$) {};
					
					\node[vertices,label = {[xshift = .3cm,yshift = -.13cm]above:{{{\tiny $a_{(n-1)}$}}}}] (3) at ($(a)+(7,0)$) {}; 
					\node[vertices,label = {[xshift = .27cm]below:{{{\tiny $c_{(n-1)}$}}}}] (4) at ($(c)+(7,0)$) {};
					\node[vertices,label = {[xshift = .3cm]below:{{{\tiny $e_{(n-1)}$}}}}] (5) at ($(e)+(7,0)$) {};
					
					\node[vertices,label = above:{{{\tiny $b_n$}}}] (6) at ($(b)+(7,0)$) {};
					\node[vertices,label = below:{{{\tiny $d_n$}}}] (7) at ($(d)+(7,0)$) {};
					
					\draw (a) -- ++(2.3,0);
					\draw (c) -- ++(2.3,0);
					\draw (e) -- ++(2.3,0);
					\draw (3) -- ++(-2.3,0);
					\draw (4) -- ++(-2.3,0);
					\draw (5) -- ++(-2.3,0);
					\draw (b) -- ++(1,1);
					\draw (b) -- ++(1,-1);
					\draw (d) -- ++(1,1);
					\draw (d) -- ++(1,-1);
					\draw (1) -- ++(-1,1);
					\draw (1) -- ++(-1,-1);
					\draw (2) -- ++(-1,1);
					\draw (2) -- ++(-1,-1);
					
					\node at ($(b)!0.5!(1)$) {$\ldots$};
					\node at ($(d)!0.5!(2)$) {$\ldots$};
					\node at ($(a)!0.5!(3)$) {$\ldots$};
					\node at ($(c)!0.5!(4)$) {$\ldots$};
					\node at ($(e)!0.5!(5)$) {$\ldots$};
					
					\foreach \vertex/\neighbors in {
						z2/{z4},
						z3/{z5},
						z4/{z6,z8},
						z5/{z9},
						z6/{z7,z8},
						z7/{a,b1},
						z8/{b1,c,d1},
						z9/{d1,e},
						b1/{d1},
						a/{b1,b},
						b/{c,d},
						c/{b1,d1,d},
						d/{e},
						e/{d1},
						1/{2,3,4},
						2/{4,5},
						3/{6},
						4/{6,7},
						5/{7},
						6/{7}}
					{
						\foreach \neighbor in \neighbors
						\draw [-] (\vertex)--(\neighbor);
					}
				\end{tikzpicture}
				\subcaption{$\mathcal{Z}_n-\{z_1\}$}
				\label{fig:Zn'}
			\end{subfigure}
			\begin{subfigure}{0.31\textwidth}
				\centering
				\begin{tikzpicture}[scale = 0.3, vertices/.style = {draw, fill = black, circle, inner sep = 0.5pt}]
					
					\node[vertices,label = below:{{\tiny $z_3$}}] (z3) at (7.5, -1.5) {};
					\node[vertices,label = below:{{\tiny $z_5$}}] (z5) at (9, -1.5) {};
					\node[vertices,label = above:{{\tiny $z_7$}}] (z7) at (10.5, 4.5) {}; 
					\node[vertices,label = below:{{\tiny $z_9$}}] (z9) at (10.5, -1.5) {}; 
					
					\node[vertices,label = above:{{\tiny $b_1$}}] (b1) at (12, 3) {};
					\node[vertices,label = below:{{\tiny $d_1$}}] (d1) at (12, 0) {};
					
					\node[vertices,label = above:{{\tiny $a_1$}}] (a) at (13.5, 4.5) {};
					\node[vertices,label = below:{{\tiny $c_1$}}] (c) at (13.5, 1.5) {};
					\node[vertices,label = below:{{\tiny $e_1$}}] (e) at (13.5, -1.5) {};
					
					\node[vertices,label = above:{{\tiny $b_2$}}] (b) at (15, 3) {};
					\node[vertices,label = below:{{\tiny $d_2$}}] (d) at (15, 0) {};
					
					\node[vertices, label = {[xshift = .32cm,yshift = -.12cm]above:{{{\tiny $b_{(n-1)}$}}}}] (1) at ($(b)+(4,0)$) {};
					\node[vertices, label = {[xshift = .3cm]below:{{{\tiny $d_{(n-1)}$}}}}] (2) at ($(d)+(4,0)$) {};
					
					\node[vertices,label = {[xshift = .3cm,yshift = -.13cm]above:{{{\tiny $a_{(n-1)}$}}}}] (3) at ($(a)+(7,0)$) {}; 
					\node[vertices,label = {[xshift = .27cm]below:{{{\tiny $c_{(n-1)}$}}}}] (4) at ($(c)+(7,0)$) {};
					\node[vertices,label = {[xshift = .3cm]below:{{{\tiny $e_{(n-1)}$}}}}] (5) at ($(e)+(7,0)$) {};
					
					\node[vertices,label = above:{{{\tiny $b_n$}}}] (6) at ($(b)+(7,0)$) {};
					\node[vertices,label = below:{{{\tiny $d_n$}}}] (7) at ($(d)+(7,0)$) {};
					
					\draw (a) -- ++(2.3,0);
					\draw (c) -- ++(2.3,0);
					\draw (e) -- ++(2.3,0);
					\draw (3) -- ++(-2.3,0);
					\draw (4) -- ++(-2.3,0);
					\draw (5) -- ++(-2.3,0);
					\draw (b) -- ++(1,1);
					\draw (b) -- ++(1,-1);
					\draw (d) -- ++(1,1);
					\draw (d) -- ++(1,-1);
					\draw (1) -- ++(-1,1);
					\draw (1) -- ++(-1,-1);
					\draw (2) -- ++(-1,1);
					\draw (2) -- ++(-1,-1);
					
					\node at ($(b)!0.5!(1)$) {$\ldots$};
					\node at ($(d)!0.5!(2)$) {$\ldots$};
					\node at ($(a)!0.5!(3)$) {$\ldots$};
					\node at ($(c)!0.5!(4)$) {$\ldots$};
					\node at ($(e)!0.5!(5)$) {$\ldots$};
					
					\foreach \vertex/\neighbors in {
						z3/{z5},
						z5/{z9},
						z7/{a,b1},
						z9/{d1,e},
						b1/{d1},
						a/{b1,b},
						b/{c,d},
						c/{b1,d1,d},
						d/{e},
						e/{d1},
						1/{2,3,4},
						2/{4,5},
						3/{6},
						4/{6,7},
						5/{7},
						6/{7}}
					{
						\foreach \neighbor in \neighbors
						\draw [-] (\vertex)--(\neighbor);
					}
				\end{tikzpicture}
				\subcaption{$\mathcal{Z}_n'$}
				\label{fig:Zn''}
			\end{subfigure}
			\vspace{-.2cm}
			\caption{} 
			\label{fig:Zn_seq1}
		\end{figure}
		
		On the other hand, $del(z_1,\Ind(\mathcal{Z}_n)) \simeq \Ind(\mathcal{Z}_n-\{z_1\})$ by \eqref{equation:lk_del_ind_complex}. Since $z_2$ is a simplicial vertex in $\mathcal{Z}_n-\{z_1\}$ with $N_{\mathcal{Z}_n-\{z_1\}}(z_2) = \{z_4\}$ (\Cref{fig:Zn'}), \Cref{theorem:simplicial_vertex} gives $\Ind(\mathcal{Z}_n-\{z_1\})\simeq \Sigma(\Ind((\mathcal{Z}_n-\{z_1\})-N_{\mathcal{Z}_n-\{z_1\}}[z_4])).$ Let $\mathcal{Z}_n': = (\mathcal{Z}_n-\{z_1\})-N_{\mathcal{Z}_n-\{z_1\}}[z_4]$. Then $\Ind(\mathcal{Z}_n-\{z_1\})\simeq \Sigma(\Ind(\mathcal{Z}_n'))$.
		
		Now, $z_3$ is a simplicial vertex in $\mathcal{Z}_n'$ with $N_{\mathcal{Z}_n'}(z_3) = \{z_5\}$ (see \Cref{fig:Zn''}). Hence $\Ind(\mathcal{Z}_n')\simeq \Sigma(\Ind(\mathcal{Z}_n'-N_{\mathcal{Z}_n'}[z_5]))$. Since $\mathcal{Z}_n'-N_{\mathcal{Z}_n'}[z_5]\cong\mathcal{U}_n$, we have $\Ind(\mathcal{Z}_n-\{z_1\})\simeq \Sigma(\Ind(\mathcal{Z}_n'))\simeq \Sigma^2(\Ind(\mathcal{U}_n))$. Therefore $del(z_1,\Ind(\mathcal{Z}_n))\simeq \Sigma^2 (\Ind(\mathcal{U}_n)).$
	\end{proof}
	
	\subsubsection{Homological results}\label{subsection:lower_dimensions_homology}
	
	In this section, we prove $\tilde{H}_i(\Ind(\Gamma_n)) = 0$ for all $i\le n-2$ and $\tilde{H}_{n-1}(\Ind(\Gamma_n))\ne0$. These results are proved by induction on $n$, using the Mayer-Vietoris sequence together with the recursive relations (derived in \Cref{subsection:recursive_homotopy_equivalences}) among the independence complexes of the auxiliary graphs and of $\Ind(\Gamma_n)$.
	
	We first recall the Mayer-Vietoris sequence (see \cite{munkres_book} for more details). Let $\Delta$ be a simplicial complex, and let $\Delta_1$ and $\Delta_2$ be two subcomplexes of $\Delta$ such that $\Delta = \Delta_1\cup\Delta_2$. Then there is a long exact sequence
	$$\cdots\to H_n(\Delta_1\cap \Delta_2) \xrightarrow{a_n} H_n(\Delta_1) \oplus H_n(\Delta_2) \xrightarrow{b_n} H_n(\Delta) \xrightarrow{c_n} H_{n-1}(\Delta_1 \cap \Delta_2)\to \cdots\to H_0(\Delta) \to 0$$
	known as the Mayer-Vietoris sequence associated with the pair $(\Delta_1,\Delta_2)$. When $\Delta_1\cap \Delta_2$ is nonempty, we obtain the following long exact sequence in reduced homology:
	\begin{equation}\label{equation:mayer_vietoris_reduced}
		\cdots\to \tilde{H}_n(\Delta_1\cap \Delta_2) \xrightarrow{a_n} \tilde{H}_n(\Delta_1) \oplus \tilde{H}_n(\Delta_2) \xrightarrow{b_n} \tilde{H}_n(\Delta) \xrightarrow{c_n} \tilde{H}_{n-1}(\Delta_1 \cap \Delta_2)\to \cdots\to \tilde{H}_0(\Delta) \to 0.
	\end{equation}
	
	For a vertex $v$ of $\Delta$, we have $$\Delta = st(v,\Delta) \cup del(v,\Delta)\text{ and }lk(v,\Delta) = st(v,\Delta) \cap del(v,\Delta).$$ Since $st(v,\Delta)$ is contractible, $\tilde{H}_i(st(v,\Delta)) = 0 $ for all $i$. Therefore, substituting $\Delta_1 = st(v,\Delta) $ and $\Delta_2 = del(v,\Delta)$ in \eqref{equation:mayer_vietoris_reduced}, we get 
	\begin{equation}\label{equation:MayerVietoris}
		\cdots\to \tilde{H}_n(lk(v,\Delta)) \xrightarrow{a_n} \tilde{H}_n(del(v,\Delta)) \xrightarrow{b_n} \tilde{H}_n(\Delta) \xrightarrow{c_n} \tilde{H}_{n-1}(lk(v,\Delta))\to \cdots\to \tilde{H}_0(\Delta) \to 0.
	\end{equation}
	By the exactness of \eqref{equation:MayerVietoris}, $Im(a_i) = Ker(b_i)$, $Im(b_i) = Ker(c_i)$ and $Im(c_i) = Ker(a_{i-1})$ for all $i$, where $Im(f)$ and $Ker(f)$ denote the image and kernel of a map $f$, respectively. 
	\begin{proposition}\label{proposition:homology_up_to_n-2} 
		Let $n\geq 1$. Then
		\begin{enumerate}[label = (\roman*)]
			\item $\tilde{H}_i(\Ind(\Gamma_n)) = 0$ for all $i\le n-2$. \label{result:Gamma_n}
			\item $\tilde{H}_i(\Ind(\mathcal{U}_n)) = 0$ for all $i\le n-2$. \label{result:U_n}
			\item $\tilde{H}_i(\Ind(\mathcal{V}_n)) = 0$ for all $i\le n-1$. \label{result:V_n}
			\item $\tilde{H}_i(\Ind(\mathcal{W}_n)) = 0$ for all $i\le n-1$. \label{result:W_n}
			\item $\tilde{H}_i(\Ind(\mathcal{X}_n)) = 0$ for all $i\le n$. \label{result:X_n}
			\item $\tilde{H}_i(\Ind(\mathcal{Y}_n)) = 0$ for all $i\le n$. \label{result:Y_n}
			\item $\tilde{H}_i(\Ind(\mathcal{Z}_n)) = 0$ for all $i\le n$. \label{result:Z_n}
		\end{enumerate}
	\end{proposition}
	
	\begin{proof}
		The proof is by induction on $n$. We consider the base cases $n = 1,2.$ For these cases, the required results follow directly from the homotopy types given in \eqref{equation:Gamma_123}--\eqref{equation:Z_12} and the fact that $\tilde{H}_{-1}(\Delta) = 0$ whenever $\Delta\neq \emptyset$. Let $n\geq 3,$ and assume that the result holds for all $1\leq l\leq n-1.$ We now prove the results for $n$.
		\begin{enumerate}[label = (\roman*)]
			\item $\Ind(\Gamma_n)$. To analyze $\Ind(\Gamma_n)$, we consider the link and deletion of the vertex $b_1$ in $\Ind(\Gamma_n)$. We have $lk(b_1, \Ind(\Gamma_n))\simeq \Ind(\Gamma_n-N_{\Gamma_n}[b_1])$ and $del(b_1, \Ind(\Gamma_n))\simeq \Ind(\Gamma_n-\{b_1\})$ by \eqref{equation:lk_del_ind_complex}. 
			Observe that $\Gamma_n-N_{\Gamma_n}[b_1]\cong\mathcal{U}_{n-1}$ and $\Gamma_n-\{b_1\}\cong\mathcal{V}_{n-1}$. Hence $lk(b_1, \Ind(\Gamma_n))\simeq \Ind(\mathcal{U}_{n-1})$ and $del(b_1, \Ind(\Gamma_n))\simeq \Ind(\mathcal{V}_{n-1})$.
			Substituting $lk(b_1, \Ind(\Gamma_n))$ and $del(b_1, \Ind(\Gamma_n))$ into the Mayer-Vietoris sequence given in \eqref{equation:MayerVietoris}, we have
			\begin{align*}
				\cdots \to& \tilde{H}_{n-1}(\Ind(\Gamma_n))\to \tilde{H}_{n-2}(\Ind(\mathcal{U}_{n-1}))\to \tilde{H}_{n-2}(\Ind(\mathcal{V}_{n-1}))\to \tilde{H}_{n-2}(\Ind(\Gamma_n))\to\\
				& \tilde{H}_{n-3}(\Ind(\mathcal{U}_{n-1}))\to \tilde{H}_{n-3}(\Ind(\mathcal{V}_{n-1}))\to \tilde{H}_{n-3}(\Ind(\Gamma_n))\to \tilde{H}_{n-4}(\Ind(\mathcal{U}_{n-1}))\to\cdots
			\end{align*}
			By the induction hypothesis, $\tilde{H}_i(\Ind(\mathcal{U}_{n-1})) = 0$ for all $i\le n-3$ and $\tilde{H}_i(\Ind(\mathcal{V}_{n-1})) = 0$ for all $i\le n-2$. This gives
			\begin{align*}
				\cdots \to & \tilde{H}_{n-1}(\Ind(\Gamma_n))\to \tilde{H}_{n-2}(\Ind(\mathcal{U}_{n-1}))\to 0\to \\
				&\tilde{H}_{n-2}(\Ind(\Gamma_n))\to 0 \to 0 \to \tilde{H}_{n-3}(\Ind(\Gamma_n))\to 0 \to\cdots
			\end{align*}
			From the exactness of the above sequence, $\tilde{H}_i(\Ind(\Gamma_n)) = 0$ for all $i\le n-2$.
			
			\item $\Ind(\mathcal{U}_n).$ For $n = 3$, the result follows from \Cref{claim:Un_relation}. Now assume that $n \geq 4$. By \Cref{claim:Un_relation},
			$\Ind(\mathcal{U}_n)\simeq \Sigma(\Ind(\mathcal{U}_{n-1}))\vee\Sigma(\Ind(\mathcal{Z}_{n-3}))\vee\Sigma^3 (\Ind(\mathcal{U}_{n-3})).$ It follows that $\tilde{H}_i(\Ind(\mathcal{U}_n)) \cong \tilde{H}_{i-1}(\Ind(\mathcal{U}_{n-1}))\oplus\tilde{H}_{i-1}(\Ind(\mathcal{Z}_{n-3}))\oplus\tilde{H}_{i-3}(\Ind(\mathcal{U}_{n-3}))$. 
			By the induction hypothesis, we have $\tilde{H}_{i-1}(\Ind(\mathcal{U}_{n-1})) = 0$, $\tilde{H}_{i-1}(\Ind(\mathcal{Z}_{n-3})) = 0$ and $\tilde{H}_{i-3}(\Ind(\mathcal{U}_{n-3})) = 0$ for all $i\leq n-2$. Therefore $\tilde{H}_i(\Ind(\mathcal{U}_n)) = 0$ for all $i\leq n-2$.
			
			\item $\Ind(\mathcal{V}_n)$. Since $\Ind(\mathcal{V}_n)\simeq \Sigma(\Ind(\mathcal{X}_{n-2}))\vee \Sigma^2 (\Ind(\Gamma_{n-1}))$ by \Cref{claim:Vn_relation}, we have $\tilde{H}_i(\Ind(\mathcal{V}_n))\cong \tilde{H}_{i-1}(\Ind(\mathcal{X}_{n-2}))\oplus \tilde{H}_{i-2}(\Ind(\Gamma_{n-1})).$ 
			For $i\leq n-1$, $\tilde{H}_{i-1}(\Ind(\mathcal{X}_{n-2})) = 0$ and $\tilde{H}_{i-2}(\Ind(\Gamma_{n-1})) = 0$ by the induction hypothesis. Hence $\tilde{H}_i(\Ind(\mathcal{V}_n)) = 0$ for all $i\leq n-1$.
			
			\item $\Ind(\mathcal{W}_n)$. By \Cref{claim:Wn_relation}, $\Ind(\mathcal{W}_n)\simeq \Ind(\mathcal{W}_n')$, where $\mathcal{W}_n' = \mathcal{W}_n-\{d_1\}$. We consider the link and deletion of the vertex $w_1$ in $\Ind(\mathcal{W}_n')$. For $n = 3$, \Cref{claim:Wn_relation} gives $lk(w_1,\Ind(\mathcal{W}_3'))\simeq \vee_3\mathbb{S}^2$ and $del(w_1,\Ind(\mathcal{W}_3'))\simeq \Ind(\mathcal{Y}_{2})$. Using the Mayer-Vietoris sequence given in \eqref{equation:MayerVietoris}, we obtain
			\begin{align*}
				\cdots\to &\tilde{H}_{3}(\Ind(\mathcal{W}_3'))\to \tilde{H}_{2}(\vee_3\mathbb{S}^2)\to \tilde{H}_{2}(\Ind(\mathcal{Y}_{2})) \to \tilde{H}_{2}(\Ind(\mathcal{W}_3'))
				\to \\&\tilde{H}_{1}(\vee_3\mathbb{S}^2)\to \tilde{H}_{1}(\Ind(\mathcal{Y}_{2})) \to \tilde{H}_{1}(\Ind(\mathcal{W}_3'))\to \tilde{H}_{0}(\vee_3\mathbb{S}^2)\to\cdots 
			\end{align*}
			Since $\tilde{H}_2(\vee_3\mathbb{S}^2) \cong\mathbb{Z}^3$ and $\tilde{H}_i(\vee_3\mathbb{S}^2) = 0$ for all $i\ne2$, and $\Ind(\mathcal{Y}_{2})\simeq\text{pt}$ by \eqref{equation:Y_12}, it follows that 
			\begin{equation*}
				\cdots\to \tilde{H}_{3}(\Ind(\mathcal{W}_3'))\to \mathbb{Z}^3\to 0 \to \tilde{H}_{2}(\Ind(\mathcal{W}_3'))\to 0 \to 0 \to \tilde{H}_{1}(\Ind(\mathcal{W}_3'))\to 0 \to\cdots
			\end{equation*}
			The exactness of the above sequence yields $\tilde{H}_i(\Ind(\mathcal{W}_3')) = 0$ for all $i\le 2$. Since $\Ind(\mathcal{W}_3)\simeq \Ind(\mathcal{W}_3')$, we conclude that $\tilde{H}_i(\Ind(\mathcal{W}_3)) = 0$ for all $i\le 2$.
			
			Now let $n\geq 4$. By \Cref{claim:Wn_relation}, we have $lk(w_1, \Ind(\mathcal{W}_n'))\simeq \Sigma^2 (\Ind(\Gamma_{n-2}))\vee \Sigma(\Ind(\mathcal{X}_{n-3}))$ and $del(w_1, \Ind(\mathcal{W}_n'))\simeq \Ind(\mathcal{Y}_{n-1})$.
			Substituting these into the Mayer-Vietoris sequence \eqref{equation:MayerVietoris}, we obtain 
			\begin{align*}
				\cdots \to & \tilde{H}_n(\Ind(\mathcal{W}_n'))\to \tilde{H}_{n-1}(\Sigma^2 (\Ind(\Gamma_{n-2}))\vee \Sigma(\Ind(\mathcal{X}_{n-3})))\to \tilde{H}_{n-1}(\Ind(\mathcal{Y}_{n-1}))\to\\
				& \tilde{H}_{n-1}(\Ind(\mathcal{W}_n'))\to \tilde{H}_{n-2}(\Sigma^2 (\Ind(\Gamma_{n-2}))\vee \Sigma(\Ind(\mathcal{X}_{n-3})))\to \tilde{H}_{n-2}(\Ind(\mathcal{Y}_{n-1})) \to\\
				& \tilde{H}_{n-2}(\Ind(\mathcal{W}_n'))\to \tilde{H}_{n-3}(\Sigma^2 (\Ind(\Gamma_{n-2}))\vee \Sigma(\Ind(\mathcal{X}_{n-3})))\to\cdots
			\end{align*}
			This implies
			\begin{align*}
				\cdots \to & \tilde{H}_n(\Ind(\mathcal{W}_n'))\to \tilde{H}_{n-3}(\Ind(\Gamma_{n-2}))\oplus \tilde{H}_{n-2}( \Ind(\mathcal{X}_{n-3})) \to \tilde{H}_{n-1}(\Ind(\mathcal{Y}_{n-1}))\to\\
				& \tilde{H}_{n-1}(\Ind(\mathcal{W}_n'))\to \tilde{H}_{n-4}(\Ind(\Gamma_{n-2}))\oplus \tilde{H}_{n-3}( \Ind(\mathcal{X}_{n-3})) \to \tilde{H}_{n-2}(\Ind(\mathcal{Y}_{n-1}))\to\\
				& \tilde{H}_{n-2}(\Ind(\mathcal{W}_n'))\to \tilde{H}_{n-5}(\Ind(\Gamma_{n-2}))\oplus \tilde{H}_{n-4}(\Ind(\mathcal{X}_{n-3}))\to\cdots
			\end{align*}
			By the induction hypothesis, $\tilde{H}_i(\Ind(\Gamma_{n-2})) = 0$ for all $i\le n-4$, $\tilde{H}_i(\Ind(\mathcal{X}_{n-3})) = 0$ for all $i\le n-3$, and $\tilde{H}_i(\Ind(\mathcal{Y}_{n-1})) = 0$ for all $i\le n-1$. Hence
			\begin{align*}
				\cdots \to & \tilde{H}_n(\Ind(\mathcal{W}_n'))\to \tilde{H}_{n-3}(\Ind(\Gamma_{n-2}))\oplus \tilde{H}_{n-2}(\Ind(\mathcal{X}_{n-3}))\to 0 \to\\
				& \tilde{H}_{n-1}(\Ind(\mathcal{W}_n'))\to 0 \to 0 \to \tilde{H}_{n-2}(\Ind(\mathcal{W}_n'))\to 0 \to\cdots
			\end{align*}
			Using the exactness of the above sequence, $\tilde{H}_i(\Ind(\mathcal{W}_n')) = 0$ for all $i\le n-1$. Since $\Ind(\mathcal{W}_n)\simeq \Ind(\mathcal{W}_n')$, it follows that $\tilde{H}_i(\Ind(\mathcal{W}_n)) = 0$ for all $i\le n-1$.
			
			\item $\Ind(\mathcal{X}_n)$. By \Cref{claim:Xn_relation}, $\Ind(\mathcal{X}_n)\simeq \vee_2 \Sigma^2(\Ind(\Gamma_n))\vee \Sigma^3 (\Ind(\mathcal{U}_{n-1})).$ This implies $\tilde{H}_i(\Ind(\mathcal{X}_n))\cong\tilde{H}_{i-2}(\Ind(\Gamma_n))\oplus \tilde{H}_{i-2}(\Ind(\Gamma_n))\oplus\tilde{H}_{i-3}(\Ind(\mathcal{U}_{n-1})).$
			For $i\le n$, we have $\tilde{H}_{i-2}(\Ind(\Gamma_n)) = 0$ by \ref{result:Gamma_n}, and $\tilde{H}_{i-3}(\Ind(\mathcal{U}_{n-1})) = 0$ by the induction hypothesis. Thus $\tilde{H}_i(\Ind(\mathcal{X}_n)) = 0$ for all $i\leq n$.
			
			\item $\Ind(\mathcal{Y}_n)$. Since $\Ind(\mathcal{Y}_n)\simeq\Sigma^2 (\Ind(\mathcal{W}_{n-1}))\vee \Sigma^3 (\Ind(\mathcal{W}_{n-2}))$ by \Cref{claim:Yn_relation}, it follows that $\tilde{H}_i(\Ind(\mathcal{Y}_n))\cong \tilde{H}_{i-2}(\Ind(\mathcal{W}_{n-1}))\oplus\tilde{H}_{i-3}(\Ind(\mathcal{W}_{n-2}))$. 
			For $i\le n$, the induction hypothesis gives $\tilde{H}_{i-3}(\Ind(\mathcal{W}_{n-2})) = 0$ and $\tilde{H}_{i-2}(\Ind(\mathcal{W}_{n-1})) = 0$.  Therefore $\tilde{H}_i(\Ind(\mathcal{Y}_n)) = 0$ for all $i\leq n$.
			
			\item $\Ind(\mathcal{Z}_n)$. By \Cref{claim:Zn_relation}, $lk(z_1, \Ind(\mathcal{Z}_n))\simeq \Ind(\mathcal{Y}_n)$ and $del(z_1, \Ind(\mathcal{Z}_n))\simeq \Sigma^2 (\Ind(\mathcal{U}_n))$. Substituting these into the Mayer-Vietoris sequence \eqref{equation:MayerVietoris}, we obtain
			\begin{align*}
				\cdots \to & \tilde{H}_{n+1}(\Ind(\mathcal{Z}_n)) \to \tilde{H}_n(\Ind(\mathcal{Y}_n))\to \tilde{H}_n(\Sigma^2 (\Ind(\mathcal{U}_n)))\to \tilde{H}_n(\Ind(\mathcal{Z}_n)) \to\\
				& \tilde{H}_{n-1}(\Ind(\mathcal{Y}_n)) \to \tilde{H}_{n-1}(\Sigma^2 (\Ind(\mathcal{U}_n)))\to \tilde{H}_{n-1}(\Ind(\mathcal{Z}_n))\to \tilde{H}_{n-2}(\Ind(\mathcal{Y}_n))\to\cdots
			\end{align*}
			This implies
			\begin{align*}
				\cdots \to & \tilde{H}_{n+1}(\Ind(\mathcal{Z}_n)) \to \tilde{H}_n(\Ind(\mathcal{Y}_n))\to \tilde{H}_{n-2}(\Ind(\mathcal{U}_n))\to \tilde{H}_n(\Ind(\mathcal{Z}_n))\to\\
				& \tilde{H}_{n-1}(\Ind(\mathcal{Y}_n)) \to \tilde{H}_{n-3}(\Ind(\mathcal{U}_n))\to \tilde{H}_{n-1}(\Ind(\mathcal{Z}_n))\to \tilde{H}_{n-2}(\Ind(\mathcal{Y}_n))\to\cdots
			\end{align*} 
			Since $\tilde{H}_i(\Ind(\mathcal{Y}_n)) = 0$ for all $i\le n$ by \ref{result:Y_n} and $\tilde{H}_i(\Ind(\mathcal{U}_n)) = 0$ for all $i\le n-2$ by \ref{result:U_n}, we have
			\begin{align*}
				\cdots & \to \tilde{H}_{n+1}(\Ind(\mathcal{Z}_n)) \to 0 \to 0 \to \tilde{H}_n(\Ind(\mathcal{Z}_n))\to 0 \to 0\to \tilde{H}_{n-1}(\Ind(\mathcal{Z}_n))\to 0 \to\cdots
			\end{align*}
			The exactness of the above sequence implies $\tilde{H}_i(\Ind(\mathcal{Z}_n)) = 0$ for all $i\le n$.
		\end{enumerate}
		This completes the proof.
	\end{proof}
	
	\begin{remark}
		In the proof of \Cref{proposition:homology_up_to_n-2}, the idea of considering auxiliary graphs (as defined in \Cref{subsection:graph_definitions}) and establishing recursive relations among the homology groups of independence complexes of these graphs and of $\Gamma_n$, in terms of link and deletion, is taken from \cite{Previous_3Xn}. 
	\end{remark}
	
	\begin{proposition}\label{proposition:homology_n-1} 
		For $n\geq1,$ $\tilde{H}_{n-1}(\Ind(\Gamma_n)) \neq 0 \text{ and }\tilde{H}_{n-1}(\Ind(\mathcal{U}_n)) \neq 0$. 
	\end{proposition}
	
	\begin{proof}
		We proceed by induction on $n$. The results for the base cases $n = 1,2,3$ are immediate from \eqref{equation:Gamma_123}, \eqref{equation:U_12} and \Cref{claim:Un_relation}. 
		
		Let $n>3$ and assume that the results hold for all $1\leq l\leq n-1$. We now prove the results for $n.$
		
		For $\Ind(\Gamma_n)$, we consider the link and deletion of the vertex $b_1$ in $\Ind(\Gamma_n)$. Since $lk(b_1, \Ind(\Gamma_n))\simeq \Ind(\mathcal{U}_{n-1})$ and $del(b_1, \Ind(\Gamma_n))\simeq \Ind(\mathcal{V}_{n-1})$, using the Mayer-Vietoris sequence \eqref{equation:MayerVietoris}, we obtain
		\begin{align*}
			\cdots & \to \tilde{H}_{n-1}(\Ind(\mathcal{V}_{n-1}))\to \tilde{H}_{n-1}(\Ind(\Gamma_n))\to \tilde{H}_{n-2}(\Ind(\mathcal{U}_{n-1}))\to \tilde{H}_{n-2}(\Ind(\mathcal{V}_{n-1}))\to \ldots
		\end{align*}
		We have $\tilde{H}_{n-2}(\Ind(\mathcal{U}_{n-1}))\ne0$ by the induction hypothesis, and $\tilde{H}_{n-2}(\Ind(\mathcal{V}_{n-1})) = 0$ by \Cref{proposition:homology_up_to_n-2}. Therefore, the exactness of the above sequence implies that $ \tilde{H}_{n-1}(\Ind(\Gamma_n)) \neq 0$.
		
		For $\Ind(\mathcal{U}_n)$, we have
		$\Ind(\mathcal{U}_n)\simeq \Sigma(\Ind(\mathcal{U}_{n-1}))\vee\Sigma(\Ind(\mathcal{Z}_{n-3}))\vee\Sigma^3 (\Ind(\mathcal{U}_{n-3}))$ by \Cref{claim:Un_relation}. 
		It follows that $\tilde{H}_{n-1}(\Ind(\mathcal{U}_n))\cong \tilde{H}_{n-2}(\Ind(\mathcal{U}_{n-1}))\oplus\tilde{H}_{n-2}(\Ind(\mathcal{Z}_{n-3}))\oplus\tilde{H}_{n-4}(\Ind(\mathcal{U}_{n-3}))$. By the induction hypothesis, $\tilde{H}_{n-2}(\Ind(\mathcal{U}_{n-1}))\ne0$. Hence $\tilde{H}_{n-1}(\Ind(\mathcal{U}_n))\neq 0$.
	\end{proof}
	
	\subsection{}\label{section:top_dimension}
	This section is divided into two subsections.
	In \Cref{subsection:top_dimension}, we prove the vanishing of the homology of $\Ind(\Gamma_n)$ in the top dimension, \textit{i.e.}, $\tilde{H}_{\dim(\Ind(\Gamma_n))}(\Ind(\Gamma_n)) = 0$ for $n\ge2$. In \Cref{subsection:simply_connected}, we establish that $\Ind(\Gamma_n)$ is simply connected.
	
	\subsubsection{Top-dimensional homology of $\Ind(\Gamma_n)$} \label{subsection:top_dimension}
	
	To prove that $\tilde{H}_{\dim(\Ind(\Gamma_n))}(\Ind(\Gamma_n)) = 0$ for $n\ge2$, we use discrete Morse theory. We begin by recalling the necessary terminology and results from discrete Morse theory. For more details, we refer to \cite{JonssonBook,Kozlov2008}. 
	
	Let $A$ be a set and let $\mathcal{F}$ be a finite family of finite subsets of $A$.
	\begin{definition}
		A \textit{matching} on $\mathcal{F}$ is a family $\mathcal{M}$ of pairs $\{\sigma,\tau\}$ with $\sigma,\tau\in\mathcal{F}$ such that each element of $\mathcal{F}$ belongs to at most one pair of $\mathcal{M}$.
	\end{definition}
	
	For any $\sigma\in\mathcal{F}$, if there exists a pair $\mu\in\mathcal{M}$ such that $\sigma\in\mu$, then $\sigma$ is said to be \textit{matched} in $\mathcal{M}$, and we write $\sigma\in\mathcal{M}$. Otherwise, $\sigma$ is said to be \textit{critical} (or \textit{unmatched}) in $\mathcal{M}$, and we write $\sigma\notin\mathcal{M}$.
	\begin{definition} 
		A matching $\mathcal{M}$ on $\mathcal{F}$ is called a \textit{partial matching} if, for every pair $\{\sigma,\tau\}\in\mathcal{M}$, either $\sigma\prec\tau$ (\textit{i.e.}, $\sigma\subset\tau$ and no $\rho\in\mathcal{F}$ satisfies $ \sigma\subset\rho\subset\tau$), or $\sigma\succ \tau$.
	\end{definition}
	Equivalently, $\mathcal{M}$ is a partial matching on $\mathcal{F}$ if and only if there exists $\mathcal{F}' \subset\mathcal{F}$ and an injective map $\phi: \mathcal{F}'\rightarrow \mathcal{F}\setminus \mathcal{F}'$ such that $\phi(\sigma)\succ \sigma$ for all $\sigma\in\mathcal{F}'$. 
	\begin{definition}
		A partial matching $\mathcal{M}$ on $\mathcal{F}$ is called {\it acyclic} if there does not exist a cycle
		$$\phi(\sigma_1) \succ \sigma_1 \prec \phi(\sigma_2) \succ \sigma_2 \prec \phi(\sigma_3) \succ \sigma_3 \prec \dots \prec \phi(\sigma_r) \succ \sigma_r \prec \phi(\sigma_1),$$
		with $\ r\ge 2$, and all $\sigma_i\in\mathcal{F}$ being distinct.
	\end{definition}
	
	We will primarily use a particular type of partial matching called an element matching.
	\begin{definition}
		Let $a\in A$. A matching $\mathcal{M}_a$ on $\mathcal{F}$ is called an \textit{element matching} using $a$ if every pair in $\mathcal{M}_a$ is of the form $\{\sigma\setminus\{a\},\sigma\cup\{a\}\}$ for some $\sigma\in\mathcal{F}$.
	\end{definition}
	
	\begin{lemma}[{\cite[Lemma 4.1]{JonssonBook}}]\label{lemma:Jonsson acyclic}
		Let $a\in A$. Define $$\mathcal{M}_a: = \{\{\sigma\setminus\{a\},\sigma\cup\{a\}\}\, |\, \sigma\setminus\{a\},\, \sigma\cup\{a\}\in\mathcal{F}\}.$$ 
		Let $\mathcal{M}'$ be an acyclic matching on $\mathcal{F}': = \{\sigma\in\mathcal{F}\, |\, \sigma\notin\mathcal{M}_a\}$. Then $\mathcal{M}: = \mathcal{M}_a \cup\mathcal{M}'$ is an acyclic matching on $\mathcal{F}$.
	\end{lemma}
	
	Using \Cref{lemma:Jonsson acyclic}, we construct an acyclic matching $\mathcal{M}_S$ on $\mathcal{F}$ associated with an ordered subset $S = \{x_1,x_2,\ldots,x_{t}\}$ of $A$. For this, we first define a sequence $\mathcal{M}_{x_1},\mathcal{M}_{x_2},\ldots,$ $\mathcal{M}_{x_{t}}$ of element matchings as follows. 
	
	Let $\mathcal{C}_0: = \mathcal{F}$. For each $1\le i\le t$, define 
	\begin{align*}
		\mathcal{M}_{x_i} &: = \{\{\di{\sigma}{x_i},\un{\sigma}{x_i}\}\, |\, \di{\sigma}{x_i},\un{\sigma}{x_i}\in\mathcal{C}_{i-1}\};\\
		\mathcal{C}_{i} &: = \{\sigma\in\mathcal{C}_{i-1}\, |\, \sigma\notin\mathcal{M}_{x_i}\}. 
	\end{align*} 
	Observe that the matchings $\mathcal{M}_{x_1},\mathcal{M}_{x_2},\ldots,$ $\mathcal{M}_{x_{t}}$ are pairwise disjoint. Define 
	\begin{equation}\label{equation:sequence_of_element_matchings}
		\mathcal{M}_S: = \bigsqcup_{i = 1}^{t}\mathcal{M}_{x_i}.
	\end{equation} 
	By successive applications of \Cref{lemma:Jonsson acyclic}, $\mathcal{M}_S$ is an acyclic matching on $\mathcal{F}$.
	
	We now turn our attention to matchings on simplicial complexes. A matching on a simplicial complex is a matching on the set of its simplices (or faces). The following is one of the main results from discrete Morse theory that we will use.
	\begin{theorem}[{\cite[Theorem 4.14]{JonssonBook}}]\label{theorem:acyclic}
		Let $\Delta$ be a simplicial complex and let $\mathcal{M}$ be an acyclic matching on $\Delta$ such that the empty set is not critical. Then $\Delta$ is homotopy equivalent to a cell complex with one cell of dimension $d \ge 0$ for each critical face of $\Delta$ of dimension $d$ plus one additional $0$-cell.
	\end{theorem}
	
	Before proceeding further, we first fix some notation. Let $G$ be a graph, and let $\aG{G}$ denote its independence number, \textit{i.e.}, the size of a maximum independent set in $G$. Define 
	$$\IG{G}: = \{\sigma\in \Ind(G)\ |\ \dim(\sigma) = \dim(\Ind(G))\}.$$
	Note that $\IG{G}$ consists of the top-dimensional faces of $\Ind(G)$, which are precisely the maximum independent sets of $G$. 
	\begin{remark}\label{remark:al_G_and_sigma_union_v_notin_Ind(G)}
		Let $\sigma\in\Ind(G)$. Then $\sigma\in\IG{G}$ if and only if $|\sigma| = \aG{G}$. Consequently, if $\sigma\in\IG{G}$ and $v\in V(G)\setminus\sigma$, then $\un{\sigma}{v}\notin\Ind(G)$.
	\end{remark}
	
	For an independent ordered subset $S$ of $V(G)$, the following lemma characterizes the elements of $\IG{G}$ that are matched in the matching $\mathcal{M}_S$ defined in \eqref{equation:sequence_of_element_matchings}.
	\begin{lemma}\label{lemma:matching_top_dimension_in_G}
		Let $\mathcal{S} = \{v_1, v_2,v_3,\cdots,v_t\}\subseteq V(G)$ be an ordered independent set in $G$. For $\sigma\in \IG{G}$, $\sigma\in \mathcal{M}_{\mathcal{S}}$ if and only if $\sigma\cap\mathcal{S}\ne\emptyset$. 
	\end{lemma}
	
	\begin{proof} 
		Assume first that $\sigma \in \mathcal{M}_{\mathcal{S}}$. Since $\mathcal{M}_{\mathcal{S}} = \bigsqcup_{i = 1}^{t}\mathcal{M}_{v_i}$, there exists $1\le j\le t$ such that $\sigma\in \mathcal{M}_{v_j}$. Suppose $v_j\notin \sigma$. Then $\sigma = \di{\sigma}{v_j}$ is matched with $\un{\sigma}{v_j}$ in $\mathcal{M}_{v_j}$. It follows that $\sigma\cup \{v_j\}\in \Ind(G)$, contradicting \Cref{lemma:matching_top_dimension_in_G} (as $\sigma\in \IG{G}$). Therefore $v_j\in \sigma$, and hence $\sigma\cap\mathcal{S}\ne\emptyset$.
		
		We now assume that $\sigma\cap\mathcal{S}\ne\emptyset$. Let $j: = \min\{i: v_i\in \sigma\}$. We prove that $\sigma\in\mathcal{M}_{v_j}$ and hence $\sigma\in\mathcal{M}_{\mathcal{S}}$. On the contrary, suppose $\sigma\notin \mathcal{M}_{v_j}$. Since $v_j\in\sigma$, it follows that $\un{\sigma}{v_j} = \sigma\in\mathcal{M}_{v_{j'}}$ or $\sigma\setminus\{v_j\}\in \mathcal{M}_{v_{j''}}$ for some $j',j''<j$. By the minimality of $j$, we have $v_{j'}\notin\sigma$ and $v_{j''}\notin\sigma\setminus\{v_j\}$.
		
		If $\sigma\in\mathcal{M}_{v_{j'}}$, then since $\sigma = \di{\sigma}{v_{j'}}$ can only be matched with $\un{\sigma}{v_{j'}}$ in $\mathcal{M}_{v_{j'}}$, we have $\un{\sigma}{v_{j'}}\in \Ind(G)$. This contradicts \Cref{remark:al_G_and_sigma_union_v_notin_Ind(G)}. Hence $\sigma\setminus\{v_j\}\in \mathcal{M}_{v_{j''}}$. Now, since $\sigma\setminus\{v_j\} = \di{\sigma}{v_{j},v_{j''}}$ can only be matched with $\un{(\sigma\setminus\{v_j\})}{v_{j''}}$ in $\mathcal{M}_{v_{j''}}$, we get $\un{(\sigma\setminus\{v_j\})}{v_{j''}}\in \Ind(G)$. 
		This implies $v_{j''}\nsim v$ for all $v\in\sigma\setminus\{v_j\}$. Since $\mathcal{S}$ is an independent set, $v_{j''}\nsim v_j$. It follows that $\sigma\cup\{v_{j''}\}\in \Ind(G)$, again a contradiction to \Cref{remark:al_G_and_sigma_union_v_notin_Ind(G)}. Thus $\sigma\in\mathcal{M}_{v_j}$, which implies $\sigma\in\mathcal{M}_{\mathcal{S}}$. 
	\end{proof}
	
	Recall that the aim of this section is to show that $\tilde{H}_{\dim(\Ind(\Gamma_n))}(\Ind(\Gamma_n)) = 0$. By \Cref{theorem:acyclic}, it suffices to construct an acyclic matching on $\Ind(\Gamma_n)$ for which no top-dimensional face is critical. If there exists an ordered independent set $\mathcal{B}_n \subseteq V(\Gamma_n)$ such that $\sigma\cap \mathcal{B}_n\neq \emptyset$ for every $\sigma\in\IG{\Gamma_n}$, then the matching $\mathcal{M}_{\mathcal{B}_n}$ has no critical face of top-dimension by \Cref{lemma:matching_top_dimension_in_G}. We define such a set $\mathcal{B}_n$ of $V(\Gamma_n)$ as 
	\begin{equation}\label{set_Bn}
		\mathcal{B}_n: = \{a_1,a_4,a_7,\cdots,a_{(\beta_n-3)},a_{\beta_n}\}\sqcup\{e_1, e_4,e_7,\cdots, e_{(\beta_n-3)},e_{\beta_n}\},
	\end{equation} where 
	$$\beta_n = \begin{dcases}
		n-3 & \text{ if } n\equiv 1(\text{mod $6$}),\\
		n-2 & \text{ if } n\equiv 3(\text{mod $6$}),\\
		n-1 & \text{ if } n\equiv 5(\text{mod $6$}),\\
		1 & \text{ otherwise}.
	\end{dcases}$$
	Observe that $\mathcal{B}_n$ is independent in $\Gamma_n$. To prove that every $\sigma\in\IG{\Gamma_n}$ satisfies $\sigma\cap\mathcal{B}_n\ne \emptyset$, we first establish two preliminary results.
	\begin{lemma}\label{lemma:dim_sigma}
		Let $\sigma\in\IG{\Gamma_n}$. Then 
		$\aG{\Gamma_n} = |\sigma| 
		= \begin{cases}
			\frac{3n}{2} & \text{if $n$ is even,}\\
			\frac{3n-1}{2} & \text{if $n$ is odd}.
		\end{cases}$
	\end{lemma}
	
	\begin{proof}
		By \Cref{remark:al_G_and_sigma_union_v_notin_Ind(G)}, $\aG{\Gamma_n} = |\sigma|$. We now compute $|\sigma|$.
		
		Since $\sigma \in \IG{\Gamma_n},$ $|\sigma| = \dim(\Ind(\Gamma_{ n}))+1$. Using $\Ind(\Gamma_n) = M(G_{3\times n}),$ we obtain $|\sigma| = \dim(M(G_{3\times n}))+1$. By the definition of the matching complex, $|\sigma|$ equals the size of a maximum matching in $G_{3\times n}$.
		
		We first assume that $n$ is even. Then $M = \{(11,21),(12,22),(13,23),(31,41),(32,42),(33,43),\ldots,((n-1)1,n1),((n-1)2,n2),((n-1)3,n3)\}$ is a matching of size $3n/2$ that covers the entire vertex set $V(G_{3\times n})$. Therefore $M$ is a maximum matching, and hence
		$|\sigma| = \frac{3n}{2}$.
		
		Now, suppose that $n$ is odd. Since $|V(G_{3\times n})| = 3n$ is odd, every matching in $G_{3\times n}$ leaves at least one vertex uncovered. Observe that $M' = \{(11,21),(12,22),(13,23),(31,41),(32,42),(33,43),\ldots,((n-2)1,(n-1)1),((n-2)2,(n-1)2),((n-2)3,(n-1)3),(n1,n2)\}$ is a matching of size $\frac{3n-1}{2}$ that covers every vertex of $G_{3\times n}$ except $n3$. Thus $M'$ is a maximum matching, and therefore $|\sigma| = \frac{3n-1}{2}$.
	\end{proof}
	
	\begin{lemma}\label{lemma:3m+1}
		For $m\ge 1,$ let $\sigma\in \IG{\Gamma_{2m+1}}$. Then $\sigma\cap\{ b_1, c_1, d_1\}\ne\emptyset$.
	\end{lemma}
	
	\begin{proof}
		The proof proceeds by induction on $m$. For the base case $m = 1$, we have
		\begin{align*}
			\IG{\Gamma_3}& = \{\{a_1, c_1, e_1, b_3\}, \{a_1, c_1, e_1, d_3\}, \{a_1, c_1, e_2, b_3\}, \{a_1, d_1, c_2, e_2\}, \{a_1, d_1, d_2, b_3\}, \{a_1, d_1, d_2, d_3\},\\
			&\qquad \{a_1, d_1, e_2, b_3\}, \{b_1, e_1, a_2, c_2\}, \{b_1, e_1, a_2, d_3\}, \{b_1, e_1, b_2, b_3\}, \{b_1, e_1, b_2, d_3\}, \{b_1, a_2, c_2, e_2\},\\
			&\qquad \{b_1, a_2, d_2, d_3\}, \{b_1, b_2, e_2, b_3\}, \{c_1, e_1, a_2, d_3\}, \{d_1, a_2, c_2, e_2\}, \{d_1, a_2, d_2, d_3\}, \{d_1, b_2, e_2, b_3\}\}.
		\end{align*}
		Clearly, every $\sigma\in \IG{\Gamma_3}$ satisfies $\sigma\cap\{ b_1, c_1, d_1\}\ne\emptyset$. Hence the result is true for $m = 1$.
		
		Let $m>1$ and assume that the result holds for all $1\le l\le m-1$. We now prove it for $m$. 
		
		Let $\sigma\in \IG{\Gamma_{2m+1}}$. Since $\Gamma_3$ is a subgraph of $\Gamma_{2m+1}$, we define $\Omega: = \Gamma_3-\{b_3,d_3\}$ and $\Omega': = \Gamma_{2m+1}- V(\Omega)$. Then $V(\Gamma_{2m+1}) = V(\Omega)\sqcup V(\Omega')$. It follows that $\sigma = \sigma_1\sqcup \sigma_2$ for some $\sigma_1\in \Ind(\Omega)$ and $\sigma_2\in \Ind(\Omega')$. Observe that $$\IG{\Omega} = \{\{a_1, d_1, c_2, e_2\}, \{b_1, e_1, a_2, c_2\}, \{b_1, a_2, c_2, e_2\}, \{d_1, a_2, c_2, e_2\}\}.$$ Since every element of $\IG{\Omega}$ has cardinality $4$, \Cref{remark:al_G_and_sigma_union_v_notin_Ind(G)} implies $\aG{\Omega} = 4$. Hence $|\sigma_1|\le4$. 
		\begin{figure}[H]
			\centering
			\begin{tikzpicture}[scale = 0.3, vertices/.style = {draw, fill = black, circle, inner sep = 0.5pt}]
				
				\node[vertices,label = above:{{\tiny $b_1$}}] (b1) at (12, 3) {};
				\node[vertices,label = below:{{\tiny $d_1$}}] (d1) at (12, 0) {};
				
				\node[vertices,label = above:{{\tiny $a_1$}}] (a1) at (13.5, 4.5) {};
				\node[vertices,label = below:{{\tiny $c_1$}}] (c1) at (13.5, 1.5) {};
				\node[vertices,label = below:{{\tiny $e_1$}}] (e1) at (13.5, -1.5) {};
				
				\node[vertices,label = above:{{\tiny $b_2$}}] (b2) at (15, 3) {};
				\node[vertices,label = below:{{\tiny $d_2$}}] (d2) at (15, 0) {};
				
				\node[vertices,label = above:{{\tiny $a_2$}}] (a2) at (16.5, 4.5) {};
				\node[vertices,label = below:{{\tiny $c_2$}}] (c2) at (16.5, 1.5) {};
				\node[vertices,label = below:{{\tiny $e_2$}}] (e2) at (16.5, -1.5) {};
				
				\node[vertices,label = above:{{\tiny $b_3$}}] (b3) at (18, 3) {};
				\node[vertices,label = below:{{\tiny $d_3$}}] (d3) at (18, 0) {};
				
				\node[vertices,label = above:{{\tiny $a_3$}}] (a) at (19.5, 4.5) {};
				\node[vertices,label = below:{{\tiny $c_3$}}] (c) at (19.5, 1.5) {};
				\node[vertices,label = below:{{\tiny $e_3$}}] (e) at (19.5, -1.5) {};
				
				\node[vertices,label = above:{{\tiny $b_4$}}] (b) at (21, 3) {};
				\node[vertices,label = below:{{\tiny $d_4$}}] (d) at (21, 0) {};
				
				\node[vertices, label = {[xshift = .08cm,yshift = -.01cm]above:{{{\tiny $b_{2m}$}}}}] (1) at ($(b)+(4,0)$) {};
				\node[vertices, label = {[xshift = .05cm]below:{{{\tiny $d_{2m}$}}}}] (2) at ($(d)+(4,0)$) {};
				
				\node[vertices,label = {[xshift = .1cm]above:{{{\tiny $a_{2m}$}}}}] (3) at ($(a)+(7,0)$) {}; 
				\node[vertices,label = {[xshift = .1cm]below:{{{\tiny $c_{2m}$}}}}] (4) at ($(c)+(7,0)$) {};
				\node[vertices,label = {[xshift = .1cm]below:{{{\tiny $e_{2m}$}}}}] (5) at ($(e)+(7,0)$) {};
				
				\node[vertices,label = {[xshift = .18cm]above:{{{\tiny $b_{2m+1}$}}}}] (6) at ($(b)+(7,0)$) {};
				\node[vertices,label = {[xshift = .19cm]below:{{{\tiny $d_{2m+1}$}}}}] (7) at ($(d)+(7,0)$) {};
				
				\draw (a) -- ++(2.3,0);
				\draw (c) -- ++(2.3,0);
				\draw (e) -- ++(2.3,0);
				\draw (3) -- ++(-2.3,0);
				\draw (4) -- ++(-2.3,0);
				\draw (5) -- ++(-2.3,0);
				\draw (b) -- ++(1,1);
				\draw (b) -- ++(1,-1);
				\draw (d) -- ++(1,1);
				\draw (d) -- ++(1,-1);
				\draw (1) -- ++(-1,1);
				\draw (1) -- ++(-1,-1);
				\draw (2) -- ++(-1,1);
				\draw (2) -- ++(-1,-1);
				
				\node at ($(b)!0.5!(1)$) {$\ldots$};
				\node at ($(d)!0.5!(2)$) {$\ldots$};
				\node at ($(a)!0.5!(3)$) {$\ldots$};
				\node at ($(c)!0.5!(4)$) {$\ldots$};
				\node at ($(e)!0.5!(5)$) {$\ldots$};
				
				\foreach \vertex/\neighbors in {
					b1/{a1,c1,d1},
					d1/{c1,e1},
					a1/{a2,b2},
					c1/{b2,c2,d2},
					e1/{d2,e2},
					b2/{a2,c2,d2},
					d2/{c2,e2}}
				{
					\foreach \neighbor in \neighbors
					\draw [blue] (\vertex)--(\neighbor);
				}
				
				\foreach \vertex/\neighbors in {
					a2/{a,b3},
					c2/{b3,c,d3},
					e2/{d3,e}}
				{
					\foreach \neighbor in \neighbors
					\draw [red] (\vertex)--(\neighbor);
				}
				
				\foreach \vertex/\neighbors in {
					b3/{d3},
					a/{b3,b},
					b/{c,d},
					c/{b3,d3,d},
					d/{e},
					e/{d3},
					1/{2,3,4},
					2/{4,5},
					3/{6},
					4/{6,7},
					5/{7},
					6/{7}}
				{
					\foreach \neighbor in \neighbors
					\draw [-] (\vertex)--(\neighbor);
				} 
			\end{tikzpicture}
			\vspace{-.2cm}
			\caption{$\Gamma_{2m+1}$ with subgraphs $\Omega$ (blue edges) and $\Omega'$ (black edges)}
			\label{fig:Gamma_2m+1}
		\end{figure}
		Further, by \Cref{lemma:dim_sigma}, $\aG{\Gamma_{2m-1}} = 3m-2$. Since $\Omega'\cong \Gamma_{2m-1}$, we have $|\sigma_2|\le \aG{\Omega'} = \aG{\Gamma_{2m-1}} = 3m-2$. Moreover, $\sigma\in\IG{\Gamma_{2m+1}}$ implies $|\sigma| = 3m+1$ by \Cref{lemma:dim_sigma}. Therefore, there are two possible cases:
		\begin{enumerate}[label = (\roman*)]
			\item $|\sigma_1| = 3$ and $|\sigma_2| = 3m-2$.\\ 
			Then $|\sigma_2| = \aG{\Omega'}$, and hence \Cref{remark:al_G_and_sigma_union_v_notin_Ind(G)} implies $\sigma_2\in \IG{\Omega'}$. Using the fact that $\Omega'\cong \Gamma_{2m-1}$, we may apply the induction hypothesis to conclude $\sigma_2 \cap\{b_3,c_3,d_3\}\ne \emptyset$. Since $c_2\sim v$ for all $v\in\{b_3,c_3,d_3\}$ in $\Gamma_{2m+1}$, it follows that $c_2\notin \sigma_1$. Let $\Omega'': = \Omega-\{b_1,c_1,d_1,c_2\}$. If $\sigma_1\cap\{b_1,c_1,d_1\} = \emptyset,$ then $\sigma_1\in \Ind(\Omega'')$ (as $\sigma_1 \in \Ind(\Omega)$). Observe that $\aG{\Omega''} = 2$, which implies $|\sigma_1|\le\aG{\Omega''} = 2$, a contradiction. Hence $\sigma_1\cap\{b_1, c_1, d_1\}\ne \emptyset$, and thus $\sigma\cap\{b_1,c_1,d_1\}\ne\emptyset$. 
			
			\item $|\sigma_1| = 4$ and $|\sigma_2| = 3m-3$.\\ 
			In this case, $\sigma_1 \in \IG{\Omega}$ (as $|\sigma_1| = \aG{\Omega}$). Since every $\sigma''\in \IG{ \Omega}$ satisfies $\sigma''\cap\{b_1,c_1,d_1\}\ne \emptyset$, we get $\sigma_1\cap\{b_1,c_1,d_1\}\ne \emptyset$. Hence $\sigma\cap\{b_1, c_1, d_1\}\ne \emptyset$. 
		\end{enumerate}
		In either case, $\sigma\cap\{b_1, c_1, d_1\}\ne \emptyset$. This completes the proof.
	\end{proof}
	
	\begin{lemma}\label{prop:even+odd}
		For $n\ge 4$, let $\sigma\in \IG{\Gamma_{n}}$. Then $\sigma\cap \mathcal{B}_{n}\ne \emptyset$
	\end{lemma}
	
	\begin{proof}
		We consider the cases of $n$ even and $n$ odd separately. 
		
		\noindent \textbf{Case 1:} $n = 2m$ for some integer $m\geq 2.$ 
		
		We have $\sigma\in \IG{\Gamma_{2m}}$. Let $\Omega: = \Gamma_2-\{b_2,d_2\}$ and $\Omega': = \Gamma_{2m}- V(\Omega)$ (as $\Gamma_2$ is a subgraph of $\Gamma_{2m}$). Then $V(\Gamma_{2m}) = V(\Omega)\sqcup V(\Omega')$. This gives $\sigma = \sigma_1\sqcup \sigma_2$ for some $\sigma_1\in \Ind(\Omega)$ and $\sigma_2\in \Ind(\Omega')$. It is easy to verify that $\IG{\Omega} = \{\{a_1,c_1,e_1\}\}$, and hence $\aG{\Omega} = 3$ by \Cref{remark:al_G_and_sigma_union_v_notin_Ind(G)}. Thus $|\sigma_1|\le 3$.
		\begin{figure}[H]
			\centering
			\begin{tikzpicture}[scale = 0.3, vertices/.style = {draw, fill = black, circle, inner sep = 0.5pt}]
				
				\node[vertices,label = above:{{\tiny $b_1$}}] (b1) at (12, 3) {};
				\node[vertices,label = below:{{\tiny $d_1$}}] (d1) at (12, 0) {};
				
				\node[vertices,label = above:{{\tiny $a_1$}}] (a1) at (13.5, 4.5) {};
				\node[vertices,label = below:{{\tiny $c_1$}}] (c1) at (13.5, 1.5) {};
				\node[vertices,label = below:{{\tiny $e_1$}}] (e1) at (13.5, -1.5) {};
				
				\node[vertices,label = above:{{\tiny $b_2$}}] (b2) at (15, 3) {};
				\node[vertices,label = below:{{\tiny $d_2$}}] (d2) at (15, 0) {};
				
				\node[vertices,label = above:{{\tiny $a_2$}}] (a2) at (16.5, 4.5) {};
				\node[vertices,label = below:{{\tiny $c_2$}}] (c2) at (16.5, 1.5) {};
				\node[vertices,label = below:{{\tiny $e_2$}}] (e2) at (16.5, -1.5) {};
				
				\node[vertices,label = above:{{\tiny $b_3$}}] (b3) at (18, 3) {};
				\node[vertices,label = below:{{\tiny $d_3$}}] (d3) at (18, 0) {};
				
				\node[vertices,label = above:{{\tiny $a_3$}}] (a) at (19.5, 4.5) {};
				\node[vertices,label = below:{{\tiny $c_3$}}] (c) at (19.5, 1.5) {};
				\node[vertices,label = below:{{\tiny $e_3$}}] (e) at (19.5, -1.5) {};
				
				\node[vertices,label = above:{{\tiny $b_4$}}] (b) at (21, 3) {};
				\node[vertices,label = below:{{\tiny $d_4$}}] (d) at (21, 0) {};
				
				\node[vertices, label = {[xshift = .05cm,yshift = -.12cm]above:{{{\tiny $b_{(2m-1)}$}}}}] (1) at ($(b)+(4,0)$) {};
				\node[vertices, label = {[xshift = .1cm, yshift = .06cm]below:{{{\tiny $d_{(2m-1)}$}}}}] (2) at ($(d)+(4,0)$) {};
				
				\node[vertices,label = {[xshift = .3cm,yshift = -.13cm]above:{{{\tiny $a_{(2m-1)}$}}}}] (3) at ($(a)+(7,0)$) {}; 
				\node[vertices,label = {[xshift = .27cm]below:{{{\tiny $c_{(2m-1)}$}}}}] (4) at ($(c)+(7,0)$) {};
				\node[vertices,label = {[xshift = .3cm]below:{{{\tiny $e_{(2m-1)}$}}}}] (5) at ($(e)+(7,0)$) {};
				
				\node[vertices,label = {[xshift = .15cm]above:{{{\tiny $b_{2m}$}}}}] (6) at ($(b)+(7,0)$) {};
				\node[vertices,label = {[xshift = .15cm]below:{{{\tiny $d_{2m}$}}}}] (7) at ($(d)+(7,0)$) {};
				
				\draw (a) -- ++(2.3,0);
				\draw (c) -- ++(2.3,0);
				\draw (e) -- ++(2.3,0);
				\draw (3) -- ++(-2.3,0);
				\draw (4) -- ++(-2.3,0);
				\draw (5) -- ++(-2.3,0);
				\draw (b) -- ++(1,1);
				\draw (b) -- ++(1,-1);
				\draw (d) -- ++(1,1);
				\draw (d) -- ++(1,-1);
				\draw (1) -- ++(-1,1);
				\draw (1) -- ++(-1,-1);
				\draw (2) -- ++(-1,1);
				\draw (2) -- ++(-1,-1);
				
				\node at ($(b)!0.5!(1)$) {$\ldots$};
				\node at ($(d)!0.5!(2)$) {$\ldots$};
				\node at ($(a)!0.5!(3)$) {$\ldots$};
				\node at ($(c)!0.5!(4)$) {$\ldots$};
				\node at ($(e)!0.5!(5)$) {$\ldots$};
				
				\foreach \vertex/\neighbors in {
					b1/{a1,c1,d1},
					d1/{c1,e1}}
				{
					\foreach \neighbor in \neighbors
					\draw [blue] (\vertex)--(\neighbor);
				}
				
				\foreach \vertex/\neighbors in {
					a1/{a2,b2},
					c1/{b2,c2,d2},
					e1/{d2,e2}}
				{
					\foreach \neighbor in \neighbors
					\draw [red] (\vertex)--(\neighbor);
				}
				
				\foreach \vertex/\neighbors in {
					b2/{a2,c2,d2},
					d2/{c2,e2},
					b3/{a2,c2,d3},
					d3/{c2,e2},
					a/{a2,b3,b},
					b/{c,d},
					c/{c2,b3,d3,d},
					d/{e},
					e/{e2,d3},
					1/{2,3,4},
					2/{4,5},
					3/{6},
					4/{6,7},
					5/{7},
					6/{7}}
				{
					\foreach \neighbor in \neighbors
					\draw [-] (\vertex)--(\neighbor);
				} 
			\end{tikzpicture}
			\vspace{-.2cm}
			\caption{$\Gamma_{2m}$ with subgraphs $\Omega$ (blue edges) and $\Omega'$ (black edges)}
			\label{fig:Gamma_2m}
		\end{figure}
		Observe that $\Omega'\cong \Gamma_{2m-1}$. Therefore, $\aG{\Omega'} = 3m-2$ by \Cref{lemma:dim_sigma}, and hence $|\sigma_2|\le 3m-2$. Furthermore, $\Omega'\cong \Gamma_{2m-1}$ implies $\sigma'\cap\{b_2,c_2,d_2\}\ne\emptyset$ by \Cref{lemma:3m+1}. 
		
		Since $2m(\text{mod $6$})\in\{0,2,4\}$, we have $\beta_{2m} = 1$, and hence $\mathcal{B}_{2m} = \{a_1\}\sqcup\{e_1\}$ (see \eqref{set_Bn}). Let $\Omega'': = \Omega-\{a_1,e_1\}$. 
		Note that $\aG{\Omega''} = 1$. Suppose that $\sigma\cap\mathcal{B}_{2m} = \emptyset$. Then $\sigma_1\cap\mathcal{B}_{2m} = \emptyset,$ which implies that $\sigma_1\in \Ind(\Omega'')$ (as $\sigma_1\in \Ind(\Omega)$). Therefore $|\sigma_1|\le\aG{\Omega''} = 1$. 
		By \Cref{lemma:dim_sigma}, $|\sigma| = 3m$. Hence $|\sigma| = |\sigma_1|+|\sigma_2| = 3m$ implies that $|\sigma_2|\ge 3m-1$, contradicting $|\sigma_2|\le 3m-2$. Thus $\sigma\cap\mathcal{B}_{2m}\ne\emptyset$. 
		
		\vspace{0.1cm}
		\noindent\textbf{Case 2:} $n = 2m+1$ for some integer $m\geq 2.$ 
		
		In this case, $\sigma\in \IG{\Gamma_{2m+1}}$. We proceed by induction on $m$. For this, we first define a subgraph $\Omega_m$ of $\Gamma_{2m+1}$ as $\Omega_m: = \Gamma_{2m+1}-V(\Gamma_{2m-1})$. This gives $V(\Gamma_{2m+1}) = V(\Gamma_{2m-1})\sqcup V(\Omega_m)$.
		\begin{figure}[H]
			\centering
			\begin{subfigure}{0.55\textwidth}
				\centering
				\begin{tikzpicture}[scale = 0.33, vertices/.style = {draw, fill = black, circle, inner sep = 0.5pt}]
					
					\node[vertices,label = above:{{\tiny $b_1$}}] (b1) at (12, 3) {};
					\node[vertices,label = below:{{\tiny $d_1$}}] (d1) at (12, 0) {};
					
					\node[vertices,label = above:{{\tiny $a_1$}}] (a1) at (13.5, 4.5) {};
					\node[vertices,label = below:{{\tiny $c_1$}}] (c1) at (13.5, 1.5) {};
					\node[vertices,label = below:{{\tiny $e_1$}}] (e1) at (13.5, -1.5) {};
					
					\node[vertices,label = above:{{\tiny $b_2$}}] (b2) at (15, 3) {};
					\node[vertices,label = below:{{\tiny $d_2$}}] (d2) at (15, 0) {};
					
					\node[vertices,label = above:{{\tiny $a_2$}}] (a) at (16.5, 4.5) {};
					\node[vertices,label = below:{{\tiny $c_2$}}] (c) at (16.5, 1.5) {};
					\node[vertices,label = below:{{\tiny $e_2$}}] (e) at (16.5, -1.5) {};
					
					\node[vertices,label = above:{{\tiny $b_3$}}] (b) at (18, 3) {};
					\node[vertices,label = below:{{\tiny $d_3$}}] (d) at (18, 0) {};
					
					\node[vertices, label = {[xshift = -.02cm,yshift = -.11cm]above:{{{\tiny $b_{(2m-2)}$}}}}] (1) at ($(b)+(4,0)$) {};
					\node[vertices, label = {[xshift = -.02cm, yshift = .06cm]below:{{{\tiny $d_{(2m-2)}$}}}}] (2) at ($(d)+(4,0)$) {};
					
					\node[vertices,label = {[xshift = -.08cm,yshift = -.06cm]above:{{{\tiny $a_{(2m-2)}$}}}}] (3) at ($(a)+(7,0)$) {}; 
					\node[vertices,label = {[xshift = .02cm,yshift = .06cm]below:{{{\tiny $c_{(2m-2)}$}}}}] (4) at ($(c)+(7,0)$) {};
					\node[vertices,label = {[xshift = -.08cm,yshift = .03cm]below:{{{\tiny $e_{(2m-2)}$}}}}] (5) at ($(e)+(7,0)$) {};
					
					\node[vertices,label = {[xshift = .03cm,yshift = -.1cm]above:{{{\tiny $b_{(2m-1)}$}}}}] (6) at ($(b)+(7,0)$) {};
					\node[vertices,label = {[xshift = .1cm, yshift = .07cm]below:{{{\tiny $d_{(2m-1)}$}}}}] (7) at ($(d)+(7,0)$) {};
					
					\node[vertices,label = {[xshift = .05cm,yshift = -.06cm]above:{{{\tiny $a_{(2m-1)}$}}}}] (8) at ($(3)+(3,0)$) {}; 
					\node[vertices,label = {[xshift = .1cm,yshift = .06cm]below:{{{\tiny $c_{(2m-1)}$}}}}] (9) at ($(4)+(3,0)$) {};
					\node[vertices,label = {[xshift = .05cm,yshift = .03cm]below:{{{\tiny $e_{(2m-1)}$}}}}] (10) at ($(5)+(3,0)$) {};
					
					\node[vertices,label = {[yshift = -.05cm]above:{{{\tiny $b_{2m}$}}}}] (11) at ($(6)+(3,0)$) {};
					\node[vertices,label = {[yshift = .06cm]below:{{{\tiny $d_{2m}$}}}}] (12) at ($(7)+(3,0)$) {};
					
					\node[vertices,label = {[yshift = -.06cm]above:{{{\tiny $a_{2m}$}}}}] (13) at ($(3)+(6,0)$) {}; 
					\node[vertices,label = {[xshift = .05cm,yshift = .06cm]below:{{{\tiny $c_{2m}$}}}}] (14) at ($(4)+(6,0)$) {};
					\node[vertices,label = {[yshift = .03cm]below:{{{\tiny $e_{2m}$}}}}] (15) at ($(5)+(6,0)$) {};
					
					\node[vertices,label = {[yshift = -.1cm]above:{{{\tiny $b_{(2m+1)}$}}}}] (16) at ($(6)+(6,0)$) {};
					\node[vertices,label = {[yshift = .07cm]below:{{{\tiny $d_{(2m+1)}$}}}}] (17) at ($(7)+(6,0)$) {};
					
					\draw (a) -- ++(2.3,0);
					\draw (c) -- ++(2.3,0);
					\draw (e) -- ++(2.3,0);
					\draw (3) -- ++(-2.3,0);
					\draw (4) -- ++(-2.3,0);
					\draw (5) -- ++(-2.3,0);
					\draw (b) -- ++(1,1);
					\draw (b) -- ++(1,-1);
					\draw (d) -- ++(1,1);
					\draw (d) -- ++(1,-1);
					\draw (1) -- ++(-1,1);
					\draw (1) -- ++(-1,-1);
					\draw (2) -- ++(-1,1);
					\draw (2) -- ++(-1,-1);
					
					\node at ($(b)!0.5!(1)$) {$\ldots$};
					\node at ($(d)!0.5!(2)$) {$\ldots$};
					\node at ($(a)!0.5!(3)$) {$\ldots$};
					\node at ($(c)!0.5!(4)$) {$\ldots$};
					\node at ($(e)!0.5!(5)$) {$\ldots$};
					
					\foreach \vertex/\neighbors in {
						b1/{a1,c1,d1},
						d1/{c1,e1},
						a1/{b2},
						c1/{b2,d2},
						e1/{d2},
						b2/{d2},
						a/{a1,b2,b},
						b/{c,d},
						c/{c1,b2,d2,d},
						d/{e},
						e/{d2,e1},
						1/{2,3,4},
						2/{4,5},
						3/{6},
						4/{6,7},
						5/{7},
						6/{7}}
					{
						\foreach \neighbor in \neighbors
						\draw [-] (\vertex)--(\neighbor);
					}
					
					\foreach \vertex/\neighbors in {
						8/{3,6},
						9/{4,6,7},
						10/{5,7}}
					{
						\foreach \neighbor in \neighbors
						\draw [red] (\vertex)--(\neighbor);
					}
					\foreach \vertex/\neighbors in {
						11/{8,9,12},
						12/{9,10},
						13/{8,11},
						14/{9,11,12},
						15/{10,12},
						16/{13,14},
						17/{14,15,16},}
					{
						\foreach \neighbor in \neighbors
						\draw [blue] (\vertex)--(\neighbor);
					} 
				\end{tikzpicture}
				\subcaption{$\Gamma_{2m+1}$ (for $m\geq 2$)}
				\label{fig:Gamma_2m+1_reverse}
			\end{subfigure}
			\begin{subfigure}{0.40\textwidth}
				\centering
				\begin{tikzpicture}[scale = 0.32, vertices/.style = {draw, fill = black, circle, inner sep = 0.5pt}]
					
					\node[vertices,label = above:{{\tiny $b_1$}}] (b1) at (12, 3) {};
					\node[vertices,label = below:{{\tiny $d_1$}}] (d1) at (12, 0) {};
					
					\node[vertices,label = above:{{\tiny $a_1$}}] (a1) at (13.5, 4.5) {};
					\node[vertices,label = below:{{\tiny $c_1$}}] (c1) at (13.5, 1.5) {};
					\node[vertices,label = below:{{\tiny $e_1$}}] (e1) at (13.5, -1.5) {};
					
					\node[vertices,label = above:{{\tiny $b_2$}}] (b2) at (15, 3) {};
					\node[vertices,label = below:{{\tiny $d_2$}}] (d2) at (15, 0) {};
					
					\node[vertices,label = above:{{\tiny $a_2$}}] (a2) at (16.5, 4.5) {};
					\node[vertices,label = below:{{\tiny $c_2$}}] (c2) at (16.5, 1.5) {};
					\node[vertices,label = below:{{\tiny $e_2$}}] (e2) at (16.5, -1.5) {};
					
					\node[vertices,label = above:{{\tiny $b_3$}}] (b3) at (18, 3) {};
					\node[vertices,label = below:{{\tiny $d_3$}}] (d3) at (18, 0) {};
					
					\node[vertices,label = above:{{\tiny $a_3$}}] (a3) at (19.5, 4.5) {};
					\node[vertices,label = below:{{\tiny $c_3$}}] (c3) at (19.5, 1.5) {};
					\node[vertices,label = below:{{\tiny $e_3$}}] (e3) at (19.5, -1.5) {};
					
					\node[vertices,label = above:{{\tiny $b_4$}}] (b4) at (21, 3) {};
					\node[vertices,label = below:{{\tiny $d_4$}}] (d4) at (21, 0) {};
					
					\node[vertices,label = above:{{\tiny $a_4$}}] (a4) at (22.5, 4.5) {};
					\node[vertices,label = below:{{\tiny $c_4$}}] (c4) at (22.5, 1.5) {};
					\node[vertices,label = below:{{\tiny $e_4$}}] (e4) at (22.5, -1.5) {};
					
					\node[vertices,label = above:{{\tiny $b_5$}}] (b5) at (24, 3) {};
					\node[vertices,label = below:{{\tiny $d_5$}}] (d5) at (24, 0) {};
					
					\foreach \to/\from in
					{a1/a2, a1/b1,a1/b2,a2/b2, b1/d1, b1/c1,d1/e1,c1/d2, d1/c1,b2/c1,c1/c2,d2/e1, b2/c2, c2/d2, b2/d2, d2/e2, e1/e2, e2/d3, c2/d3, c2/d2, a2/b3, b3/c2, b3/d3} \draw [black] (\to)--(\from);
					
					\foreach \to/\from in
					{a3/b4, a3/a4, a4/b4, a4/b5, b4/c4, b4/c3, b4/d4, b5/d5, b5/c4, c3/c4, c3/d4, c4/d4, c4/d5, e3/d4, e3/e4, e4/d4, e4/d5} \draw [blue] (\to)--(\from);
					
					\foreach \to/\from in
					{a3/b3, a2/a3, b3/c3, c3/c2, d3/c3, e3/d3, e3/e2} \draw [red] (\to)--(\from);
					
				\end{tikzpicture}
				\subcaption{$\Gamma_{5}$ (for $m=2$)}
				\label{fig:Gamma_5}
			\end{subfigure}  
			\caption{$\Gamma_{2m+1}$ with subgraphs $\Gamma_{2m-1}$ (black edges) and $\Omega_m$ (blue edges)
			}
			\label{fig:gamma_2m+1}
		\end{figure}
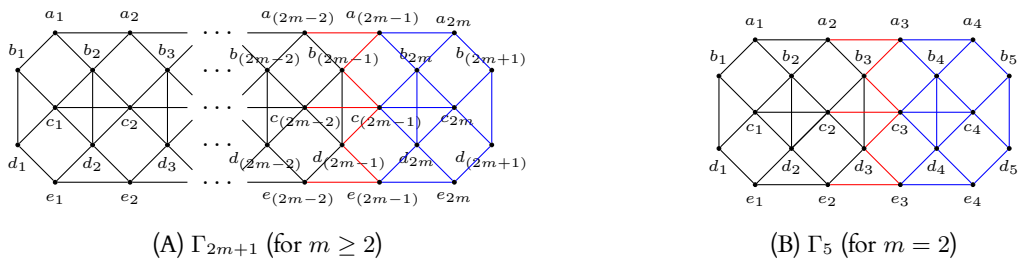
		
		\begin{claim}\label{claim:al_Omega_m}
			$\aG{\Omega_m} = 4$.
		\end{claim}
		
		\begin{proof}[Proof of the \Cref{claim:al_Omega_m}]
			Observe that $\Omega_m$ is isomorphic to an induced subgraph of $\Gamma_3$. Since $\aG{\Gamma_3} = 4$ by \Cref{lemma:dim_sigma}, and $\{a_{(2m-1)},c_{(2m-1)},e_{(2m-1)},b_{(2m+1)}\}$ is an independent set of size $4$ in $\Omega_m$, it follows that $\aG{\Omega_m} = 4$.
		\end{proof}
		
		We consider the base case $m = 2$. Let $\sigma\in\IG{\Gamma_5}$. We show that $\sigma\cap\mathcal{B}_5\ne\emptyset$. From \eqref{set_Bn}, $\mathcal{B}_5 = \{a_1,a_4\}\sqcup\{e_1,e_4\}$ (as $\beta_5 = 4$). Since $V(\Gamma_5) = V(\Gamma_3)\sqcup V(\Omega_2)$, where $\Omega_2 = \Gamma_5-V(\Gamma_3)$, it follows that $\sigma = \sigma_1\sqcup \sigma_2$ for some $\sigma_1\in \Ind(\Gamma_3)$ and some $\sigma_2\in \Ind(\Omega_2)$. 
		By \Cref{lemma:dim_sigma}, $|\sigma| = 7$ and $|\sigma_1|\le 4$. Since $\aG{\Omega_2} = 4$ by \Cref{claim:al_Omega_m}, we have $|\sigma_2|\le 4$. Thus, there are two possible cases:
		\begin{enumerate}[label = (\roman*)]
			\item $|\sigma_1| = 3$ and $|\sigma_2| = 4$.
			
			By \Cref{remark:al_G_and_sigma_union_v_notin_Ind(G)}, $\sigma_2\in\IG{\Omega_2}$ (as $|\sigma_2| = \aG{\Omega_2}$). We have $$\IG{\Omega_2} = \{\{a_3, c_3, e_3, b_5\}, \{a_3, c_3, e_3, d_5\}, \{a_3, c_3, e_4, b_5\}, \{c_3, e_3, a_4, d_5\}\}.$$
			If $\sigma_2\cap\{a_4,e_4\}\ne\emptyset$, then $\sigma_2\subset\sigma$ and $\{a_4,e_4\}\subset\mathcal{B}_5$ imply $\sigma\cap\mathcal{B}_5\ne\emptyset$. 
			
			Now, assume $\sigma_2\cap\{a_4,e_4\} = \emptyset$. Then $\sigma_2\in\{\{a_3, c_3, e_3, b_5\}, \{a_3, c_3, e_3, d_5\}\}.$ Using the fact that $\sigma = \sigma_1\sqcup\sigma_2$ is an independent set in $\Gamma_5$, we have $a_2,c_2,e_2,b_3,d_3\notin\sigma_1$.
			Let $\Omega': = \Gamma_3-\{a_1, e_1,a_2,e_2,c_2, b_3,d_3\}$. Observe that $\aG{\Omega'} = 2$. If $\sigma_1\cap\{a_1,e_1\} = \emptyset$, then $\sigma_1\in \Ind(\Omega')$ (as $\sigma_1\in \Ind(\Gamma_3)$), which yields $|\sigma_1|\le\aG{\Omega'} = 2$, a contradiction. Hence $\sigma_1\cap\{a_1,e_1\}\ne\emptyset$. Since $\sigma_1\subset\sigma$ and $\{a_1,e_1\}\subset\mathcal{B}_5$, we have $\sigma\cap\mathcal{B}_5\ne\emptyset$.
			
			\item $|\sigma_1| = 4$ and $|\sigma_2| = 3$.
			
			Since $\aG{\Gamma_3} = 4$ by \Cref{lemma:dim_sigma}, $|\sigma_1| = \aG{\Gamma_3}$. Hence $\sigma_1\in\IG{\Gamma_3}$. If $\sigma_1\cap\{a_1,e_1\}\ne\emptyset$, then $\sigma\cap\mathcal{B}_5\ne\emptyset$.
			
			Now, assume $\sigma_1\cap\{a_1,e_1\} = \emptyset$. Let $\Omega': = \Gamma_3-\{a_1,e_1\}$. Then $\sigma_1\in\IG{\Omega'}$.
			We have $\IG{\Omega'} = \{\{b_1, a_2, c_2, e_2\}, \{b_1, a_2, d_2, d_3\}, \{b_1, b_2, e_2, b_3\}, \{d_1, a_2, c_2, e_2\}, \{d_1, a_2, d_2, d_3\}, \{d_1, b_2, e_2, b_3\}\}.$
			It is easy to verify that $a_3, c_3, e_3\in N_{\Gamma_5}(\sigma')$ for every $\sigma'\in\IG{\Omega'}$, and hence $a_3, c_3, e_3\in N_{\Gamma_5}(\sigma_1)$. Since $\sigma = \sigma_1\sqcup\sigma_2$ is an independent set in $\Gamma_5$, we get $a_3, c_3, e_3\notin\sigma_2$.
			Let $\Omega'': = \Omega_2-\{a_3, c_3, e_3, a_4, e_4\}$. If $\sigma_2\cap\{a_4,e_4\} = \emptyset$, then $\sigma_2\in \Ind(\Omega'')$. Observe that $\aG{\Omega''} = 2$, which implies $|\sigma_2|\le\aG{\Omega''} = 2$, a contradiction. Hence $\sigma_2\cap\{a_4,e_4\}\ne\emptyset$. Therefore, $\sigma_2\subset\sigma$ and $\{a_4,e_4\}\subset\mathcal{B}_5$ imply $\sigma\cap\mathcal{B}_5\ne\emptyset$.
		\end{enumerate}
		In either case, $\sigma\cap\mathcal{B}_5\ne\emptyset$. Hence the result is true for $m = 2$.
		
		Let $m>2$ and assume that the result holds for all $2\le l\le m-1$. We now prove it for $m$. 
		
		Let $\sigma\in \IG{\Gamma_{2m+1}}$. We show that $\sigma\cap\mathcal{B}_{2m+1}\ne \emptyset$. Note that $\beta_{2m-1}\le\beta_{2m+1}$, which implies $\mathcal{B}_{2m-1}\subseteq\mathcal{B}_{2m+1}$. Since $V(\Gamma_{2m+1}) = V(\Gamma_{2m-1})\sqcup V(\Omega_m)$, we have $\sigma = \sigma_1\sqcup \sigma_2$ for some $\sigma_1\in \Ind(\Gamma_{2m-1})$ and some $\sigma_2\in \Ind(\Omega_m)$. By \Cref{lemma:dim_sigma}, $|\sigma| = 3m+1$ and $|\sigma_1|\le 3m-2$. Since $\aG{\Omega_m} = 4$ by \Cref{claim:al_Omega_m}, $|\sigma_2|\le 4$. There are two possible cases:
		\begin{enumerate}[label = (\roman*)]
			\item $|\sigma_1| = 3m-3$ and $|\sigma_2| = 4$.
			
			Then $\sigma_2\in\IG{\Omega_m}$, where 
			\begin{align*}
				\IG{\Omega_m}& = \{\{a_{2m-1}, c_{2m-1}, e_{2m-1}, b_{2m+1}\}, \{a_{2m-1}, c_{2m-1}, e_{2m-1}, d_{2m+1}\},\\
				&\qquad \{a_{2m-1}, c_{2m-1}, e_{2m}, b_{2m+1}\}, \{c_{2m-1}, e_{2m-1}, a_{2m}, d_{2m+1}\}\}.
			\end{align*}
			Since $\mathcal{B}_{2m+1}$ depends on $2m+1(\text{mod $6$})\in\{1,3,5\}$ (see \eqref{set_Bn}), we consider three subcases:
			\begin{enumerate}
				\item Let $2m+1\equiv 1(\text{mod $6$})$. Then $\beta_{2m+1} = 2m-2$, which implies $$\mathcal{B}_{2m+1} = \{a_1,a_4,a_7,\cdots,a_{2m-5},a_{2m-2}\}\sqcup\{e_1, e_4,e_7,\cdots,e_{2m-5}, e_{2m-2}\}.$$
				
				First, assume $\sigma_1\cap N_{\Gamma_{2m-1}}[b_{2m-1}] = \emptyset$. Then $\sigma_1\sqcup \{b_{2m-1}\}\in \Ind(\Gamma_{2m-1})$ with $|\sigma_1\sqcup \{b_{2m-1}\}| = 3m-2$. Since $\aG{\Gamma_{2m-1}}=3m-2$ by \Cref{lemma:dim_sigma}, it follows that $\sigma_1\sqcup \{b_{2m-1}\}\in\IG{\Gamma_{2m-1}}$. By induction hypothesis, $(\sigma_1\sqcup \{b_{2m-1}\})\cap\mathcal{B}_{2m-1}\ne\emptyset$. Note that $b_{2m-1}\notin \mathcal{B}_{2m-1}$. This yields $\sigma_1\cap\mathcal{B}_{2m-1}\ne \emptyset$. Therefore, $\sigma_1\subset \sigma$ and $\mathcal{B}_{2m-1}\subseteq \mathcal{B}_{2m+1}$ imply $\sigma\cap\mathcal{B}_{2m+1}\ne \emptyset.$
				
				We now assume $\sigma_1\cap N_{\Gamma_{2m-1}}[b_{2m-1}]\ne\emptyset$. Then $N_{\Gamma_{2m-1}}[b_{2m-1}] = \{a_{2m-2},c_{2m-2},b_{2m-1},d_{2m-1}\}$ implies $\sigma_1\cap\{a_{2m-2},c_{2m-2},b_{2m-1},d_{2m-1}\}\ne\emptyset$. Since $c_{2m-1}\in \sigma'$ for every $\sigma'\in\IG{\Omega_m}$, we have $c_{2m-1}\in \sigma_2$ (as $\sigma_2\in\IG{\Omega_m}$). 
				Using the fact that $\sigma= \sigma_1\sqcup\sigma_2$ is an independent set in $\Gamma_{2m+1}$, and $c_{2m-2},b_{2m-1},d_{2m-1}\in N_{\Gamma_{2m+1}}(c_{2m-1})$, it follows that $c_{2m-2},b_{2m-1},d_{2m-1}\notin\sigma_1$. Therefore $a_{2m-2}\in\sigma_1$. Since $a_{2m-2}\in\mathcal{B}_{2m+1}$ and $\sigma_1\subset\sigma$, we obtain $\sigma\cap\mathcal{B}_{2m+1}\ne\emptyset$. 
				
				\item Let $2m+1\equiv 3(\text{mod $6$})$. Then $\beta_{2m+1} = 2m-1$, which implies $$\mathcal{B}_{2m+1} = \{a_1,a_4,a_7,\cdots,a_{2m-4},a_{2m-1}\}\sqcup\{e_1, e_4,e_7,\cdots,e_{2m-4}, e_{2m-1}\}.$$
				Note that for every $\sigma'\in\IG{\Omega_m}$, we have $\sigma'\cap\{a_{2m-1},e_{2m-1}\}\ne \emptyset$. Since $\sigma_2\in\IG{\Omega_m}$, it follows that $\sigma_2\cap\{a_{2m-1},e_{2m-1}\}\ne \emptyset$. Therefore, $\sigma_2\subset\sigma$ and $\{a_{2m-1},e_{2m-1}\}\subseteq\mathcal{B}_{2m+1}$ imply $\sigma\cap\mathcal{B}_{2m+1}\ne \emptyset$. 
				
				\item Let $2m+1\equiv 5(\text{mod $6$})$. Then $\beta_{2m+1} = 2m$, which implies $$\mathcal{B}_{2m+1} = \{a_1,a_4,a_7,\cdots,a_{2m-3},a_{2m}\}\sqcup\{e_1, e_4,e_7,\cdots,e_{2m-3},e_{2m}\}.$$
				
				If $\sigma_2\cap\{a_{2m}, e_{2m}\}\ne\emptyset$, then $\sigma_2\subset\sigma$ and $\{a_{2m}, e_{2m}\}\subset\mathcal{B}_{2m+1}$ imply $\sigma\cap\mathcal{B}_{2m+1}\ne\emptyset$.
				
				Now, assume $\sigma_2\cap\{a_{2m}, e_{2m}\} = \emptyset$. We consider two subcases: 
				\begin{itemize}
					\item Let $\sigma_1\cap N_{\Gamma_{2m-1}}[a_{2m-2}] = \emptyset$ or $\sigma_1\cap N_{\Gamma_{2m-1}}[e_{2m-2}] = \emptyset$. Then $\sigma_1\sqcup \{a_{2m-2}\}\in \Ind(\Gamma_{2m-1})$ with $|\sigma_1\sqcup \{a_{2m-2}\}| = 3m-2$ or $\sigma_1\sqcup \{e_{2m-2}\}\in \Ind(\Gamma_{2m-1})$ with $|\sigma_1\sqcup \{e_{2m-2}\}| = 3m-2$. This implies $\sigma_1\sqcup \{a_{2m-2}\}\in\IG{\Gamma_{2m-1}}$ or $\sigma_1\sqcup \{e_{2m-2}\}\in\IG{\Gamma_{2m-1}}$ (as $\aG{\Gamma_{2m-1}}=3m-2$). By induction hypothesis, $(\sigma_1\sqcup \{a_{2m-2}\})\cap\mathcal{B}_{2m-1}\ne\emptyset$ or $(\sigma_1\sqcup \{e_{2m-2}\})\cap\mathcal{B}_{2m-1}\ne\emptyset$. Since $a_{2m-2},e_{2m-2}\notin \mathcal{B}_{2m-1}, $ we get $\sigma_1\cap\mathcal{B}_{2m-1}\ne \emptyset$. Therefore, $\sigma_1\subset \sigma$ and $\mathcal{B}_{2m-1}\subseteq \mathcal{B}_{2m+1}$ imply $\sigma\cap\mathcal{B}_{2m+1}\ne \emptyset.$
					
					\item Let $\sigma_1\cap N_{\Gamma_{2m-1}}[a_{2m-2}]\ne\emptyset$ and $\sigma_1\cap N_{\Gamma_{2m-1}}[e_{2m-2}] \ne\emptyset$. 
					We have $\sigma_2\in \IG{\Omega_m}$ and $\sigma_2\cap\{a_{2m}, e_{2m}\} = \emptyset$. This implies either $\sigma_2=\{a_{2m-1}, c_{2m-1}, e_{2m-1}, b_{2m+1}\}$ or $\sigma_2=\{a_{2m-1}, c_{2m-1}, e_{2m-1}, d_{2m+1}\}$. Note that $a_{2m-2},e_{2m-2},b_{2m-1},d_{2m-1}\in N_{\Gamma_{2m+1}}(\sigma_2)$. Since $\sigma = \sigma_1\sqcup\sigma_2$ is an independent set in $\Gamma_{2m+1}$, it follows that $a_{2m-2},e_{2m-2},b_{2m-1},d_{2m-1}\notin \sigma_1$. Therefore, $N_{\Gamma_{2m-1}}[a_{2m-2}] = \{a_{2m-3},a_{2m-2},b_{2m-2},b_{2m-1}\}$ and $N_{\Gamma_{2m-1}}[e_{2m-2}] = \{d_{2m-2},d_{2m-1},e_{2m-3},e_{2m-2}\}$ imply $\sigma_1\cap\{a_{2m-3},b_{2m-2}\} \ne\emptyset$ and $\sigma_1\cap\{d_{2m-2},e_{2m-3}\} \ne\emptyset$. 
					If $\sigma_1\cap\{a_{2m-3},e_{2m-3}\} = \emptyset$, then $b_{2m-2},d_{2m-2}\in \sigma_1$, which contradicts $\sigma_1\in \Ind(\Gamma_{2m-1})$ (as $b_{2m-2}\sim d_{2m-2}$ in $\Gamma_{2m-1}$). Hence $\sigma_1\cap\{a_{2m-3},e_{2m-3}\}\ne\emptyset$. Since $\sigma_1\subset\sigma$ and $\{a_{2m-3},e_{2m-3}\}\subset\mathcal{B}_{2m+1}$, we obtain $\sigma\cap\mathcal{B}_{2m+1}\ne\emptyset$.
				\end{itemize}
			\end{enumerate}
			
			\item $|\sigma_1| = 3m-2$ and $|\sigma_2| = 3$.
			
			Then $\sigma_1\in \IG{\Gamma_{2m-1}}$. By induction hypothesis, $\sigma_1\cap\mathcal{B}_{2m-1}\ne \emptyset$. Hence, $\sigma_1\subset \sigma$ and $\mathcal{B}_{2m-1}\subseteq \mathcal{B}_{2m+1}$ imply $\sigma\cap\mathcal{B}_{2m+1}\ne \emptyset.$
		\end{enumerate}
		In all cases, $\sigma\cap\mathcal{B}_{2m+1}\ne \emptyset$. This completes the proof.
	\end{proof}
	
	\begin{proposition} \label{proposition:top_dimension}
		For $n\ge 2$, $\tilde{H}_{\dim(\Ind(\Gamma_n))}(\Ind(\Gamma_n)) = 0$.
	\end{proposition}
	
	\begin{proof}
		For $n = 2, 3$, the result follows from \eqref{equation:Gamma_123}.
		
		Now, let $n\ge4$. Recall that $\mathcal{B}_n$ is an independent ordered set in $\Gamma_n$ (see \eqref{set_Bn}), and $\mathcal{M}_{\mathcal{B}_n} = (\bigsqcup_{i=1}^{\beta_n}\mathcal{M}_{a_i})\sqcup(\bigsqcup_{j=1}^{\beta_n}\mathcal{M}_{e_j})$ (see \eqref{equation:sequence_of_element_matchings}) is an acyclic matching on $\Ind(\Gamma_n)$. By \Cref{lemma:matching_top_dimension_in_G,prop:even+odd}, every top-dimensional face of $\Ind(\Gamma_n)$ is matched in $\mathcal{M}_{\mathcal{B}_n}$. It follows by \Cref{theorem:acyclic} that $\Ind(\Gamma_n)$ is homotopy equivalent to a cell complex of dimension strictly less than $\dim(\Ind(\Gamma_n))$. Thus $\tilde{H}_{\dim(\Ind(\Gamma_n))}(\Ind(\Gamma_n)) = 0$. 
	\end{proof}
	
	\Cref{proposition:homology_up_to_n-2,proposition:homology_n-1,proposition:top_dimension} together prove \Cref{theorem:ind_homology}, and hence \Cref{theorem:intro_homology}. 
	
	\subsubsection{Simply connectedness of $\Ind(\Gamma_n)$} \label{subsection:simply_connected}
	
	To prove \Cref{theorem:intro_simply_connected}, we now show that $\Ind(\Gamma_n)$ is simply connected.
	
	\begin{theorem}\label{theorem:simply_connected}
		For $n\geq 3$, $\Ind(\Gamma_n)$ is simply connected.
	\end{theorem}
	
	\begin{proof}
		The proof is by induction on $n.$ For the base case $n = 3$, we have $\Ind(\Gamma_3)\simeq\vee_5\mathbb{S}^2$ by \eqref{equation:Gamma_123}. Hence, $\Ind(\Gamma_3)$ is simply connected.  Let $n>3$, and assume that $\Ind(\Gamma_l)$ is simply connected for all $3\leq l\leq n-1$. We show that $\Ind(\Gamma_n)$ is simply connected. 
		
		Let $p: \mathbb{S}^1\rightarrow \Ind(\Gamma_n)$ be a closed path. Since $\Ind(\Gamma_n)$ is a simplicial complex, we may assume that $p$ is homotopic to an edge loop $\gamma = p_1, p_2, \ldots, p_k,p_1$, where $p_i\in V(\Ind(\Gamma_n))$ and $\{p_i, p_{i+1}\} \in \Ind(\Gamma_n)$ for all $1 \leq i \leq k$, with indices taken modulo $k$ throughout the proof. It suffices to prove that $\gamma$ is null-homotopic. 
		
		Let $B = \{a_{n-1},b_n,c_{n-1},d_{n},e_{n-1}\}$ and $A = V(\Gamma_n)\setminus B$. Since $\Gamma_n[A]\cong \Gamma_{n-1}$, the induction hypothesis implies that every loop in $\Ind(\Gamma_n[A])$ is null-homotopic.
		
		If $p_i\in A$ for all $1 \leq i \leq k$, then $\gamma$ is a loop in $\Ind(\Gamma_n[A])$, and hence null-homotopic by the induction hypothesis. Otherwise, $p_i\in B$ for some $1 \leq i \leq k$. In this case, we proceed by successively eliminating the vertices of $B$ from $\gamma$ to obtain a loop $\tilde{\gamma}$ in $ \Ind(\Gamma_n[A])$ homotopic to $\gamma$. Let $j:=\min \{i\ | \ p_i\in B\}$. We consider the following two cases. 
		\begin{enumerate}[label=(\roman*)]
			\item First, assume that $\{p_{j-1},p_{j+1}\}\in \Ind(\Gamma_n)$. Since $\{p_{j-1},p_j\},\{p_{j},p_{j+1}\}\in \Ind(\Gamma_n) $, it follows that $\{p_{j-1}, p_j, p_{j+1}\} \in \Ind(\Gamma_n)$. Thus, the loop $p_1,\ldots, p_{j-1}, p_{j+1}, \ldots, p_{k},p_1$ obtained by removing the vertex $p_j$ is homotopic to $\gamma$.
			
			\item We now assume that $\{p_{j-1},p_{j+1}\}\notin \Ind(\Gamma_n)$. Note that $p_j \in B$ and $n \geq 4$. Hence, we can choose a vertex $q\in A$ such that $\{p_{j-1}, q, p_j\} \in \Ind(\Gamma_n)$ and $\{q,p_j, p_{j+1}\} \in \Ind(\Gamma_n)$. Replacing the vertex $p_j$ by $q$, we get the loop $p_1, \ldots, p_{j-1}, q, p_{j+1}, \ldots, p_k,p_1$, which is homotopic to $\gamma$. 
		\end{enumerate}
		
		In either case, we obtain a loop homotopic to $\gamma$ that contains one fewer vertex from $B$ than $\gamma$ does. Repeating this procedure finitely many times, we obtain the required loop $\tilde{\gamma}$ in $\Ind(\Gamma_n[A])$ that is homotopic to $\gamma$. By the induction hypothesis, $\tilde{\gamma}$ is null-homotopic, and hence so is $\gamma$. Thus, $\Ind(\Gamma_n)$ is simply connected.
	\end{proof}

	\section{Conclusion and Future Directions} \label{section:conclusion}
	In \cite{Kozlov_directed_trees} and \cite{Matsushita_grid}, the authors studied the matching complexes $M(G_{1\times n})$ and $M(G_{2\times n})$ of the $1\times n$ and $2\times n$ grid graphs, respectively. They proved that these complexes are either contractible or homotopy equivalent to a wedge of spheres. For $M(G_{3\times n})$, an attempt to determine its homotopy type was made in \cite{Previous_3Xn}, but a gap was later found in the proof, leaving the problem open. In this article, we established the topological connectivity of $M(G_{3\times n})$. To this end, we showed that $\tilde{H}_i(M(G_{3\times n})) = 0$ for $i\leq n-2$ and $\tilde{H}_{n-1}(M(G_{3\times n}))\neq 0$.
	We also proved that $\tilde{H}_{\dim(M(G_{3\times n}))}(M(G_{3\times n})) = 0$.
	
	We investigated $M(G_{3\times n})$ using SageMath by computing its homology groups (with coefficients in $\mathbb{Z}$). For $n\leq 9$, the homology data of $M(G_{3\times n})$ is given in \Cref{table:M_G}. Here, the notation \homo{i}{j} indicates that the $i^{th}$ homology group is $\mathbb{Z}^{j}$. 
	\begin{table}[H]
		\centering
		\begin{tabular}{|c|c|c|c|c|c|c|c|c|c|}
			\hline 
			$n$ & $1$ & $2$ & $3$ & $4$ & $5$ & $6$ & $7$ & $8$ & $9$ \\ \hline $M({G_{3\times n}})$
			& \homo{0}{1}& \homo{1}{2}& \homo{2}{5} & \homo{3}{9} & \homo{4}{16} & \homo{5}{31} & \homo{6}{55} & \homo{7}{94} & \homo{8}{163} \\ \hline
		\end{tabular}
		\caption{Non-trivial homology of $M(G_{3\times n})$ for $n\leq 9$}
		\label{table:M_G}
	\end{table}
	
	The explicit computations in \eqref{equation:Gamma_123} show that $M(G_{3\times n})$ is homotopy equivalent to a wedge of $(n-1)$- dimensional spheres for $n=1,2$. Moreover, for $3\le n\le 9$, it can be seen from \Cref{table:M_G} that non-trivial reduced homology occurs only in dimension $n-1$. Combining this with the fact  that, for $n\geq 3$, $M(G_{3\times n})$ is simply connected, we conclude that $M(G_{3\times n})$ is homotopy equivalent to a wedge of $(n-1)$-dimensional spheres for $3\le n\le 9$.  These observations lead us to the following conjecture.
	\begin{conjecture}\label{conjecture}
		For $n\geq 1$, $M(G_{3\times n})$ is homotopy equivalent to a wedge of $(n-1)$-dimensional spheres.
	\end{conjecture}
	
	To prove \Cref{conjecture}, it is sufficient to prove that $\tilde{H}_{i}(M(G_{3\times n})) = 0$ for all $i \geq n$. For this, we tried an approach similar to the one in \Cref{proposition:homology_up_to_n-2,proposition:homology_n-1}, where the Mayer-Vietoris sequence was applied to the independence complex of $\Gamma_n$ (the line graph of $G_{3\times n}$) and to the independence complexes of auxiliary graphs.
	We find that this method, when applied for $n \leq 8$, yields non-trivial homology for the independence complexes of the auxiliary graphs and for $\Ind(\Gamma_n)$ in exactly the same dimensions as given in \Cref{table:M_G,table:2}, obtained using SageMath. However, in extending this argument to general $n$, the major difficulty was control over the exact Betti numbers of the independence complexes of auxiliary graphs (arising in the induction).
	We believe a refinement of the Mayer-Vietoris approach that provides better control over the Betti numbers may allow us to extend the argument further.
	
	\begin{table}[H]
		\centering
		\begin{tabular}{|*{11}{c|}}\hline
			$\mathbf{n}$ & $1$ & $2$ & $3$ & $4$ & $5$ & $6$ & $7$ \\\hline
			
			$\Ind(\mathcal{U}_n)$ & \homo{0}{1} & \homo{1}{2} & \homo{2}{4} &\homo{3}{6} & \homo{4}{10} & \homo{5}{17} & \homo{6}{26}, \homo{7}{1}\\\hline
			
			$\Ind(\mathcal{V}_n)$ &\homo{1}{1} &\homo{2}{3} & \homo{3}{5} & \homo{4}{10} &\homo{5}{21} & \homo{6}{38}& \homo{7}{69} \\\hline
			
			$\Ind(\mathcal{W}_n)$ & &\homo{2}{1} & \homo{3}{3} & \homo{4}{6} &\homo{5}{14} &\homo{6}{30}& \homo{7}{58} \\\hline
			
			$\Ind(\mathcal{X}_n)$ &\homo{2}{3} & \homo{3}{5} &\homo{4}{12} & \homo{5}{22} & \homo{6}{38} & \homo{7}{72} &\homo{8}{127} \\\hline
			
			$\Ind(\mathcal{Y}_n)$ & & & \homo{4}{1} &\homo{5}{4} & \homo{6}{9} & \homo{7}{20} & \homo{8}{44}\\\hline
			
			$\Ind(\mathcal{Z}_n)$ & \homo{2}{1}& \homo{3}{2} & \homo{4}{3} &\homo{5}{3}, \homo{6}{1} & \homo{6}{4}, \homo{7}{3} & \homo{7}{5}, \homo{8}{8} &\homo{8}{5}, \homo{9}{24} \\\hline
		\end{tabular}
		\caption{Non-trivial homology of independence complexes of the auxiliary graphs for $n\leq 7$ (Blank entries denote that the homology is trivial in all dimensions)}
		\label{table:2}
	\end{table}
	
	\section*{Acknowledgement}
	
	The first and second authors are supported by HTRA fellowship by IIT Mandi, India. The third author is supported by the seed grant project IITM/SG/SMS/95 by IIT Mandi, India.
	
	\bibliographystyle{abbrv}
	\bibliography{ref}

@article{Adamaszek_power_of_cyles,
	AUTHOR = {Adamaszek, Micha\l},
	TITLE = {Splittings of independence complexes and the powers of cycles},
	JOURNAL = {J. Combin. Theory Ser. A},
	FJOURNAL = {Journal of Combinatorial Theory. Series A},
	VOLUME = {119},
	YEAR = {2012},
	NUMBER = {5},
	PAGES = {1031--1047},
	ISSN = {0097-3165,1096-0899},
	MRCLASS = {05E45 (05C76)},
	MRNUMBER = {2891380},
	MRREVIEWER = {Sonja\ \v Cuki\'c},
	DOI = {10.1016/j.jcta.2012.01.009},
	URL = {https://doi.org/10.1016/j.jcta.2012.01.009},
}

@article{Bayer2024Cutcomplex,
	AUTHOR = {Margaret Bayer and Mark Denker and Marija Jeli{\'c} Milutinovi{\'c} and Rowan Rowlands and Sheila Sundaram and Lei Xue},
	TITLE = {Topology of cut complexes of graphs},
	JOURNAL = {SIAM J. Discrete Math.},
	FJOURNAL = {SIAM Journal on Discrete Mathematics},
	VOLUME = {38},
	YEAR = {2024},
	NUMBER = {2},
	PAGES = {1630--1675},
	ISSN = {0895-4801,1095-7146},
	MRCLASS = {57M15 (05C69 05E18 05E45 57Q70)},
	MRNUMBER = {4750271},
	MRREVIEWER = {Anurag\ Singh},
	DOI = {10.1137/23M1569034},
	URL = {https://doi.org/10.1137/23M1569034},
}

@article{Bayer2024TotalCutcomplex,
	AUTHOR = {Margaret Bayer and Mark Denker and Marija Jeli{\'c} Milutinovi{\'c} and Rowan Rowlands and Sheila Sundaram and Lei Xue},
	TITLE = {Total cut complexes of graphs},
	JOURNAL = {Discrete Comput. Geom.},
	FJOURNAL = {Discrete \& Computational Geometry. An International Journal of Mathematics and Computer Science},
	VOLUME = {73},
	YEAR = {2025},
	NUMBER = {2},
	PAGES = {500--527},
	ISSN = {0179-5376,1432-0444},
	MRCLASS = {57M15 (05C69 05E45 57Q70)},
	MRNUMBER = {4865931},
	MRREVIEWER = {Priyavrat\ Deshpande},
	DOI = {10.1007/s00454-024-00630-4},
	URL = {https://doi.org/10.1007/s00454-024-00630-4},
}

@article{bayer2025topology,
	title={Topology of Cut Complexes {II}},
	author={Margaret Bayer and Mark Denker and Marija Jeli{\'c} Milutinovi{\'c} and Sheila Sundaram and Lei Xue},
	JOURNAL = {SIAM J. Discrete Math.},
	FJOURNAL = {SIAM Journal on Discrete Mathematics},
	volume={39},
	number={2},
	pages={1123--1157},
	year={2025},
	publisher={SIAM}
}

@book{bondy1976graph,
	AUTHOR = {John Adrian Bondy and Uppaluri Siva Ramachandra Murty},
	TITLE = {Graph theory},
	SERIES = {Grad. Texts in Math.},
	VOLUME = {244},
	PUBLISHER = {Springer, New York},
	YEAR = {2008},
	PAGES = {xii+651},
	ISBN = {978-1-84628-969-9},
	MRCLASS = {05-01 (05Cxx)},
	MRNUMBER = {2368647},
	MRREVIEWER = {Arthur\ M.\ Hobbs},
	DOI = {10.1007/978-1-84628-970-5},
	URL = {https://doi.org/10.1007/978-1-84628-970-5},
}

@book{hatcher2005algebraic,
	AUTHOR = {Hatcher, Allen},
	TITLE = {Algebraic topology},
	PUBLISHER = {Cambridge University Press, Cambridge},
	YEAR = {2002},
	PAGES = {xii+544},
	ISBN = {0-521-79160-X; 0-521-79540-0},
	MRCLASS = {55-01 (55-00)},
	MRNUMBER = {1867354},
	MRREVIEWER = {Donald\ W.\ Kahn},
}

@book{munkres_book,
	AUTHOR = {James R.\ Munkres},
	TITLE = {Elements of algebraic topology},
	PUBLISHER = {Chapman and Hall/CRC},
	YEAR = {1984},
}

@book{west,
	author = {West, Douglas B.},
	isbn = {0130144002, 9780130144003},
	publisher = {Prentice Hall, Inc., Upper Saddle River, NJ},
	refid = {43903794},
	title = {Introduction to graph theory},
	year = {2001},
}

@book{JonssonBook,
	AUTHOR = {Jonsson, Jakob},
	TITLE = {Simplicial complexes of graphs},
	SERIES = {Lecture Notes in Mathematics},
	VOLUME = {1928},
	PUBLISHER = {Springer-Verlag, Berlin},
	YEAR = {2008},
	PAGES = {xiv+378},
	ISBN = {978-3-540-75858-7},
	MRCLASS = {05C10 (05C20 05C40 05E30 06A11 55U10)},
	MRNUMBER = {2368284},
	DOI = {10.1007/978-3-540-75859-4},
	URL = {https://doi.org/10.1007/978-3-540-75859-4},
}

@book{Kozlov2008,
	AUTHOR = {Kozlov, Dmitry},
	TITLE = {Combinatorial algebraic topology},
	SERIES = {Algorithms Comput. Math.},
	PUBLISHER = {Springer, Berlin},
	YEAR = {2008},
	PAGES = {xx+389},
	ISBN = {978-3-540-71961-8},
	MRCLASS = {55-02 (05C15 05C25 06A07 52B70 55U10 57-02)},
	MRNUMBER = {2361455},
	MRREVIEWER = {Rade\ \v Zivaljevi\'c},
	DOI = {10.1007/978-3-540-71962-5},
	URL = {https://doi.org/10.1007/978-3-540-71962-5},
}

@article{Bjorner1994,
	AUTHOR = {Bj\"{o}rner, A. and Lov\'{a}sz, L. and Vre\'{c}ica, S. T. and
		\v{Z}ivaljevi\'{c}, R. T.},
	TITLE = {Chessboard complexes and matching complexes},
	JOURNAL = {J. London Math. Soc. (2)},
	FJOURNAL = {Journal of the London Mathematical Society. Second Series},
	VOLUME = {49},
	YEAR = {1994},
	NUMBER = {1},
	PAGES = {25--39},
	ISSN = {0024-6107,1469-7750},
	MRCLASS = {52B20 (05C70 13F50 55U10)},
	MRNUMBER = {1253009},
	MRREVIEWER = {Rafael\ H.\ Villarreal},
	DOI = {10.1112/jlms/49.1.25},
	URL = {https://doi.org/10.1112/jlms/49.1.25},
}

@article{Wachs_survey,
	AUTHOR = {Wachs, Michelle L.},
	TITLE = {Topology of matching, chessboard, and general bounded degree
		graph complexes},
	JOURNAL = {Algebra Universalis},
	FJOURNAL = {Algebra Universalis},
	VOLUME = {49},
	YEAR = {2003},
	NUMBER = {4},
	PAGES = {345--385},
	ISSN = {0002-5240,1420-8911},
	MRCLASS = {05E25 (05C10 05E10 55U10)},
	MRNUMBER = {2022345},
	MRREVIEWER = {Edward\ B.\ Swartz},
	DOI = {10.1007/s00012-003-1825-1},
	URL = {https://doi.org/10.1007/s00012-003-1825-1},
}

@article{csorba2009subdivision,
	AUTHOR = {Csorba, P\'eter},
	TITLE = {Subdivision yields {A}lexander duality on independence
		complexes},
	JOURNAL = {Electron. J. Combin.},
	FJOURNAL = {Electronic Journal of Combinatorics},
	VOLUME = {16},
	YEAR = {2009},
	NUMBER = {2},
	PAGES = {Research  Paper 11, 7},
	ISSN = {1077-8926},
	MRCLASS = {55P10 (05C69 05E18 57M15)},
	MRNUMBER = {2515774},
	MRREVIEWER = {Alberto\ Cavicchioli},
	DOI = {10.37236/77},
	URL = {https://doi.org/10.37236/77},
}

@article{chandrakar2024topology,
	AUTHOR = {Chandrakar, Himanshu and Hazra, Nisith Ranjan and Rout,
		Debotosh and Singh, Anurag},
	TITLE = {Topology of total cut complexes and cut complexes of grid
		graphs},
	JOURNAL = {SIAM J. Discrete Math.},
	FJOURNAL = {SIAM Journal on Discrete Mathematics},
	VOLUME = {40},
	YEAR = {2026},
	NUMBER = {2},
	PAGES = {760--788},
	ISSN = {0895-4801,1095-7146},
	MRCLASS = {05C69 (05E45 55P15 57M15)},
	MRNUMBER = {5074457},
	DOI = {10.1137/24M1687716},
	URL = {https://doi.org/10.1137/24M1687716},
}

@article{bouc1992homologie,
	title={Homologie de certains ensembles de 2-sous-groupes des groupes sym{\'e}triques},
	author={Bouc, Serge},
	journal = {J. Algebra},
	fjournal={Journal of Algebra},
	volume={150},
	number={1},
	pages={158--186},
	year={1992},
	publisher={Academic Press}
}

@article{braun2016matching,
	AUTHOR = {Braun, Benjamin and Hough, Wesley K.},
	TITLE = {Matching and independence complexes related to small grids},
	JOURNAL = {Electron. J. Combin.},
	FJOURNAL = {Electronic Journal of Combinatorics},
	VOLUME = {24},
	YEAR = {2017},
	NUMBER = {4},
	PAGES = {Paper No. 4.18, 20},
	ISSN = {1077-8926},
	MRCLASS = {05C69 (05C25 05C70 05C76 05E45)},
	MRNUMBER = {3723557},
	DOI = {10.37236/6212},
	URL = {https://doi.org/10.37236/6212},
}

@book{garst1979cohen,
	title={Cohen-{M}acaulay complexes and group actions.},
	author={Garst, Peter Freedman},
	PUBLISHER = {Ph.D.\ Thesis, University of Wisconsin - Madison},
	YEAR = {1979},
	PAGES = {145},
	MRCLASS = {99-05},
	MRNUMBER = {2628756},
	URL =
	{http://gateway.proquest.com/openurl?url_ver=Z39.88-2004&rft_val_fmt=info:ofi/fmt:kev:mtx:dissertation&res_dat=xri:pqdiss&rft_dat=xri:pqdiss:7924170},
}

@article{bayer2023general,
title={General polygonal line tilings and their matching complexes},
author={Bayer, Margaret and Milutinovi{\'c}, Marija Jeli{\'c} and Vega, Julianne},
journal={Discrete Math.},
volume={346},
number={7},
pages={113428},
year={2023},
publisher={Elsevier}
}

@incollection{bjorner1987some,
AUTHOR = {Bj\"orner, Anders},
TITLE = {Some {C}ohen-{M}acaulay complexes arising in group theory},
BOOKTITLE = {Commutative algebra and combinatorics ({K}yoto, 1985)},
SERIES = {Adv. Stud. Pure Math.},
VOLUME = {11},
PAGES = {13--19},
PUBLISHER = {North-Holland, Amsterdam},
YEAR = {1987},
ISBN = {0-444-70314-4},
MRCLASS = {05B30 (06A10 20D30 20F16 20J05)},
MRNUMBER = {951193},
MRREVIEWER = {Stephen\ D.\ Smith},
DOI = {10.2969/aspm/01110013},
URL = {https://doi.org/10.2969/aspm/01110013},
}

@article{Reiner_Matching,
AUTHOR = {Reiner, Victor and Roberts, Joel},
TITLE = {Minimal resolutions and the homology of matching and
	chessboard complexes},
JOURNAL = {J. Algebraic Combin.},
FJOURNAL = {Journal of Algebraic Combinatorics. An International Journal},
VOLUME = {11},
YEAR = {2000},
NUMBER = {2},
PAGES = {135--154},
ISSN = {0925-9899,1572-9192},
MRCLASS = {13D02 (05C70 05E10 13F50)},
MRNUMBER = {1761911},
DOI = {10.1023/A:1008728115910},
URL = {https://doi.org/10.1023/A:1008728115910},
}

@article{Jiang_chessboard,
AUTHOR = {Jiang, Chengyao and Zhao, Yakun and Wang, Hong and Zhu,
	Guangjun},
TITLE = {The facet ideals of chessboard complexes},
JOURNAL = {Rocky Mountain J. Math.},
FJOURNAL = {The Rocky Mountain Journal of Mathematics},
VOLUME = {53},
YEAR = {2023},
NUMBER = {4},
PAGES = {1155--1175},
ISSN = {0035-7596,1945-3795},
MRCLASS = {13F55 (13C15 13D02)},
MRNUMBER = {4634995},
MRREVIEWER = {Hassan\ Haghighi},
DOI = {10.1216/rmj.2023.53.1155},
URL = {https://doi.org/10.1216/rmj.2023.53.1155},
}

@article{Shareshian,
AUTHOR = {Shareshian, John and Wachs, Michelle L.},
TITLE = {Torsion in the matching complex and chessboard complex},
JOURNAL = {Adv. Math.},
FJOURNAL = {Advances in Mathematics},
VOLUME = {212},
YEAR = {2007},
NUMBER = {2},
PAGES = {525--570},
ISSN = {0001-8708,1090-2082},
MRCLASS = {55U10 (05C70 05E10 05E25)},
MRNUMBER = {2329312},
MRREVIEWER = {Rade\ \v Zivaljevi\'c},
DOI = {10.1016/j.aim.2006.10.014},
URL = {https://doi.org/10.1016/j.aim.2006.10.014},
}

@article{Jojic_Tverberg,
AUTHOR = {Joji\'{c}, Du\{v}sko and Vre\'{c}ica, Sini\v{s}a T. and \v{Z}ivaljevi\'{c}, Rade T.}

@article{Athanasiadis,
AUTHOR = {Athanasiadis, Christos A.},
TITLE = {Decompositions and connectivity of matching and chessboard
complexes},
JOURNAL = {Discrete Comput. Geom.},
FJOURNAL = {Discrete \& Computational Geometry. An International Journal
of Mathematics and Computer Science},
VOLUME = {31},
YEAR = {2004},
NUMBER = {3},
PAGES = {395--403},
ISSN = {0179-5376,1432-0444},
MRCLASS = {05E25 (55U10)},
MRNUMBER = {2036946},
MRREVIEWER = {Hugh\ Ross\ Thomas},
DOI = {10.1007/s00454-003-2869-x},
URL = {https://doi.org/10.1007/s00454-003-2869-x},
}

@article{Karaguezian_Matching,
AUTHOR = {Karaguezian, Dikran B. and Reiner, Victor and Wachs, Michelle
L.},
TITLE = {Matching complexes, bounded degree graph complexes, and weight
spaces of {${\rm GL}_n$}-complexes},
JOURNAL = {J. Algebra},
FJOURNAL = {Journal of Algebra},
VOLUME = {239},
YEAR = {2001},
NUMBER = {1},
PAGES = {77--92},
ISSN = {0021-8693,1090-266X},
MRCLASS = {05E25 (05C99 13D25 20G05)},
MRNUMBER = {1827875},
MRREVIEWER = {David\ A.\ Jorgensen},
DOI = {10.1006/jabr.2000.8654},
URL = {https://doi.org/10.1006/jabr.2000.8654},
}

@article{Zivaljevic_colored_Tverberg,
AUTHOR = { \v{Z}ivaljevi\'{c}, Rade T. and Vre\'{c}ica, Sini\v{s}a T.},
TITLE = {The colored {T}verberg's problem and complexes of injective
functions},
JOURNAL = {J. Combin. Theory Ser. A},
FJOURNAL = {Journal of Combinatorial Theory. Series A},
VOLUME = {61},
YEAR = {1992},
NUMBER = {2},
PAGES = {309--318},
ISSN = {0097-3165,1096-0899},
MRCLASS = {52A35 (05D05)},
MRNUMBER = {1185000},
MRREVIEWER = {Egon\ Schulte},
DOI = {10.1016/0097-3165(92)90028-S},
URL = {https://doi.org/10.1016/0097-3165(92)90028-S},
}

@article{Kozlov_directed_trees,
AUTHOR = {Kozlov, Dmitry N.},
TITLE = {Complexes of directed trees},
JOURNAL = {J. Combin. Theory Ser. A},
FJOURNAL = {Journal of Combinatorial Theory. Series A},
VOLUME = {88},
YEAR = {1999},
NUMBER = {1},
PAGES = {112--122},
ISSN = {0097-3165,1096-0899},
MRCLASS = {05C10 (05C05 05E25)},
MRNUMBER = {1713484},
MRREVIEWER = {Timothy\ Y.\ Chow},
DOI = {10.1006/jcta.1999.2984},
URL = {https://doi.org/10.1006/jcta.1999.2984},
}

@article{Milutinovic_trees,
AUTHOR = {Milutinovi\'c, Marija Jeli\'c{} and Jenne, Helen and
McDonough, Alex and Vega, Julianne},
TITLE = {Matching complexes of trees and applications of the matching
tree algorithm},
JOURNAL = {Ann. Comb.},
FJOURNAL = {Annals of Combinatorics},
VOLUME = {26},
YEAR = {2022},
NUMBER = {4},
PAGES = {1041--1075},
ISSN = {0218-0006,0219-3094},
MRCLASS = {05C70 (05C69 05E45)},
MRNUMBER = {4505054},
MRREVIEWER = {Siamak\ Yassemi},
DOI = {10.1007/s00026-022-00605-3},
URL = {https://doi.org/10.1007/s00026-022-00605-3},
}

@article{Matsushita_grid,
AUTHOR = {Matsushita, Takahiro},
TITLE = {Matching complexes of small grids},
JOURNAL = {Electron. J. Combin.},
FJOURNAL = {Electronic Journal of Combinatorics},
VOLUME = {26},
YEAR = {2019},
NUMBER = {3},
PAGES = {Paper No. 3.1, 8},
ISSN = {1077-8926},
MRCLASS = {05C70 (05C69 05E45 55U10)},
MRNUMBER = {3982310},
MRREVIEWER = {Henry\ Hugh\ Adams},
DOI = {10.37236/8480},
URL = {https://doi.org/10.37236/8480},
}

@article{Marietti_forests,
AUTHOR = {Marietti, Mario and Testa, Damiano},
TITLE = {A uniform approach to complexes arising from forests},
JOURNAL = {Electron. J. Combin.},
FJOURNAL = {Electronic Journal of Combinatorics},
VOLUME = {15},
YEAR = {2008},
NUMBER = {1},
PAGES = {Research Paper 101, 18},
ISSN = {1077-8926},
MRCLASS = {05C10 (05C05 52B70 57Q05)},
MRNUMBER = {2426164},
MRREVIEWER = {Alberto\ Cavicchioli},
DOI = {10.37236/825},
URL = {https://doi.org/10.37236/825},
}

@article{Matsushita_polygonal,
AUTHOR = {Matsushita, Takahiro},
TITLE = {Matching complexes of polygonal line tilings},
JOURNAL = {Hokkaido Math. J.},
FJOURNAL = {Hokkaido Mathematical Journal},
VOLUME = {51},
YEAR = {2022},
NUMBER = {3},
PAGES = {339--359},
ISSN = {0385-4035},
MRCLASS = {55P10 (05C70 52C20)},
MRNUMBER = {4517338},
MRREVIEWER = {Shuchita\ Goyal},
DOI = {10.14492/hokmj/2019-213},
URL = {https://doi.org/10.14492/hokmj/2019-213},
}

@article{bayer_outerplanar,
title={Matching Complexes of Outerplanar Graphs},
author={Margaret Bayer and Marija Jeli\'{c} Milutinovi\'{c} and Julianne Vega},
journal={arXiv preprint arXiv:2411.04601},
year={2024}
}

@article{filakovsky_robust_clique,
title={The Topology of $k$-Robust Clique Complexes in Grid-like Graphs}, 
author={Marek Filakovsk\'{y}},
journal={arXiv preprint arXiv:2602.11365},
year={2026} 
}

@article{Bate_representation,
AUTHOR = {Bate, Michael and Everitt, Brent and Ford, Sam and Ramos,
Eric},
TITLE = {Homology of matching complexes and representations of
symmetric groups},
JOURNAL = {Proc. Amer. Math. Soc.},
FJOURNAL = {Proceedings of the American Mathematical Society},
VOLUME = {154},
YEAR = {2026},
NUMBER = {4},
PAGES = {1467--1478},
ISSN = {0002-9939,1088-6826},
MRCLASS = {55U10 (18A25 20C30)},
MRNUMBER = {5042198},
DOI = {10.1090/proc/17573},
URL = {https://doi.org/10.1090/proc/17573},
}

@article{Categorical_product,
AUTHOR = {Gupta, Raju Kumar and Sarkar, Sourav and Sawant, Sagar S. and
Shukla, Samir},
TITLE = {On the matching complexes of the categorical product of path graphs},
JOURNAL = {J. Appl. Comput. Topol.},
FJOURNAL = {Journal of Applied and Computational Topology},
VOLUME = {10},
YEAR = {2026},
NUMBER = {3},
PAGES = {Paper No. 19, 44},
ISSN = {2367-1726,2367-1734},
MRCLASS = {55P10 (05E45 55U10)},
MRNUMBER = {5102471},
DOI = {10.1007/s41468-026-00238-y},
URL = {https://doi.org/10.1007/s41468-026-00238-y},
}

@article{Jonsson_unpublished,
title={Matching complexes on grids},
author={Jakob Jonsson},
journal={Unpublished manuscript},
year={2005}
}

@article {Previous_3Xn,
AUTHOR = {Goyal, Shuchita and Shukla, Samir and Singh, Anurag},
TITLE = {Matching complexes of {$3 \times n$} grid graphs},
JOURNAL = {Electron. J. Combin.},
FJOURNAL = {Electronic Journal of Combinatorics},
VOLUME = {28},
YEAR = {2021},
NUMBER = {4},
PAGES = {Paper No. 4.16, 26},
ISSN = {1077-8926},
MRCLASS = {05E45 (55P15)},
MRNUMBER = {4394659},
MRREVIEWER = {Oana\ Stefania\ Olteanu},
DOI = {10.37236/10496},
URL = {https://doi.org/10.37236/10496},
NOTE = {(This article has been withdrawn by the authors)}
}

@article {Matsushita_grid_independence_1,
AUTHOR = {Matsushita, Takahiro and Wakatsuki, Shun},
TITLE = {Independence complexes of {$(n \times 4)$} and {$(n \times
	5)$}-grid graphs},
	JOURNAL = {Topology Appl.},
	FJOURNAL = {Topology and its Applications},
	VOLUME = {334},
	YEAR = {2023},
	PAGES = {Paper No. 108541, 18},
	ISSN = {0166-8641,1879-3207},
	MRCLASS = {55P10 (05C69)},
	MRNUMBER = {4587515},
	MRREVIEWER = {Bennet\ Goeckner},
	DOI = {10.1016/j.topol.2023.108541},
	URL = {https://doi.org/10.1016/j.topol.2023.108541},
}

@article {Matsushita_grid_independence_2,
AUTHOR = {Matsushita, Takahiro and Wakatsuki, Shun},
TITLE = {Independence complexes of {$(n\times 6)$}-grid graphs},
JOURNAL = {Homology Homotopy Appl.},
FJOURNAL = {Homology, Homotopy and Applications},
VOLUME = {26},
YEAR = {2024},
NUMBER = {1},
PAGES = {15--27},
ISSN = {1532-0073,1532-0081},
MRCLASS = {55P10 (05C69)},
MRNUMBER = {4700575},
MRREVIEWER = {Dae-Woong\ Lee},
DOI = {10.4310/hha.2024.v26.n1.a2},
URL = {https://doi.org/10.4310/hha.2024.v26.n1.a2},
}

@article{adamaszek2012hardsquarescylindersrevisited,
title={Hard squares on cylinders revisited}, 
author={Michal Adamaszek},
journal={arXiv preprint arXiv:1202.1655},
year={2012} 
}

@article{engstrom2009complexes,
title={Complexes of directed trees and independence complexes},
author={Engstr{\"o}m, Alexander},
journal={Discrete Math.},
volume={309},
number={10},
pages={3299--3309},
year={2009},
publisher={Elsevier}
}
	
	\addcontentsline{toc}{section}{References}
	
\end{document}